\documentclass[12pt,leqno,twoside]{amsproc}
\usepackage{amsfonts}
\usepackage{amsmath}
\usepackage{amssymb}
\usepackage{amsthm}
\usepackage{mathtools}
\usepackage{epsf}
\usepackage{graphics}
\usepackage{graphicx}
\usepackage{latexsym}
\usepackage{enumitem}
\usepackage{psfrag}
\usepackage{tikz}
\usepackage{verbatim}
\usepackage{hyperref}
\usetikzlibrary{positioning}
\usetikzlibrary{arrows}
\usetikzlibrary{decorations.pathreplacing}

\newtheorem{thm}{Theorem}[section]
\newtheorem{lem}{Lemma}[section]
\newtheorem{cor}[thm]{Corollary}
\newtheorem{definition}{Definition}[section]

\newtheorem{rem}{Remark}[section]

\newtheorem{ass}{Assumption}[section]

\def\g{\gamma}
\def\G{\Gamma}
\def\d{\delta}
\def\e{\epsilon}

\def\R{{\mathbb R}}
\def\C{{\mathbb C}}
\def\E{{\mathbb E}}
\def\EE{{\mathcal E}}
\def\LL{{\mathcal L}}
\def\OO{{\mathcal O}}

\def\supp{{\text{Supp}}}

\title[Gelfand-Tsetlin polytopes and random contractions]{Gelfand-Tsetlin polytopes and random contractions
away from the limiting shape.}

\author{Beno\^\i t Collins}
\address{Beno\^\i t Collins, Kyoto University, Mathematics Department, Kyoto, Japan}
\email{collins@math.kyoto-u.ac.jp}

\author{Anthony Metcalfe}
\address{Anthony Metcalfe, Hagavagen 16, 16969, Solna, Sweden}
\email{metcalf@kth.se}

\date{\today}

\begin{document}

\begin{abstract}
In this paper, we consider a sequence of selfadjoint matrices $A_n$ having a limiting spectral distribution as $n\to \infty$,
and we consider a sequence of  full flags $\{0\le p_1^n\le\ldots\le p_i^n\le\ldots\le 1_n\}$
chosen at random according to the uniform measure on full flag manifolds.
We are interested in the behaviour of the extremal eigenvalues of $p_i^nA_np_i^n$.
This problem is known to be equivalent to the study of uniform probability measures on Gelfand-Tsetlin polytopes.
Our main results consist in explicit uniform estimates for extremal eigenvalues, and the fact that an outlier behavior has
an exponentially small probability. 
This problem is of intrinsic interest in random matrix theory, but it was motivated from a problem in 
Quantum Information Theory, which we discuss. 
The proofs rely on a reinterpretation of the problem with the help of determinantal point processes and the techniques are based 
on steepest descent analysis.
\end{abstract}

\maketitle

\section{Introduction}

\subsection{Two facets of the same problem}

A (weak) Gelfand-Tsetlin pattern is an $n$-tuple,
$(y^{(1)},y^{(2)},\ldots,y^{(n)}) \in  \R \times \R^2 \times \cdots \times \R^n$,
which satisfies the constraints
\begin{equation*}
y_1^{(r+1)} \; \ge \; y_1^{(r)} \; \ge \; y_2^{(r+1)} \; \ge \; y_2^{(r)}
\; \ge \cdots \ge \; y_r^{(r)} \; \ge \; y_{r+1}^{(r+1)},
\end{equation*}
for all $r \in \{1,\ldots,n-1\}$. We refer to subsection \ref{sectdsogtp} for precise definitions and properties. 
The study of this subset of $\R^{n(n+1)/2}$ is very natural and has led to many deep results. 
For example, if $y^{(n)}$ is fixed, the collection of weak Gelfand-Tsetlin patterns form a polytope, and the study 
of the uniform probability measure on it is the object of many research results. 
We refer for example to \cite{Duse15a,Duse16,Duse17} and references therein. 

For the above uniform measure and under some assumptions on $n,j$ and $y^{(n)}$ to be specified subsequently, 
 it is known that some regions of $\R$ are highly unlikely to have elements $y_i^{(j)}$. 
 While the description of these
 zones is well understood, quantifying the un-likelihood remained to be studied and it is one purpose of this paper to
 provide answers to this problem. 

Let us now turn to the following random matrix problem.
For selfadjoint matrices $A_n$, we consider a sequence of  full flags 
$\{0\le p_1^n\le\ldots\le p_i^n\le\ldots\le 1_n\}$.
Recall that a full flag is a maximal sequence of (selfadjoint) projections whose images are increasing for the inclusion order. 
In particular, in our setup, $rk p_i^n=i, Im p_i^n\subset Im p_{i+1}^n$. 
The collection of full flags is a compact subset of $n$-tuples of matrices, on which unitary matrices act transitively by global
conjugation, therefore there exists a unique invariant probability measure on full flags (another description of the measure 
would be the collection of projections onto the $i$ first columns of a Haar distributed unitary matrix). 
We consider a random maximal flag according to this measure 
and we are interested in the joint set of eigenvalues of $p_i^nA_np_i^n$.
It is well-known (\cite{Bar01}) 
that this yields a Gelfand-Tsetlin pattern, and its distribution is the uniform measure discussed above. 

In this paper, we actually focus on the behaviour of the extremal eigenvalues of $p_i^nA_np_i^n$. 
This unexpected connection allows to exploit properties from both facets to derive analytic estimates. For example, the fact
that the uniform measure can be seen as the push forward of a measure on the unitary group implies some Gaussian
concentration for each $y_i^{(j)}$ (see e.g. the book \cite{MR2760897}) 
typically, there exists constants $C,c$ such that for any $n$ and for any $\varepsilon >0$,
\begin{equation}\label{from-concentration}
P(|y_i^{(j)}-E(y_i^{(j)})|\ge \varepsilon)\le C\exp (-nc\varepsilon^2)
\end{equation}
Such estimates are far from obvious from the study of uniform measure in polytopes in general 
(see for example partial results in the special case of random polytopes \cite{Vu})
and they hint at the fact that the Gelfand-Tsetlin polytope has an exceptional behaviour.

\subsection{Motivations from Quantum Information theory}

Quantum Information Theory 
-- often abbreviated by QIT in this paper -- 
questions the information theoretic possibilities and limitations of using quantum protocols, e.g. 
quantum measurements and quantum channels. It has made very important progress in the last decades with a need
for ever increasingly involved mathematics. In particular, random techniques have proven to be very useful for solving important
problems, such as the problem of additivity of the Minimum Output Entropy. 

Let us recall here briefly this problem. For further details, we refer to \cite{CoNe}.
A \emph{Quantum Channel} $\Phi$ is a map  $M_n(\C)\to M_k(\C)$ that 
is linear, preserves the trace, and such that for any $l$, 
$$\Phi\otimes Id_l: M_n(\C)\otimes M_l(\C)\to M_k(\C)\otimes M_l(\C)$$
takes a positive matrix to a positive matrix (the map $\Phi$ is said to be \emph{completely positive}). 
A \emph{density matrix} is a positive matrix of trace $1$, and for $\rho$ a density matrix, its von Neumann entropy is
$H(\rho )= -\sum \lambda_i (\rho )\log (\lambda_i (\rho ))$, where $\lambda_1(\rho )\ge\lambda_2(\rho)\ge\ldots$ are the
eigenvalues of $\rho$.
Here, the entropy function $x\log x: (0 , 1)\to\mathbb{R}_{-}$ is extended by continuity to $[0,1]$ and takes value $0$ at $0$ and
$1$.
The \emph{Miminum Output Entropy} (aka MOE) of a quantum channel $\Phi$ is 
$$H_{min}(\Phi)=\min_{\rho \,\, density \,\, matrix} H(\Phi (\rho )),$$
and the problem of additivity asks whether it is true, for any $\Phi_1,\Phi_2$ quantum channels, 
$$H_{min}(\Phi_1\otimes \Phi_2)=H_{min}(\Phi_1)+H_{min}(\Phi_2).$$ 
The importance of the question relies in the fact that
a systematic equality implies the additivity of the classical capacity of quantum channels (i.e. the amount of classical 
information that can be sent through quantum channels is additive). 
This result has been proved to be false, i.e. there exist quantum channels $\Phi_1,\Phi_2$ such that
$H_{min}(\Phi_1\otimes \Phi_2)< H_{min}(\Phi_1)+H_{min}(\Phi_2)$, see \cite{Ha09} and \cite{MR2443305}
for important preliminary results. 
However all constructions
so far rely on the probabilistic method, i.e. on finding adequate sequences of random channels that
 satisfy the strict inequality with high probability. 
No non-random example is known at this point. Actually, it is very difficult, if not impossible, 
to estimate the size of matrices involved in creating a counterexample. 
While some strategies \cite{Ha09, FuKiMo10, AuSzWe11, BrHo10} might in principle yield 
dimensions that can actually be described numerically, they yield extremely small violations.
On the other hand, the strategy known to yield the best violation \cite{BeCoNe12,BeCoNe16}, while giving an optimal estimate on 
the output (iff more than 183), makes it even more difficult to estimate the required dimension for the input. 

Let us now outline why this dimension estimate 
is difficult. The results of  \cite{BeCoNe12,BeCoNe16} rely on the fact that the largest eigenvalue of 
random matrix models converge almost surely. 
Typically, the matrix models involved are as follows:
$$p (A\otimes 1_n) p$$
where $A\in M_k$ is selfadjoint deterministic and $p$ is a random uniform projection in $M_k\otimes M_n$ of rank approximately
$tkn$ (for a fixed $t\in (0,1]$).
The spectrum of such an operator has been known since Voiculescu to converge almost surely to the free contraction of 
the spectral distribution of $A$ by the relative dimension of $p$. The operator norm of this object is called $||A||_t$.
Since this part is just a motivation, but not essential to the main results,
we refer to  \cite{BeCoNe12,BeCoNe16} for a thorough introduction and detail.
In the core of this paper, we will not use the notation $||A||_t$ and rather study the Gelfand Tsetlin cone globally, 
so here, to link the topics, we will just note that  
\begin{equation}\label{soft-def-free-compressed}
||A||_t=\sup \{x, (x,t)\in \LL\},
\end{equation}
provided that the eigenvalues of $A$ correspond to the top eigenvalues of the Gelfand Tsetlin cone. 
In the above equation, for the definition of $\LL$, we refer to definition \ref{defLiq} in the body of the manuscript. 
The papers of \cite{BeCoNe12,CoMa} are the first ones that prove that the largest eigenvalue converges almost surely to $||A||_t$.
However, nothing is known about the speed of convergence, except in the notable case where $A$ itself is a projection,
\cite{Co05} but the techniques at hand in \cite{BeCoNe12,CoMa} do not allow us to quantify the speed of convergence.

On the other hand, the set of eigenvalues of $p (A\otimes 1_n) p$ is known to be a determinantal point process.
Such a determinantal point process is actually a particular case of a more general determinantal point process on 
Gelfand-Tsetlin patterns, as per Defosseux' results \cite{De10}.
For us, this potential of applications to mathematical physics was a compelling motivation to undertake
in this paper a systematic study of the top elements in the Gelfand Tsetlin cone. 

The large dimension limit study of this determinantal point process has been initiated by the second author 
and his coauthors, with very 
fine asymptotic results inside the spectrum and at the boundary \cite{Duse15a,Duse16,Duse17,Met13}.
In this respect, our paper is a continuation of the aforementioned papers.
The main result is Theorem \ref{thmdecay}. 
Since it is quite technical, instead of stating it here, let us mention that it means that uniformly, 
$$P(|y_1^{(j)}-f(j/n))|\ge \varepsilon)\le C\exp (-nh(\varepsilon))$$ 
for some constants $C$ and a strictly increasing function $h:\mathbb{R}_+\to \mathbb{R}_+$ satisfying $h(0)=0$.
For a precise statement, we refer to Theorem \ref{thmdecay}.

A seemingly technical, yet necessary contribution of our work is to replace
$E(y_1^{(j)})$ 
that appears for example in Equation \eqref{from-concentration} 
by an explicit $f(j/n)$. Specifically, under reasonable assumptions, 
such as in the case of subsection \ref{secAtoms} (i.e. the case that motivated us in QIT, and 
that satisfies all technical assumptions of Theorem \ref{thmdecay}),
one rules out the possibility of a tame fluctuation with 
respect to the mean or median, but misbehaved  with respect to a limiting quantity -- for example due to the mean or 
median converging too slowly towards a limit. As of today, all these computations are possible only thanks to the determinantal
structures and the algebra and steepest descent analysis behind. 
Note also that in principle, our results allow us to systematically compute $h(\varepsilon )$.
For the sake of keeping things within a reasonable pages number, we do not discuss this question systematically in this manuscript. 
This question will be discussed with completely different methods, highly specific to the largest eigenvalue, in a future work of F. Parraud.

To close this introduction, we would like to make the following remark: Many advanced analytical techniques have proved to be 
very useful towards solving problems in quantum information theory. 
This includes notably random matrix theory, but also large deviation theory, free probability theory, 
large dimensional convex analysis. We hope that this paper will also serve as an invitation to consider 
saddle point methods, determinantal point process, and possibly Riemann Hilbert techniques 
as possible additional mathematical techniques in the toolbox that can be used in quantum information theory

Acknowledgements: BC. was supported by JSPS KAKENHI 17K18734 and 17H04823.
He would like to thank Ion Nechita for  inspiring discussions related to this project, especially during a summer school held in Autrans in 2016. 
Both authors are indebted to Neil O'Connell and Kurt Johannson
for discussions at a preliminary stage of this project. AM would also like to thank
Maurice Duits for stimulating discussions.
Last but not least, this is a very long paper, which has been reviewed  thoroughly by a very careful referee in a  short amount of 
time. Our communication
with the referee resulted in  substantial improvements of the paper, and we would like to express our gratitude to
the referee for her/his important contribution towards improving our paper.

\subsection{Aims, structure and assumptions}
\label{secAssAndTerm}

In this section, to mitigate reader confusion, we give a brief description of the aims and
structure of the paper, and the assumptions to be used in each section. The mathematical objects
referenced in this section will be defined where appropriate.

The main result of this paper, theorem \ref{thmdecay}, concerns explicit uniform estimates for extremal
eigenvalues. The eigenvalues form Gelfand-Tsetlin patterns of particles, and we prove an exponentially
small probability for the local asymptotic behaviour of the relevant eigenvalues/particles. We use
steepest descent analysis to obtain the explicit bounds. This is highly technical: First we must
understand the global asymptotic behaviour. We then use this understanding to identify
the global asymptotic region which contains the extremal eigenvalues. Finally, we perform a steepest
descent analysis within that region to understand the local asymptotic behaviour.

Steepest descent
analysis is powerful but involved by nature, and by far the most complex part of such an
analysis is proving the existence of appropriate contours of steepest descent/ascent.
Additionally, steepest descent authors are normally only interested in proving convergence,
and not in the explicit bounds we obtain in theorem \ref{thmdecay}. These bounds, essential for
our intended purpose, necessitates that we find exact contours of descent/ascent
(see definition \ref{defDesAsc}), a problem greatly more complex than simply proving existence.
We must also prove explicit bounds at each step of the calculation.
The length of this paper reflects these unavoidable technical obstacles.

The assumptions at work throughout much of the paper are, in fact, weaker than those
ultimately used in our main result, theorem \ref{thmdecay}. 
Although the stronger assumptions of theorem \ref{thmdecay} are sufficient for our
intended application to QIT, many of the steepest
descent related results hold in greater generality, 
and we try to be as general as possible where we can in
the hope that the steepest descent related results will trigger subsequent interest in random matrix
theory. 

Section \ref{sec:preliminaries} contains the mathematical preliminaries of the steepest descent problem,
and a statement of the main result, theorem \ref{thmdecay}. We give the minimum amount of information
in order to state theorem \ref{thmdecay} without ambiguity, but leave the definitions of some quantities
and statements of some results until later in the paper where they can more naturally be introduced.
Section \ref{sec:enop} applies theorem \ref{thmdecay} to an example relevant to QIT.

Section \ref{sec:main} examines the global asymptotic behaviour of the Gelfand-Tsetlin patterns.
The assumptions here are quite broad, and stated at the beginning of the section, to be of greater
interest to the random matrix theory community. We identify the {\em liquid region}, $\LL$.
Metcalfe, \cite{Met13}, proved universal bulk asymptotic behaviour in $\LL$ using steepest
descent analysis. We next identity the {\em edge}, $\EE$, a natural subset of $\partial \LL$
where steepest descent analysis suggests universal edge asymptotic behaviour, and perhaps other novel
universal asymptotic behaviours (see remarks \ref{remOtherUniversalKernels} and \ref{remNonAsyEdge}).
This is beyond the scope of this paper. Finally, we identify $\OO$, the region in which we perform
the steepest descent analysis in this paper. We then restrict our scope with the additional assumption
that $\mu[\{b\}] > 0$ (see lemma \ref{lemLowRigEdge} and section \ref{secLowRig}), since it is sufficient
for our applications to QIT, and since it allows us to obtain a simpler
global description of $\OO$ (see figure \ref{figU}) and to simplify the setup of the steepest
descent analysis.

Section \ref{secsdatak} examines the local asymptotic behaviour around a fixed point
$(\chi,\eta) \in \OO$ using steepest descent techniques. The assumptions used are stated
clearly at the beginning of the section. As stated above we assume that $\mu[\{b\}] > 0$, and
we assume that $u_n,r_n,v_n,s_n$ are defined as in equation (\ref{equnrnvnsn2}).
Note, theorem \ref{thmdecay} additionally assumes that $r_n = s_n$, which trivially gives
$\phi_{r_n, s_n} (u_n,v_n) = 0$. This additional assumption is not otherwise
used, and all asymptotic results of section \ref{secsdatak} hold without it. We use this condition
as it is sufficient for our applications to QIT, and it avoids an involved asymptotic analysis of
$\phi_{r_n, s_n}(u_n,v_n)$ for general $r_n$ and $s_n$. Nevertheless, we conjecture that
the steepest descent techniques of section \ref{secsdatak} are sufficient to examine the
asymptotic behaviour of $\phi_{r_n, s_n} (u_n,v_n)$ for general $r_n$ and $s_n$, and
that theorem \ref{thmdecay} holds in this case also. The assumption $\mu[\{b\}] > 0$ can
also be weakened. Indeed, whenever $\mu[\{b\}] = 0$ and $(\chi,\eta) \in \OO$, $f_{(\chi,\eta)}'$
has either $2$ distinct roots of multiplicity $1$ in $(b,+\infty)$ and no other roots in $(b,+\infty)$,
or simply $1$ distinct root of multiplicity $1$ in $(b,+\infty)$ and no other roots in $(b,+\infty)$.
The geometric interpretation of $\OO$ is therefore more complex than that shown in
figure \ref{figU}. Regardless, the asymptotic techniques in section \ref{secsdatak} prove that
theorem \ref{thmdecay} holds whenever $2$ distinct roots of multiplicity $1$ exist. The other case
would require a more detailed analysis, and is beyond the scope of this paper.

Section \ref{sectrofn'} is a necessary technical examination of the behaviour of the roots of the
relevant steepest descent functions. The assumptions needed are again weaker than
those in theorem \ref{thmdecay}. Indeed, only assumptions about the behaviour of the asymptotic measure,
$\mu$, are required, and these are stated clearly at the beginning of the section. Note, it
is necessary to understand the behaviour of the roots in their entirety, and not just
in the interval where the steepest descent analysis is carried out, as theorem \ref{thmf'} employs
a subtle counting argument. We exhaust all possible root behaviours, some of
which are not directly related to the asymptotic situations in this paper, for the sake of
completeness and general interest in the random matrix theory community.

We finish with Section \ref{sec:application} that contains applications to random geometry and QIT.

\section{Mathematical preliminaries}
\label{sec:preliminaries}

\subsection{The determinantal structure of Gelfand-Tsetlin patterns}
\label{sectdsogtp}

A Gelfand-Tsetlin pattern of depth $n$ is an $n$-tuple,
$(y^{(1)},y^{(2)},\ldots,y^{(n)}) \in  \R \times \R^2 \times \cdots \times \R^n$,
which satisfies the interlacing constraint
\begin{equation*}
y_1^{(r+1)} \; \ge \; y_1^{(r)} \; > \; y_2^{(r+1)} \; \ge \; y_2^{(r)}
\; > \cdots \ge \; y_r^{(r)} \; > \; y_{r+1}^{(r+1)},
\end{equation*}
for all $r \in \{1,2,\ldots,n-1\}$, denoted $y^{(r+1)} \succ y^{(r)}$. Equivalently,
this can be considered as an interlaced configuration of $\frac12 n (n+1)$
particles in $\R \times \{1,2,\ldots,n\}$ by placing a particle at position
$(u,r) \in \R \times \{1,2,\ldots,n\}$ whenever $u$ is an element of $y^{(r)}$.
An example of such a configuration is given in figure \ref{figGT}. Note,
the particles obtained from $y^{(r)}$, for all $r \in \{1,2,\ldots,n\}$, are
referred to as the particles on {\em row $r$} of the interlaced configuration. 

\begin{figure}[h]
\centering
\begin{tikzpicture}

\draw (0,0) node {$y_4^{(4)}$};
\draw (2,0) node {$y_3^{(4)}$};
\draw (4,0) node {$y_2^{(4)}$};
\draw (6,0) node {$y_1^{(4)}$};
\draw (1,-1.5) node {$y_3^{(3)}$};
\draw (3,-1.5) node {$y_2^{(3)}$};
\draw (5,-1.5) node {$y_1^{(3)}$};
\draw (2,-3) node {$y_2^{(2)}$};
\draw (4,-3) node {$y_1^{(2)}$};
\draw (3,-4.5) node {$y_1^{(1)}$};

\draw (.45,-.75) node [rotate=-55] {$<$};
\draw (1.45,-.75) node [rotate=55] {$\le$};
\draw (2.45,-.75) node [rotate=-55] {$<$};
\draw (3.45,-.75) node [rotate=55] {$\le$};
\draw (4.45,-.75) node [rotate=-55] {$<$};
\draw (5.45,-.75) node [rotate=55] {$\le$};
\draw (1.45,-2.25) node [rotate=-55] {$<$};
\draw (2.45,-2.25) node [rotate=55] {$\le$};
\draw (3.45,-2.25) node [rotate=-55] {$<$};
\draw (4.45,-2.25) node [rotate=55] {$\le$};
\draw (2.45,-3.75) node [rotate=-55] {$<$};
\draw (3.45,-3.75) node [rotate=55] {$\le$};

\draw (8,0) node {row $4$};
\draw (8,-1.5) node {row $3$};
\draw (8,-3) node {row $2$};
\draw (8,-4.5) node {row $1$};

\end{tikzpicture}
\caption{A visualisation of a Gelfand-Testlin pattern of depth $4$.}
\label{figGT}
\end{figure}
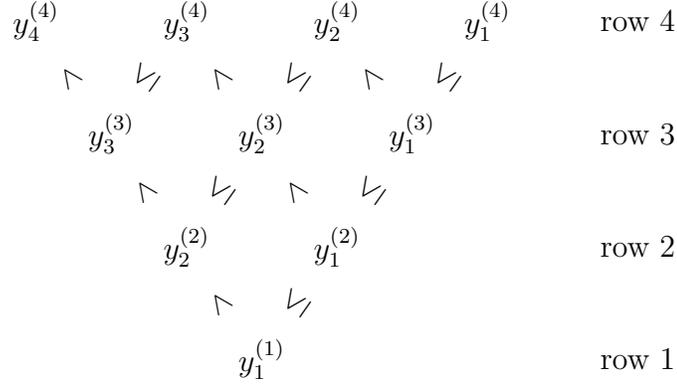

For each $n\ge1$, fix $x^{(n)} \in \R^n$ with $x_1^{(n)} > x_2^{(n)} > \cdots > x_n^{(n)}$.
Let $\Omega_n$ represent the set of Gelfand-Tsetlin patterns of depth $n$ with
the particles on row $n$ in the deterministic positions defined by $x^{(n)}$, and let
$\nu_n$ represent the uniform probability measure on $\Omega_n$:
\begin{equation*}
d \nu_n [y^{(1)},\ldots,y^{(n)}] =
\frac1{Z_n} \cdot \begin{cases}
\delta_{x^{(n)}} (y^{(n)}) dy^{(n)} dy^{(n-1)} \ldots dy^{(1)} & ;
\text{if} \; y^{(n)} \succ y^{(n-1)} \succ \cdots \succ y^{(1)}, \\
0 & ; \mbox{otherwise},
\end{cases}
\end{equation*}
where $Z_n$ is a normalisation constant.
Let $E_n := \R \times \{1,2,\ldots,n\}$ and $N : = \frac12 n(n+1)$, and recall
the above equivalence of $\Omega_n$ as a set of configurations of $N$ particles
in $E_n$. $(\Omega_n, \nu_n)$ is therefore equivalent to a probability
space on configurations of $N$ particles in $E_n$. Such probability spaces are
commonly referred to as {\em random point fields}. Baryshnikov, \cite{Bar01},
showed that this field arises naturally as an {\em eigenvalue minor
process}: $y^{(n)} = x^{(n)}$ are the fixed eigenvalues of a random Hermitian matrix
of size $n$ with unitarily invariant distribution, and $y^{(r)}$ are the random
eigenvalues of the principal minor of size $r$ for all $r \in \{1,2,\ldots,n-1\}$
(consisting of the first $r$ rows and columns).

The above random point field was studied in Metcalfe, \cite{Met13}, and we now
recall some important properties. First, for each $m \leq N$ define a measure,
$\mathbb{M}_m$, on $E_n^m$ by:
\begin{equation*}
\mathbb{M}_m[B] := \E \left[ \sum_{1 \le i_1 \ne i_2 \ne \cdots \ne i_m \le N}
1_{\{\omega \in \Omega_n : (\omega_{i_1},\omega_{i_2},\ldots,\omega_{i_m}) \in B\}} \right],
\end{equation*}
for any Borel subset $B \subset E_n^m$, where the expectation is with respect to $\nu_n$.
Note, $\mathbb{M}_m[B]$ is the expected number of $m$-tuples of particles from $\Omega_n$
that are contained in $B$. In particular note that, when $m=1$ and
$B = A \times \{r\}$ for any Borel $A \subset \R$ and $r \in \{1,2,\ldots,n\}$,
$\mathbb{M}_1[B] = \mathbb{M}_1[A \times \{r\}]$ is the expected number
of particles on row $r$ that are contained in $A$.

In \cite{Met13} it is shown, for all $m \in \{1,2,\ldots,n\}$ and Borel
subsets $B \subset E_n^m$, that
\begin{equation*}
\mathbb{M}_m [B] = \int_B \det[K_n((u_i,r_i),(u_j,r_j))]_{i,j=1}^m d\lambda^m[(u,r)],
\end{equation*}
for some function $K_n : E_n^2 \to \C$, where $\lambda$ is the
direct product of Lebesgue measure (on $\R$) with counting measure
(on $\{1,2,\ldots,n\}$). In words, the Radon-Nikodym derivative of
$\mathbb{M}_m$ with respect to the reference measure $\lambda^m$ exists,
and is given by a determinant of a function of pairs of particle
positions. Such random point fields are called {\em determinantal},
and the function $K_n : E_n^2 \to \C$ is called the
{\em correlation kernel}. In particular note that, when $m=1$ and
$B = A \times \{r\}$ for any Borel $A \subset \R$ and $r \in \{1,2,\ldots,n\}$,
\begin{equation*}
\mathbb{M}_1[B] = \mathbb{M}_1[A \times \{r\}]
= \int_A K_n((u,r),(u,r)) du,
\end{equation*}
where integration is with respect to Lebesgue measure. Therefore,
the expected number of particles on row $r$ is a measure on $\R$
which is absolutely continuous with respect to Lebesgue measure,
with density given by $u \mapsto K_n((u,r),(u,r))$ for all $u \in \R$.

\begin{ass}
\label{assWeakConv}
Let $\mu$ be a probability measure on $\R$ with compact support,
$\supp(\mu) \subset [a,b]$ with $b > a$ and $\{a,b\} \subset \supp(\mu)$.
Assume,
\begin{equation*}
\frac1n \sum_{i=1}^n \delta_{x_i^{(n)}} \to \mu \text{ weakly}.
\end{equation*}
\end{ass}

Then, rescaling the Gelfand-Tsetlin patterns vertically by $\frac1n$,
the bulk of the rescaled particles asymptotically lie in $[a,b] \times [0,1]$
as $n \to \infty$. Indeed, as we shall see, the asymptotic bulk lies in a
natural open subset of $[a,b] \times [0,1]$, which we denote below by $\LL$.
We provide global descriptions of $\LL$ which arise naturally from steepest
descent considerations (see theorem \ref{thmBulk}), some examples of which
are given in figure \ref{figAtoms}.
The local asymptotic behaviour of particles near a fixed point,
$(\chi,\eta) \in [a,b] \times [0,1]$, is studied by considering
$K_n((u_n,r_n),(v_n,s_n))$ as $n \to \infty$,
where $\{(u_n,r_n)\}_{n\ge1} \subset \R \times \{1,2,\ldots,n-1\}$ and
$\{(v_n,s_n)\}_{n\ge1} \subset \R \times \{1,2,\ldots,n-1\}$ satisfy:
\begin{equation}
\label{equnrnvnsn0}
(u_n,\tfrac{r_n}n) = (\chi,\eta) + o(1)
\hspace{.5cm} \text{and} \hspace{.5cm}
(v_n,\tfrac{s_n}n) = (\chi,\eta) + o(1)
\hspace{.5cm} \text{as} \hspace{.5cm}
n \to \infty.
\end{equation}
The main result of this paper, theorem \ref{thmdecay}, can then be stated at a high level as follows:
\begin{thm}
\label{thmHighLevel}
Assume $\mu[\{b\}] > 0$. Then there exists an open subset ($\OO$) in the lower right
corner of $[a,b] \times [0,1]$, which lies outside the asymptotic bulk ($\LL$). A
global description of $\OO$ arises naturally from steepest descent considerations.
Moreover, the following is satisfied: Assume that
$(u_n,\frac{r_n}n)$ and $(v_n,\frac{s_n}n)$ are contained in neighbourhoods of
$\OO$ whose description also arise naturally from steepest decent considerations. Take
$r_n = s_n$ for all $n$ so that the particles are on the same level of the Gelfand-Tsetlin
patterns. Then $K_n((u_n,r_n),(v_n,s_n))$ decays exponentially as $n \to \infty$. Moreover,
we obtain explicit bounds on the rates of decay, and explicit conditionals
which describe how big we should take $n$ (definition \ref{defxiN} and lemma \ref{lemN2}).
\end{thm} 

With the above result, and in particular the explicit bounds and description of $n$,
we aimed to find explicit exponentially decaying bounds for the expected number
of particles in subsets of $\OO$, and over sets of the asymptotic measure $\mu$.
However the bounds we obtained were very
complex, and this goal proved to be intractable. Instead, we consider an example calculation
in section \ref{sec:enop} where $\mu := \tfrac14 \delta_1 + \tfrac34 \delta_{-1}$, and we take
specific choices of $x^{(n)}, (v_n,s_n), (u_n,r_n)$. The bulk $\LL$,
and $\OO$, for this example are depicted in figure \ref{fig2Atoms}.
We obtain corollary \ref{corAtoms}, which can be stated at a high level as follows:
\begin{cor}
\label{corHighLevel}
Take $\mu, x^{(n)}, (v_n,s_n), (u_n,r_n)$ as described in the previous paragraph. Fix $l \ge 2$.
Then $[.5,.99] \times \{(1 - \tfrac1l) \tfrac14\}$ is a subset of $\OO$,
is the horizontal line depicted in figure \ref{fig2Atoms}, and the closest vertical
distance between this and the asymptotic bulk is $1/4l$. Moreover, we can find $C>0$,
and independent integers $N,L$ for which the following is satisfied for all
$l \ge L$ and $n \ge N$:
\begin{equation*}
\mathbb{M}_1[[.5,.99] \times \{n \eta\}]
< C l \; \exp( - n \tfrac{5}{12 \sqrt{6}} \; (\tfrac1{l})^\frac32).
\end{equation*}
\end{cor}

As stated above, we ultimately wished to find explicit values for $C,N,L$.
However, while explicit values can in principle be obtained from theorem \ref{thmdecay},
and the conditionals in definition \ref{defxiN} and lemma \ref{lemN2}, providing
these proved to be impractical giving the already considerable length of 
this paper.

Let us now continue with the analysis.
Metcalfe, \cite{Met13}, gives the following expression for $K_n$:
\begin{equation}
\label{eqKnrusvFixTopLine}
K_n((u,r),(v,s)) = \widetilde{K}_n ((u,r),(v,s)) - \phi_{r,s}(u,v),
\end{equation}
for all $(u,r),(v,s) \in \R \times \{1,2,\ldots,n-1\}$, where
\begin{equation*}
\widetilde{K}_n ((u,r),(v,s)) = \sum_{j=1}^n 1_{(x_j^{(n)} > u)}
\frac{(x_j^{(n)} - u)^{n-r-1}}{(n-r-1)!} \frac{\partial^{n-s}}{\partial v^{n-s}}
\prod_{i \neq j} \left( \frac{v - x_i^{(n)}}{x_j^{(n)} - x_i^{(n)}} \right),
\end{equation*}
and
\begin{equation*}
\phi_{r,s} (u,v)
:= 1_{(v>u)} \cdot \left\{
\begin{array}{rcl}
0 & ; & \text{when} \; s \le r, \\
1 & ; & \text{when} \; s = r+1, \\
(v-u)^{s-r-1}/(s-r-1)! & ; & \text{when} \; s > r+1.
\end{array}
\right.
\end{equation*}
Next take the particle positions as in equation (\ref{equnrnvnsn0}) for some fixed
$(\chi,\eta) \in [a,b] \times [0,1]$. Note, whenever $r_n \in \{1,2,\ldots,n-2\}$,
equation (\ref{eqKnrusvFixTopLine}) and the Residue Theorem give,
\begin{equation}
\label{eqKnrnunsnvn1}
K_n((u_n,r_n),(v_n,s_n))
= \frac{(n-s_n)!}{(n-r_n-1)!} \; J_n - \phi_{r_n,s_n}(u_n,v_n),
\end{equation}
where
\begin{equation}
\label{eqJn}
J_n := \frac1{(2\pi i)^2} \int_{c_n} dw \int_{C_n} dz \;
\frac1{w-z} \frac{(z - u_n)^{n-r_n-1}}{(w - v_n)^{n-s_n+1}}
\prod_{i=1}^n \left( \frac{w - x_i^{(n)}}{z - x_i^{(n)}} \right),
\end{equation}
where $c_n$ and $C_n$ are any counter-clockwise simple closed contours
which satisfy the following: $C_n$ contains $\{x_j^{(n)} : x_j^{(n)} > u_n\}$
and does not contain any of $\{x_j^{(n)} : x_j^{(n)} < u_n\}$, and $c_n$
contains $v_n$ and $C_n$. Also note that for all $w,z \in \C \setminus \R$
the integrand can be written as:
\begin{equation}
\label{eqKnConInt2}
\frac{\exp(n f_n(w) - n \tilde{f}_n(z))}{w-z},
\end{equation}
where
\begin{align}
\label{eqfn}
f_n(w)
& := \frac1n \sum_{i=1}^n \log (w - x_i^{(n)}) - \frac{n-s_n+1}n \log(w - v_n), \\
\label{eqtildefn}
\tilde{f}_n(z)
& := \frac1n \sum_{i=1}^n \log (z - x_i^{(n)}) - \frac{n-r_n-1}n \log(z - u_n),
\end{align}
and log is the principal branch of the logarithm. Inspired by these, and by
assumption \ref{assWeakConv} and equation (\ref{equnrnvnsn0}), define:
\begin{equation}
\label{eqf}
f_{(\chi,\eta)}(w) := \int_a^b \log (w-x) \mu[dx] - (1-\eta) \log (w-\chi),
\end{equation}
for all $w \in \C \setminus \R$.

\begin{rem}
Let us point out that the above series of equations also appear frequently in free probability theory in the context of
calculations on R-transforms and S-transforms. For a closely related example, we refer to
\cite{BeCoNe12}.
\end{rem}

Steepest descent analysis, and the above structure, intuitively suggests that the
behaviour of $K_n((u_n,r_n),(v_n,s_n))$ as $n \to \infty$ depends on the roots of
$f_{(\chi,\eta)}'$. In order to discuss this, we must first identify the largest possible
domain of analytic extensions of $f_{(\chi,\eta)}'$.
Recall that $\supp(\mu) \subset [a,b]$ with $b > a$ and
$\{a,b\} \subset \supp(\mu)$, and $b \ge \chi \ge a$. Thus, for all $w \in \C \setminus \R$,
it is natural to write,
\begin{equation}
\label{eqf2}
f_{(\chi,\eta)}(w) 
= \int_{(\chi,b]} \log (w-x) \mu[dx]
- (1-\eta - \mu[\{\chi\}]) \log(w - \chi)
+ \int_{[a,\chi)} \log (w-x) \mu[dx].
\end{equation}
Also note, for all $w \in \C \setminus \R$, equation (\ref{eqf}) gives,
\begin{equation}
\label{eqf'0}
f_{(\chi,\eta)}'(w) = C(w) - \frac{1-\eta}{w - \chi},
\end{equation}
where $C : \C \setminus \supp(\mu) \to \C$ denotes the {\em Cauchy}
transform of $\mu$:
\begin{equation}
\label{eqCauTrans}
C(w) := \int_a^b \frac{\mu[dx]}{w-x},
\end{equation}
for all $w \in \C \setminus \supp(\mu)$. Note that the above
expression of $f_{(\chi,\eta)}'$ has a unique analytic extension
to $\C \setminus (\supp(\mu) \cup \{\chi\})$. Alternatively, for
all $w \in \C \setminus \R$, equation (\ref{eqf2}) gives, 
\begin{equation}
\label{eqf'}
f_{(\chi,\eta)}'(w)
= \int_{(\chi,b]} \frac{\mu[dx]}{w-x}
- \frac{1-\eta - \mu[\{\chi\}]}{w-\chi}
+ \int_{[a,\chi)} \frac{\mu[dx]}{w-x}.
\end{equation}
Finally note that the above expression has a unique analytic extension to the
(possibly) larger set $\C \setminus (S_1 \cup S_2 \cup S_3)$,
where $S_i := S_i(\chi,\eta)$ for all $i \in \{1,2,3\}$ are defined by: 
\begin{equation*}
S_1 := \supp(\mu |_{(\chi,b]}),
\hspace{0.5cm}
S_2 := \left\{ \begin{array}{rcl}
\{\chi\} & ; & \text{when } \mu[\{\chi\}] \neq 1-\eta, \\
\emptyset & ; & \text{when } \mu[\{\chi\}] = 1-\eta,
\end{array} \right.
\hspace{0.5cm}
S_3 := \supp(\mu |_{[a,\chi)}).
\end{equation*}
Note $S_1 = \emptyset$ when $b = \chi$, and $S_1 \neq \emptyset$ when $b > \chi$ (since $b \in \supp(\mu)$).
Similarly $S_3 = \emptyset$ when $\chi = a$, and $S_3 \neq \emptyset$ when $\chi > a$ (since $a \in \supp(\mu)$).
Theorem \ref{thmf'} characterises all possible behaviours of the roots of
$f_{(\chi,\eta)}'$ in this domain. We now identify regions of $[a,b] \times [0,1]$ which
can be defined by particular behaviours of the roots relevant to steepest descent considerations.

\begin{definition}
\label{defLiq}
The liquid region, $\LL$, is the set of all $(\chi,\eta) \in [a,b] \times [0,1]$
for which $f_{(\chi,\eta)}'$ has a root in $\mathbb{H} := \{ w \in \C : \text{Im}(w) > 0 \}$.
\end{definition}
Theorem \ref{thmf'} implies that $(\chi,\eta) \in \LL$ if and only if $f_{(\chi,\eta)}'$
has exactly $1$ root in $\mathbb{H}$, counting multiplicities. Steepest descent analysis
then suggests, and Metcalfe \cite{Met13} confirmed, universal bulk asymptotic behaviour
whenever $(\chi,\eta) \in \LL$:
Fixing $(\chi,\eta) \in \LL$, and choosing the parameters $(u_n,r_n)$ and $(v_n,s_n)$ of
equation (\ref{equnrnvnsn0}) appropriately, $K_n((u_n,r_n),(v_n,s_n))$
converges to the {\em Sine} kernel as $n \to \infty$. The Sine kernel is a so-called
{\em universal} kernel as it has been observed asymptotically in the spectrum of other
ensembles of random matrices and in related systems (see, for example, \cite{Erdos2010a},
\cite{Joh01}, \cite{Pas97}). 

Note, it is clear from the above observations that there is a natural map,
\begin{equation*}
(\chi_\LL(\cdot),\eta_\LL(\cdot)): \mathbb{H} \to \LL.
\end{equation*}
Theorem \ref{thmBulk} obtains an explicit expression for this and shows it is a
homeomorphism, and so $\LL$ is open.
Lemma \ref{lemBddEdge} examines $\partial \LL$. Part (2) of that lemma
shows that $(\chi_\LL(\cdot),\eta_\LL(\cdot)) : \mathbb{H} \to \LL$ has
a unique continuous extension to the following open subset of $\R$:
\begin{equation}
\label{eqR1}
R := (\R \setminus \supp(\mu)) \cup R_1,
\end{equation}
where $R_1$ is the set of all {\em isolated atoms} of $\mu$ (see equation
(\ref{eqR2})). We also obtain an explicit
expression for this extension (see equation (\ref{eqEdge})), denoted
\begin{equation*}
(\chi_\EE(\cdot),\eta_\EE(\cdot)) : R \to \partial \LL \subset [a,b] \times [0,1].
\end{equation*}
We now define:
\begin{definition}
\label{defEdge0}
The edge, $\EE \subset \partial \LL$, is the image of the curve $(\chi_\EE(\cdot),\eta_\EE(\cdot)) : R \to \partial \LL \subset [a,b] \times [0,1]$.
The curve itself is called the edge curve.
\end{definition}

Theorem \ref{thmEdge} give an alternative definition of $\EE$ which is analogous to
that of $\LL$. Recall that $C : \C \setminus \supp(\mu) \to \C$ denotes the {\em Cauchy}
transform of $\mu$ (see equation (\ref{eqCauTrans})) and note that $R$ is
given by the disjoint union,
\begin{equation}
\label{eqR2}
R := R^+ \cup R^- \cup R_0 \cup R_1,
\end{equation}
where:
\begin{itemize}
\item
$R^+ := \{ t \in \R \setminus \supp(\mu) : C(t) > 0 \}$.
\item
$R^- := \{ t \in \R \setminus \supp(\mu) : C(t) < 0 \}$.
\item
$R_0 := \{ t \in \R \setminus \supp(\mu) : C(t) = 0 \}$.
\item
$R_1 := \{ t \in \supp(\mu) : \mu[\{t\}] > 0
\text{ and there exists an open interval }
I \subset \R \text{ with } t \in I \text{ and }
I \setminus \{t\} \subset \R \setminus \supp(\mu) \}$.
\end{itemize}
Lemma \ref{lemBddEdge} then gives the following for
$(\chi_\EE(\cdot),\eta_\EE(\cdot)) : R \to \EE \subset \partial \LL \subset [a,b] \times [0,1]$:
\begin{align}
\label{eqEdge}
\chi_\EE(t) = t + \frac{C(t)}{C'(t)}
&\text{ and }
\eta_\EE(t) = 1 + \frac{C(t)^2}{C'(t)}
\text{ when } t \in R^+ \cup R^- \cup R_0 = \R \setminus \supp(\mu), \\
\nonumber
\chi_\EE(t) = t
&\text{ and }
\eta_\EE(t) = 1
\text{ when } t \in R_0, \\
\nonumber
\chi_\EE(t) = t
&\text{ and }
\eta_\EE(t) = 1 - \mu[\{t\}]
\text{ when } t \in R_1.
\end{align}
The above mapping is continuous in any open sub-interval of $R$. Next, define:
\begin{definition}
\label{defEdge}
The edge, $\EE$, is the disjoint union $\EE = \EE^+ \cup \EE^- \cup \EE_0 \cup \EE_1$ where:
\begin{itemize}
\item
$\EE^+$ is the set of all $(\chi,\eta) \in [a,b] \times [0,1]$
for which $f_{(\chi,\eta)}'$ has a
repeated root in $(\chi,+\infty) \setminus \supp(\mu)$.
\item
$\EE^-$ is the set of all $(\chi,\eta) \in [a,b] \times [0,1]$
for which $f_{(\chi,\eta)}'$ has a
repeated root in $(-\infty,\chi) \setminus \supp(\mu)$.
\item
$\EE_0 := \{(\chi,\eta) \in [a,b] \times [0,1] : \chi \in R_0 \text{ and } \eta = 1 \}$.
\item
$\EE_1 := \{(\chi,\eta) \in [a,b] \times [0,1] : \chi \in R_1 \text{ and } \eta = 1 - \mu[\{\chi\}] \}$.
\end{itemize}
\end{definition}
In section \ref{secEdge}, we show that definitions \ref{defEdge0} and
\ref{defEdge} are equivalent: First, starting with definition \ref{defEdge},
part (4) of corollary \ref{corf'} shows that $\{\LL, \EE^+, \EE^-, \EE_0, \EE_1\}$
are pairwise disjoint. Moreover, part (1) of corollary \ref{corf'} shows that
$f_{(\chi,\eta)}'$ has a unique real-valued repeated root in $\R \setminus \{\chi\}$
when $(\chi,\eta) \in \EE^+ \cup \EE^-$. Next, map each
$(\chi,\eta) \in \EE^+ \cup \EE^-$ to the unique real-valued repeated root
in $\R \setminus \{\chi\}$, and map each $(\chi,\eta) \in \EE_0 \cup \EE_1$
to $\chi$. Then, theorem \ref{thmEdge} implies that this bijectively
maps $\EE$ to $R$, and the inverse of this map is the edge curve of definition
\ref{defEdge0}. Therefore the definitions are trivially equivalent, and equation
(\ref{eqEdge}) is a convenient parameterisation of the edge which is continuous in
any open sub-interval of $R$, with the relevant root as parameter.

In section \ref{secEdge}, we also examine the geometric behaviour of the edge curve.
First, fix $(\chi,\eta) \in \EE$ and the corresponding $t \in R$. Next, let
$m = m(t)$ denote the multiplicity of $t$ as a root of $f_{(\chi,\eta)}'$.
Then, lemma \ref{lemLocGeo} proves:
\begin{itemize}
\item
The edge curve behaves like a parabola in neighbourhoods of $(\chi,\eta)$
when $t \in R^+ \cup R^-$ and $(\chi,\eta) \in \EE^+ \cup \EE^-$ and $m = 2$,
when $t \in R_0$ and $(\chi,\eta) \in \EE_0$ and $m = 1$, and when $t \in R_1$
and $(\chi,\eta) \in \EE_1$ and $m = 0$.
\item
The edge curve behaves like an algebraic
cusp of first order in neighbourhoods of $(\chi,\eta)$ when $t \in R^+ \cup R^-$ and
$(\chi,\eta) \in \EE^+ \cup \EE^-$ and $m = 3$, and when $t \in R_1$ and
$(\chi,\eta) \in \EE_1$ and $m = 1$.
\end{itemize}
For clarity we state that the above exhaust
all possibilities, and $m=0$ means that $f_{(\chi,\eta)}'(t) \neq 0$.

\begin{rem}
\label{remOtherUniversalKernels}
Steepest descent analysis, and the above root behaviour, suggests universal
edge asymptotic behaviour whenever $(\chi,\eta) \in \EE^+ \cup \EE^-$:
\begin{itemize}
\item
Fixing $(\chi,\eta) \in \EE^+ \cup \EE^-$, and choosing the parameters $(u_n,r_n)$ and
$(v_n,s_n)$ appropriately, we conjecture that $K_n((u_n,r_n),(v_n,s_n))$
converges to the Airy kernel as $n \to \infty$ when $m=2$.
\item
Fixing $(\chi,\eta) \in \EE^+ \cup \EE^-$, and choosing the parameters $(u_n,r_n)$ and
$(v_n,s_n)$ appropriately, we conjecture that $K_n((u_n,r_n),(v_n,s_n))$
converges to the Pearcey kernel as $n \to \infty$ when $m=3$.
\end{itemize}
Analogous behaviour when $m=2$ is observed in Duse and
Metcalfe, \cite{Duse17}, for discrete interlaced Gelfand-Tsetlin patterns.
Steepest descent analysis also suggests possible novel universal
behaviours in the following situations:
\begin{itemize}
\item
$(\chi,\eta) \in \EE_0$ and $m = 0$.
\item
$(\chi,\eta) \in \EE_1$ and $m = 0$.
\item
$(\chi,\eta) \in \EE_1$ and $m = 1$.
\end{itemize}
See Duse and Johansson and Metcalfe, \cite{Duse16}, for an analysis of
edge points of discrete interlaced Gelfand-Tsetlin patterns which
produced a previously unobserved universal kernel, which we called the
Cusp-Airy kernel. A further discussion on these points is beyond
the scope of this paper.
\end{rem}

\begin{figure}
\centering
\mbox{\includegraphics{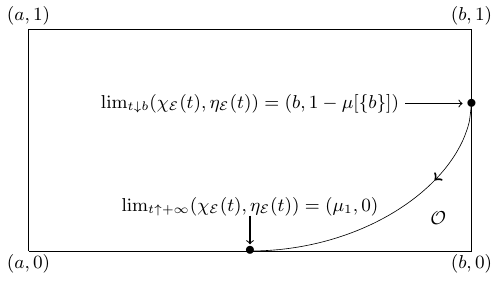}}
\caption{$(\chi_\EE(\cdot),\eta_\EE(\cdot)) : (b,+\infty) \to \EE$
when $\mu[\{b\}] > 0$. The arrow represents the direction of the increasing parameter,
and $\mu_1 := \int_a^b x \mu[dx]$.}
\label{figU}
\end{figure}

Next define:
\begin{definition}
\label{defLowRig}
$\OO$ is the set of all $(\chi,\eta) \in [a,b] \times [0,1]$
for which $\chi < b$, $\eta > 0$, and $f_{(\chi,\eta)}'$ has a
root of multiplicity $1$ in $(b,+\infty)$.
\end{definition}
Corollary \ref{corf'} implies that $\{\LL, \EE, \OO\}$ are pairwise disjoint.

The main result of this paper, theorem \ref{thmdecay}, examines the
the local asymptotic behaviour in $\OO$ using steepest descent techniques.
The additional assumption that $\mu[\{b\}] > 0$ is used as
it is sufficient for our application to QIT, and greatly simplifies
the analysis. First, under this additional assumption, lemma \ref{lemLowRigEdge} shows
that $(\chi_\EE(\cdot),\eta_\EE(\cdot)) : (b, +\infty) \to \EE$ always
behaves as in figure \ref{figU}. Next note, a simpler description of $\OO$ exists when $\mu[\{b\}] > 0$
(see definition \ref{defLowRig2}): $\OO$ is the set of all $(\chi,\eta) \in (a,b) \times (0,1)$
for which $1 - \eta > \mu[\{\chi\}]$, $f_{(\chi,\eta)}'$ has $2$ distinct
roots of multiplicity $1$ in $(b,+\infty)$, and $f_{(\chi,\eta)}'$
has no other roots in $(b,+\infty)$. This defines a map from $\OO$ to
\begin{equation}
\label{eqangle}
\angle := \{(t,s) \in (b,+\infty)^2 : t > s\}.
\end{equation}
Theorem \ref{thmLowRig}
shows that this map is a homeomorphism, and finds an explicit expression for the
inverse of the homeomorphism, denoted,
\begin{equation*}
(\chi_\OO(\cdot),\eta_\OO(\cdot)) : \angle \to \OO.
\end{equation*}
Finally, lemma \ref{lemLowRig} gives the following
simple geometric interpretation of $\OO$ in this case: $\OO$ is that open subset
of $(a,b) \times (0,1)$ in figure \ref{figU} bounded by
$(\chi_\EE(\cdot),\eta_\EE(\cdot)) |_{(b, +\infty)}$ and the bounding box of
$[a,b] \times [0,1]$. Steepest descent analysis, and the above root behaviour, suggest universal
asymptotic behaviour whenever $\mu[\{b\}] > 0$ and $(\chi,\eta) \in \OO$.
Indeed, the correlation kernel should decay exponentially as $n \to \infty$.
Theorem \ref{thmdecay} confirms this intuition.

\subsection{$\LL$, $\EE$ and $\OO$ when $\mu$ is atomic}
\label{secAtoms}

In this section we restrict to the case of purely atomic measures to illustrate
the global geometric behaviours of $\LL$, $\EE$ and $\OO$ as discussed in the
previous section. Suppose,
\begin{equation*}
\mu = \sum_{i=1}^k \alpha_i \delta_{b_i},
\end{equation*}
for some $k \ge 2$, $b_1, b_2, \ldots, b_k \in \R$ with
$b = b_1 > b_2 > \cdots > b_k = a$, and $\alpha_1, \alpha_2, \ldots, \alpha_k > 0$
with $\alpha_1 + \alpha_2 + \cdots + \alpha_k = 1$. In this case, equation
(\ref{eqR1}) easily gives $R = \R$. Definition \ref{defEdge} then implies
that the edge curve is a map
$(\chi_\EE(\cdot),\eta_\EE(\cdot)) : \R \to \partial \LL \subset [a,b] \times [0,1]$,
and $\EE$ is the image space of this map. Moreover, theorem \ref{thmEdge} implies that
this map is bijective. Also, definition \ref{defEdge} and
lemma \ref{lemBddEdge} imply that $\partial \LL = (\int_a^b x \mu[dx],0) \cup \EE$.
Finally, equations (\ref{eqCauTrans}, \ref{eqEdge}) imply
that,
\begin{equation}
\label{eqEdgekAtoms}
\chi_\EE(t) = t - \frac{\sum_{i=1}^k \alpha_i (t-b_i) \prod_{j \neq i} (t-b_j)^2}
{\sum_{i=1}^k \alpha_i \prod_{j \neq i} (t-b_j)^2},
\hspace{0.5cm}
\eta_\EE(t) = 1 - \frac{(\sum_{i=1}^k \alpha_i \prod_{j \neq i} (t-b_j))^2}
{\sum_{i=1}^k \alpha_i \prod_{j \neq i} (t-b_j)^2},
\end{equation}
for all $t \in \R$. In particular, note this gives
$(\chi_\EE(b_l),\eta_\EE(b_l)) = (b_l, 1 - \alpha_l) = (b_l, 1 - \mu[\{b_l\}])$
for all atoms $b_l \in \{b_1, b_2, \ldots, b_k\}$.

Suppose $k = 2$, $b = b_1 = 1$, $a = b_k = b_2 = -1$, and
$\mu = \alpha \delta_1 + (1-\alpha) \delta_{-1}$ for some
$\alpha \in (0,1)$. Equation (\ref{eqEdgekAtoms}) then gives,
\begin{align}
\label{eqEdge2Atoms}
\chi_\EE(t) &= t - \frac{\alpha (t-1) (t+1)^2 + (1-\alpha) (t+1) (t-1)^2}
{\alpha (t+1)^2 + (1-\alpha) (t-1)^2}, \\
\nonumber
\eta_\EE(t) &= 1 - \frac{(\alpha (t+1) + (1-\alpha) (t-1))^2}
{\alpha (t+1)^2 + (1-\alpha) (t-1)^2},
\end{align}
for all $t \in \R$.
The case where $\alpha = \frac14$, i.e., $\mu = \frac14 \delta_1 + \frac34 \delta_{-1}$,
is shown on the left of figure \ref{figAtoms}. Note the atoms on the upper level at
$(1,1)$ and $(-1,1)$, of size $\frac14$ and $\frac34$ respectively.
In Metcalfe \cite{Met13}, it is shown that the sizes of these atoms decay linearly as $\eta$ decreases.
More exactly, since there is an the atom of size $\frac14$ at the point $(1,1)$ on
the upper level, then there is an atom of size $\frac14 - (1-\eta) = \eta - \frac34$ at
the point $(1,\eta)$ for all $1 \ge \eta > \frac34$, and no atom at the point
$(1,\eta)$ when $\frac34 \ge \eta \ge 0$. Similarly there is an atom of size
$\frac34 - (1-\eta) = \eta - \frac14$ at the point $(-1,\eta)$ for all
$1 \ge \eta > \frac14$, and no atom at the point $(-1,\eta)$ when $\frac14 \ge \eta \ge 0$.
The vertical solid lines in figure \ref{figAtoms} represent these atoms. Note that the
edge curve is tangent to the boundary box at the points $(1,\frac34)$
and $(-1,\frac14)$ where the atoms `disappear'.

\begin{figure}
\centering
\mbox{\includegraphics{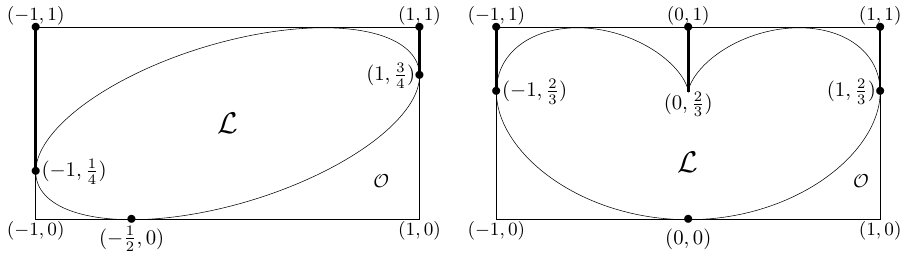}}
\caption{Left: $\mu = \frac14 \delta_1 + \frac34 \delta_{-1}$.
$\EE$ is composed of all points on $\partial \LL$ except
the lower tangent point $(\int_a^b x \mu[dx],0) = (-\frac12,0)$.
Right: $\mu = \frac13 \delta_1 + \frac13 \delta_0 + \frac13 \delta_{-1}$.
$\EE$ is composed of all points on $\partial \LL$ except
the lower tangent point $(\int_a^b x \mu[dx],0) = (0,0)$. In both cases
$\lim_{t \to \pm \infty} (\chi_\EE(t),\eta_\EE(t))$ equals the
lower tangent point, and parameter increases in the clockwise direction.}
\label{figAtoms}
\end{figure}

The case where $\mu = \frac13 \delta_1 + \frac13 \delta_0 + \frac13 \delta_{-1}$
($k=3$, $b = b_1 = 1$, $b_2 = 0$, $a = b_3 = -1$)
is shown on the right of figure \ref{figAtoms}. Now we distinguish between the `outer'
top level atoms at $(-1,1)$ and $(1,1)$, and the `inner' top level atom at $(0,1)$.
All top level atoms are of size $\frac13$. Moreover,
similar to before, the sizes of the `outer' atoms decay linearly as $\eta$ decreases,
and the edge curve is tangent to the boundary box at their points of disappearance
($(-1,\frac23)$ and $(1,\frac23))$. In contrast, though the size of the `inner' atom also
decays linearly at $\eta$ decreases, there is a cusp in the edge curve at the point of
disappearance ($(0,\frac23)$). 

For the general atomic measure, though the exact details are non-trivial and
beyond the scope of this paper, a high level overview of analogous results is
illuminating. We state the following without proof: We call the top level atoms at
$(b_1,1)$ and $(b_k,1)$, of size $\alpha_1$ and $\alpha_k$ respectively, the `outer' atoms.
All other top level atoms when $k > 2$ are called the `inner' atoms. A similar decay
in the size of all top level atoms occurs as $\eta$ decreases. Moreover, the edge
curve is tangent to the boundary box at the points of disappearance of the `outer' atoms
($(b_1,1-\alpha_1)$ and $(b_k,1-\alpha_k)$). Finally, each `inner' atom has an
associated cusp, but the location of the cusp can be distinct
from the point of disappearance of the atom.

Figure \ref{figAtoms} also depicts $\OO$ for these examples. 
$\LL$, $\EE$ and $\OO$ for general atomic measures behaves similarly,
irrespective of the number of cusps in the edge curve. For more general
measures $\mu$, however, $\LL$, $\EE$ and $\OO$ are highly non-trivial to
completely characterise. That is why, in this paper, we restrict to the
case $\mu[\{b\}] > 0$. Then
$(\chi_\EE(\cdot),\eta_\EE(\cdot)) : (b, +\infty) \to \EE$ and $\OO$ always
behave as described in the previous section, depicted in figure \ref{figU}.

\subsection{Statement of main asymptotic result}

In \cite{Met13}, Metcalfe examined the local asymptotic behaviour in $\LL$ for the Gelfand-Tsetlin
particle systems discussed in section \ref{sectdsogtp} as $n \to \infty$, and found universal bulk
asymptotic behaviour. In \cite{Duse17}, Duse and Metcalfe examined the local asymptotic behaviour of
particles in $\EE$ for analogous systems of discrete Gelfand-Tsetlin patterns
and found universal edge asymptotic behaviour, and the authors have every expectation that analogous
results hold in this case also (see remark \ref{remNonAsyEdge}). The main asymptotic result of
this paper, theorem \ref{thmdecay}, concerns the local asymptotic behaviour of particles
in neighbourhoods of $\mathcal{O}$ (see definition \ref{defLowRig}), under the assumption that
$\mu[\{b\}] > 0$.

As we discussed at the end of section \ref{sectdsogtp}, the assumption
that $\mu[\{b\}] > 0$ allows us to refine the definition of $\OO$: $\OO$ is the set of all
$(\chi,\eta) \in (a,b) \times (0,1)$ for which
$1 - \eta > \mu[\{\chi\}]$, $f_{(\chi,\eta)}'$ has $2$ distinct
roots of multiplicity $1$ in $(b,+\infty)$, and $f_{(\chi,\eta)}'$
has no other roots in $(b,+\infty)$. This defines a map from $\OO$ to
$\angle = \{(t,s) \in (b,+\infty)^2 : t > s\}$, a homeomorphism with
inverse $(\chi_\OO(\cdot),\eta_\OO(\cdot)) : \angle \to \OO$.

More specifically, theorem \ref{thmdecay} examines the asymptotic behaviour of
$K_n((u_n,r_n),(v_n,s_n))$ (see equation (\ref{eqKnrusvFixTopLine}))
under the following:

\begin{ass}
\label{AssWeakerAssumption}
Assume that $\mu[\{b\}] > 0$. Fix $(\chi,\eta) \in \OO \subset (a,b) \times (0,1)$
and the corresponding $(t,s) \in \angle$ with $(\chi,\eta) = (\chi_\OO(t,s), \eta_\OO(t,s))$.
Also let $\{(u_n,r_n)\}_{n\ge1} \subset \R \times \{1,2,\ldots,n-1\}$ and
$\{(v_n,s_n)\}_{n\ge1} \subset \R \times \{1,2,\ldots,n-1\}$ be
sequences of particle positions which satisfy:
\begin{equation}
\label{equnrnvnsn}
(u_n,\tfrac{r_n}n) = (\chi,\eta) + o(1)
\hspace{.25cm} \text{and} \hspace{.25cm}
(v_n,\tfrac{s_n}n) = (\chi,\eta) + o(1)
\hspace{.25cm} \text{as} \hspace{.25cm} n \to \infty.
\end{equation}
Consequently, for the remainder of this section it is natural to index
$f_{(\chi,\eta)}'$ with $(t,s) \in \angle$ instead of $(\chi,\eta)$. In
other words:
\begin{equation*}
f_{(t,s)} := f_{(\chi,\eta)}.
\end{equation*}
\end{ass}

Theorem \ref{thmdecay} obtains the asymptotics of $K_n((u_n,r_n),(v_n,s_n))$ by performing a
steepest descent analysis of the contour integral expression in equation
(\ref{eqKnrnunsnvn1}). Note, part (4) of lemma \ref{lemLowRig} implies
that $f_{(t,s)} |_{(b,+\infty)}$ is real-valued, strictly increasing
in $(b,s)$, has a local maximum at $s$, is strictly decreasing in $(s,t)$,
has a local minimum at $t$, and is strictly increasing in $(t,+\infty)$.
Also, equations (\ref{eqfn}, \ref{eqtildefn}, \ref{eqf}, \ref{equnrnvnsn})
and assumption \ref{assWeakConv} imply that
$f_n(t) - \tilde{f}_n(s) \to f_{(t,s)} (t) - f_{(t,s)}(s) < 0$
as $n \to \infty$. Equations (\ref{eqKnrnunsnvn1}, \ref{eqJn}, \ref{eqKnConInt2})
and intuition from steepest descent analysis then imply that
$\exp(n f_n (t) - n \tilde{f}_n(s))
\sim \exp(n f_{(t,s)} (t) - n f_{(t,s)}(s))$ will dominate the
asymptotics as $n \to \infty$, i.e., exponential decay. The main 
result of this section, theorem \ref{thmdecay}, proves this result,
and gives exact bounds on the rate of convergence.

To state theorem \ref{thmdecay}, we must motivate the choice of the $o(1)$
terms in equation (\ref{equnrnvnsn}) using steepest descent considerations.
First recall that $t$ and $s$ are the unique roots of $f_{(t,s)}'$ in
$(b,+\infty)$ and $t>s>b$ (see equation (\ref{eqangle})). 
Next recall that $x^{(n)} = (x_1^{(n)},x_2^{(n)},\ldots,x_n^{(n)})$ are the
deterministic positions of the particles on row $n$ (see section \ref{sectdsogtp}),
$x_1^{(n)} > x_2^{(n)} > \cdots > x_n^{(n)}$, and $\supp(\mu) \subset [a,b]$
(see assumption \ref{assWeakConv}). Note, an element of $x^{(n)}$ may act as
a pole for the contour integral expression of equation (\ref{eqKnrnunsnvn1}),
and so a problem may arise in the steepest descent analysis if these are not
{\em eventually isolated} from the roots $t$ and $s$. It is therefore
convenient to assume:
\begin{ass}
\label{assIsol}
Assume that there exists an $\xi = \xi(t,s) > 0$ and $N = N(t,s) \ge 1$ for
which $t-4\xi > s+4\xi > s-4\xi > b+4\xi > x_1^{(n)} > b-4\xi$ and
$a+4\xi > x_n^{(n)} > a-4\xi$ for all $n>N$.
\end{ass}

Next define,
\begin{equation}
\label{eqPn}
P_n := \{x_1^{(n)},x_2^{(n)},\ldots,x_n^{(n)}\}
\hspace{.25cm} \text{and} \hspace{.25cm}
\mu_n := \frac1n \sum_{x \in P_n} \delta_x
\hspace{.25cm} \text{and} \hspace{.25cm}
C_n(w) := \frac1n \sum_{x \in P_n} \frac1{w-x},
\end{equation}
for all $n \ge 1$ and $w \in \C \setminus \supp(\mu_n) = \C \setminus P_n$.
Note, $\mu_n \to \mu$ weakly as $n \to \infty$ (see assumption \ref{assWeakConv}), 
$C_n : \C \setminus \supp(\mu_n) \to \C$ is the Cauchy Transform of $\mu_n$, and
$\{t,s\} \subset (b+4\xi,+\infty) \subset \C \setminus P_n$ for all $n>N$. Next,
inspired by the explicit expression for
$(\chi_\OO(\cdot),\eta_\OO(\cdot)) : \angle \to \OO$ in theorem \ref{thmLowRig}, define:
\begin{definition}
\label{defLowRigNonAsy}
Define $\angle(\xi) := \{(T,S) \in (b+4\xi,+\infty)^2 : T > S\}$.
Also, for all $n>N$, define
\begin{equation*}
\chi_n(T,S) = \frac{T C_n(T) -  S C_n(S)}{C_n(T) - C_n(S)}
\hspace{0.5cm} \text{and} \hspace{0.5cm}
\eta_n(T,S) = 1 + \frac{C_n(T) C_n(S) (T-S)}{C_n(T) - C_n(S)},
\end{equation*}
for all $(T,S) \in \angle(\xi)$. Note, assumption \ref{assIsol}
implies that $(t,s) \in \angle(\xi)$, and
define $(\chi_n,\eta_n) := (\chi_n(t,s),\eta_n(t,s))$ for all $n>N$.
\end{definition}
Note that, for any fixed $w \in (\C \setminus \R) \cup (b+4\xi,+\infty)$
and $(T,S) \in \angle(\xi)$, assumption \ref{assIsol} and definition
\ref{defLowRigNonAsy} imply that $C_n(w)$ and $\chi_n(T,S)$ and $\eta_n(T,S)$
are all well-defined for $n>N$. Moreover, since $\mu_n \to \mu$ weakly as
$n \to \infty$:
\begin{equation}
\label{eqL4}
C_n(w) \to C(w), \;\;\;
\chi_n(T,S) \to \chi_\OO(T,S), \;\;
\eta_n(T,S) \to \eta_\OO(T,S),
\end{equation}
for all $w \in (\C \setminus \R) \cup (b+4\xi,+\infty)$ and $(T,S) \in \angle(\xi)$.

\begin{rem}
\label{remNonAsyEdge}
Given a fixed $(\chi, \eta) = (\chi_\OO(t,s),\eta_\OO(t,s)) \in \OO$, $(\chi_n,\eta_n)$
can be considered as the equivalent non-asymptotic point. Since we have no control
of the rate of convergence in assumption \ref{assWeakConv}, it is therefore natural to
examine particles in neighborhoods of $(\chi_n,\eta_n)$ rather than neighborhoods
of $(\chi,\eta)$, as we do in equation (\ref{equnrnvnsn2}) below. Note, though beyond
the scope of this paper, we stated at the beginning of this section that we expect
universal edge asymptotic behaviour in neighborhoods of $\EE$. Proceeding analogously,
first define the equivalent non-asymptotic edge
by simply replacing the asymptotic quantities in equation (\ref{eqEdge}) with
their non-asymptotic equivalents ($\chi_{n,\EE}(t) := t + \frac{C_n(t)}{C_n'(t)}$
when $t \in R^+ \cup R^-$ etc), and then examine particles in neighborhoods of the non-asymptotic
edge. See \cite{Duse17} for the analogous result for discrete Gelfand-Tsetlin patterns.
\end{rem}

Next recall (see above) that $\exp(n f_n (t) - n \tilde{f}_n(s))$
intuitively dominates the asymptotics as $n \to \infty$.
Note, equations (\ref{eqfn}, \ref{eqtildefn}) imply that
$f_n(t)$ depends on $t,v_n,s_n$, and $\tilde{f}_n(s)$
depends on $s,u_n,r_n$. Intuition from steepest descent
analysis then imply that $v_n$ and $s_n$ must depend on $t$, and
$u_n$ and $r_n$ must depend on $s$.
Moreover, equations (\ref{eqfn}, \ref{eqtildefn}, \ref{eqf}, \ref{equnrnvnsn}),
and part (4) of lemma \ref{lemLowRig} imply the following as $n \to \infty$:
\begin{itemize}
\item
$f_n'(t) \to f_{(t,s)}'(t) = 0$ and
$f_n''(t) \to f_{(t,s)}''(t) \neq 0$.
\item
$f_n'(s) \to f_{(t,s)}'(s) = 0$ and
$f_n''(s) \to f_{(t,s)}''(s) \neq 0$.
\end{itemize}
Equations (\ref{eqKnrnunsnvn1}, \ref{eqJn}, \ref{eqKnConInt2}),
definition \ref{defLowRigNonAsy}, and intuition from steepest
descent analysis then imply the following refinement of
equation (\ref{equnrnvnsn}) for all $n>N$, a stronger assumption
than assumption \ref{AssWeakerAssumption}:
\begin{equation}
\label{equnrnvnsn1.5}
(v_n, \tfrac{s_n}n) = (\chi_n,\eta_n) + \mathbf{X}_n(t) n^{-\frac12}
\hspace{.25cm} \text{and} \hspace{.25cm}
(u_n, \tfrac{r_n}n) = (\chi_n,\eta_n) + \tilde{\mathbf{X}}_n(s) n^{-\frac12},
\end{equation}
where $\mathbf{X}_n(t)$ is a vector depending on $t$,
$\tilde{\mathbf{X}}_n(s)$ is a vector depending on $s$,
and $||\mathbf{X}_n(t)|| = O(1)$ and $||\tilde{\mathbf{X}}_n(s)|| = O(1)$
for all $n$ sufficiently large.

It remains to provide natural choices for $\mathbf{X}_n(t)$ and
$\tilde{\mathbf{X}}_n(s)$. Note, when $n>N$, we can proceed similarly
to the proof of lemma \ref{lemLowRigTay} to get:
\begin{align*}
(\chi_n(T,S),\eta_n(T,S))
&= (\chi_n,\eta_n) + (T-t) \; c_{1,n} \; \mathbf{x_n}(s)
+ (S-s) \; c_{2,n} \; \mathbf{x_n}(t) \\
&\;\;\;\; + O((|T-t|+|S-s|)^2),
\end{align*}
for all $(T,S) \in \angle_\xi$ with $|T-t|$ and $|S-s|$ sufficiently small,
where $\mathbf{x}_n(T) := (1,C_n(T))$ for all $T \in (b+2\xi,+\infty)$,
and where $c_{1,n} = c_{1,n}(t,s) \to c_1(t,s) < 0$ and
$c_{2,n} = c_{2,n}(t,s) \to c_2(t,s) < 0$ as $n \to \infty$
(see equation (\ref{eqc1c2OO})). Note that the linear part of the above
Taylor expansion is non-trivial, and is naturally decomposed in terms
of the vectors $\mathbf{x}_n(t)$ and $\mathbf{x}_n(s)$. It therefore
seems natural to assume the following stronger assumption than equation (\ref{equnrnvnsn1.5})
for all $n>N$:
\begin{align}
\label{equnrnvnsn2}
(v_n, \tfrac{s_n-1}n)
&= (\chi_n,\eta_n) + m_n \mathbf{x}_n(t) n^{-\frac12} + (y_{1,n}, y_{2,n}) n^{-1}, \\
\nonumber
(u_n, \tfrac{r_n+1}n)
&= (\chi_n,\eta_n) + \tilde{m}_n \mathbf{x}_n(s) n^{-\frac12}
+ (\tilde{y}_{1,n}, \tilde{y}_{2,n}) n^{-1},
\end{align}
where $m_n, \tilde{m}_n, y_{1,n}, y_{2,n}, \tilde{y}_{1,n}, \tilde{y}_{2,n} = O(1)$
for all $n$ sufficiently large. Using $\tfrac{s_n-1}n$ and $\tfrac{r_n+1}n$
above, rather that simply $\tfrac{s_n}n$ and $\tfrac{r_n}n$, simplifies
some expressions later.

Finally we give additional conditions on $\xi$ and $N$ (see assumption
\ref{assIsol}) which are sufficient to obtain exact steepest descent bounds
for $K_n((u_n,r_n), (v_n,s_n))$ for all $n>N$:
\begin{definition}
\label{defxiN}
Assume that $\mu[\{b\}] > 0$, and fix $(\chi,\eta) \in \OO$ and the corresponding
$(t,s) \in \angle$ with $(\chi,\eta) = (\chi_\OO(t,s), \eta_\OO(t,s))$. Recall that $t>s>b>\chi>a$
and $1>\eta>0$ (see equation (\ref{eqangle}) and definition \ref{defLowRig2}).
Also recall that $\chi_n = \chi_n(t,s)$, $\eta_n = \eta_n(t,s)$, $u_n = u_n(t,s)$,
$r_n = r_n(t,s)$ $v_n = v_n(t,s)$ and $s_n = s_n(t,s)$ (see
definition \ref{defLowRigNonAsy} and equation (\ref{equnrnvnsn2})), and
$\chi_n, u_n, v_n \to \chi$ and $\eta_n, \frac{r_n}n, \frac{s_n}n \to \eta$
as $n \to \infty$ (see equations (\ref{eqL4}, \ref{equnrnvnsn2})). Finally
recall (see assumption \ref{assIsol}) that exists an $\xi = \xi(t,s) > 0$ and
$N = N(t,s) \ge 1$ for which $t-4\xi > s+4\xi > s-4\xi > b+4\xi > x_1^{(n)} > b-4\xi$
for all $n>N$. We first choose the above $\xi = \xi(t,s) > 0$ sufficiently small
such that the following are also satisfied:
\begin{itemize}
\item
$t - 4\xi > s + 4\xi >
s - 4\xi > b + 4\xi >
b - 4\xi > \chi + 4\xi >
\chi - 4\xi > a + 4\xi$.
\item
$1 - 2\xi > 1 - \eta + 2\xi >
1 - \eta - 2\xi > 2\xi $.
\end{itemize}
Next, given this new $\xi$, we
choose the above $N = N(t,s) \ge 1$ sufficiently large such that the
following are also satisfied for all $n>N$:
\begin{itemize}
\item
$b+4\xi > x_1^{(n)} > b-4\xi$ and $a+4\xi > x_n^{(n)} > a-4\xi$,
\item
$\chi + 4\xi > \{\chi_n, v_n , u_n\} > \chi - 4\xi$,
\item
$1 - \eta + 2\xi > \{1-\eta_n, 1 - \tfrac{s_n-1}n, 1 - \tfrac{r_n+1}n\}
> 1 - \eta - 2\xi$,
\item
$\{x \in P_n : x > v_n \vee u_n \} \neq \emptyset$
and
$\{x \in P_n : x < v_n \wedge u_n \} \neq \emptyset$.
\end{itemize}
Above, $\alpha > \{x,y,z\} > \beta$ denotes
$\alpha > x \vee y \vee z \ge x \wedge y \wedge z > \beta$.
Next, fix $\theta \in (\frac13, \frac12)$, and choose the above
$N = N(t,s) \ge 1$ sufficiently large such that the following are
also satisfied for all $n>N$:
\begin{itemize}
\item
$n^{\frac13-\theta} < \frac12, \hspace{.25cm}
n^{-\frac12+\theta} < \frac12, \hspace{.25cm}
n^{-\theta} < \xi, \hspace{.25cm}
n^{-\frac12} < \tfrac12 \xi, \hspace{.25cm}
|v_n - u_n| < \tfrac12 \xi$, \hspace{.25cm}
$n^{-1} < \xi$
\item
$n^{-\theta} <
2^{-6} (t-\chi) (t-b)^3 (b-a)^{-1} |f_{(t,s)}''(t)|$.
\item
$n^{1-3\theta} (E_{2,n} + \tilde{E}_{2,n}) < 1$,
where $E_{2,n}, \tilde{E}_{2,n}$ are defined in parts (6,7)
of lemma \ref{lemTay}.
\end{itemize}
\end{definition}

The above conditions on $N = N(t,s) \ge 1$ are still not yet sufficient.
To obtain the remaining conditions we need to examine the root behaviour
of $f_{(t,s)}', f_n', \tilde{f}_n'$ more closely. We also consider
the following `non-asymptotic' functions inspired by definition
\ref{defLowRigNonAsy} and by equations (\ref{eqf'0}, \ref{eqCauTrans}, \ref{eqPn}):
\begin{equation}
\label{eqftsn'}
f_{(t,s),n}'(w)
:= C_n(w) - \frac{1-\eta_n}{w-\chi_n}
= \frac1n \sum_{x \in P_n} \frac1{w-x} - \frac{1-\eta_n}{w-\chi_n},
\end{equation}
for all $w \in \C \setminus (P_n \cup \{\chi_n\})$. The function $f_{(t,s),n}$
is unimportant and left undefined. Moreover, equations
(\ref{eqfn}, \ref{eqtildefn}, \ref{eqPn}) give,
\begin{align}
\label{eqfn'}
f_n'(w)
&= C_n(w) - \frac{1-\frac{s_n-1}n}{w-v_n}
= \frac1n \sum_{x \in P_n} \frac1{w-x} - \frac{1-\frac{s_n-1}n}{w-v_n}, \\
\label{eqtildefn'}
\tilde{f}_n'(w)
&= C_n(w) - \frac{1-\frac{r_n+1}n}{w-v_n}
= \frac1n \sum_{x \in P_n} \frac1{w-x} - \frac{1-\frac{r_n+1}n}{w-u_n},
\end{align}
for all $w \in \C \setminus \R$. We extend these functions analytically
to $\C \setminus (P_n \cup \{v_n\})$ and $\C \setminus (P_n \cup \{u_n\})$
respectively. The following result will be proved in
section \ref{secRootsOO}:
\begin{lem}
\label{lemN2}
Fix $\xi = \xi(t,s) > 0$ and $N = N(t,s) \ge 1$ as in definition \ref{defxiN}.
Let $B(t,2\xi) \subset \mathbb{C}$ represent the open ball of radius $2\xi$ centered on $t$,
and similarly for $B(s,2\xi)$. Then $B(t,2\xi)$ and $B(s,2\xi)$ are disjoint open
subsets of $(\C \setminus \R) \cup (b+4\xi,+\infty)$, and $f_{(t,s)}'$
is well-defined and analytic in $B(t,2\xi) \cup B(s,2\xi)$. Moreover:
\begin{enumerate}
\item
$f_{(t,s)}'(t) = f_{(t,s)}'(s) = 0$.
\item
$f_{(t,s)}''(t) > 0$ and $f_{(t,s)}''(s) < 0$.
\item
$t$ and $s$ are the unique roots of $f_{(t,s)}'$ in $B(t,2\xi)$ and
$B(s,2\xi)$ respectively.
\end{enumerate}
Also, for all $n>N$, $f_{(t,s),n}'$ is well-defined and analytic in
$B(t,2\xi) \cup B(s,2\xi)$, and:
\begin{enumerate}
\setcounter{enumi}{3}
\item
$f_{(t,s),n}'(t) = f_{(t,s),n}'(s) = 0$.
\end{enumerate}
Moreover, we can choose the above $N = N(t,s) \ge 1$ sufficiently large
such that the following are also satisfied for all $n>N$:
\begin{enumerate}
\setcounter{enumi}{4}
\item
$f_{(t,s),n}''(t) > \tfrac12 f_{(t,s)}''(t) > 0$
and
$f_{(t,s),n}''(s) < \tfrac12 f_{(t,s)}''(s) < 0$.
\item
$t$ and $s$ are the unique roots of $f_{(t,s),n}'$ in $B(t,\xi)$ and
$B(s,\xi)$ respectively.
\end{enumerate}
Also, for all $n>N$, $B(t,2n^{-\frac12}) \subset B(t,\xi)$
and $B(s,2n^{-\frac12}) \subset B(s,\xi)$, $f_n'$ and $\tilde{f}_n'$ are
well-defined and analytic in $B(t,2\xi) \cup B(s,2\xi)$, and:
\begin{enumerate}
\setcounter{enumi}{6}
\item
$|f_n'(t)|, |\tilde{f}_n'(s)| = O(n^{-1})$, and
$|f_n''(t) - f_{(t,s),n}''(t)|, |\tilde{f}_n''(s) - f_{(t,s),n}''(s)| = O(n^{-\frac12})$
for all $n$ sufficiently large (we give explicit bounds in the proof).
\end{enumerate}
Moreover, we can choose the above $N = N(t,s) \ge 1$ sufficiently large
such that the following are also satisfied for all $n>N$:
\begin{enumerate}
\setcounter{enumi}{7}
\item
$f_n''(t) > \tfrac14 f_{(t,s)}''(t) > 0$
and
$\tilde{f}_n''(s) < \tfrac14 f_{(t,s)}''(s) < 0$.
\item
Counting multiplicities, $f_n'$ has exactly $1$ root (denoted $t_n$) in $B(t,n^{-\frac12})$ and
exactly $1$ root (denoted $s_n$) in $B(s,\xi)$. Also,
$t_n \in (t-n^{-\frac12}, t+n^{-\frac12}) \subset (t- \frac\xi2, t+\frac\xi2)$
and $s_n \in (s-\xi,s+\xi)$.
\item
Counting multiplicities, $\tilde{f}_n'$ has exactly $1$ root (denoted $\tilde{t}_n$) in $B(t,\xi)$ and
exactly $1$ root (denoted $\tilde{s}_n$) in $B(s,n^{-\frac12})$. Also,
$\tilde{t}_n \in (t-\xi,t+\xi)$ and
$\tilde{s}_n \in (s-n^{-\frac12}, s+n^{-\frac12}) \subset (s- \frac\xi2, s+\frac\xi2)$.
\item
$f_n''(t_n) > \tfrac14 f_{(t,s)}''(t) > 0$
and
$\tilde{f}_n''(\tilde{s}_n) < \tfrac14 f_{(t,s)}''(s) < 0$.
\end{enumerate}
\end{lem}

With the above lemma, we can finally state the main asymptotic result (proved in section \ref{secproof}):
\begin{thm}
\label{thmdecay}
Assume that $\mu[\{b\}] > 0$, and fix $(\chi,\eta) \in \OO$ and the corresponding 
$(t,s) \in \angle$ with $(\chi,\eta) = (\chi_\OO(t,s), \eta_\OO(t,s))$. Define
$u_n,r_n,v_n,s_n$ as in equation (\ref{equnrnvnsn2}), fix $\theta \in (\frac13,\frac12)$,
and choose $N = N(t,s) \ge 1$ sufficiently large that the conditions of definition
\ref{defxiN} and the results lemma \ref{lemN2} are both satisfied.
Then, for all $n>N$,
\begin{align*}
&\left| n J_n
- \frac{\exp( n f_n(t) - n \tilde{f}_n(s))}{4 \pi (t-s) D_n \tilde{D}_n} \right|
< \frac{\exp( n f_n(t) - n \tilde{f}_n(s))}{4 \pi (t-s) D_n \tilde{D}_n}
\; n^{1-3\theta} \; F_n \\
&+ \frac{\exp(n f_n(t) - n \tilde{f}_n(s))}{t-s} \;
\; \exp( - \tfrac14 n^{1-2\theta} (D_n^2 \wedge \tilde{D}_n^2))
\; n^{1-\theta} \; G_n,
\end{align*}
where $J_n$ is defined in equation (\ref{eqJn}),
$F_n > 0$ and $G_n > 0$ are defined in the proofs of lemmas \ref{lemJn11} and \ref{lemJn12}
(respectively) and satisfy $F_n = O(1)$ and $G_n = O(1)$ for all $n$ sufficiently large,
and $D_n := ( \tfrac12 |f_n''(t)| )^\frac12 \ge 0$ and
$\tilde{D}_n := ( \tfrac12 |\tilde{f}_n''(s)| )^{\frac12} \ge 0$. Finally,
$\phi_{r_n, s_n}(u_n,v_n) = 0$ when $r_n = s_n$ for all $n > N$
(see equation (\ref{eqKnrusvFixTopLine})), and:
\begin{equation*}
K_n((u_n,r_n),(v_n,s_n)) = (1-\tfrac{s_n}n) \; nJ_n
\;\;\; \text{when } r_n = s_n \text{ for all } n > N.
\end{equation*}
\end{thm}

Note that $f_n(t) - \tilde{f}_n(s) \to f_{(t,s)}(t) - f_{(t,s)}(s) < 0$
as $n \to \infty$ (see equations (\ref{eqfn}, \ref{eqtildefn}, \ref{eqf})
and part (1) of lemma \ref{lemf'}). Also note that $D_n^2 > \frac18 |f_{(t,s)}''(t)| > 0$
and  $\tilde{D}_n^2 > \frac18 |f_{(t,s)}''(s)| > 0$ for all $n>N$ (see
part (3) of lemma \ref{lemTay}). Moreover, definition \ref{defxiN} gives
$n^{-1} < \xi$ and $1 - \tfrac{s_n-1}n > 1 - \eta - 2\xi$ and $\xi < \frac14 (1-\eta)$
for all $n>N$, and so $1 > 1 - \tfrac{s_n}n > \frac14 (1-\eta) > 0$. Finally,
recall that $\theta \in (\frac13,\frac12)$. Theorem \ref{thmdecay} thus shows
that $| K_n((u_n,r_n),(v_n,s_n))|$ decays exponentially when $r_n = s_n$ as $n \to \infty$,
and gives exact rates of decay.

Finally, write the denominator $(t-s) D_n \tilde{D}_n$ of theorem \ref{thmdecay} as follows:
\begin{equation*}
(t-s) D_n \tilde{D}_n
= (t-s) \tfrac12 (|f_n''(t)| |\tilde{f}_n''(s)|)^\frac12
= (t-s)^2 \tfrac12 \bigg( \frac{|f_n''(t)|}{t-s} \frac{|\tilde{f}_n''(s)|}{t-s} \bigg)^\frac12.
\end{equation*}
Then, part (4) of lemma \ref{lemf'} shows that there exists natural
bounds $c_1 = c_1(t,s)>0$ and $c_2 = c_2(t,s)>0$ for which the following
is satisfied for all $n$ sufficiently large:
\begin{equation*}
c_1 (t-s)^2 > (t-s) D_n \tilde{D}_n > c_2 (t-s)^2.
\end{equation*}
This demonstrates the natural dependence of the denominator on the term $t-s$.
We will see this explicitly for an example in the next section.

\subsection{Expected number of particles}
\label{sec:enop}

Theorem \ref{thmdecay} proves exponential decay for the correlation kernel in neighbourhoods
of $\OO$ as $n \to \infty$. Moreover, explicit bounds and rates of convergence
have been obtained. However, it is clear that the bounds are very complex for the
general case. In this section we consider an example calculation. We demonstrate
how explicit bounds may be obtained in principle, but do not actually obtain these
for brevity.

First we define the asymptotic measure of assumption \ref{assWeakConv}:
\begin{itemize}
\item
$\mu := \tfrac14 \delta_1 + \tfrac34 \delta_{-1}$  (note, $1=b>a=-1$).
\end{itemize}
With this $\mu$, we will see below that $[.5,.99] \times (0,\frac14) \subset \OO$.
Recall, definition \ref{defLowRig2} and theorem \ref{thmLowRig} imply that
for each unique point in $\OO$, there exists a corresponding unique
point in $\angle = \{(t,s) : t > s > 1\}$  with $(\chi,\eta) = (\chi_\OO(t,s), \eta_\OO(t,s))$.
We can therefore apply theorem \ref{thmdecay} to the following:
\begin{itemize}
\item
Consider all
$(\chi,\eta) \in [.5,.99] \times \{(1 - \frac1l) \frac14 \} \subset \OO$
and all corresponding $(t,s) \in \angle$ with $(\chi,\eta) = (\chi_\OO(t,s), \eta_\OO(t,s))$,
where $l \ge L \ge 2$ are integers.
\end{itemize}
Note that $\eta = (1 - \frac1l) \frac14$ for all above $(\chi,\eta)$, and recall that $t > s > 1$
for all above $(t,s)$. The integer $L \ge 2$ will not be specified for brevity. However, we will see
below that $L$ can be fixed sufficiently large such that the relevant results hold
for all $l \ge L$. We also adopt the following terminology for brevity: Whenever
we say a statement holds for all $l$ and corresponding pairs, we mean the statement holds for all
$l \ge L$ and all $(\chi,\eta) \in [.5,.99] \times \{(1 - \tfrac1l) \tfrac14\}$ and all corresponding
$(t,s) \in \angle$. Whenever we say a statement holds for any fixed $l$ and all corresponding pairs,
we mean if we fix any specific $l \ge L$, the statement holds for all
$(\chi,\eta) \in [.5,.99] \times \{(1 - \tfrac1l) \tfrac14\}$ and all corresponding $(t,s) \in \angle$.
Next choose parameters in definition \ref{defxiN} for all $l$ and all corresponding pairs:
\begin{itemize}
\item
Fix $\theta := \frac{5}{12} \in (\frac13,\frac12)$, $\xi > 0$, and integers $n \ge N \ge 1$.
\end{itemize}
Note, we do not specify explicit values for $\xi$ and $N$ for brevity.
However, we show below that $\xi$ can be fixed sufficiently small and $N$ sufficiently
large such that the requirements
of theorem \ref{thmdecay} (definition \ref{defxiN} and lemma \ref{lemN2}) are satisfied
for any fixed $n \ge N$, and all $l$ and corresponding pairs. We also demonstrate how, in
principle, explicit values may be found. Next choose the remaining parameters of theorem
\ref{thmdecay} for any fixed $l$ and all corresponding pairs:
\begin{itemize}
\item
Restrict the above $n$ to integer multiples of $4l$.
\item
$x_1^{(n)} := 1$ and $x_i^{(n)} := 1 - (i-1) \tfrac1{n^2}$ for all
$k \in \{2,3,\ldots, \frac{n}4 \}$, and
$x_n^{(n)} := -1$ and $x_i^{(n)} := -1 + (n-i) \tfrac1{n^2}$ for all
$k \in \{\frac{n}4 + 1, \frac{n}4 + 2, \ldots, n-1\}$:
The particles on the top row of the Gelfand-Tsetlin pattern.
\item
$(u_n,r_n) := (\chi, n \eta)$ and $(v_n,s_n) := (\chi, n \eta)$: The parameters
in equation (\ref{equnrnvnsn2}).
More exactly, in equation (\ref{equnrnvnsn2}), we take $m_n = \tilde{m}_n = 0$,
$y_{1,n} = \tilde{y}_{1,n} = n(\chi - \chi_n)$, $y_{2,n} = n(\eta - \eta_n) - 1$,
and $\tilde{y}_{2,n} = n(\eta - \eta_n) + 1$.
\end{itemize}
Note, the top level particles are distinct (a requirement of section \ref{sectdsogtp}),
and assumption \ref{assWeakConv} is trivially satisfied.  Note also, that $r_n$ and $s_n$
are integers as required, since $\eta = (1 - \frac1l) \frac14$ and $n$ is a multiple of $4l$,
and we will show below that the above choices of $(u_n,r_n)$ and $(v_n,s_n)$ satisfy the
requirements of equation (\ref{equnrnvnsn2}).

We then use theorem \ref{thmdecay} to estimate the following:
\begin{equation}
\label{eqExpNum}
\mathbb{M}_1[[.5,.99] \times \{n \eta\}]
= \int_{.5}^{.99} K_n((\chi, n \eta), (\chi, n \eta)) d\chi
= \int_{.5}^{.99} K_n((u_n, r_n), (v_n, s_n)) d\chi,
\end{equation}
where integration is with respect to Lebesgue measure. This is
the expected number of particles on row $n \eta = n (1 - \tfrac1l) \tfrac14$
that are contained in $[.5,.99]$ (see section \ref{sectdsogtp}).

First recall (see section \ref{secAtoms}) that $\LL$, $\EE$ and $\OO$ for the above
$\mu$ are shown in figure \ref{figAtoms} (reproduced in figure \ref{fig2Atoms}),
the edge curve $(\chi_\EE(\cdot),\eta_\EE(\cdot)) : \R \to \EE$ is
given by equation (\ref{eqEdge2Atoms}) with $\alpha = \frac14$, the
restriction $(\chi_\EE(\cdot),\eta_\EE(\cdot)) : (1,+\infty) \to \EE$ is that
lower right section of the edge curve in figure \ref{fig2Atoms} between
$(1,\frac34)$ and $(-\frac12,0)$, and $\OO$ is that open subset of $(-1,1) \times (0,1)$
bounded by $(\chi_\EE(\cdot),\eta_\EE(\cdot)) |_{(1, +\infty)}$ and the bounding box of
$[-1,1] \times [0,1]$. It follows that $(\frac12, \frac14) = (\chi_\EE(2),\eta_\EE(2))$
is a point of the lower right edge, and $[.5,.99] \times (0,\frac14) \subset \OO$.
Then, $[.5,.99] \times \{(1 - \tfrac1l) \tfrac14\}$ is a horizontal line in $\OO$ for any $l \ge L$.

\begin{figure}
\centering
\mbox{\includegraphics{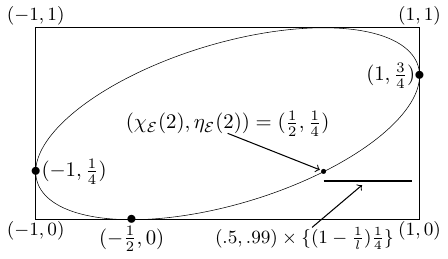}}
\caption{$\LL$, $\EE$ and $\OO$ when $\mu = \frac14 \delta_1 + \frac34 \delta_{-1}$,
and $[.5,.99] \times \{(1 - \tfrac1l) \tfrac14\} \subset \OO$. See figure \ref{figAtoms} for
a more explicit description of $\LL$, $\EE$ and $\OO$ in this case.}
\label{fig2Atoms}
\end{figure}

Consider the relevant asymptotic quantities in theorem \ref{thmdecay}.
Note, since $\mu = \tfrac14 \delta_1 + \tfrac34 \delta_{-1}$, equation (\ref{eqf}) gives:
\begin{equation}
\label{eqf2Atoms}
f_{(\chi,\eta)}(w) = \tfrac14 \log (w-1) + \tfrac34 \log(w+1) - (1-\eta) \log (w-\chi),
\end{equation}
for all $(\chi,\eta) \in [-1,1] \times [0,1]$ and $w \in \C \setminus \{1,-1,\chi\}$,
where $\log$ represents principal value. Note that $(.5, (1 - \tfrac1l) \tfrac14) \in \OO$ is
the leftmost point of $[.5,.99] \times \{(1 - \tfrac1l) \tfrac14\}$ for any fixed $l$.
Lemma \ref{lemexpBeh} then gives the following for any fixed $l$ and all corresponding pairs:
\begin{itemize}
\item
$f_{(\chi,\eta)}(t) - f_{(\chi,\eta)}(s) < 0$ is maximised when
$(\chi,\eta) = (.5,(1 - \tfrac1l) \tfrac14)$.
\item
$t>1$ is minimised when $(\chi,\eta) = (.5,(1 - \tfrac1l) \tfrac14)$.
\item
$s>1$ is maximised when $(\chi,\eta) = (.5,(1 - \tfrac1l) \tfrac14)$.
\end{itemize}
Let $(t_l, s_l) \in \angle$ denote the point in $\angle$ which
corresponds $(.5,(1 - \tfrac1l) \tfrac14)$, i.e, $(.5,(1 - \tfrac1l) \tfrac14) =
(\chi_\OO(t_l,s_l), \eta_\OO(t_l,s_l))$.
The above bounds thus imply the following for any fixed $l$ and all corresponding pairs:
\begin{itemize}
\item
$f_{(\chi,\eta)}(t) - f_{(\chi,\eta)}(s)
< f_{(.5, (1 - \frac1l) \frac14)}(t_l) - f_{(.5, (1 - \frac1l) \frac14)}(s_l)
< 0$.
\item
$t > t_l > 1$.
\item
$s_l > s > 1$.
\end{itemize}
We now apply lemma \ref{lemexpBenEdge} to analyse these further.
Consider the corresponding points $2 \in (1,+\infty) = (b,+\infty)$ and
$(\chi,\eta) = (\frac12,\frac14) = (\chi_\EE(2), \eta_\EE(2))\in \EE$
(see theorem \ref{thmEdge}). Note equation (\ref{eqf2Atoms}) gives
$f_{(\frac12,\frac14)}'(2) = f_{(\frac12,\frac14)}''(2) = 0$
and $f_{(\frac12,\frac14)}'''(2) = \frac19$.
Then, since $l \ge L$, lemma \ref{lemexpBenEdge} implies that we can fix $L$
sufficiently large such that the following is satisfied for any fixed $l \ge L$
and all corresponding pairs:
\begin{itemize}
\item
$f_{(.5, (1 - \frac1l) \frac14)}(t_l) - f_{(.5, (1 - \frac1l) \frac14)}(s_l)
< - \frac{5}{12 \sqrt{6}} \; (\tfrac1{l})^\frac32$.
\item
$2 + ( \tfrac6l )^\frac12 > t_l > 2 + \tfrac12 ( \tfrac6l )^\frac12
> 2$.
\item
$2 >
2 - \frac12 ( \tfrac6l )^\frac12 > s_l > 2 - ( \tfrac6l )^\frac12
> 1 + ( \tfrac6l )^\frac12$.
\end{itemize}
The above then prove the following for any fixed $l$ and all corresponding pairs:
\begin{itemize}
\item
$f_{(\chi,\eta)}(t) - f_{(\chi,\eta)}(s)
< - \frac{5}{12 \sqrt{6}} \; (\tfrac1{l})^\frac32$.
\item
$t-s > ( \tfrac6l )^\frac12$.
\end{itemize}
The above also show the following for all $l$ and corresponding pairs:
\begin{itemize}
\item
$t > 2 > s$.
\end{itemize}

Next note that $(.99, (1 - \tfrac1l) \tfrac14) \in \OO$ is
the rightmost point of $[.5,.99] \times \{(1 - \tfrac1l) \tfrac14\}$ for any fixed $l$,
and $(1 - \tfrac1l) \tfrac14 \ge (1 - \tfrac1L) \tfrac14$ for all $l$.
Lemma \ref{lemexpBeh} then gives the following for all $l \ge L$ and corresponding pairs:
\begin{itemize}
\item
$s>1$ is minimised when $(\chi,\eta) = (.99,(1 - \tfrac1L) \tfrac14)$.
\item
$t>1$ is maximised when $(\chi,\eta) = (.99,(1 - \tfrac1L) \tfrac14)$.
\end{itemize}
Note, equation (\ref{eqf2Atoms}) gives the following when $(\chi,\eta) = (.99,(1 - \tfrac1L) \tfrac14)$:
\begin{equation*}
f_{(\chi,\eta)}(w) = \tfrac14 \log (w-1) + \tfrac34 \log(w+1) - (\tfrac34 + \tfrac1{4L}) \log (w-.99),
\end{equation*}
for all $w > 1$. Then, similar methods to those used in lemma \ref{lemexpBeh}
easily give the following for all $l$ and corresponding pairs for some constants $D,d > 0$:
\begin{itemize}
\item
$D > t > s > 1 + d$.
\end{itemize}
In principle we can obtain explicit expressions for $D$ and $d$, but we do not do so for brevity.

Next note, since $\mu := \tfrac14 \delta_1 + \tfrac34 \delta_{-1}$, equation (\ref{eqCauTrans}) gives,
\begin{equation*}
C(w) = \frac14 \frac1{w-1} + \frac34 \frac1{w+1},
\end{equation*}
for all $w>1$. The above bounds then give the following for all $l$ and corresponding pairs:
\begin{itemize}
\item
$\frac14 \frac1{D-1} + \frac34 \frac1{D+1}
< C(t) <
\frac14 \frac1{2-1} + \frac34 \frac1{2+1} = \frac12$.
\item
$\frac12 = \frac14 \frac1{2-1} + \frac34 \frac1{2+1}
< C(s) <
\frac14 \frac1{d} + \frac34 \frac1{d+2}$.
\end{itemize}
Moreover, for all $l$ and corresponding pairs:
\begin{equation}
\label{eqCauAtomEx}
- \frac{C(t) - C(s)}{t-s}
= \frac14 \frac1{(t-1)(s-1)} + \frac34 \frac1{(t+1)(s+1)}.
\end{equation}
The above bounds then give the following for all $l$ and all corresponding pairs:
\begin{itemize}
\item
$- \frac{C(t) - C(s)}{t-s}
> \frac14 \frac1{(D-1)(2-1)} + \frac34 \frac1{(D+1)(2+1)}$.
\item
$- \frac{C(t) - C(s)}{t-s}
< \frac14 \frac1{(2-1)(d)} + \frac34 \frac1{(2+1)(d+2)}$.
\end{itemize}

Next consider $f_{(\chi,\eta)}''(t)$ and $f_{(\chi,\eta)}''(s)$ for all $l$
and all corresponding pairs. Note, part (4) of lemma \ref{lemf'} gives,
\begin{align*}
|f_{(\chi,\eta)}''(t)|
&= \int_{-1}^1 \int_{-1}^1 \frac{(t-s) (x-y)^2 \mu[dx] \mu[dy]}{2 C(s) (t-x)^2 (t-y)^2 (s-x) (s-y)}, \\
|f_{(\chi,\eta)}''(s)|
&= \int_{-1}^1 \int_{-1}^1 \frac{(t-s) (x-y)^2 \mu[dx] \mu[dy]}{2 C(t) (s-x)^2 (s-y)^2 (t-x) (t-y)}.
\end{align*}
The, since $\mu = \frac14 \delta_1 + \frac34 \delta_{-1}$,
\begin{align}
\label{eqf''AtomEx}
|f_{(\chi,\eta)}''(t)|
&= \frac{(t-s) \frac{3}{8}}{2 C(s) (t-1)^2 (t+1)^2 (s-1) (s+1)}, \\
\nonumber
|f_{(\chi,\eta)}''(s)|
&= \frac{(t-s) \frac{3}{8}}{2 C(t) (s-1)^2 (s+1)^2 (t-1) (t+1)}.
\end{align}
It is thus clear from the above bounds
that we can find in principle explicit $D_1,D_2,d_1,d_2>0$
such that the following is satisfied for all $l$ and corresponding pairs:
\begin{itemize}
\item
$d_1 < |f_{(\chi,\eta)}''(t)| (t-s)^{-1} < D_1$.
\item
$d_2 < |f_{(\chi,\eta)}''(s)| (t-s)^{-1} < D_2$.
\end{itemize}

Next consider non-asymptotic quantities. Recall that $n \ge 1$ is a multiple of $4l$,
and the above definition of $x^{(n)}$. Equations (\ref{eqCauTrans}, \ref{eqPn}) then give,
\begin{align*}
&C(w) - C_n(w) \\
&= \tfrac1n \sum_{i=1}^{\frac{n}4}
\left( \frac1{w-1} - \frac1{w - 1 + (i-1) \frac1{n^2}} \right)
+ \tfrac1n \sum_{i= \frac{n}4 + 1}^n
\left( \frac1{w+1} - \frac1{w + 1 - (n-i) \frac1{n^2}} \right),
\end{align*}
for all $w > 1$. It is thus clear from the above bounds
that we can find in principle $B>0$
such that the following is satisfied for all $l$ and corresponding pairs:
\begin{itemize}
\item
$|C(t) - C_n(t)| < \frac{B}n$.
\item
$|C(s) - C_n(s)| < \frac{B}n$.
\end{itemize}
We can similarly show that we can choose $B$ such
that the following is satisfied for all $l$ and corresponding pairs:
\begin{itemize}
\item
$|C'(t) - C_n'(t)| < \frac{B}n$ and $|C''(t) - C_n''(t)| < \frac{B}n$.
\item
$|C'(s) - C_n'(s)| < \frac{B}n$ and $|C''(s) - C_n'(s)| < \frac{B}n$.
\end{itemize}
Also, since $(v_n,s_n) = (\chi, n\eta)$, equations (\ref{eqfn}, \ref{eqf2Atoms}) give:
\begin{align*}
f_{(\chi,\eta)}(t) - f_n(t)
&= \tfrac1n \sum_{i=1}^{\frac{n}4} \bigg( \log(t-1) - \log (t - 1 + (i-1) \tfrac1{n^2}) \bigg) \\
&+ \tfrac1n \sum_{i= \frac{n}4 + 1}^n \bigg( \log(t+1) - \log (t + 1 - (n-i) \tfrac1{n^2}) \bigg)
+ \tfrac1n \log(t-\chi).
\end{align*}
The above bounds thus show that we can choose $B$ and $N$ (recall $n>N$) such
that the following is satisfied for all $l$ and corresponding pairs:
\begin{itemize}
\item
$|f_{(\chi,\eta)}(t) - f_n(t)| < \frac{B}n$.
\end{itemize}
Similarly, since $(u_n,r_n) = (\chi, n\eta)$, equations (\ref{eqtildefn}, \ref{eqf2Atoms})
and the above bounds show that we can choose $B$ and $N$ such that:
\begin{itemize}
\item
$|f_{(\chi,\eta)}(s) - \tilde{f}_n(s)| < \frac{B}n$.
\end{itemize}
Similarly, we can choose $B$ and $N$ such that:
\begin{itemize}
\item
$|f_{(\chi,\eta)}'(t) - f_n'(t)| < \frac{B}n$
and $|f_{(\chi,\eta)}''(t) - f_n''(t)| < \frac{B}n$.
\item
$|f_{(\chi,\eta)}'(s) - \tilde{f}_n'(s)| < \frac{B}n$
and $|f_{(\chi,\eta)}''(s) - \tilde{f}_n''(s)| < \frac{B}n$.
\end{itemize}

Next note, for all $l$ and all corresponding pairs, equation (\ref{eqPn}) gives:
\begin{align*}
- \frac{C_n(t) - C_n(s)}{t-s}
&= \tfrac1n \sum_{i=1}^{\frac{n}4}
\left( \frac1{(t - 1 + (i-1) \frac1{n^2}) (s - 1 + (i-1) \frac1{n^2})} \right) \\
&+ \tfrac1n \sum_{i= \frac{n}4 + 1}^n
\left( \frac1{(t + 1 - (n-i) \frac1{n^2}) (s + 1 - (n-i) \frac1{n^2})} \right).
\end{align*}
Therefore, since $t > s > 1$ and $n \ge 1$,
\begin{equation*}
- \frac{C_n(t) - C_n(s)}{t-s}
> \frac14 \left( \frac1{(t - 1 + \frac14) (s - 1 + \frac14)} \right) \\
+ \frac34 \left( \frac1{(t + 1 + 1) (s + 1 + 1)} \right).
\end{equation*}
The above bounds thus give the following for all $l$ and corresponding pairs:
\begin{itemize}
\item
$- \frac{C_n(t) - C_n(s)}{t-s}
> \frac14 \frac1{(D - 1 + \frac14) (2 - 1 + \frac14)}
+ \frac34 \frac1{(D + 1 + 1) (2 + 1 + 1)}$.
\end{itemize}

Next recall (see definition \ref{defLowRigNonAsy} and theorem \ref{thmLowRig}
that, $(\chi,\eta) = (\chi_\OO(t,s), \eta_\OO(t,s))$ and
$(\chi_n,\eta_n) = (\chi_n(t,s), \eta_n(t,s))$. Definition \ref{defLowRig2}
and theorem \ref{thmLowRig} (replace $\mu$ by $\mu_n = \sum_i \delta_{x_i^{(n)}}$
etc) then give the following for all $l$ and corresponding pairs:
\begin{itemize}
\item
$1 = x_1^{(n)} > \chi_n > x_n^{(n)}$.
\item
$1 > \eta_n > 0$.  
\end{itemize}
Moreover, the expressions for $\chi$ and $\chi_n$ give:
\begin{equation*}
(\chi_n - \chi) (t-s)
= \frac{t-s}{C_n(t) - C_n(s)} \frac{t-s}{C(t) - C(s)} [- (C_n(t) - C(t)) C(s) + (C_n(s) - C(s)) C(t)].
\end{equation*}
It is thus clear from the above bounds that we can choose the $B>0$ such that
the following is satisfied for all $l$ and all corresponding pairs:
\begin{itemize}
\item
$|\chi_n - \chi| (t-s) < \frac{B}n$.
\end{itemize}
Thus, since $t - s > (\frac6l)^\frac12$ for any fixed $l$ and all corresponding pairs (see above):
\begin{itemize}
\item
$|\chi_n - \chi| < \frac{B \sqrt{l}}{n \sqrt6}$.
\end{itemize}
Similarly we can choose the $B>0$ such that
the following is satisfied for any fixed $l$ and all corresponding pairs:
\begin{itemize}
\item
$|\eta_n - \eta| < \frac{B \sqrt{l}}{n \sqrt6}$.
\end{itemize}

Next recall that $m_n = \tilde{m}_n = 0$, $y_{1,n} = \tilde{y}_{1,n} = n(\chi - \chi_n)$,
$y_{2,n} = n(\eta - \eta_n) - 1$, and $\tilde{y}_{2,n} = n(\eta - \eta_n) + 1$.
The above bounds then give the following for any fixed $l$ and all corresponding pairs,
which we note trivially satisfy the requirements of equation (\ref{equnrnvnsn2}):
\begin{itemize}
\item
$|m_n| = |\tilde{m}_n| = 0$.
\item
$|y_{1,n}| < B \sqrt{l/6}$ and $|\tilde{y}_{1,n}| < B \sqrt{l/6}$.
\item
$|y_{2,n}| < 1 + B \sqrt{l/6}$ and $|\tilde{y}_{2,n}| < 1 + B \sqrt{l/6}$.
\end{itemize}

Next consider the requirements of definition \ref{defxiN} and lemma \ref{lemN2}. Recall
that $\theta = \frac5{12} \in (\frac13, \frac12)$, $v_n = u_n = \chi$,
and $s_n = r_n = n \eta$. With these choices, and the above bounds, it is
easy to see that $\xi>0$ can be fixed sufficiently small, and $N\ge1$ (recall $n > N$)
can be fixed sufficiently large, such that all requirements are satisfied
for all $l$ and corresponding pairs. Moreover, we can in principle find explicit values but
we do not do this for brevity.

Finally, we apply theorem \ref{thmdecay} to equation (\ref{eqExpNum}).
Recall that $(u_n, r_n) = (v_n, s_n) = (\chi, n\eta)$, where $\eta = (1-\frac1l) \frac14$,
where $l \ge L$, $n \ge N$ is a multiple of $4l$, and $\theta = \frac{5}{12}$.
We have shown above that the conditions of theorem \ref{thmdecay} are
satisfied, and applying theorem \ref{thmdecay} for any fixed $l$ we get:
\begin{align}
\label{eqKnCalculation}
| K_n((\chi, n\eta),(\chi, n\eta)) |
&<  \frac{\exp( n f_n(t) - n \tilde{f}_n(s))}{4 \pi (t-s) D_n \tilde{D}_n} \\
\nonumber
&+ \frac{\exp( n f_n(t) - n \tilde{f}_n(s))}{4 \pi (t-s) D_n \tilde{D}_n}
\; n^{-\frac14} \; F_n \\
\nonumber
&+ \frac{\exp(n f_n(t) - n \tilde{f}_n(s))}{t-s} \;
\; \exp( - \tfrac14 n^{\frac16} (D_n^2 \wedge \tilde{D}_n^2))
\; n^{\frac7{12}} \; G_n,
\end{align}
where $F_n > 0$ and $G_n > 0$ are defined in the proofs of lemmas \ref{lemJn11} and \ref{lemJn12}
(respectively) and satisfy $F_n = O(1)$ and $G_n = O(1)$ for all $n$ sufficiently large,
and $D_n \tilde{D}_n = \tfrac12 (|f_n''(t)| |\tilde{f}_n''(s)|)^\frac12 \ge 0$.
Recall (see  part (8) of lemma \ref{lemN2}) that $(|f_n''(t)| |\tilde{f}_n''(s)|)^\frac12
> \frac14 (|f_{(\chi,\eta)}''(t)| |f_{(\chi,\eta)}''(s)|)^\frac12$, and
(see above) $|f_{(\chi,\eta)}''(t)| > d_1 (t-s)$ and
$|f_{(\chi,\eta)}''(s)| > d_2 (t-s)$ for all $l$ and corresponding pairs, where
in principle we can find explicit constants for $d_1,d_2 > 0$.
Recall also that $t-s > (\frac6l)^\frac12$ for any fixed $l$ and all corresponding pairs.
Also, we have shown that $f_{(\chi,\eta)}(t) - f_{(\chi,\eta)}(s)
< - \frac{5}{12 \sqrt{6}} \; (\tfrac1{l})^\frac32$ ,
and $|f_n(t) - f_{(\chi,\eta)}(t)| < \tfrac{B}n$ and $|\tilde{f}_n(s) - f_{(\chi,\eta)}(s)| < \tfrac{B}n$.
Combined, the above show that we can choose $N$ sufficiently
large such that the following is satisfied for any fixed $l$ and all corresponding
$(\chi,\eta) \in [.5,.99] \times \{(1 - \tfrac1l) \tfrac14\}$ 
and $(t,s) \in \angle$ with $(\chi,\eta) = (\chi_\OO(t,s), \eta_\OO(t,s))$:
\begin{itemize}
\item
$f_n(t) - \tilde{f}_n(s)
< - \frac{5}{12 \sqrt{6}} \; (\tfrac1{l})^\frac32 + \frac{2B}n$.
\item
$t - s > (\frac6l)^\frac12$.
\item
$D_n \tilde{D}_n > \tfrac18 (t-s) \sqrt{d_1 d_2} > \tfrac18 (\frac6l)^\frac12 \sqrt{d_1 d_2}$.
\end{itemize}
The first term on the RHS of equation (\ref{eqKnCalculation})
thus satisfies the following for any fixed $l$ and all corresponding pairs:
\begin{equation*}
\frac{\exp( n f_n(t) - n \tilde{f}_n(s))}{4 \pi (t-s) D_n \tilde{D}_n}
< \frac{\exp( - n \frac{5}{12 \sqrt{6}} \; (\tfrac1{l})^\frac32 + 2B)}
{3 \pi (\frac1l) \sqrt{d_1 d_2}}.
\end{equation*}
Finally, we state that we can find explicit bounds for $|F_n|$ and $|G_n|$ using similar
methods to those discussed above. It thus follows that we can choose $N$ sufficiently
large such that the second and third terms on the RHS of equation (\ref{eqKnCalculation})
are also bounded by the above term for any fixed $l$ and all corresponding pairs.
Finally, equations (\ref{eqExpNum}, \ref{eqKnCalculation}) give the following corollary of
theorem \ref{thmdecay}
\begin{cor}
\label{corAtoms}
Take $\mu := \tfrac14 \delta_1 + \tfrac34 \delta_{-1}$,
and define $x^{(n)}, (v_n,s_n), (u_n,r_n)$, $N$, $B$ etc, as above. Fix $l \ge L$,
and $n > N$ a multiple of $4l$. Then the expected number of particles on row
$n \eta = n (1 - \tfrac1l) \tfrac14$
that are contained in $[.5,.99]$ satisfies the following:
\begin{equation*}
\mathbb{M}_1[[.5,.99] \times \{n \eta\}]
< C l \; \exp( - n \tfrac{5}{12 \sqrt{6}} \; (\tfrac1{l})^\frac32),
\end{equation*}
where $C := \frac{\exp(2B)}{2 \pi \sqrt{d_1 d_2}}$ is a constant independent of $l$.
\end{cor}

\section{The global asymptotic behaviour}\label{sec:main}

In this section we examine the global asymptotic behaviours of $\LL$, $\EE$ and $\OO$,
defined in section \ref{sectdsogtp}. The analysis here is analogous to that given in Duse
and Metcalfe, \cite{Duse15a}, for discrete interlaced Gelfand-Tsetlin patterns, and
many of the methods and results are similar. However, it is still necessary to carry
out the analysis in this context as understanding the global asymptotic behaviour
is an essential first step to identifying natural regions in which universal local asymptotic
behaviours can occur. Unless otherwise stated, only the following assumptions are
required in this section:
\begin{itemize}
\item
$\mu$ is a probability measure on $\R$ with compact support, $\supp(\mu) \subset [a,b]$ with
$\{a,b\} \subset \supp(\mu)$, and $(\chi,\eta) \in [a,b] \times [0,1]$ is fixed.
\item
Assume that $b > a$ to avoid that degenerate
case where $\mu$ is a single atom of mass 1. This implies that
$\mu[\{\chi\}] \in [0,1)$.
\end{itemize}

\subsection{The liquid region}
\label{secLiq}

Recall (see definition \ref{defLiq}) that the liquid region, $\LL$,
is the set of all $(\chi,\eta) \in [a,b] \times [0,1]$
for which the following function has non-real roots (see equation (\ref{eqf'0})):
\begin{equation}
\label{eqf'1}
f_{(\chi,\eta)}'(w) = C(w) - \frac{1-\eta}{w - \chi},
\end{equation}
for all $w \in \C \setminus \R$, where $C$ is the Cauchy transform of $\mu$
(see equation (\ref{eqCauTrans})). We denote $f_{(\chi,\eta)}'$ simply by $f'$
where no confusion is possible. Note, definition \ref{defLiq} and part (1) of corollary
\ref{corf'} imply the following, more refined, definition of $\LL$:
\begin{definition}
\label{defLiq2}
The liquid region, $\LL$, is the set of all $(\chi,\eta) \in (a,b) \times (0,1)$
for which $f'$ has a unique root in $\mathbb{H} := \{ w \in \C : \text{Im}(w) > 0 \}$.
This root has multiplicity $1$.
\end{definition}

\begin{thm}
\label{thmBulk}
Let $W_\LL : \LL \to \mathbb{H}$ map $(\chi,\eta) \in \LL$ to the
corresponding root of $f'$ in $\mathbb{H}$. This is a homeomorphism
with inverse $(\chi_\LL(\cdot),\eta_\LL(\cdot)) : \mathbb{H} \to \LL$
given by,
\begin{equation*}
\chi_\LL(w) = w + \frac{C(\bar{w}) (w - \bar{w})}{C(w) - C(\bar{w})}
\hspace{0.5cm} \text{and} \hspace{0.5cm}
\eta_\LL(w) = 1 + \frac{C(w) C(\bar{w}) (w-\bar{w})}{C(w) - C(\bar{w})}.
\end{equation*}
\end{thm}

\begin{proof}
We first show:
\begin{enumerate}
\item[(i)]
$\LL$ is non-empty.
\item[(ii)]
$\LL$ is open.
\item[(iii)]
$W_\LL : \LL \to \mathbb{H}$ is continuous.
\item[(iv)]
$W_\LL : \LL \to \mathbb{H}$ is injective.
\end{enumerate}
The {\em invariance of domain theorem} then implies that $W_\LL(\LL)$ is open
and $W_\LL : \LL \to W_\LL(\LL)$ is a homeomorphism. We complete
the result by showing:
\begin{enumerate}
\item[(v)]
$W_\LL : \LL \to W_\LL(\LL)$ has inverse
$w \mapsto (\chi_\LL(w),\eta_\LL(w))$ for all $w \in W_\LL(\LL)$.
\item[(vi)]
$W_\LL(\LL) = \mathbb{H}$.
\end{enumerate}

Consider (i). Fix $w \in \mathbb{H}$ and define
$(\chi,\eta) := (\chi_\LL(w),\eta_\LL(w))$, where $\chi_\LL$ and $\eta_\LL$
are defined in the statement of the theorem. We will show that: 
\begin{enumerate}
\item[(ia)]
$f'(w) = 0$.
\item[(ib)]
$(\chi,\eta) \in (a,b) \times (0,1)$ when $|w|$ is
sufficiently large.
\end{enumerate}
Definition \ref{defLiq} then implies that $(\chi,\eta) \in \LL$ when $|w|$ is
sufficiently large. This proves (i).

Consider (ia). First note, the definitions of $\chi = \chi_\LL(w)$ and
$\eta = \eta_\LL(w)$ easily give $1 - \eta = (w - \chi) C(w)$. Equation
(\ref{eqf'1}) then trivially gives $f'(w) = 0$. This proves (ia).

Consider (ib). First recall $\chi = \chi_\LL(w)$ and $\eta = \eta_\LL(w)$,
and write $\chi$ and $\eta$ as in equation (\ref{eqlemBddEdge3}) (below)
to get $(\chi,\eta) \in \R^2$. Next note, Taylor expansions of equation
(\ref{eqCauTrans}) give
\begin{align*}
C(w)
&= \frac1{w} + \frac{\mu_1}{w^2} + \frac{\mu_2}{w^3} + O \left( |w|^{-4} \right), \\
C(w) - C(\bar{w})
&= \left( \frac1{w} - \frac1{\bar{w}} \right)
\left( 1 + \mu_1 \left( \frac1{w} + \frac1{\bar{w}} \right)  +
\mu_2 \left( \frac1{w^2} + \frac1{|w|^2} + \frac1{\bar{w}^2} \right)
+ O \left( |w|^{-3} \right) \right),
\end{align*}
where $\mu_1 := \int_a^b x \mu[dx]$ and $\mu_2 := \int_a^b x^2 \mu[dx]$.
Combine these with the expressions for $\chi = \chi_\LL(w)$ and $\eta = \eta_\LL(w)$
given in the statement of this lemma to get,
\begin{equation*}
\chi = \mu_1 + O \left( |w|^{-1} \right)
\hspace{0.5cm} \text{and} \hspace{0.5cm}
\eta = \left( \mu_2 - \mu_1^2 \right) \frac1{|w|^2}
+ O \left( |w|^{-3} \right).
\end{equation*}
Finally recall (see assumption \ref{assWeakConv})
that $\mu$ is a probability measure on $[a,b]$, $b>a$,
and $\{a,b\} \in \supp(\mu)$. Therefore,
\begin{equation*}
\mu_1
= \int_a^b x \mu[dx]
< \int_{\{b\}} x \; \delta_b[dx]
= b.
\end{equation*}
Similarly $\mu_1 > a$, and
\begin{equation*}
\mu_2 - \mu_1^2 = \frac12 \int_a^b \mu[dx] \int_a^b \mu[dy] (x-y)^2  >
\frac12 \int_{\{0\}} \delta_0[dx] \int_{\{0\}} \delta_0[dy] \; (x-y)^2 = 0.
\end{equation*}
Therefore $(\chi,\eta) \in (a,b) \times (0,1)$
when $|w|$ is sufficiently large. This proves (ib).

Consider (ii). Fix $(\chi_1,\eta_1), (\chi_2,\eta_2) \in (a,b) \times (0,1)$
with $(\chi_1,\eta_1) \in \LL$. Define,
\begin{itemize}
\item 
$f_1'(w) := C(w) - (1-\eta_1)/(w-\chi_1)$,
\item
$f_2'(w) := C(w) - (1-\eta_2)/(w-\chi_2)$,
\end{itemize}
for all $w \in \mathbb{H}$. Note, since $(\chi_1,\eta_1) \in \LL$,
definition \ref{defLiq2} implies that $f_1'$ has a unique root in
$\mathbb{H}$. Denote this root by $w_1$, and fix $\e>0$ such that
$B(w_1,2\e) \subset \mathbb{H}$. Next note, since $w_1$ is the unique
root of $f_1'$ in $\mathbb{H}$, the extreme value theorem gives,
\begin{equation*}
\inf_{w \in \partial B(w_1,\e)} |f_1'(w)| > 0.
\end{equation*}
Finally, $|f_1'(w) - f_2'(w)| \le |\frac{1-\eta_1}{w-\chi_1} - \frac{1-\eta_2}{w-\chi_2}|$
for all $w \in \mathbb{H}$. Thus, whenever $|\chi_1-\chi_2|$ and $|\eta_1-\eta_2|$ are
sufficiently small, $|f_1'(w)| > |f_1'(w) - f_2'(w)|$ for all
$w \in \partial B(w_1,\e)$. Rouch\'{e}'s Theorem thus implies that $f_2'$ has a root
in $B(w_1,\e) \subset \mathbb{H}$. Definition \ref{defLiq} thus implies that
$(\chi_2,\eta_2) \in \LL$ whenever $|\chi_1-\chi_2|$ and $|\eta_1-\eta_2|$ are
sufficiently small. This proves (ii).

Consider (iii). Fix $(\chi_1,\eta_1), (\chi_2,\eta_2) \in \LL$, and define
$f_1', f_2'$ as in (ii). Also define $w_1$ and $\e$ as in (ii), and let $w_2$
denote the unique root of $f_2'$ in $\mathbb{H}$ (see definition \ref{defLiq2}).
Next, proceed as in (ii) to show that $f_2'$ has a root in $B(w_1,\e) \subset \mathbb{H}$
whenever $|\chi_1-\chi_2|$ and $|\eta_1-\eta_2|$ are sufficiently small. Thus we must
have $w_2 \in B(w_1,\e)$ whenever $|\chi_1-\chi_2|$ and $|\eta_1-\eta_2|$
are sufficiently small. Next recall that $w_1 = W_\LL(\chi_1,\eta_1)$ and
$w_2 = W_\LL(\chi_2,\eta_2)$ (see statement of this lemma). Therefore
$|W_\LL(\chi_1,\eta_1) - W_\LL(\chi_2,\eta_2)| < \e$ whenever $|\chi_1-\chi_2|$
and $|\eta_1-\eta_2|$ are sufficiently small. Finally note that we can repeat
the above analysis with $\e$ replaced by any $\d \in (0,\e)$. This proves (iii).

Consider (iv). Fix $(\chi_1,\eta_1), (\chi_2,\eta_2) \in \LL$ with
$W_\LL(\chi_1,\eta_1) = W_\LL(\chi_2,\eta_2) = w \in \mathbb{H}$.
Equation (\ref{eqf'1}) and the above definition of $W_\LL$ then give,
\begin{equation*}
C(w)
= \frac{1-\eta_1}{w-\chi_1}
= \frac{1-\eta_2}{w-\chi_2}.
\end{equation*}
Therefore $(\eta_2-\eta_1) w = (1-\eta_1) \chi_2 - (1-\eta_2) \chi_1$.
Then $w \in \R$ whenever $\eta_1 \neq \eta_2$, which contradicts
$w \in \mathbb{H}$. Thus $\eta_1 = \eta_2$, and so $(1-\eta_1) (\chi_1 - \chi_2) = 0$.
Finally, $\eta_1 < 1$ since $(\chi_1,\eta_1) \in \LL$ (see definition \ref{defLiq2}),
and so $\chi_1 = \chi_2$. This proves (iv).

Consider (v). Fix $(\chi,\eta) \in \LL$ and let
$w := W_\LL(\chi,\eta) \in W_\LL(\LL)$. Equation (\ref{eqf'1}) and
the above definition of $W_\LL$ then give $1-\eta = (w-\chi) C(w)$.
Complex conjugation then gives,
\begin{equation*}
1-\eta = (w-\chi) C(w) = (\bar{w}-\chi) C(\bar{w}).
\end{equation*}
Solving gives $(\chi,\eta) = (\chi_\LL(w),\eta_\LL(w))$. This proves (v).

Consider (vi). Recall that $W_\LL(\LL)$ is open and that
$W_\LL : \LL \to W_\LL(\LL)$ is a homeomorphism with inverse
$w \mapsto (\chi_\LL(w),\eta_\LL(w))$. Assume that $W_\LL(\LL)$ is a
proper subset of $\mathbb{H}$, i.e., that there exists a point
$w \in \partial W_\LL(\LL)$ with $w \in \mathbb{H} \setminus W_\LL(\LL)$.
Choose a sequence $\{w_k\}_{k\ge1} \subset W_\LL(\LL)$ with $w_k \to w$ as
$k \to \infty$, and let $(\chi_k,\eta_k) := (\chi_\LL(w_k),\eta_\LL(w_k))$
for all $k\ge1$. Note that we can always choose so that
$\{(\chi_k,\eta_k)\}_{k\ge1}$ is convergent as $k \to \infty$,
$(\chi_k,\eta_k) \to (\chi,\eta)$ say. Also note equation (\ref{eqf'1})
and the above definition of $W_\LL$ gives
$C(w_k) - (1-\eta_k)/(w_k-\chi_k) = 0$
for all $k\ge1$. Letting $k \to \infty$ we get
$C(w) - (1-\eta)/(w-\chi) = 0$, and so
$(\chi,\eta) \in \LL$ and $w = W_\LL(\chi,\eta)$. This contradicts
the assumption that $w \in \mathbb{H} \setminus W_\LL(\LL)$, and so
$W_\LL(\LL) = \mathbb{H}$. This proves (vi).
\end{proof}

Note the following trivial corollary of theorem \ref{thmBulk}:
\begin{cor}
\label{corBndryLim}
$\LL$ is a non-empty, open, simply connected subset of $(a,b) \times (0,1)$.
Moreover, $\partial \LL$ is the set of all $(\chi,\eta) \in [a,b] \times [0,1]$
for which there exists a sequence, $\{w_k\}_{k\ge1} \subset \mathbb{H}$, with
$(\chi_\LL(w_k),\eta_\LL(w_k)) \to (\chi,\eta)$ as $k \to \infty$,
and either $|w_k| \to \infty$ or $w_k \to t \in \R$ as $k \to \infty$.
\end{cor}
We end this section by using the above to examine $\partial \LL$:
\begin{lem}
\label{lemBddEdge}
First we consider those parts of $\partial \LL$ which exist for
any choice of $\mu$:
\begin{enumerate}
\item
$(\int_a^b x \mu[dx],0) \in \partial \LL$. Moreover
$(\chi_\LL(w_k),\eta_\LL(w_k)) \to (\int_a^b x \mu[dx],0)$ as
$k \to \infty$ for all $\{w_k\}_{k\ge1} \subset \mathbb{H}$ with
$|w_k| \to \infty$.
\item
$(\chi_\EE(t),\eta_\EE(t)) \in \partial \LL$ for all $t \in R$, where
$R \subset \R$ is open and given by the disjoint union
$R = R^+ \cup R^- \cup R_0 \cup R_1$ (see equation (\ref{eqR2})), and where
\begin{align*}
\chi_\EE(t) := t + \frac{C(t)}{C'(t)}
\hspace{0.5cm} &\text{and} \hspace{0.5cm}
\eta_\EE(t) := 1 + \frac{C(t)^2}{C'(t)}
\hspace{.5cm} \text{ when } t \in R^+ \cup R^-, \\
\chi_\EE(t) := t
\hspace{0.5cm} &\text{and} \hspace{0.5cm}
\eta_\EE(t) := 1
\hspace{.5cm} \text{ when } t \in R_0, \\
\chi_\EE(t) := t
\hspace{0.5cm} &\text{and} \hspace{0.5cm}
\eta_\EE(t) := 1 - \mu[\{t\}]
\hspace{.5cm} \text{ when } t \in R_1.
\end{align*}
Moreover, whenever $t \in R$,
$(\chi_\LL(w_k),\eta_\LL(w_k)) \to (\chi_\EE(t),\eta_\EE(t))$ as
$k \to \infty$ for all $\{w_k\}_{k\ge1} \subset \mathbb{H}$ with
$w_k \to t$.
\end{enumerate}

Next we impose restrictions on $\mu$ to examine other possible parts of $\partial \LL$:
\begin{enumerate}
\setcounter{enumi}{2}
\item
$(t,1) \in \partial \LL$ when there exists an interval $I = (t_2,t_1)$
with $t \in I \subset \supp(\mu)$, $\mu$ is absolutely continuous on $I$, and
the density of $\mu$ (denoted $\varphi$) satisfies one of the following:
\begin{itemize}
\item
$\sup_{x \in (t_2,t_1)} \varphi (x) < +\infty$ and $\inf_{x \in (t_2,t_1)} \varphi (x) > 0$.
\item
$\sup_{x \in (t_2,t)} \varphi (x) < +\infty$, $\inf_{x \in (t_2,t)} \varphi (x) > 0$,
$\varphi(x) = 0$ for all $x \in (t,t_1)$.
\item
$\sup_{x \in (t,t_1)} \varphi (x) < +\infty$, $\inf_{x \in (t,t_1)} \varphi (x) > 0$,
$\varphi(x) = 0$ for all $x \in (t_2,t)$.
\end{itemize}
Moreover, $(\chi_\LL(w_k),\eta_\LL(w_k)) \to (t,1)$ as
$k \to \infty$ for all $\{w_k\}_{k\ge1} \subset \mathbb{H}$ with
$w_k \to t$.
\end{enumerate}
\end{lem}

\begin{proof}
Consider (1). Fix $\{w_k\}_{k\ge1} \subset \mathbb{H}$ with
$|w_k| \to \infty$ as $k \to \infty$. The proof
of step (ib) in theorem \ref{thmBulk} then gives
$(\chi_\LL(w_k),\eta_\LL(w_k)) \to (\int_a^b x \mu[dx],0)$.
Corollary \ref{corBndryLim} then gives
$(\int_a^b x \mu[dx],0) \in \partial \LL$. This proves (1).

Consider (2) when $t \in R^+ \cup R^- \cup R_0 = \R \setminus \supp(\mu)$
(see equation (\ref{eqR2})). Fix $\{w_k\}_{k\ge1} \subset \mathbb{H}$ with
$w_k \to t$ as $k \to \infty$. First write (see theorem \ref{thmBulk}),
\begin{align}
\label{eqlemBddEdge1}
\chi_\LL(w_k)
&= w_k + C(\overline{w_k}) \frac{ w_k - \overline{w_k}}{C(w_k) - C(\overline{w_k})}, \\
\nonumber
\eta_\LL(w_k)
&= 1 + C(w_k) C(\overline{w_k}) \frac{w_k-\overline{w_k}}{C(w_k) - C(\overline{w_k})}.
\end{align}
Thus, since $w_k \to t$ and $\overline{w_k} \to t$ as $k \to \infty$,
where $w_k \in \mathbb{H}$ and $t \in \R \setminus \supp(\mu)$,
and since $C : \C \setminus \supp(\mu) \to \C$ is analytic (see equation
(\ref{eqCauTrans})),
\begin{equation*}
\chi_\LL(w_k) \to t + C(t) \frac1{C'(t)}
\hspace{0.5cm} \text{and} \hspace{0.5cm}
\eta_\LL(w_k) \to 1 + C(t) C(t) \frac1{C'(t)}
\hspace{0.5cm} \text{as } k \to \infty.
\end{equation*}
Therefore $(\chi_\LL(w_k),\eta_\LL(w_k)) \to (\chi_\EE(t),\eta_\EE(t))$ when
$t \in R^+ \cup R^-$. Similarly when $t \in R_0$ (recall that $C(t) = 0$
in this case by equation (\ref{eqR2})). Corollary \ref{corBndryLim} then gives
$(\chi_\EE(t),\eta_\EE(t)) \in \partial \LL$ when $t \in R^+ \cup R^- \cup R_0$.
This proves (2) when $t \in R^+ \cup R^- \cup R_0$.

Consider (2) when $t \in R_1$. Fix $\{w_k\}_{k\ge1} \subset \mathbb{H}$ with
$w_k \to t$ as $k \to \infty$. Recall (see equation (\ref{eqR2})) that
$\mu[\{t\}] > 0$, and there exists an open interval $I \subset \R$ with
$t \in I$ and $I \setminus \{t\} \subset \R \setminus \supp(\mu)$.
Equation (\ref{eqCauTrans}) thus gives,
\begin{equation*}
C(w) = \frac{\mu[\{t\}]}{w-t} + C_I(w),
\end{equation*}
for all $w \in \C \setminus \supp(\mu)$, where
$C_I(w) := \int_{[a,b] \setminus I} \frac{\mu[dx]}{w-x}$.
Therefore,
\begin{equation*}
\frac{C(w) - C(\bar{w})}{w - \bar{w}}
= - \frac{\mu[\{t\}]}{(w-t) (\bar{w} - t)}
+ \frac{C_I(w) - C_I(\bar{w})}{w - \bar{w}},
\end{equation*}
Recall that $w_k, \overline{w_k} \to t \in I$ as $k \to \infty$, and
note that $C_I$ has a unique analytic extension to $I$. 
Thus, combined, the above give the following as $k \to \infty$:
\begin{align*}
&C(w_k) = \frac{\mu[\{t\}]}{w_k-t} + C_I(t) + o(1),
\hspace{.5cm}
C(\overline{w_k}) = \frac{\mu[\{t\}]}{\overline{w_k}-t} + C_I(t) + o(1), \\
&\frac{C(w_k) - C(\overline{w_k})}{w_k - \overline{w_k}}
= - \frac{\mu[\{t\}]}{(w_k-t) (\overline{w_k} - t)}
+ C_I'(t) + o(1).
\end{align*}
Equation (\ref{eqlemBddEdge1}) thus gives the following for all $k$ sufficiently large:
\begin{align*}
\chi_\LL(w_k)
&= w_k + \bigg( \frac{\mu[\{t\}]}{\overline{w_k}-t} + O(1) \bigg)
\; \bigg( -\frac{\mu[\{t\}]}{(w_k-t) (\overline{w_k}-t)} + O(1) \bigg)^{-1}, \\
\eta_\LL(w_k)
&= 1 + \bigg( \frac{\mu[\{t\}]}{w_k-t} + O(1) \bigg)
\bigg( \frac{\mu[\{t\}]}{\overline{w_k}-t} + O(1) \bigg)
\bigg( -\frac{\mu[\{t\}]}{(w_k-t) (\overline{w_k}-t)} + O(1) \bigg)^{-1}.
\end{align*}
Therefore, since $w_k, \overline{w_k} \to t$ as
$k \to \infty$, $(\chi_\LL(w_k),\eta_\LL(w_k)) \to (t,1 - \mu[\{t\}])$ when
$t \in R_1$. Corollary \ref{corBndryLim} then gives
$(t,1 - \mu[\{t\}]) \in \partial \LL$ when $t \in R_1$.
This proves (2) when $t \in R_1$.

Consider (3) when $\sup_{x \in (t_2,t_1)} \varphi (x) < +\infty$
and $\inf_{x \in (t_2,t_1)} \varphi (x) > 0$. Fix
$\{w_k\}_{k\ge1} \subset \mathbb{H}$ with $w_k \to t$ as $k \to \infty$.
Denote $u_k := \text{Re}(w_k)$, $v_k := \text{Im}(w_k)$,
$R_k := \text{Re}(C(w_k))$, and $I_k := - \text{Im}(C(w_k))$, where $C$ is
the Cauchy transform of $\mu$ (see equation (\ref{eqCauTrans})). Then
$u_k \to t$ and $v_k \searrow 0$ as $k \to \infty$, and
\begin{equation}
\label{eqlemBddEdge2}
R_k = \int_a^b \frac{(u_k-x) \mu[dx]}{(u_k-x)^2 + v_k^2}
\hspace{0.5cm} \text{and} \hspace{0.5cm}
I_k = \int_a^b \frac{v_k \mu[dx]}{(u_k-x)^2 + v_k^2}
\hspace{.5cm} \text{ for all } k.
\end{equation}
Letting $\varphi^+ := \sup_{x \in (t_2,t_1)} \varphi (x)$ and
$\varphi^- := \inf_{x \in (t_2,t_1)} \varphi (x)$, we will show:
\begin{enumerate}
\item[(3a)]
$\pi \varphi^+ + O(v_k) \ge I_k \ge \pi \varphi^- + O(v_k)$ for all $k$ sufficiently large.
\item[(3b)]
$|R_k| \le (\varphi^+ - \varphi^-) |\log(v_k)| + O(1)$ for all $k$ sufficiently large.
\end{enumerate}
Next write (see theorem \ref{thmBulk}),
\begin{equation}
\label{eqlemBddEdge3}
\chi_\LL(w_k)
= u_k - \frac{v_k R_k}{I_k}
\hspace{0.5cm} \text{and} \hspace{0.5cm}
\eta_\LL(w_k)
= 1 - \frac{v_k (R_k^2 + I_k^2)}{I_k},
\end{equation}
for all $k$. Then, since $u_k \to t$ and $v_k \searrow 0$ as $k \to \infty$,
and since $+\infty > \varphi^+ \ge \varphi^- > 0$, (3a,3b) and equation
(\ref{eqlemBddEdge3}) give $(\chi_\LL(w_k),\eta_\LL(w_k)) \to (t,1)$ as
$k \to \infty$. Corollary \ref{corBndryLim} then gives
$(t,1) \in \partial \LL$. This proves (3) when $\sup_{x \in (t_2,t_1)} \varphi (x) < +\infty$
and $\inf_{x \in (t_2,t_1)} \varphi (x) > 0$. Part (3) for the other cases
follows similarly.

Consider (3a). Recall that $t \in (t_2,t_1) \subset \supp(\mu)$, and $\mu$
is absolutely continuous on $(t_2,t_1)$ with density $\varphi$. Equation
(\ref{eqlemBddEdge2}) then gives,
\begin{equation*}
I_k
= \int_a^{t_2} \frac{v_k \mu[dx]}{(u_k-x)^2 + v_k^2}
+ \int_{t_2}^{t_1} \frac{v_k \varphi(x) dx}{(u_k-x)^2 + v_k^2}
+ \int_{t_1}^b \frac{v_k \mu[dx]}{(u_k-x)^2 + v_k^2}.
\end{equation*}
Recall that $u_k \to t \in (t_2,t_1)$ and $v_k \searrow 0$ as $k \to \infty$. Therefore,
\begin{equation*}
I_k
= \int_{t_2}^{t_1} \frac{v_k \varphi(x) dx}{(u_k-x)^2 + v_k^2} + O(v_k),
\end{equation*}
for all $k$ sufficiently large. Recall also that
$\varphi^+ = \sup_{x \in (t_2,t_1)} \varphi (x) < +\infty$. Therefore,
\begin{align*}
I_k
&\le \int_{t_2}^{t_1} \frac{v_k (\varphi^+) dx}{(u_k-x)^2 + v_k^2} + O(v_k) \\
&= - \varphi^+ \arctan \left( \frac{u_k-t_1}{v_k} \right)
+ \varphi^+ \arctan \left( \frac{u_k-t_2}{v_k} \right)
+ O(v_k).
\end{align*}
Thus, since $u_k \to t \in (t_2,t_1)$ and $v_k \searrow 0$ as $k \to \infty$,
$I_k \le - \varphi^+ (-\frac{\pi}2 + O(v_k)) +  \varphi^+ (\frac{\pi}2 + O(v_k)) + O(v_k)
= \pi \varphi^+ + O(v_k)$ for all $k$ sufficiently large.
Similarly, since $\varphi^- = \inf_{x \in (t_2,t_1)} \varphi (x) > 0$,
\begin{equation*}
I_k \ge \int_{t_2}^{t_1} \frac{v_k (\varphi^-) dx}{(u_k-x)^2 + v_k^2} + O(v_k),
\end{equation*}
for all $k$ sufficiently large. Proceed as before to get
$I_k \ge \pi \varphi^- + O(v_k)$. Combining both inequalities proves (3a).

Consider (3b). Recall that $t \in (t_2,t_1) \subset \supp(\mu)$, and $\mu$
is absolutely continuous on $(t_2,t_1)$ with density $\varphi$. Equation
(\ref{eqlemBddEdge2}) then gives,
\begin{equation*}
R_k
= \int_a^{t_2} \frac{(u_k-x) \mu[dx]}{(u_k-x)^2 + v_k^2}
+ \int_{t_2}^{t_1} \frac{(u_k-x) \varphi(x) dx}{(u_k-x)^2 + v_k^2}
+ \int_{t_1}^b \frac{(u_k-x) \mu[dx]}{(u_k-x)^2 + v_k^2}.
\end{equation*}
Recall that $u_k \to t \in (t_2,t_1)$ and $v_k \searrow 0$ as $k \to \infty$.
Therefore,
\begin{equation*}
R_k
= \int_{t_2}^{u_k} \frac{(u_k-x) \varphi(x) dx}{(u_k-x)^2 + v_k^2}
+ \int_{u_k}^{t_1} \frac{(u_k-x) \varphi(x) dx}{(u_k-x)^2 + v_k^2} + O(1),
\end{equation*}
for all $k$ sufficiently large. Recall also that
$\varphi^+ = \sup_{x \in (t_2,t_1)} \varphi (x) < +\infty$ and
$\varphi^- = \inf_{x \in (t_2,t_1)} \varphi (x) > 0$. Therefore,
\begin{align*}
R_k
&\le \int_{t_2}^{u_k} \frac{(u_k-x) (\varphi^+) dx}{(u_k-x)^2 + v_k^2}
+ \int_{u_k}^{t_1} \frac{(u_k-x) (\varphi^-) dx}{(u_k-x)^2 + v_n^2}
+ O(1) \\
&= - \left. \frac{\varphi^+}2 \log ( (u_k-x)^2 + v_k^2 ) \right|_{t_2}^{u_k}
- \left. \frac{\varphi^-}2 \log ( (u_k-x)^2 + v_k^2 ) \right|_{u_k}^{t_1}
+ O(1),
\end{align*}
for all $k$ sufficiently large.
Thus, since $u_k \to t \in (t_2,t_1)$ and $v_k \searrow 0$ as $k \to \infty$,
$R_n \le -(\varphi^+ - \varphi^-) \log(v_k) + O(1)$. Similarly,
\begin{equation*}
R_k
\ge \int_{t_2}^{u_k} \frac{(u_k-x) (\varphi^-) dx}{(u_k-x)^2 + v_k^2}
+ \int_{u_k}^{t_1} \frac{(u_k-x) (\varphi^+) dx}{(u_k-x)^2 + v_n^2}
+ O(1) ,
\end{equation*}
for all $k$ sufficiently large. Proceed as before to get
$R_k \ge (\varphi^+ - \varphi^-) \log(v_k) + O(1)$. Combining both
inequalities proves (3b).
\end{proof}

\begin{rem}
\label{remBdd}
Note, parts (1,2) of lemma \ref{lemBddEdge} finds parts of $\partial \LL$ which
exist for any choice of $\mu$. Also note, equation (\ref{eqR2}) implies that
$R = \R$ when $\mu$ is purely atomic. In that case, corollary \ref{corBndryLim}
implies that parts (1,2) of lemma \ref{lemBddEdge} give a complete description
of $\partial \LL$. Finally note, part (3) imposes restrictions on $\mu$ to
examine other possible parts of $\partial \LL$. It is beyond the scope of this
paper to ease these restrictions since the resulting technicalities are highly
non-trivial.
\end{rem}

\subsection{The Edge, $\EE$}
\label{secEdge}

In this section we define $\EE$ as in definition \ref{defEdge}, and prove
an analogous result for $\EE$ to theorem \ref{thmBulk} for $\LL$. As in
section \ref{secLiq}, we denote $f'_{(\chi,\eta)}$ simply by $f'$.
Recall in theorem \ref{thmBulk}, $W_\LL: \LL \to \mathbb{H}$
maps each $(\chi,\eta) \in \LL$ to the corresponding unique root
of $f'$ in $\mathbb{H}$, and $W_\LL$ is a homeomorphism
with inverse $(\chi_\LL(\cdot),\eta_\LL(\cdot)) : \mathbb{H} \to \LL$. Recall also,
lemma \ref{lemBddEdge} implies that $(\chi_\EE(\cdot),\eta_\EE(\cdot)) : R \to \partial \LL$
is the curve which is the unique continuous extension to
$R = R^+ \cup R^- \cup R_0 \cup R_1 \subset \R$ (see equations (\ref{eqR1}, \ref{eqR2}))
of $(\chi_\LL(\cdot),\eta_\LL(\cdot)) : \mathbb{H} \to \LL$, and which is
continuous in any open sub-interval of $R$.
Finally note, definition \ref{defEdge} and corollary
\ref{corf'} imply the following, more refined, definition of $\EE$:
\begin{definition}
\label{defEdge3}
The edge, $\EE$, is the disjoint union $\EE = \EE^+ \cup \EE^- \cup \EE_0 \cup \EE_1$ where:
\begin{itemize}
\item
$\EE^+$ is the set of all $(\chi,\eta) \in (a,b) \times (0,1)$ for which
$1-\eta > \mu[\{\chi\}]$, and $f'$ has a unique repeated root in
$(\chi,+\infty) \setminus \supp(\mu)$. This root has multiplicity $2$ or $3$.
\item
$\EE^-$ is the set of all $(\chi,\eta) \in (a,b) \times (0,1)$ for which
$1-\eta > \mu[\{\chi\}]$, and $f'$ has a unique repeated root in
$(-\infty,\chi) \setminus \supp(\mu)$. This root has multiplicity $2$ or $3$.
\item
$\EE_0 := \{(\chi,\eta) : \chi \in R_0 \text{ and } \eta = 1 \}$. Moreover,
when $(\chi,\eta) \in \EE_0$, $\chi$ is a root of $f'$ of multiplicity $1$.
\item
$\EE_1 := \{(\chi,\eta) : \chi \in R_1 \text{ and } \eta = 1 - \mu[\{\chi\}] \}$.
Moreover, when $(\chi,\eta) \in \EE_1$, either $f'(\chi) \neq 0$ or
$\chi$ is a root of $f'$ of multiplicity $1$.
\end{itemize}
\end{definition}

\begin{thm}
\label{thmEdge}
Let $W_\EE : \EE \to \R$ map each $(\chi,\eta) \in \EE^+ \cup \EE^-$ to
the corresponding real-valued repeated root, and map each
$(\chi,\eta) \in \EE_0 \cup \EE_1$ to $\chi$.
Then $W_\EE(\EE) = R$ and $W_\EE : \EE \to R$ is a bijection with inverse
$t \mapsto (\chi_\EE(t),\eta_\EE(t))$. Moreover, the image spaces
of $\EE^+, \EE^-, \EE_0, \EE_1$ are (respectively) $R^+, R^-, R_0, R_1$.
\end{thm}

\begin{proof}
We will show:
\begin{enumerate}
\item
$W_\EE(\EE^+) = R^+$ and $W_\EE : \EE^+ \to R^+$ is a bijection with inverse
$t \mapsto (\chi_\EE(t),\eta_\EE(t))$.
\item
$W_\EE(\EE^-) = R^-$ and $W_\EE : \EE^- \to R^-$ is a bijection with inverse
$t \mapsto (\chi_\EE(t),\eta_\EE(t))$.
\item
$W_\EE(\EE_0) = R_0$ and $W_\EE : \EE_0 \to R_0$ is a bijection with inverse
$t \mapsto (\chi_\EE(t),\eta_\EE(t))$.
\item
$W_\EE(\EE_1) = R_1$ and $W_\EE : \EE_1 \to R_1$ is a bijection with inverse
$t \mapsto (\chi_\EE(t),\eta_\EE(t))$.
\end{enumerate}
Note, equation (\ref{eqR2}) implies that $R$ is the disjoint union,
$R = R^+ \cup R^- \cup R_0 \cup R_1$. We will prove (1). Part (2)
follows from similar considerations. Parts (3,4) trivially follow
from equation (\ref{eqR2}), definition \ref{defEdge3}, and part (2)
of lemma \ref{lemBddEdge}. Parts (1-4) give the required result.

Consider (1). We prove this by showing:
\begin{enumerate}
\item[(1a)]
Fix $(\chi,\eta) \in \EE^+$ and let $t := W_\EE(\chi,\eta)$. Then
$t \in R^+$ and $(\chi,\eta) = (\chi_\EE(t),\eta_\EE(t))$.
\item[(1b)]
Fix $t \in R^+$ and let $(\chi,\eta) := (\chi_\EE(t),\eta_\EE(t))$. Then
$(\chi,\eta) \in \EE^+$ and $W_\EE(\chi,\eta) = t$.
\end{enumerate}

Consider (1a). First note, definition \ref{defEdge3}, and the definition
of $W_\EE$ given in the statement of this theorem, imply that
$(\chi,\eta) \in (a,b) \times (0,1)$, $1-\eta > \mu[\{\chi\}]$, and
$t \in (\chi,+\infty) \setminus \supp(\mu)$ is a repeated root of $f'$.
Also, equations (\ref{eqf'0}) gives,
\begin{equation}
\label{eqf'Rmu}
f'(w) = C(w) - \frac{1-\eta}{w-\chi}
\hspace{0.5cm} \text{and} \hspace{0.5cm}
f''(w) = C'(w) + \frac{1-\eta}{(w-\chi)^2},
\end{equation}
for all $w \in \C \setminus (\supp(\mu) \cup \{\chi\})$. Then, since
$t \in (\chi,+\infty) \setminus \supp(\mu)$ and $f'(t) = f''(t) = 0$,
this gives
\begin{equation*}
C(t) = \frac{1-\eta}{t-\chi}
\hspace{0.5cm} \text{and} \hspace{0.5cm}
C'(t) = - \frac{1-\eta}{(t-\chi)^2}.
\end{equation*}
The first part gives $C(t) > 0$, since $t > \chi$ and $1 > \eta$, and so
$t \in R^+$ (see equation (\ref{eqR2})). Also, solving the above equations
gives $(\chi,\eta) = (\chi_\EE(t),\eta_\EE(t))$ (see part (2) of lemma
\ref{lemBddEdge}). This proves (1a). 

Consider (1b). First note, part (2) of lemma \ref{lemBddEdge} implies that
$(\chi,\eta) \in \partial \LL \subset [a,b] \times [0,1]$, and
\begin{equation*}
\chi = t + \frac{C(t)}{C'(t)}
\hspace{0.5cm} \text{and} \hspace{0.5cm}
\eta = 1 + \frac{C(t)^2}{C'(t)}.
\end{equation*}
Next note, since $t \in R^+$, equation (\ref{eqR2}) implies that
$t \in \R \setminus \supp(\mu)$ and $C(t) > 0$. Also, since
$t \in \R \setminus \supp(\mu)$, equation (\ref{eqCauTrans}) implies
that $C'(t) < 0$. The first part of the above equation thus implies
that $t > \chi$. Also, equation (\ref{eqf'Rmu}) holds as above,
for all $w \in \C \setminus (\supp(\mu) \cup \{\chi\})$. Substitute the
above expressions for $\chi$ and $\eta$ into equation (\ref{eqf'Rmu})
to get $f'(t)= f''(t) = 0$. Therefore,
$(\chi,\eta) \in [a,b] \times [0,1]$, and $f'$ has a
repeated root in $t \in (\chi,+\infty) \setminus \supp(\mu)$.
Definition \ref{defEdge} thus implies that $(\chi,\eta) \in \EE^+$,
and the definition of $W_\EE$ given in the statement of this
theorem gives $W_\EE(\chi,\eta) = t$. This proves (1b).
\end{proof}

Note (see part (2) of lemma \ref{lemBddEdge}) that
$(\chi_\EE(\cdot),\eta_\EE(\cdot)) : R \to \partial \LL$
is continuous in any open sub-interval of $R$. Also recall that definitions \ref{defEdge}
and \ref{defEdge3} for  $\EE$ are equivalent. Note, theorem
\ref{thmEdge} uses these definitions, and shows that
$(\chi_\EE(\cdot),\eta_\EE(\cdot)) : R \to \partial \LL$ bijectively
maps $R$ to $\EE$. Finally recall that definition \ref{defEdge0}
defines $\EE$ as the image of
$(\chi_\EE(\cdot),\eta_\EE(\cdot)) : R \to \partial \LL$. Therefore:
\begin{cor}
\label{corEdgeDefEquiv}
Definitions \ref{defEdge0}, \ref{defEdge}, \ref{defEdge3} of $\EE$ are equivalent.
\end{cor}

The curve, $(\chi_\EE(\cdot),\eta_\EE(\cdot)) : R \to \EE$, is called the
{\em edge curve}. We now consider the geometric behaviour
of the edge curve. Fix $(\chi,\eta) \in \EE$ and the corresponding $t \in R$ 
with $(\chi,\eta) = (\chi_\EE(t), \eta_\EE(t))$, and again denote $f'_{(\chi,\eta)}$
simply by $f'$. Recall that $t = W_\EE(\chi,\eta)$ (see theorem \ref{thmEdge}),
and let $m = m(t)$ denote the multiplicity of $t$ as a root of $f'$ (see definition
\ref{defEdge3}). Note, definition \ref{defEdge3} and theorem \ref{thmEdge} imply that
the following exhaust all possibilities:
\begin{itemize}
\item
$t \in R^+ \cup R^-$, $(\chi,\eta) \in \EE^+ \cup \EE^-$, and $m \in \{2,3\}$.
\item
$t \in R_0$, $(\chi,\eta) \in \EE_0$, and $m = 1$.
\item
$t \in R_1$, $(\chi,\eta) \in \EE_1$, and $m \in \{0,1\}$.
\end{itemize}
Above, $m=0$ means that $f'(t) \neq 0$. We now show how the local
geometric behaviour of the edge curve in a neighbourhood of $(\chi,\eta)$
depends on $m$:
\begin{lem}
\label{lemLocGeo}
Define $t, (\chi,\eta), m$, as above. Also define the (un-normalised)
orthogonal vectors, $\mathbf{x} = \mathbf{x}(t)$ and $\mathbf{y} = \mathbf{y}(t)$ as,
\begin{itemize}
\item
$\mathbf{x} := (1,C(t))$ and $\mathbf{y} := (C(t),-1)$ when
$t \in R^+ \cup R^-$.
\item
$\mathbf{x} := (1,0)$ and $\mathbf{y} := (0,1)$ when $t \in R_0$.
\item
$\mathbf{x} := (0,1)$ and $\mathbf{y} := (1,0)$ when $t \in R_1$.
\end{itemize}
Write,
\begin{equation}
\label{eqrtvt}
(\chi_\EE(s),\eta_\EE(s)) - (\chi,\eta)
= a(s) \mathbf{x} + b(s) \mathbf{y},
\end{equation}
for all $s \in R$ sufficiently close to $t$. Then,
\begin{align}
\label{eqas}
a(s) &= a_1 (s-t) + a_2 (s-t)^2 + O((s-t)^3), \\
\label{eqbs}
b(s) &= b_1 (s-t)^2 + b_2 (s-t)^3 + O((s-t)^4),
\end{align}
where $a_1 = a_1(t)$, $a_2 = a_2(t)$, $b_1 = b_1(t)$ and $b_2 = b_2(t)$
satisfy the following:
\begin{itemize}
\item
$a_1 \neq 0$ and $b_1 \neq 0$ when $t \in R^+ \cup R^-$ and $m = 2$.
Similarly when $t \in R_0$ and $m = 1$, and when $t \in R_1$ and $m = 0$.
Expressions for $a_1, b_1$ for these cases are given in equations
(\ref{eqa1b1R+R-}, \ref{eqa1b1R0}, \ref{eqa1b1R1}).
\item
$a_1 = b_1 = 0$, $a_2 \neq 0$ and $b_2 \neq 0$ when $t \in R^+ \cup R^-$ and $m = 3$.
Similarly when $t \in R_1$ and $m = 1$. Expressions for $a_2, b_2$ for these cases
are given in equations (\ref{eqa2b2R+R-}, \ref{eqa2b2R1}).
\end{itemize}
\end{lem}

\begin{proof}
Consider equations (\ref{eqas}, \ref{eqbs}) when $t \in R^+ \cup R^-$.
Fix an interval, $I := (t_2,t_1)$, with $t \in I$, $I \subset R^+$ when
$t \in R^+$, and $I \subset R^-$ when $t \in R^-$.
Recall that $(\chi,\eta) = (\chi_\EE(t), \eta_\EE(t))$, and solve equation
(\ref{eqrtvt}) to get,
\begin{align*}
(1 + C(t)^2) a(s)
&= (\chi_\EE(s) - \chi_\EE(t)) + (\eta_\EE(s) - \eta_\EE(t)) C(t), \\
(1 + C(t)^2) b(s)
&= (\chi_\EE(s) - \chi_\EE(t)) C(t) - (\eta_\EE(s) - \eta_\EE(t)),
\end{align*}
for all $s \in I$. Also, part (2) of lemma \ref{lemBddEdge} gives,
\begin{equation}
\label{eqchi'eta'}
\chi_\EE'(t) = 2 - \frac{C''(t) C(t)}{C'(t)^2}
\hspace{0.3cm} \text{and} \hspace{0.3cm}
\eta_\EE'(t) = \chi_\EE'(t) C(t).
\end{equation}
Note, the second part of this equation and Taylor expansions give
equations (\ref{eqas}, \ref{eqbs}) with:
\begin{itemize}
\item
$a_1 = \chi_\EE'(t)$.
\item
$2 a_2 = \chi_\EE''(t) + \chi_\EE'(t) C'(t) C(t) (1 + C(t)^2)^{-1}$.
\item
$2 b_1 = - \chi_\EE'(t) C'(t) (1 + C(t)^2)^{-1}$.
\item
$6 b_2 = - (2 \chi_\EE''(t) C'(t) + \chi_\EE'(t) C''(t))  (1 + C(t)^2)^{-1}$.
\end{itemize}

Consider equations (\ref{eqas}, \ref{eqbs}) when $t \in R_0$.
Recall (see equation (\ref{eqR2})) that $C(t) = 0$, and there
exists an interval, $I = (t_2,t_1)$, with
$t \in I \subset \R \setminus \supp(\mu)$. Note,
solving equation (\ref{eqrtvt}) gives,
\begin{equation*}
a(s) = \chi_\EE(s) - \chi_\EE(t)
\hspace{0.5cm} \text{and} \hspace{0.5cm}
b(s) = \eta_\EE(s) - \eta_\EE(t),
\end{equation*}
for all $s \in I$. Also note, similar to above, equation (\ref{eqchi'eta'})
holds. Moreover, since $C(t) = 0$, this equation gives $\chi_\EE'(t) = 2$,
$\eta'(t) = 0$, and $\eta''(t) = 2 C'(t)$. Taylor expansions then give
equations (\ref{eqas}, \ref{eqbs}) with $a_1 = 2$ and $b_1 = C'(t)$.
We ignore $a_2$ and $b_2$ here.

Consider equations (\ref{eqas}, \ref{eqbs}) when $t \in R_1$.
Recall (see equation (\ref{eqR2})) that $\mu[\{t\}] > 0$ and
there exists an open interval $I \subset \R$ with $t \in I$ and
$I \setminus \{t\} \subset \R \setminus \supp(\mu)$.
Note, solving equation (\ref{eqrtvt}) gives,
\begin{equation*}
a(s) = \eta_\EE(s) - \eta_\EE(t)
\hspace{0.5cm} \text{and} \hspace{0.5cm}
b(s) = \chi_\EE(s) - \chi_\EE(t),
\end{equation*}
for all $s \in I$. Next note, since $\mu[\{t\}] > 0$, equation (\ref{eqCauTrans})
gives,
\begin{equation*}
C(w) = \frac{\mu[\{t\}]}{w-t} + C_I(w),
\end{equation*}
for all $w \in \C \setminus \supp(\mu)$
where $C_I(w) := \int_{[a,b] \setminus I} \frac{\mu[dx]}{w-x}$.
Part (2) of lemma \ref{lemBddEdge} then gives,
\begin{align*}
\chi_\EE(s)
&= s - (s-t) \frac{\mu[\{t\}] + (s-t) C_I(s)}{\mu[\{t\}] - (s-t)^2 C_I'(s)}, \\
\eta_\EE(s)
&= 1 - \frac{(\mu[\{t\}] + (s-t) C_I(s))^2}{\mu[\{t\}] - (s-t)^2 C_I'(s)},
\end{align*}
for all $s \in I$. Taylor expansions then give
equations (\ref{eqas}, \ref{eqbs}) with:
\begin{itemize}
\item
$a_1 = -2 C_I(t)$.
\item
$a_2 = -3 C_I'(t) - C_I(t)^2 \mu[\{t\}]^{-1}$.
\item
$b_1 = - C_I(t) \mu[\{t\}]^{-1}$.
\item
$b_2 = - 2 C_I'(t) \mu[\{t\}]^{-1}$.
\end{itemize}

Consider $a_1,b_1$ when $t \in R^+ \cup R^-$ and $m \in \{2,3\}$. First,
proceed as in the proof of part (1a) in theorem \ref{thmEdge}
to get $f'(w) = C(w) - (1-\eta)/(w-\chi)$ for all
$w \in \C \setminus (\supp(\mu) \cup \{\chi\})$. Differentiating
and taking $w = t$ (recall $t \in \R \setminus (\supp(\mu) \cup \{\chi\}$
by definition \ref{defEdge3}) gives $f'''(t) = C''(t) - 2(1-\eta)/(t-\chi)^3$
and $f^{(4)}(t) = C'''(t) + 6(1-\eta)/(t-\chi)^4$.
Next recall that $(\chi,\eta) = (\chi_\EE(t), \eta_\EE(t))$ (see statement
of this lemma). Part (2) of lemma \ref{lemBddEdge} thus gives,
\begin{equation}
\label{eqf3f4R+R-}
f'''(t) = C''(t) - 2 \frac{C'(t)^2}{C(t)}
\hspace{0.5cm} \text{and} \hspace{0.5cm}
f^{(4)} = C'''(t) - 6 \frac{C'(t)^3}{C(t)^2}.
\end{equation}
Next note, since $t \in \R \setminus \supp(\mu)$, equation (\ref{eqCauTrans}) gives
$C'(t) < 0$. Equations (\ref{eqchi'eta'}, \ref{eqf3f4R+R-}) then give,
\begin{equation*}
\chi_\EE'(t) = - \frac{C(t)}{C'(t)^2} f'''(t)
\hspace{0.5cm} \text{and} \hspace{0.5cm}
\eta_\EE'(t) = - \frac{C(t)^2}{C'(t)^2} f'''(t).
\end{equation*}
Moreover, the expressions for $a_1,b_1$ (see above) then give,
\begin{equation}
\label{eqa1b1R+R-}
a_1 = - \frac{C(t)}{C'(t)^2} f'''(t)
\hspace{0.5cm} \text{and} \hspace{0.5cm}
b_1 = \frac{C(t)}{2 C'(t) (1 + C(t)^2)} f'''(t).
\end{equation}
Finally recall $C(t) \neq 0$ since $t \in R^+ \cup R^-$ (see equation (\ref{eqR2})),
$C'(t) < 0$, and $m \in \{2,3\}$ is the multiplicity of $t$ as a root of $f'$ (see
statement of this lemma). Therefore $a_1 \neq 0$ and $b_1 \neq 0$ when $t \in R^+ \cup R^-$
and $m=2$, and $a_1 = b_1 = 0$ when $t \in R^+ \cup R^-$ and $m=3$.

Consider $a_2,b_2$ when $t \in R^+ \cup R^-$ and $m=3$. First note,
equation (\ref{eqf3f4R+R-}) again holds. Therefore, since $m=3$ and so
$f'''(t) = 0$,
\begin{equation*}
C''(t) = 2 \frac{C'(t)^2}{C(t)}
\hspace{0.5cm} \text{and} \hspace{0.5cm}
C'''(t) = f^{(4)} + 6 \frac{C'(t)^3}{C(t)^2}.
\end{equation*}
Substitute these into the expressions for $a_2,b_2$ (see above) to get,
\begin{equation}
\label{eqa2b2R+R-}
a_2 = - \frac{C(t)}{2 C'(t)^2} f^{(4)}(t)
\hspace{0.5cm} \text{and} \hspace{0.5cm}
b_2 = \frac{C(t)}{3 C'(t) (1 + C(t)^2)} f^{(4)}(t).
\end{equation}
Finally recall $C(t) \neq 0$, $C'(t) < 0$, and $m = 3$ and so $f^{(4)}(t) \neq 0$.
Therefore $a_2 \neq 0$ and $b_2 \neq 0$ when $t \in R^+ \cup R^-$
and $m=3$.

Consider $a_1,b_1$ when $t \in R_0$. First recall (see above) that,
\begin{equation}
\label{eqa1b1R0}
a_1 = 2
\hspace{0.5cm} \text{and} \hspace{0.5cm}
b_1 = C'(t).
\end{equation}
Next note, since $t \in \R \setminus \supp(\mu)$ (see equation (\ref{eqR2})),
equation (\ref{eqCauTrans}) gives $C'(t) < 0$. Therefore
$a_1 \neq 0$ and $b_1 \neq 0$ when $t \in R_0$.

Consider $a_1,b_1$ when $t \in R_1$ and $m \in \{0,1\}$. Recall
(see equation (\ref{eqR2})) that $t \in \supp(\mu)$, $\mu[\{t\}] > 0$,
and there exists an open interval,
$I \subset \R$ with $t \in I$ and $I \setminus \{t\} \subset \R \setminus \supp(\mu)$.
Moreover $(\chi,\eta) = (t,1-\mu[\{t\}])$ since $(\chi,\eta) \in \EE_1$ (see
definition \ref{defEdge3} and theorem \ref{thmEdge}).
Equation (\ref{eqf'}) then gives $f'(w) = C_I(w)$ for all
$w \in (\C \setminus \R) \cup I$, where
$C_I(w) := \int_{[a,b] \setminus I} \frac{\mu[dx]}{w-x}$. Therefore,
\begin{equation}
\label{eqf1f2R1}
f'(t) = C_I(t)
\hspace{0.5cm} \text{and} \hspace{0.5cm}
f''(t) = C_I'(t).
\end{equation}
The expressions for $a_1,b_1$ (see above) then give,
\begin{equation}
\label{eqa1b1R1}
a_1 = -2 f'(t)
\hspace{0.5cm} \text{and} \hspace{0.5cm}
b_1 = - \frac{f'(t)}{\mu[\{t\}]}.
\end{equation}
Finally recall that $m \in \{0,1\}$ is the multiplicity of $t$ as a root of $f'$ (see
statement of this lemma). Therefore $a_1 \neq 0$ and $b_1 \neq 0$ when $t \in R_1$
and $m=0$, and $a_1 = b_1 = 0$ when $t \in R_1$ and $m=1$.

Consider $a_2,b_2$ when $t \in R_1$ and $m=1$. First note,
equation (\ref{eqf1f2R1}) again holds. Therefore, since $m=1$ and so
$f'(t) = 0$,
\begin{equation*}
C_I(t) = 0
\hspace{0.5cm} \text{and} \hspace{0.5cm}
C_I'(t) = f''(t).
\end{equation*}
Substitute these into the expressions for $a_2,b_2$ (see above) to get,
\begin{equation}
\label{eqa2b2R1}
a_2 = -3 f''(t)
\hspace{0.5cm} \text{and} \hspace{0.5cm}
b_2 = - \frac{2 f''(t)}{\mu[\{t\}]}.
\end{equation}
Finally recall that $m = 1$ and so $f''(t) \neq 0$. Therefore
$a_2 \neq 0$ and $b_2 \neq 0$ when $t \in R_1$ and $m=1$.
\end{proof}

Note that equations (\ref{eqCauTrans}, \ref{eqR2}) imply that
$(b,+\infty) \subset R^+$. We end this section by considering
the edge restricted to this interval:
\begin{lem}
\label{lemLowRigEdge}
Recall $(b,+\infty) \subset R^+$ and consider $(\chi_\EE(\cdot),\eta_\EE(\cdot)) : (b,+\infty) \to \EE^+ \subset \EE$:
\begin{enumerate}
\item
$\chi_\EE : (b, +\infty) \to [a,b]$ is strictly decreasing with
$\lim_{t \uparrow +\infty} \chi_\EE(t) = \mu_1 := \int_a^b x \mu[dx] \in (a,b)$.
Moreover, when $\mu[\{b\}] > 0$, $\lim_{t \downarrow b} \chi_\EE(t) = b$.
\item
$\eta_\EE : (b, +\infty) \to [0,1] $ is strictly
decreasing with $\lim_{t \uparrow +\infty} \eta_\EE(t) = 0$.
Moreover, when $\mu[\{b\}] > 0$,
$\lim_{t \downarrow b} \eta_\EE(t) = 1-\mu[\{b\}]$.
\item
$\chi_\EE'(\cdot)/\eta_\EE'(\cdot)  : (b, +\infty) \to \R$ is positive and strictly
decreasing with $\lim_{t \uparrow +\infty} \chi_\EE'(t)/\eta_\EE'(t) = 0$.
Moreover, when $\mu[\{b\}] > 0$,
$\lim_{t \downarrow b} \chi_\EE'(t)/\eta_\EE'(t) = +\infty$.
\end{enumerate}
$(\chi_\EE(\cdot),\eta_\EE(\cdot)) : (b,+\infty) \to \EE$, when $\mu[\{b\}] > 0$
is depicted in figure \ref{figU}. Next:
\begin{enumerate}
\setcounter{enumi}{3}
\item
Fix $(\chi,\eta) \in \EE$ and the corresponding $t \in (b,+\infty)$ 
with $(\chi,\eta) = (\chi_\EE(t), \eta_\EE(t))$
(see definition \ref{defEdge3} and theorem \ref{thmEdge}). Then
$f_{(\chi,\eta)}'(s) > 0$ for all $s \in (b,t)$,
$f_{(\chi,\eta)}'(t) = f_{(\chi,\eta)}''(t) = 0$ and $f_{(\chi,\eta)}'''(t) > 0$,
and $f_{(\chi,\eta)}'(s) > 0$ for all $s \in (t,+\infty)$.
\end{enumerate}
\end{lem}

\begin{proof}
Consider (1). First recall that $(b,+\infty) \subset R^+$.
Next recall (see equations (\ref{eqchi'eta'}, \ref{eqf3f4R+R-}))
that $\chi_\EE'(t) = - \frac{C(t)}{C'(t)^2} f'''(t)$ for
all $t \in R^+$. Thus $\chi_\EE'(t) < 0$ for all $t \in (b,+\infty)$
since $C(t) > 0$ ($t \in (b,+\infty) \subset R^+$), and since
$f'''(t) > 0$ (see proof of part (4), below). Next note,
part (2) of lemma \ref{lemBddEdge} and equation (\ref{eqCauTrans}) give,
\begin{equation*}
\chi_\EE(t)
= \frac{t C'(t) + C(t)}{C'(t)}
= \frac{\int_a^b \mu[dx] \frac{x}{(t-x)^2}}{\int_a^b \mu[dx] \frac1{(t-x)^2}},
\end{equation*}
for all $t \in (b,+\infty)$. Therefore
$\lim_{t \uparrow +\infty} \chi_\EE(t) = \int_a^b \mu[dx] x = \mu_1$.
Moreover, when $\mu[\{b\}] > 0$,
\begin{equation*}
\chi_\EE(t)
= \frac{\mu[\{b\}] \frac{b}{(t-b)^2} + \int_{[a,b)} \mu[dx] \frac{x}{(t-x)^2}}
{\mu[\{b\}] \frac1{(t-b)^2} + \int_{[a,b)} \mu[dx] \frac1{(t-x)^2}},
\end{equation*}
for all $t \in (b,+\infty)$. Finally note, since $\mu$ is a
probability measure, $\lim_{\e \downarrow 0} \mu[(b-\e,b)] = 0$, and
so $\int_{[a,b)} \mu[dx] \frac{x}{t-x} = o((t-b)^{-1})$ and
$\int_{[a,b)} \mu[dx] \frac1{(t-x)^2} = o((t-b)^{-1})$ as
$t \downarrow b$. Therefore $\lim_{t \downarrow b} \chi_\EE(t) = b$
when $\mu[\{b\}] > 0$. This proves (1). Part (2,3) follow similarly.

Consider (4). Recall that $t \in (b,+\infty)$, and $f_{(\chi,\eta)}'$
has a root of multiplicity $2$ or $3$ at $t$ (see definition
\ref{defEdge3}). Indeed, since $(b,+\infty) = J_1$ (see equation
(\ref{eqf'domain2})), part (a) of theorem \ref{thmf'} implies that $t$ is
a root of $f_{(\chi,\eta)}'$ multiplicity $2$, and $f_{(\chi,\eta)}'$
has no roots in $(b,+\infty) \setminus \{t\} = (b,t) \cup (t,+\infty)$.
Therefore, it is sufficient to show that there exists an
$s \in (t,+\infty)$ with $f_{(\chi,\eta)}'(s) > 0$. To see this, note
equation (\ref{eqf'0}) gives
\begin{equation*}
f_{(\chi,\eta)}' (s)
= \int_a^b \frac{\mu[dx]}{s-x} - \frac{1-\eta}{s-\chi},
\end{equation*}
for all $s \in (b,+\infty)$. Thus, since
$\mu[a,b] = 1$, $\chi \in (a,b)$ and $\eta \in (0,1)$ (see definition 
\ref{defEdge3}), $\lim_{s \to +\infty} s f_{(\chi,\eta)}'(s) = \eta > 0$.
This proves (4).
\end{proof}

\subsection{Outside the liquid region, $\OO$}
\label{secLowRig}

In this section we additionally assume,
\begin{equation*}
\mu[\{b\}] > 0.
\end{equation*}
We define $\OO$ as in definition \ref{defLowRig}, and we
will prove an analogous result for $\OO$ to theorems \ref{thmBulk} and
\ref{thmEdge}. Again, we denote $f'_{(\chi,\eta)}$ simply by $f'$.
First note, equation (\ref{eqf'0}) gives
\begin{equation}
\label{eqf'4}
f'(w) = C(w) - \frac{1-\eta}{w - \chi},
\end{equation}
for all $w \in (\C \setminus \R) \cup (b,+\infty)$. Next note, since
$(b,+\infty) = J_1$ (see equation (\ref{eqf'domain2})), definition
\ref{defLowRig} and part (1) of corollary \ref{corf'} imply the following,
more refined, definition of $\OO$: $\OO$ is the set of all
$(\chi,\eta) \in (a,b) \times (0,1)$ for which $1 - \eta > \mu[\{\chi\}]$,
$f'$ has a root of multiplicity $1$ in $(b,+\infty)$, and $f'$ has at
most $2$ roots in $(b,+\infty)$ counting multiplicities. Also, since
$\mu[\{b\}] > 0$, equations (\ref{eqCauTrans}, \ref{eqf'4}) give,
\begin{equation*}
f'(w) = \frac{\mu[\{b\}]}{w - b} + \int_{[a,b)} \frac{\mu[dx]}{w-x}
- \frac{1-\eta}{w - \chi},
\end{equation*}
for all $w \in (\C \setminus \R) \cup (b,+\infty)$. Then, since $\mu[a,b] = 1$,
$\mu[\{b\}] > 0$, $\chi < b$, and $\eta > 0$,
\begin{equation*}
\lim_{w \in (b,+\infty), w \downarrow b} f'(w) = +\infty
\hspace{.5cm} \text{and} \hspace{.5cm}
\lim_{w \in (b,+\infty), w \uparrow +\infty} w f'(w) = \eta > 0.
\end{equation*}
It easily follows that $f'$ has an even number of roots in $(b,+\infty)$,
counting multiplicities. Therefore, we can further refine the definition
of $\OO$:
\begin{definition}
\label{defLowRig2}
When $\mu[\{b\}] > 0$, $\OO$ is the set of all
$(\chi,\eta) \in (a,b) \times (0,1)$ for which
$1 - \eta > \mu[\{\chi\}]$, $f'$ has $2$ distinct
roots of multiplicity $1$ in $(b,+\infty)$, and $f'$
has no other roots in $(b,+\infty)$.
\end{definition}
Corollary \ref{corf'} implies that $\{\LL, \EE, \OO\}$ are pairwise disjoint.
Next, as in theorem \ref{thmBulk}, we prove that each point in $\OO$
maps homeomorphically to its corresponding pair of roots:
\begin{thm}
\label{thmLowRig}
Define $\angle := \{(t,s) \in (b,+\infty)^2 : t > s\}$.
Let $W_\OO : \OO \to \angle$ map each $(\chi,\eta) \in \OO$ to
the corresponding pair of roots of $f'$ in $(b,+\infty)$.
Then $W_\OO : \OO \to \R$ is a homeomorphism with inverse
$(\chi_\OO(\cdot,\cdot),\eta_\OO(\cdot,\cdot)) : \angle \to \OO$
given by,
\begin{equation*}
\chi_\OO(t,s) = \frac{t C(t) -  s C(s)}{C(t) - C(s)}
\hspace{0.5cm} \text{and} \hspace{0.5cm}
\eta_\OO(t,s) = 1 + \frac{C(t) C(s) (t-s)}{C(t) - C(s)}.
\end{equation*}
\end{thm}

\begin{proof}
We prove this result by proving the analogues of parts (i-vi) in the
proof of theorem \ref{thmBulk}. We will be more brief here,
highlighting the differences only when necessary.

Consider (i). Fix $(t,s) \in \angle$ and define
$(\chi,\eta) := (\chi_\OO(t,s),\eta_\OO(t,s))$. First note, the
definitions of $\chi = \chi_\OO(t,s)$ and $\eta = \eta_\LL(t,s)$
and equation (\ref{eqf'4}) trivially imply that $f'(t) = f'(s) = 0$.
Next, proceed similarly to part (ib) in the proof of theorem
\ref{thmBulk} to get $\chi = \mu_1 + O (t^{-1})$ and
$\eta =  (\mu_2 - \mu_1^2)/(ts) + O(t^{-3})$ whenever $t \in (b,+\infty)$
is sufficiently large and $s^{-1} = O(t^{-1})$, where $\mu_1 := \int_a^b x \mu[dx]$
and $\mu_2 := \int_a^b x^2 \mu[dx]$. Finally recall that $b > \mu_1 > a$
and $\mu_2 - \mu_1^2 > 0$ (see part (ib) in the proof of theorem
\ref{thmBulk}). Therefore $f'(t) = f'(s) = 0$ and
$(\chi,\eta) \in (a,b) \times (0,1)$ whenever $t \in (b,+\infty)$ is
sufficiently large and $s^{-1} = O(t^{-1})$. Definition \ref{defLowRig} then
implies that $(\chi,\eta) \in \OO$. This proves (i).

Consider (ii). Fix $(\chi_1,\eta_1), (\chi_2,\eta_2) \in (a,b) \times (0,1)$
with $(\chi_1,\eta_1) \in \OO$. Define $f_1'(w) := C(w) - (1-\eta_1)/(w-\chi_1)$
and $f_2'(w) := C(w) - (1-\eta_2)/(w-\chi_2)$ for all
$w \in (\C \setminus \R) \cup (b,+\infty)$. Let $(t_1,s_1) \in \angle$ denote
the unique pair roots of $f_1'$ in $(b,+\infty)$ (see definition \ref{defLowRig2}). Fix
$\e>0$ such that $t_1$ is the unique root of $f_1'$ in $B(t_1,2\e)$,
$s_1$ is the unique root of $f_1'$ in $B(s_1,2\e)$,
$B(t_1,2\e) \cap B(s_1,2\e) = \emptyset$, and
$B(t_1,2\e) \cup B(s_1,2\e) \subset (\C \setminus \R) \cup (b,+\infty)$.
Then, whenever $|\chi_1-\chi_2|$ and $|\eta_1-\eta_2|$ are
sufficiently small, proceed as in part (ii) in the proof of theorem \ref{thmBulk} to show
that $f_2'$ has exactly $1$ root in $B(t_1,\e)$, counting multiplicities,
and exactly $1$ root in $B(s_1,\e)$. Denote these by $t_2$ and $s_2$ respectively,
and note that $t_2 \neq s_2$ since $B(t_1,2\e) \cap B(s_1,2\e) = \emptyset$.
Next note that roots of $f_2'$ occur in
complex conjugate pairs, and so we must have $t_2 \in (t_1-\e, t_1+\e) \subset (b,+\infty)$ and
$s_2 \in (s_1-\e, s_1+\e) \subset (b,+\infty)$. Definition \ref{defLiq} thus implies that
$(\chi_2,\eta_2) \in \OO$ whenever $|\chi_1-\chi_2|$ and $|\eta_1-\eta_2|$ are
sufficiently small. This proves (ii).

Consider (iii). This follows from similar arguments to those used
to prove part (iii) in theorem \ref{thmBulk}. Consider (iv). Fix
$(\chi_1,\eta_1), (\chi_2,\eta_2) \in \OO$ with
$W_\OO(\chi_1,\eta_1) = W_\OO(\chi_2,\eta_2) = (t,s) \in \angle$.
Definition \ref{defLowRig2}, the definition of $W_\OO$
(see statement of this theorem), and equation (\ref{eqf'4}), then give,
\begin{equation*}
C(t)
= \frac{1-\eta_1}{t-\chi_1}
= \frac{1-\eta_2}{t-\chi_2}
\hspace{.5cm} \text{and} \hspace{.5cm}
C(s)
= \frac{1-\eta_1}{s-\chi_1}
= \frac{1-\eta_2}{s-\chi_2}.
\end{equation*}
Therefore $(\eta_2-\eta_1) t = (1-\eta_1) \chi_2 - (1-\eta_2) \chi_1$
and $(\eta_2-\eta_1) s = (1-\eta_1) \chi_2 - (1-\eta_2) \chi_1$.
Then $t = s$ whenever $\eta_1 \neq \eta_2$, which contradicts
$(t,s) \in \angle$. Thus $\eta_1 = \eta_2$, and so $(1-\eta_1) (\chi_1 - \chi_2) = 0$.
Finally, $\eta_1 < 1$ since $(\chi_1,\eta_1) \in \OO$ (see definition \ref{defLowRig2}),
and so $\chi_1 = \chi_2$. This proves (iv).

Consider (v). Fix $(\chi,\eta) \in \OO$ and let $(t,s) := W_\OO(\chi,\eta)$.
Definition \ref{defLowRig2}, the definition of $W_\OO$, and equation (\ref{eqf'4})
then give $1-\eta = (t-\chi) C(t) = (s-\chi) C(s)$.
Solving gives $(\chi,\eta) = (\chi_\OO(t,s),\eta_\OO(t,s))$. This proves (v).
Consider (vi). This follows from similar arguments to those used
to prove part (vi) in theorem \ref{thmBulk}.
\end{proof}

Note, theorem \ref{thmLowRig} implies that $\OO$ is a non-empty, open,
simply connected subset of $(a,b) \times (0,1)$. We end this section by
proving analogous results to lemma \ref{lemBddEdge} and \ref{lemLowRigEdge}.
We will be more brief here. As we will see, $\OO$ is that open region
bounded by $(\chi_\EE(\cdot),\eta_\EE(\cdot)) |_{(b, +\infty)}$ and
the bounding box of $[a,b] \times [0,1]$ in figure \ref{figU}: 
\begin{lem}
\label{lemLowRig}
Consider $\partial \OO$:
\begin{enumerate}
\item
$(\chi_\EE(t),\eta_\EE(t)) \in \partial \OO$ for all $t \in (b,+\infty)$.
Moreover, $(\chi_\OO(t_k,s_k),\eta_\OO(t_k,s_k)) \to (\chi_\EE(t),\eta_\EE(t))$
as $k \to \infty$ for all $t \in (b,+\infty)$ and 
$\{(t_k,s_k)\}_{k\ge1} \subset \angle$ with $(t_k,s_k) \to (t,t) \in \partial \angle$.
\item
$(g(s),0) \in \partial \OO$ for all $s \in (b,+\infty)$ where
$g(s) := s - C(s)^{-1}$ for all $s \in (b,+\infty)$. Moreover,
$g : (b,+\infty) \to \R$ is strictly decreasing with
$\lim_{s \uparrow +\infty} g(s) = \mu_1 := \int_a^b x \mu[dx]$
and $\lim_{s \downarrow b} g(s) = b$. 
Finally, $(\chi_\OO(t_k,s_k),\eta_\OO(t_k,s_k)) \to (g(s),0)$ as
$k \to \infty$ for all $s \in (b,+\infty)$ and
$\{(t_k,s_k)\}_{k\ge1} \subset \angle$ with
$(t_k,s_k) \to (+\infty,s) \in \partial \angle$.
\item
$(b,h(t)) \in \partial \OO$ for all $t \in (b,+\infty)$,
where $h(t) := 1 - (t-b) C(t)$ for all $t \in (b,+\infty)$.
Moreover, $h : (b,+\infty) \to \R$ is strictly decreasing with
$\lim_{t \uparrow +\infty} h(t) = 0$ and
$\lim_{t \downarrow b} h(t) = 1 - \mu[\{b\}]$. Finally, 
$(\chi_\OO(t_k,s_k),\eta_\OO(t_k,s_k)) \to (b,h(t))$
as $k \to \infty$ for all $t \in (b,+\infty)$ and
$\{(t_k,s_k)\}_{k\ge1} \subset \angle$ with
$(t_k,s_k) \to (t,b) \in \partial \angle$.
\end{enumerate}
Moreover:
\begin{enumerate}
\setcounter{enumi}{3}
\item
Fix $(\chi,\eta) \in \OO$ and the corresponding $(t,s) \in \angle$ 
with $(\chi,\eta) = (\chi_\OO(t,s), \eta_\OO(t,s))$
(see definition \ref{defEdge3} and theorem \ref{thmLowRig}). Then
$f_{(\chi,\eta)}'(y) > 0$ for all $y \in (b,s)$,
$f_{(\chi,\eta)}'(s) = 0$ and $f_{(\chi,\eta)}''(s) < 0$,
$f_{(\chi,\eta)}'(y) < 0$ for all $y \in (s,t)$,
$f_{(\chi,\eta)}'(t) = 0$ and $f_{(\chi,\eta)}''(t) > 0$,
and $f_{(\chi,\eta)}'(y) > 0$ for all $y \in (t,+\infty)$.
The resulting behaviour of the real-valued function
$y \mapsto f_{(\chi,\eta)}(y)$ for all $y \in (b,\infty)$
is shown in figure \ref{figRef}.
\end{enumerate}
\end{lem}

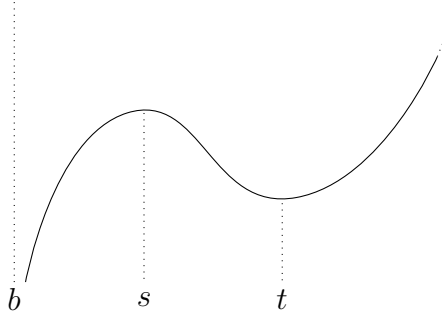
\begin{figure}[t]
\centering
\begin{tikzpicture}[scale=0.75]
\draw [dotted] (0,0) --++(0,5);
\draw (0,-.3) node {$b$};
\draw plot [smooth, tension=1] coordinates
{ (.2,0) (2,3) (5,1.5) (7.5,4)};
\draw [dotted] (7.5,4) --++(.2,.5);
\draw [dotted] (2.3,3) --++(0,-3);
\draw (2.3,-.3) node {$s$};
\draw [dotted] (4.75,1.4) --++(0,-1.4);
\draw (4.75,-.3) node {$t$};

\end{tikzpicture}
\caption{The behaviour of $y \mapsto f_{(\chi,\eta)}(y)$ for $y \in (b,\infty)$,
for $(\chi,\eta) \in \OO$ and $(t,s) \in \angle$
with $(\chi,\eta) = (\chi_\OO(t,s), \eta_\OO(t,s))$ when $\mu[\{b\}] > 0$.
The function is strictly increasing in $(b,s)$, strictly decreasing in
$(s,t)$, and strictly increasing in $(t,\infty)$.}
\label{figRef}
\end{figure}

\begin{proof}
Consider (1). Fix $t \in (b,+\infty)$ and
$\{(t_k,s_k)\}_{k\ge1} \subset \angle$
with $(t_k,s_k) \to (t,t) \in \partial \angle$.
Write (see theorem \ref{thmLowRig}),
\begin{align*}
\chi_\OO(t_k,s_k)
&= t_k + C(s_k) \frac{t_k-s_k}{C(t_k) - C(s_k)}, \\
\eta_\OO(t_k,s_k)
&= 1 + C(t_k) C(s_k) \frac{t_k-s_k}{C(t_k) - C(s_k)}.
\end{align*}
Thus, since $t_k,s_k \to t \in (b,+\infty)$ as $k \to \infty$, and
$C$ is analytic in $(b,+\infty)$,
\begin{align*}
\chi_\OO(t_k,s_k)
&\to t + C(t) \frac1{C'(t)} = \chi_\EE(t)
\hspace{0.5cm} \text{as } k \to \infty, \\
\eta_\OO(t_k,s_k)
&\to 1 + C(t) C(t) \frac1{C'(t)} = \eta_\EE(t)
\hspace{0.5cm} \text{as } k \to \infty.
\end{align*}
This proves (1).

Consider (2). First note, $g'(s) = (C(s)^2 + C'(s))/C(s)^2$ for all $s \in (b,+\infty)$.
Write as $C(s)^2 g'(s) = C(s) C(s) + \tfrac12 C'(s) + \tfrac12 C'(s)$, and use equation
(\ref{eqCauTrans}) to get,
\begin{align*}
C(s)^2 g'(s)
&= \left( \int_a^b \frac{\mu[dx]}{s-x} \right) \left( \int_a^b \frac{\mu[dy]}{s-y} \right)
- \frac12 \int_a^b \frac{\mu[dx]}{(s-x)^2} - \frac12 \int_a^b \frac{\mu[dy]}{(s-y)^2} \\
&= - \frac12 \int_a^b \mu[dx] \int_a^b \mu[dy] \left( \frac1{s-x} - \frac1{s-y} \right)^2.
\end{align*}
Thus $g'(s) < 0$ for all $s \in (b,+\infty)$, and so $g$ is strictly decreasing.
Next write (see equation (\ref{eqCauTrans})),
\begin{equation*}
g(s) = \frac{s C(s) - 1}{C(s)}
= \frac{\int_a^b \mu[dx] \frac{x}{s-x}}{\int_a^b \mu[dx] \frac1{s-x}}.
\end{equation*}
Therefore $\lim_{s \uparrow +\infty} g(s) = \int_a^b \mu[dx] x = \mu_1$.
Next note, since $\mu[\{b\}] > 0$, we can write:
\begin{equation*}
g(s)
= \frac{\mu[\{b\}] \frac{b}{s-b} + \int_{[a,b)} \mu[dx] \frac{x}{s-x}}
{\mu[\{b\}] \frac1{s-b} + \int_{[a,b)} \mu[dx] \frac1{s-x}}
= \frac{\mu[\{b\}] \frac{b}{s-b} + \int_{[a,b)} \mu[dx] \frac{x}{s-x}}
{\mu[\{b\}] \frac1{s-b} + \int_{[a,b)} \mu[dx] \frac1{s-x}},
\end{equation*}
for all $s \in (b,+\infty)$.
Also, since $\lim_{\e \downarrow 0} \mu[(b-\e,b)] = 0$,
$\int_{[a,b)} \mu[dx] \frac1{s-x} = o((s-b)^{-1})$ and
$\int_{[a,b)} \mu[dx] \frac{x}{s-x} = o((s-b)^{-1})$ as
$s \downarrow b$. Therefore,
\begin{equation*}
g(s)
= \frac{\mu[\{b\}] \frac{b}{s-b} + o(\frac1{s-b})}
{\mu[\{b\}] \frac1{s-b} + o(\frac1{s-b})} \to b
\hspace{.5cm} \text{as } s \downarrow b.
\end{equation*}
Finally, fix $s \in (b,+\infty)$ and $\{(t_k,s_k)\}_{k\ge1} \subset \angle$
with $(t_k,s_k) \to (+\infty,s) \in \partial \angle$. Recall
(see theorem \ref{thmLowRig}),
\begin{equation*}
\chi_\OO(t_k,s_k) = \frac{t_k C(t_k) -  s_k C(s_k)}{C(t_k) - C(s_k)}
\hspace{0.5cm} \text{and} \hspace{0.5cm}
\eta_\OO(t_k,s_k) = 1 + \frac{C(t_k) C(s_k) (t_k-s_k)}{C(t_k) - C(s_k)}.
\end{equation*}
Therefore, since $t_k \to +\infty$ and $s_k \to s \in (b,+\infty)$ as $k \to \infty$,
equation (\ref{eqCauTrans}) gives the following for all $k$ sufficiently large:
\begin{align*}
\chi_\OO(t_k,s_k)
&= \frac{t_k (\frac1{t_k} + O(\frac1{t_k^2})) -  (s C(s) + O(|s_k-s|))}
{O(\frac1{t_k}) - (C(s) + O(|s_k-s|))}, \\
\eta_\OO(t_k,s_k)
&= 1 + \frac{(\frac1{t_k} + O(\frac1{t_k^2})) (C(s) + O(|s_k-s|)) (t_k + O(1))}
{O(\frac1{t_k}) - (C(s) + O(|s_k-s|))}.
\end{align*}
Therefore $\chi_\OO(t_k,s_k) \to (1-s C(s))/(-C(s)) = g(s)$ and
$\eta_\OO(t_k,s_k) \to 1 + (C(s))/(-C(s)) = 0$ as $k \in \infty$.
This proves (2).

Consider (3). First note $h'(t) = - C(t) - (t-b) C'(t)$ for all $t \in (b,+\infty)$.
Equation (\ref{eqCauTrans}) then gives,
\begin{equation*}
h'(t)
= - \int_a^b \frac{\mu[dx]}{t-x} + (t-b) \int_a^b \frac{\mu[dx]}{(t-x)^2}
= \int_a^b \mu[dx] \frac{x-b}{(t-x)^2}.	
\end{equation*}
Thus $h'(t) < 0$ for all $t \in (b,+\infty)$, and so $h$ is strictly decreasing.
Next, write (see equation (\ref{eqCauTrans})),
\begin{equation*}
h(t)
= 1 - (t-b) C(t)
= 1 - \int_a^b \mu[dx] \frac{t-b}{t-x}.
\end{equation*}
Therefore $\lim_{t \uparrow +\infty} h(t) = 1 - 1 = 0$.
Next note, since $\mu[\{b\}] > 0$, we can write:
\begin{equation*}
h(t)
= 1 - \mu[\{b\}] \frac{t-b}{t-b} - \int_{[a,b)} \mu[dx] \frac{t-b}{t-x}
= 1 - \mu[\{b\}] - \int_{[a,b)} \mu[dx] \frac{t-b}{t-x},
\end{equation*}
for all $t \in (b,+\infty)$. Also, since
$\lim_{\e \downarrow 0} \mu[(b-\e,b)] = 0$,
$\int_{[a,b)} \mu[dx] \frac1{t-x} = o((t-b)^{-1})$ as
$t \downarrow b$. Therefore,
$h(t) = 1 - \mu[\{b\}] + o(1) \to 1 - \mu[\{b\}]$ as $t \downarrow b$.
Finally, fix $t \in (b,+\infty)$ and $\{(t_k,s_k)\}_{k\ge1} \subset \angle$
with $(t_k,s_k) \to (t,b) \in \partial \angle$.
Recall (see theorem \ref{thmLowRig}),
\begin{equation*}
\chi_\OO(t_k,s_k) = \frac{t_k C(t_k) -  s_k C(s_k)}{C(t_k) - C(s_k)}
\hspace{0.5cm} \text{and} \hspace{0.5cm}
\eta_\OO(t_k,s_k) = 1 + \frac{C(t_k) C(s_k) (t_k-s_k)}{C(t_k) - C(s_k)}.
\end{equation*}
Then, since $\mu[\{b\}] > 0$, equation (\ref{eqCauTrans}) gives,
\begin{align*}
\chi_\OO(t_k,s_k)
&= \frac{t_k C(t_k) -  s_k (\frac{\mu[\{b\}]}{s_k-b} + \int_{[a,b)} \frac{\mu[dx]}{s_k-x})}
{C(t_k) - (\frac{\mu[\{b\}]}{s_k-b} + \int_{[a,b)} \frac{\mu[dx]}{s_k-x})}, \\
\eta_\OO(t_k,s_k)
&= 1 + \frac{C(t_k) (\frac{\mu[\{b\}]}{s_k-b} + \int_{[a,b)} \frac{\mu[dx]}{s_k-x}) (t_k-s_k)}
{C(t_k) - (\frac{\mu[\{b\}]}{s_k-b} + \int_{[a,b)} \frac{\mu[dx]}{s_k-x})}.
\end{align*}
Therefore, since $t_k \to t \in (b,+\infty)$ and $s_k \to b$ as $k \to \infty$,
and since $\lim_{\e \downarrow 0} \mu[(b-\e,b)] = 0$, the following are satisfied
as $k \to \infty$:
\begin{align*}
\chi_\OO(t_k,s_k)
&= \frac{(t C(t) + o(1)) -  (b + o(1)) (\frac{\mu[\{b\}]}{s_k-b} + o(\frac1{s_k-b}))}
{(C(t) + o(1)) - (\frac{\mu[\{b\}]}{s_k-b} + o(\frac1{s_k-b}))}, \\
\eta_\OO(t_k,s_k)
&= 1 + \frac{(C(t) + o(1)) (\frac{\mu[\{b\}]}{s_k-b} + o(\frac1{s_k-b})) (t-b + o(1))}
{(C(t) + o(1)) - (\frac{\mu[\{b\}]}{s_k-b} + o(\frac1{s_k-b}))}.
\end{align*}
Therefore, when $\mu[\{b\}] > 0$, $\chi_\OO(t_k,s_k) \to b$ and
$\eta_\OO(t_k,s_k) \to 1 - (t-b) C(t) = h(t)$ as $k \to \infty$.
This proves (3).

Consider (4). Recall that $t,s \in (b,+\infty)$ with $t>s$,
$f_{(\chi,\eta)}'$ has a root of multiplicity $1$ at both $t$ and $s$,
and $f_{(\chi,\eta)}'$ has $0$ roots in $(b,+\infty) \setminus \{t,s\}$
(see definition \ref{defLowRig2} and theorem \ref{thmLowRig}). Therefore,
it is sufficient to show that there exists an $y \in (t,+\infty)$ with
$f_{(\chi,\eta)}'(y) > 0$. This follows similarly to the proof of part
(4) of lemma \ref{lemLowRigEdge}.
\end{proof}

Next we prove an analogous result for $\OO$ to
lemma \ref{lemLocGeo} for $\EE$:
\begin{lem}
\label{lemLowRigTay}
Fix $(\chi,\eta) \in \OO$ and the corresponding $(t,s) \in \angle$ 
with $(\chi,\eta) = (\chi_\OO(t,s), \eta_\OO(t,s))$. Define the vectors
$\mathbf{x}(T) :=  (1,C(T))$ for all $T \in (b,+\infty)$. Then,
\begin{align*}
(\chi_\OO(T,S),\eta_\OO(T,S))
&= (\chi,\eta) + (T-t) \; c_1 \; \mathbf{x}(s)
+ (S-s) \; c_2 \; \mathbf{x}(t) \\
&\;\;\;\; + O((|T-t|+|S-s|)^2), 
\end{align*}
for all $(T,S) \in \angle$ with $|T-t|$ and $|S-s|$ sufficiently small,
where $c_1 = c_1(t,s)$ is negative, and $c_2 = c_2(t,s)$ is negative. Expressions
for $c_1$ and $c_2$ are given in equation (\ref{eqc1c2OO}).
\end{lem}

\begin{proof}
This proof is similar to the proof of lemma \ref{lemLocGeo}, and so we
will be brief here. First recall (see theorem \ref{thmLowRig}),
\begin{equation*}
\chi_\OO(T,S) = \frac{T C(T) -  S C(S)}{C(T) - C(S)},
\end{equation*}
for all $(T,S) \in \angle$. Next note, since $\chi = \chi_\OO(t,s)$,
Taylor expansions give,
\begin{align*}
\chi_\OO(T,S) - \chi =
&- \frac{(T-t) (t-s) C(s) C'(t)}{(C(t)-C(s))^2}
+ \frac{(T-t) C(t)}{C(t)-C(s)} \\
&+ \frac{(S-s) (t-s) C(t) C'(s)}{(C(t)-C(s))^2}
- \frac{(S-s) C(s)}{C(t)-C(s)} + O((|T-t|+|S-s|)^2),
\end{align*}
for all $(T,S) \in \angle$ with $|T-t|$ and $|S-s|$ sufficiently small.
Next note, since $(t,s) \in \angle \subset (b,+\infty)^2$, equation
(\ref{eqf'4}) gives $f''(t) = C'(t) + (1-\eta)/(t-\chi)^2$ and
$f''(s) = C'(s) + (1-\eta)/(s-\chi)^2$. Substitute for
$(\chi,\eta) = (\chi_\OO(t,s), \eta_\OO(t,s))$ (see theorem \ref{thmLowRig})
to get,
\begin{equation*}
f''(t) = C'(t) - \frac{C(t) (C(t) - C(s))}{C(s) (t-s)}
\hspace{0.5cm} \text{and} \hspace{0.5cm}
f''(s) = C'(s) - \frac{C(s) (C(t) - C(s))}{C(t) (t-s)}.
\end{equation*}
Substitute $C'(t)$ and $C'(s)$ from the above expressions
into the Taylor expansion to get,
\begin{align*}
\chi_\OO(T,S) - \chi
&= \frac{- (T-t) (t-s) C(s) f''(t) + (S-s) (t-s) C(t) f''(s)}{(C(t)-C(s))^2} \\
&+ O((|T-t|+|S-s|)^2),
\end{align*}
for all $(T,S) \in \angle$ with $|T-t|$ and $|S-s|$ sufficiently small.
Similarly we can show that,
\begin{align*}
\eta_\OO(T,S) - \eta
&= \frac{- (T-t) (t-s) C(s)^2 f''(t) + (S-s) (t-s) C(t)^2 f''(s)}{(C(t)-C(s))^2} \\
&+ O((|T-t|+|S-s|)^2),
\end{align*}
for all $(T,S) \in \angle$ with $|T-t|$ and $|S-s|$ sufficiently small.
Finally recall that $f''(t) > 0$ and $f''(s) < 0$ (see part (4) of lemma
\ref{lemLowRig}). This proves the required result with,
\begin{equation}
\label{eqc1c2OO}
c_1(t,s) := - \frac{ (t-s) C(s) f''(t)}{(C(t)-C(s))^2}
\hspace{0.5cm} \text{and} \hspace{0.5cm}
c_2(t,s) := \frac{ (t-s) C(t) f''(s)}{(C(t)-C(s))^2}.
\end{equation}
\end{proof}

Next consider $(\chi,\eta) \in \OO$ and the corresponding $(t,s) \in \angle$
with $(\chi,\eta) = (\chi_\OO(t,s), \eta_\OO(t,s))$.
Recall that $\OO$ is depicted in figure \ref{figU}, and is that region
to the lower right of that sub-section of edge curve given by
$\theta \mapsto (\chi_\EE(\theta), \eta_\EE(\theta))$
for all $T \in (b,+\infty)$. Recall also,
part (4) of lemma \ref{lemLowRig} proves that
$f_{(\chi,\eta)}(s) - f_{(\chi,\eta)}(t) < 0$ (see also figure \ref{figRef}).
Moreover, theorem \ref{thmdecay} shows that correlation kernels of particles
in neighborhoods of $(\chi,\eta) \in \OO$ decay exponentially
with approximate exponent of decay given by
$f_{(\chi,\eta)}(s) - f_{(\chi,\eta)}(t) < 0$. We end this section by
examining the behaviour of the exponent as $(\chi,\eta) \in \OO$
changes. Lemma \ref{lemexpBeh} examines what happens to the
exponent as $(\chi,\eta) \in \OO$ is moved closer to the edge curve
along either horizontal or vertical paths (see figure \ref{figU2}),
and lemma \ref{lemexpBenEdge} examines the behaviour of the exponent
in neighborhoods of $\EE$.

\begin{lem}
\label{lemexpBeh}
Fix $(\chi,\eta) \in \OO$ and the corresponding $(t,s) \in \angle$ with
$(\chi,\eta) = (\chi_\OO(t,s), \eta_\OO(t,s))$. Similarly fix $(X,Y) \in \OO$
and the corresponding $(T,S) \in \angle$ with $(X,Y) = (\chi_\OO(T,S), \eta_\OO(T,S))$.
Assume that one of the possibilities is satisfied:
\begin{itemize}
\item
$\chi < X$ and $\eta = Y$.
\item
$\chi = X$ and $\eta > Y$.
\end{itemize}
These possibilities are depicted on the left of figure \ref{figU2}.
Then the following are satisfied:
\begin{enumerate}
\item
$T > t > s > S$.
\item
$f_{(X,Y)}(T) - f_{(X,Y)}(S)
< f_{(\chi,\eta)}(t) - f_{(\chi,\eta)}(s)
< 0$.
\end{enumerate}
\end{lem}

\begin{proof}
We will prove the result only when $\chi < X$ and $\eta = Y$.
The other case follows from similar considerations. 

\begin{figure}[t]
\centering
\begin{tikzpicture}[scale=0.75]
\draw [dotted] (0,0) --++(0,5);
\draw (0,-.3) node {$b$};
\draw plot [smooth, tension=1] coordinates
{ (.2,0) (2,3) (5,1.5) (7.5,4)};
\draw [dotted] (7.5,4) --++(.2,.5);
\draw [dotted] (2.3,3) --++(0,-3);
\draw (2.3,-.3) node {$s$};
\draw [dotted] (4.75,1.4) --++(0,-1.4);
\draw (4.75,-.3) node {$t$};

\draw [dotted] (10,0) --++(0,5);
\draw (10,-.3) node {$b$};
\draw plot [smooth, tension=1] coordinates
{ (10.2,0) (12,3) (15,1.5) (17.5,4)};
\draw [dotted] (17.5,4) --++(.2,.5);
\draw [dotted] (12.3,3) --++(0,-3);
\draw (12.3,-.3) node {$S$};
\draw [dotted] (14.75,1.4) --++(0,-1.4);
\draw (14.75,-.3) node {$T$};

\end{tikzpicture}
\caption{Left: The behaviour of $y \mapsto f_{(\chi,\eta)}(y)$ for
$y \in (b,\infty)$. Right: The behaviour of $y \mapsto f_{(X,Y)}(y)$ for
$y \in (b,\infty)$.}
\label{figRef2}
\end{figure}
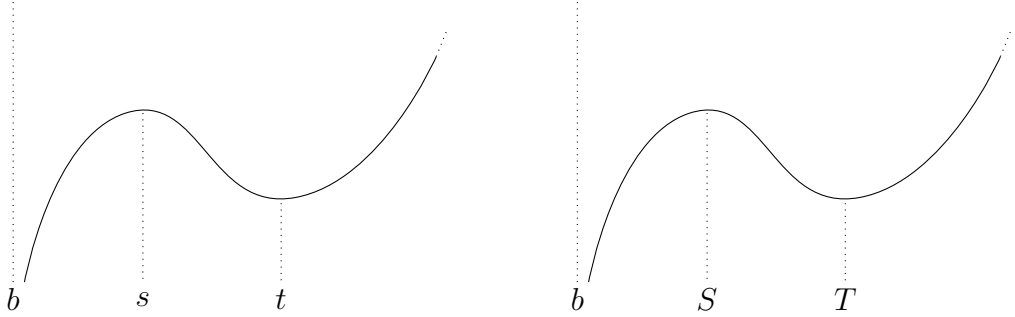

Take $\chi < X$ and $\eta = Y$. Consider (1).
First note, similarly to figure
\ref{figRef}, figure \ref{figRef2} depicts the behaviours of the real-valued
functions $y \mapsto f_{(\chi,\eta)}(y)$ and $y \mapsto f_{(X,Y)}(y)$ for all
$y \in (b,\infty)$. Note, equation (\ref{eqf}) gives,
\begin{equation}
\label{eqlemexpBeh}
f_{(X, Y)} (w)
= f_{(\chi, \eta)} (w)
+ (1-\eta) \log(w-\chi) - (1-Y) \log(w-X),
\end{equation}
for all $w \in (\C \setminus \R) \cup (b,+\infty)$, where $\log$ represents
principal value of the logarithm. Thus, since $t > s > b > \max \{\chi, X\}$
(see definition \ref{defLowRig2} and theorem \ref{thmLowRig}), and
$f_{(\chi, \eta)}'(t) = f_{(\chi, \eta)}'(s) = 0$
(see part (4) of lemma \ref{lemLowRig}),
\begin{equation*}
f_{(X, Y)}' (t)
= 0 + \frac{1-\eta}{t-\chi} - \frac{1-Y}{t-X},
\hspace{1cm}
f_{(X,Y)}' (s)
= 0 + \frac{1-\eta}{s-\chi} - \frac{1-Y}{s-X}.
\end{equation*}
It follows that $f_{(X, Y)}' (t) < 0$ and $f_{(X, Y)}' (s) < 0$
since $\chi < X$ and $\eta = Y$, $t > s > b > X$,
and $1 > Y > 0$ (see definition \ref{defLowRig2} and theorem
\ref{thmLowRig}). Finally note that $y \mapsto f_{(X,Y)}(y)$ for all
$y \in (b,+\infty)$ is strictly decreasing only when $y \in (S,T)$
(see figure \ref{figRef2}). This proves (1).

Consider (2). Recall, figure \ref{figRef2} depicts the behaviours of the real-valued
functions $y \mapsto f_{(\chi,\eta)}(y)$ and $y \mapsto f_{(X,Y)}(y)$ for all
$y \in (b,\infty)$. In particular note that $f_{(\chi,\eta)}(s) - f_{(\chi,\eta)}(t) < 0$
and $f_{(X,Y)}(S) - f_{(X,Y)}(T) < 0$. Recall also, part (1) gives $T > t > s > S$. 
Thus, since $y \mapsto f_{(X,Y)}(y)$ for all
$y \in (S,T)$ is strictly decreasing in $(S,T)$ (see figure \ref{figRef2})
$f_{(X,Y)}(T) - f_{(X,Y)}(S) < f_{(X,Y)}(t) - f_{(X,Y)}(s) < 0$. We
can thus prove (2) by showing that
$f_{(X,Y)}(t) - f_{(X,Y)}(s) < f_{(\chi,\eta)}(t) - f_{(\chi,\eta)}(s) < 0$.

To see the above, first recall that $t > s > b > \max \{\chi, X\}$.
Equation (\ref{eqlemexpBeh}) then gives,
\begin{align*}
f_{(X, Y)} (t) - f_{(X,Y)}(s)
&= f_{(\chi, \eta)} (t) - f_{(\chi, \eta)} (s) \\
&+ (1-\eta) \log \left( \frac{t-\chi}{s-\chi} \right)
- (1-Y) \log \left( \frac{t-X}{s-X} \right).
\end{align*}
Then, since $\eta = Y$,
\begin{align*}
f_{(X, Y)} (t) - f_{(X,Y)}(s)
&= f_{(\chi, \eta)} (t) - f_{(\chi, \eta)} (s)
+ (1-Y) \log \left( \frac{t-\chi}{s-\chi} \frac{s-X}{t-X} \right) \\
&= f_{(\chi, \eta)} (t) - f_{(\chi, \eta)} (s)
+ (1-Y) \log \left( 1 - \frac{(t-s) (X-\chi)}{(s-\chi) (t-X)} \right).
\end{align*}
Finally note that the logarithmic term on the right hand side is strictly
negative since $\chi < X$ and $\eta = Y$, $t > s > b > X$, and $1 > Y > 0$.
This proves that $f_{(X,Y)}(t) - f_{(X,Y)}(s) < f_{(\chi,\eta)}(t) - f_{(\chi,\eta)}(s) < 0$,
which proves (2).
\end{proof}

\begin{figure}
\centering
\mbox{\includegraphics{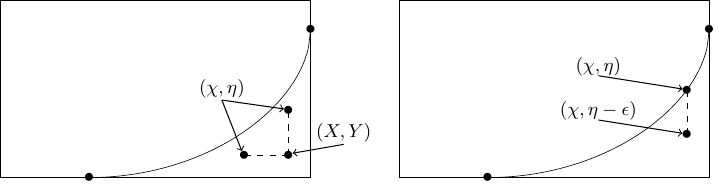}}
\caption{Left: The two possibilities of lemma \ref{lemexpBeh}. Right: The situation
in lemma \ref{lemexpBenEdge}. In both the curve is that sub-section of the edge,
$\EE$, given by $\theta \mapsto (\chi_\EE(\theta), \eta_\EE(\theta))$ for all
$\theta \in (b,+\infty)$.}
\label{figU2}
\end{figure}

\begin{lem}
\label{lemexpBenEdge}
Fix $(\chi,\eta) \in \EE$ and the corresponding $\theta \in (b,+\infty)$
with $(\chi,\eta) = (\chi_\EE(\theta), \eta_\EE(\theta))$
(see definition \ref{defEdge3} and theorem \ref{thmEdge}. See also
figure \ref{figU}). Recall that $\theta > b > \chi$, and
$f_{(\chi,\eta)}'(\theta) = f_{(\chi,\eta)}''(\theta) = 0$, and
$f_{(\chi,\eta)}'''(\theta) > 0$ (see part (4) of lemma \ref{lemLowRigEdge}),
and define $c = c(\theta) > 0$ by,
\begin{equation*}
c := (\theta - \chi)^{-1} f_{(\chi,\eta)}'''(\theta)^{-1}.
\end{equation*}
Next, fix $\epsilon > 0$ sufficiently small such that
$\sqrt{c \epsilon} < \frac14 (\theta-b)$, $\eta - \epsilon > 0$, and such
that equations (\ref{eqlemexpBenEdge2}, \ref{eqlemexpBenEdge3}) are satisfied.
Finally note that $(\chi, \eta - \epsilon) \in \OO$ since $\eta - \epsilon > 0$
(see right of figure \ref{figU2}), and let $(t_\epsilon, s_\epsilon) \in \angle$
denote the point in $\angle$ which corresponds to $(\chi, \eta - \epsilon) \in \OO$
(i.e., $(\chi, \eta - \epsilon)
= (\chi_\OO(t_\epsilon,s_\epsilon), \eta_\OO(t_\epsilon,s_\epsilon))$.
Then the following are satisfied:
\begin{enumerate}
\item
$\theta + 2 \sqrt{c \epsilon} > t_\epsilon > \theta + \sqrt{c \epsilon}
> \theta > \theta - \sqrt{c \epsilon} > s_\epsilon > \theta - 2 \sqrt{c \epsilon}
> b + 2 \sqrt{c \epsilon}$.
\item
$f_{(\chi, \eta - \epsilon)}(t_\epsilon) - f_{(\chi, \eta - \epsilon)}(s_\epsilon)
< -\frac56 \frac{\sqrt{c}}{\theta - \chi} \; \epsilon^\frac32$.
\end{enumerate}
\end{lem}

\begin{proof}
Consider (1). Recall that $(\chi,\eta-\epsilon) \in \OO$, and
$(t_\epsilon, s_\epsilon)$ is the corresponding point in $\angle$.
Recall also that the behaviour of the real-valued function
$y \mapsto f_{(\chi, \eta - \epsilon)}(y)$ for all $y \in (b,+\infty)$
is described by part (4) of lemma \ref{lemLowRig} and depicted in
figure \ref{figRef} (replace $t$ by $t_\epsilon$ and $s$ by $s_\epsilon$).
Also note that $\theta + 2 \sqrt{c \epsilon} > \theta + \sqrt{c \epsilon}
> \theta > \theta - \sqrt{c \epsilon} > \theta - 2 \sqrt{c \epsilon} > b + 2 \sqrt{c \epsilon}$
since $\theta > b$ and $\sqrt{c \epsilon} < \frac14 (\theta-b)$.
(1) thus follows if we can prove the following:
\begin{enumerate}
\item[(i)]
$f_{(\chi, \eta - \epsilon)}'(\theta) < 0$.
\item[(ii)]
$f_{(\chi, \eta - \epsilon)}'(\theta + \sqrt{c \epsilon}) < 0$
and
$f_{(\chi, \eta - \epsilon)}'(\theta - \sqrt{c \epsilon}) < 0$.
\item[(iii)]
$f_{(\chi, \eta - \epsilon)}'(\theta + 2 \sqrt{c \epsilon}) > 0$
and
$f_{(\chi, \eta - \epsilon)}'(\theta - 2 \sqrt{c \epsilon}) > 0$.
\end{enumerate}

Consider (i). First note, equation (\ref{eqf}) gives,
\begin{equation}
\label{eqlemexpBenEdge}
f_{(\chi, \eta-\epsilon)} (w)
= f_{(\chi, \eta)} (w) - \epsilon \log(w-\chi),
\end{equation}
for all $w \in (\C \setminus \R) \cup (b,+\infty)$, where $\log$ represents
principal value of the logarithm. Thus, since $\theta > b > \chi$,
and since $f_{(\chi, \eta)}'(\theta) = 0$ (see definition \ref{defEdge3} and
theorem \ref{thmEdge}),
\begin{equation*}
f_{(\chi, \eta-\epsilon)}'(\theta)
= 0 - \frac{\epsilon}{\theta-\chi}.
\end{equation*}
This proves (i).

Consider (ii). First note, since
$f_{(\chi, \eta)}'(\theta) = f_{(\chi, \eta)}''(\theta) = 0$
(see definition \ref{defEdge3} and
theorem \ref{thmEdge}), Taylors theorem gives,
\begin{equation*}
f_{(\chi,\eta)}'(\theta \pm \sqrt{c \epsilon}) =
\tfrac12 f_{(\chi,\eta)}'''(\theta) (\pm \sqrt{c \epsilon})^2
+ \tfrac12 \int_\theta^{\theta \pm \sqrt{c \epsilon}} \;
f_{(\chi,\eta)}^{(4)}(y) (\theta \pm \sqrt{c \epsilon}-y)^2 dy.
\end{equation*}
Recall 
$c = (\theta - \chi)^{-1} f_{(\chi,\eta)}'''(\theta)^{-1}$ (see statement of this lemma) and equation (\ref{eqlemexpBenEdge}). Then,
\begin{align*}
\left| f_{(\chi,\eta-\epsilon)}'(\theta \pm \sqrt{c \epsilon})
+ \frac12 \frac{\epsilon}{\theta - \chi} \right|
&\le \tfrac12 \left| \int_\theta^{\theta \pm \sqrt{c \epsilon}} \;
f_{(\chi,\eta)}^{(4)}(y) (\theta \pm \sqrt{c \epsilon}-y)^2 dy \right| \\
&+ \left| \frac{\epsilon}{\theta - \chi}
- \frac{\epsilon}{\theta \pm \sqrt{c \epsilon} - \chi} \right|.
\end{align*}
Next note, equation (\ref{eqf}) gives,
\begin{equation*}
\left| f_{(\chi,\eta)}^{(4)}(y) \right|
\le \int_a^b \frac6{|y-x|^4} \mu[dx] + (1-\eta) \frac6{|y-\chi|^4},
\end{equation*}
for all $y \in [\theta - \sqrt{c \epsilon}, \theta + \sqrt{c \epsilon}]$.
Thus, since $\theta > b > \chi$ and $0 < \eta < 1$
(see definition \ref{defEdge3} and theorem \ref{thmEdge}), and since
$\theta - \sqrt{c \epsilon} - b > \frac12 (\theta - b) > 0$
(recall $4 \sqrt{c \epsilon} < \theta-b$),
\begin{equation*}
\tfrac12 \left| \int_\theta^{\theta \pm \sqrt{c \epsilon}} \;
f_{(\chi,\eta)}^{(4)}(y) (\theta \pm \sqrt{c \epsilon}-y)^2 dy \right|
< \tfrac12 \frac{2^8}{(\theta-b)^4} (\sqrt{c \epsilon})^3.
\end{equation*}
Moreover, since $\theta > b > \chi$, and
$\theta - \sqrt{c \epsilon} - b > \frac12 (\theta - b) > 0$,
\begin{equation*}
\left| \frac{\epsilon}{\theta - \chi}
- \frac{\epsilon}{\theta \pm \sqrt{c \epsilon} - \chi} \right|
\le \frac{2 \sqrt{c} \epsilon^\frac32}{(\theta-b)^2}.
\end{equation*}
Finally, choose $\epsilon > 0$ sufficiently small such that,
\begin{equation}
\label{eqlemexpBenEdge2}
\frac14 \frac{\epsilon}{\theta - \chi}
> \frac{2^7 c^\frac32}{(\theta-b)^4} \; \epsilon^\frac32
+ \frac{2 c^\frac12}{(\theta-b)^2} \; \epsilon^\frac32.
\end{equation}
Combined, the above give $f_{(\chi,\eta-\epsilon)}'(\theta \pm \sqrt{c \epsilon}) < 0$,
which proves (ii). Part (iii) follows similarly.

Consider (2). Recall that the behaviour of the real-valued function
$y \mapsto f_{(\chi, \eta - \epsilon)}(y)$ for all $y \in (b,+\infty)$
is described by part (4) of lemma \ref{lemLowRig} and depicted in
figure \ref{figRef} (replace $t$ by $t_\epsilon$ and $s$ by $s_\epsilon$).
Part (1) thus implies that
$f_{(\chi, \eta - \epsilon)}(t_\epsilon) - f_{(\chi, \eta - \epsilon)}(s_\epsilon)
< f_{(\chi, \eta - \epsilon)}(\theta + \sqrt{c \epsilon})
- f_{(\chi, \eta - \epsilon)}(\theta - \sqrt{c \epsilon})$.
We will show:
\begin{enumerate}
\item[(iv)]
$f_{(\chi, \eta - \epsilon)}(\theta + \sqrt{c \epsilon})
- f_{(\chi, \eta - \epsilon)}(\theta - \sqrt{c \epsilon})
< -\frac56 \frac{\sqrt{c}}{\theta - \chi} \; \epsilon^\frac32$.
\end{enumerate}
This proves (2).

Consider (iv). First note, equation (\ref{eqlemexpBenEdge}) gives,
\begin{align*}
f_{(\chi, \eta - \epsilon)}(\theta + \sqrt{c \epsilon})
- f_{(\chi, \eta - \epsilon)}(\theta - \sqrt{c \epsilon})
&= f_{(\chi, \eta)}(\theta + \sqrt{c \epsilon})
- f_{(\chi, \eta)}(\theta - \sqrt{c \epsilon}) \\
&- \epsilon \log(\theta + \sqrt{c \epsilon}-\chi)
+ \epsilon \log(\theta - \sqrt{c \epsilon}-\chi).
\end{align*}
Thus, since $f_{(\chi, \eta)}'(\theta) = f_{(\chi, \eta)}''(\theta) = 0$
(see definition \ref{defEdge3} and theorem \ref{thmEdge}), Taylors theorem
applied to each term on the RHS gives,
\begin{align*}
&f_{(\chi, \eta - \epsilon)}(\theta + \sqrt{c \epsilon})
- f_{(\chi, \eta - \epsilon)}(\theta - \sqrt{c \epsilon}) \\
&= f_{(\chi, \eta)}(\theta) + \tfrac16 f_{(\chi, \eta)}'''(\theta) (\sqrt{c \epsilon})^3
+ \tfrac16 \int_\theta^{\theta + \sqrt{c \epsilon}} \;
f_{(\chi,\eta)}^{(4)}(y) (\theta + \sqrt{c \epsilon}-y)^3 dy \\
&- f_{(\chi, \eta)}(\theta) - \tfrac16 f_{(\chi, \eta)}'''(\theta) (-\sqrt{c \epsilon})^3
- \tfrac16 \int_\theta^{\theta - \sqrt{c \epsilon}} \;
f_{(\chi,\eta)}^{(4)}(y) (\theta - \sqrt{c \epsilon}-y)^3 dy \\
&- \epsilon \log(\theta - \chi) - \epsilon \frac{\sqrt{c \epsilon}}{\theta - \chi}
+ \epsilon \int_\theta^{\theta + \sqrt{c \epsilon}} \;
\frac{\theta + \sqrt{c \epsilon} - y}{(y - \chi)^2} dy \\
&+ \epsilon \log(\theta - \chi) + \epsilon \frac{-\sqrt{c \epsilon}}{\theta - \chi}
- \epsilon \int_\theta^{\theta - \sqrt{c \epsilon}} \;
\frac{\theta - \sqrt{c \epsilon} - y}{(y - \chi)^2} dy.
\end{align*}
Next recall (see statement of this lemma) that
$c = (\theta - \chi)^{-1} f_{(\chi,\eta)}'''(\theta)^{-1}$.
Therefore,
\begin{align*}
&\left| f_{(\chi, \eta - \epsilon)}(\theta + \sqrt{c \epsilon})
- f_{(\chi, \eta - \epsilon)}(\theta - \sqrt{c \epsilon})
+ \frac53 \frac{\sqrt{c}}{\theta - \chi} \; \epsilon^\frac32 \right| \\
&\le \tfrac16 \left| \int_\theta^{\theta + \sqrt{c \epsilon}} \;
f_{(\chi,\eta)}^{(4)}(y) (\theta + \sqrt{c \epsilon}-y)^3 dy \right|
+ \tfrac16 \left| \int_\theta^{\theta - \sqrt{c \epsilon}} \;
f_{(\chi,\eta)}^{(4)}(y) (\theta - \sqrt{c \epsilon}-y)^3 dy \right| \\
&+\epsilon \left| \int_\theta^{\theta + \sqrt{c \epsilon}} \;
\frac{\theta + \sqrt{c \epsilon} - y}{(y - \chi)^2} dy \right|
+ \epsilon \left| \int_\theta^{\theta - \sqrt{c \epsilon}} \;
\frac{\theta - \sqrt{c \epsilon} - y}{(y - \chi)^2} dy \right|.
\end{align*}
Next proceed similarly to the proof of part (ii) above to get,
\begin{align*}
\tfrac16 \left| \int_\theta^{\theta \pm \sqrt{c \epsilon}} \;
f_{(\chi,\eta)}^{(4)}(y) (\theta \pm \sqrt{c \epsilon}-y)^3 dy \right|
&< \frac{2^5}{(\theta-b)^4} (\sqrt{c \epsilon})^4, \\
\left| \int_\theta^{\theta \pm \sqrt{c \epsilon}} \;
\frac{\theta \pm \sqrt{c \epsilon} - y}{(y - \chi)^2} dy \right|
&< \frac{2^2}{(\theta-b)^2} (\sqrt{c \epsilon})^2.
\end{align*}
Finally, choose $\epsilon > 0$ sufficiently small such that,
\begin{equation}
\label{eqlemexpBenEdge3}
\frac56 \frac{\sqrt{c}}{\theta - \chi} \; \epsilon^\frac32
> \frac{2^6}{(\theta-b)^4} \; c^2 \; \epsilon^4
+ \frac{2^3}{(\theta-b)^2} \; c \; \epsilon^2.
\end{equation}
Combined, the above prove (iv).
\end{proof}

\section{Steepest descent analysis}
\label{secsdatak}

In this section we prove theorem \ref{thmdecay} via steepest descent
analysis. Recall the following conditions from theorem \ref{thmdecay},
which we assume throughout section \ref{secsdatak}:
\begin{itemize}
\item
Assume $\mu[\{b\}] > 0$.
\item
Fix $(\chi,\eta) \in \OO$ and the corresponding $(t,s) \in \angle$ with
$(\chi,\eta) = (\chi_\OO(t,s), \eta_\OO(t,s))$.
\item
Define $u_n,r_n,v_n,s_n$ as in equation (\ref{equnrnvnsn2}).
\item
Fix $\theta \in (\frac13,\frac12)$.
\item
Define $\xi = \xi(t,s) > 0$ and $N = N(t,s) \ge 1$ as in definition \ref{defxiN}.
\end{itemize}
Using only the above, we will show that $N$ can be chosen sufficiently
large that lemma \ref{lemN2} is satisfied. We will then prove theorem \ref{thmdecay} for
this choice of $N$. Note, theorem \ref{thmdecay} assumes that $r_n = s_n$ for all $n>N$.
As stated in section \ref{secAssAndTerm}, this trivially gives $\phi_{r_n, s_n} (u_n,v_n) = 0$,
but is not used elsewhere. All other asymptotic results in this section hold for general
$r_n$ and $s_n$.

\subsection{The roots, and the local asymptotic behaviour, of the steepest descent functions}
\label{secRootsOO}

In this section we examine the roots of the steepest descent functions
under the above conditions, and the local asymptotic behaviour of these
functions in the neighbourhood of important roots. We begin with the roots of
$f_{(\chi,\eta)}$. Note, it is now natural to index $f_{(\chi,\eta)}'$ with
$(t,s) \in \angle$ instead of with $(\chi,\eta) \in \OO$. Equations
(\ref{eqf'0}, \ref{eqf'}) thus give,
\begin{align}
\label{eqfts'}
f_{(t,s)}'(w)
&= C(w) - \frac{1-\eta}{w-\chi}, \\
\nonumber
&= \int_{(\chi,b]} \frac{\mu[dx]}{w-x}
- \frac{1-\eta - \mu[\{\chi\}]}{w-\chi}
+ \int_{[a,\chi)} \frac{\mu[dx]}{w-x}
\end{align}
for all $w \in \C \setminus (S_1 \cup S_2 \cup S_3)$, where
\begin{equation*}
S_1 := \supp(\mu |_{(\chi,b]}),
\hspace{0.5cm}
S_2 := \left\{ \begin{array}{rcl}
\{\chi\} & ; & \text{when } \mu[\{\chi\}] \neq 1-\eta, \\
\emptyset & ; & \text{when } \mu[\{\chi\}] = 1-\eta,
\end{array} \right.
\hspace{0.5cm}
S_3 := \supp(\mu |_{[a,\chi)}).
\end{equation*}
Note, assumption \ref{assWeakConv} and definition \ref{defLowRig2} give:
\begin{align*}
&S_1 \neq \emptyset
\text{ and }
\mu[S_1] > 0,
\hspace{.3cm}
1-\eta - \mu[\{\chi\}] > 0,
\hspace{.3cm}
S_3 \neq \emptyset
\text{ and }
\mu[S_3] > 0. \\
&\mu[S_1] - (1-\eta - \mu[\{\chi\}]) + \mu[S_3]
= \mu[a,b] - (1-\eta) = \eta \in (0,1). \\
&b > \chi > a
\text{ and }
b = \sup S_1 \ge \inf S_1 \ge \chi \ge \sup S_3 \ge \inf S_3 = a.
\end{align*}
Partition the domain of $f_{(t,s)}'$ as follows:
\begin{equation}
\label{eqf'domain2}
\C \setminus (S_1 \cup S_2 \cup S_3) = (\C \setminus \R) \cup J \cup K,
\end{equation}
where $J := \cup_{i=1}^4 J_i$, $K := \cup_{i=1,3} K^{(i)}$, and
\begin{itemize}
\item
$J_1 := (\sup S_1,+\infty) = (b,+\infty)$.
\item
$J_2 := (-\infty, \inf S_3) = (-\infty,a)$.
\item
$J_3 := (\sup S_2,\inf S_1) = (\chi,\inf S_1)$ (empty if $\inf S_1 = \chi$).
\item
$J_4 := (\sup S_3,\inf S_2) = (\sup S_3,\chi)$ (empty if $\sup S_3 = \chi$).
\item
$K^{(i)} := [\inf S_i, \sup S_i] \setminus S_i$ for all $i \in \{1,3\}$
(note, the indices are chosen to match those of $S_i$, and so there is no $K^{(2)}$).
\end{itemize}
This partition is depicted in figure \ref{figf'domain}.
Partition each $K^{(i)}$ as $\{K_1^{(i)}, K_2^{(i)}, \ldots\}$,
a set of pairwise disjoint open intervals, unique up to order,
either empty or finite or countable, and which satisfy
$\{\inf I, \sup I\} \subset S_i$ for any
$I \in \{K_1^{(i)}, K_2^{(i)}, \ldots\}$.

\begin{lem}
\label{lemf'}
We have:
\begin{enumerate}
\item
$f_{(t,s)}'$ has roots of multiplicity $1$ at $t,s \in J_1 = (b,+\infty)$
where $t>s$, and has $0$ roots in $J_1 \setminus \{t,s\}$. Moreover,
$f_{(t,s)} |_{(b,+\infty)}$ is real-valued, is strictly increasing in $(b,s)$,
has a local maximum at $s$ ($f_{(t,s)}'(s) = 0$ and $f_{(t,s)}''(s) < 0$),
is strictly decreasing in $(s,t)$, has a local minimum at $t$
($f_{(t,s)}'(t) = 0$ and $f_{(t,s)}''(t) > 0$), and is strictly
increasing in $(t,+\infty)$.
\item
$f_{(t,s)}'$ has $0$ roots in $\C \setminus \R$, and in each of
$\{J_2,J_3,J_4\}$.
\item
$f_{(t,s)}'$ has at most $1$ root, counting multiplicities, in each of
$\cup_{i=1,3} \{K_1^{(i)},K_2^{(i)},\ldots\}$.
\item
The following are expressions for $f_{(t,s)}''(t) > 0$ and $f_{(t,s)}''(s) < 0$:
\begin{align*}
f_{(t,s)}''(t)
&= \int_a^b \frac{\mu[dx] (\chi-x)}{(t-x)^2 (t-\chi)}
= \int_a^b \int_a^b \frac{(t-s) (x-y)^2 \mu[dx] \mu[dy]}{2 C(s) (t-x)^2 (t-y)^2 (s-x) (s-y)}, \\
f_{(t,s)}''(s)
&= \int_a^b \frac{\mu[dx] (\chi-x)}{(s-x)^2 (s-\chi)}
= - \int_a^b \int_a^b \frac{(t-s) (x-y)^2 \mu[dx] \mu[dy]}{2 C(t) (s-x)^2 (s-y)^2 (t-x) (t-y)}.
\end{align*}
\end{enumerate}
\end{lem}

\begin{proof}
Consider (1,2,3). Note, since $(\chi,\eta) \in \OO$, part (4) of lemma \ref{lemLowRig},
and possibility (a) of theorem \ref{thmf'} trivially imply parts (1,2,3).

Consider (4). First recall,
part (1) gives $f_{(t,s)}'(t) = 0$. Equation (\ref{eqfts'})
then gives,
\begin{equation*}
\frac{1-\eta}{t-\chi} = \int_a^b \frac{\mu[dx]}{t-x}.
\end{equation*}
Moreover, equation (\ref{eqfts'}) gives,
\begin{equation*}
f_{(t,s)}''(t)
= - \int_a^b \frac{\mu[dx]}{(t-x)^2} + \frac{1-\eta}{(t-\chi)^2}.
\end{equation*}
Combined, the above give,
\begin{equation*}
f_{(t,s)}''(t)
= - \int_a^b \frac{\mu[dx]}{(t-x)^2}
+ \int_a^b \frac{\mu[dx]}{(t-x) (t-\chi)}
= \int_a^b \frac{\mu[dx] (\chi-x)}{(t-x)^2 (t-\chi)}.
\end{equation*}
This proves the first expression for $f_{(t,s)}''(t)$.

Consider the second expression for $f_{(t,s)}''(t)$. 
Recall that $\chi = \chi_\OO(t,s)$, where an expression
for $\chi_\OO(t,s)$ is given in the statement of theorem \ref{thmLowRig}.
This gives,
\begin{equation*}
f_{(t,s)}''(t)
= \int_a^b \frac{\mu[dx]}{(t-x)^2} \; \frac{\chi-x}{t-\chi}
= \int_a^b \frac{\mu[dx]}{(t-x)^2} \; \frac{(t-x) C(t) - (s-x) C(s)}{-(t-s) C(s)}.
\end{equation*}
Equation (\ref{eqCauTrans}) then gives,
\begin{align*}
f_{(t,s)}''(t)
&= \int_a^b \frac{\mu[dx]}{(t-x)^2} \; \frac1{-(t-s) C(s)} \;
\int_a^b \left( \frac{t-x}{t-y} - \frac{s-x}{s-y} \right) \mu[dy] \\
&= - \int_a^b \int_a^b \frac{(x-y)}{C(s) (t-x)^2 (t-y) (s-y)} \mu[dx] \mu[dy]. 
\end{align*}
Thus, since $x$ and $y$ are dummy parameters,
\begin{align*}
f_{(t,s)}''(t)
&= - \frac12 \int_a^b \int_a^b \left( \frac{(x-y)}{C(s) (t-x)^2 (t-y) (s-y)} 
+ \frac{(y-x)}{C(s) (t-y)^2 (t-x) (s-x)} \right) \mu[dx] \mu[dy] \\
&= \frac12 \int_a^b \int_a^b \frac{(t-s) (x-y)^2 \mu[dx] \mu[dy]}
{C(s) (t-x)^2 (t-y)^2 (s-x) (s-y)}.
\end{align*}
This gives the second expression for $f_{(t,s)}''(t)$. We can similarly prove
the first and second expression for $f_{(t,s)}''(s)$. This proves (4).
\end{proof}

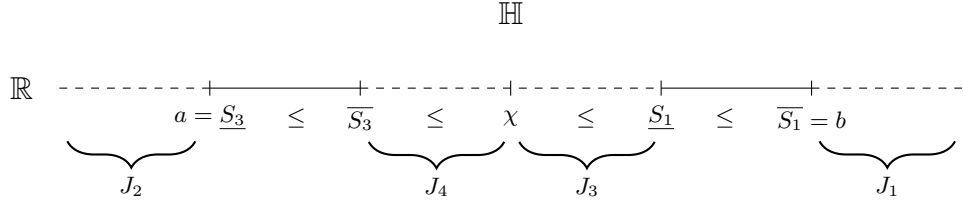
\begin{figure}[t]
\centering
\begin{tikzpicture}

\draw [dashed] (-2,0) --++ (2,0);
\draw (0,0) --++ (2,0);
\draw [dashed] (2,0) --++ (4,0);
\draw (6,0) --++ (2,0);
\draw [dashed] (8,0) --++ (2,0);
\draw (-2.5,0) node {$\R$};
\draw (4,1) node {$\mathbb{H}$};

\draw (0,.1) --++ (0,-.2);
\draw (0,-.4) node {\scriptsize $a = \underline{S_3}$};
\draw (1.15,-.4) node {\scriptsize $\le$};
\draw (2,.1) --++ (0,-.2);
\draw (2,-.4) node {\scriptsize $\overline{S_3}$};
\draw (3,-.4) node {\scriptsize $\le$};
\draw (4,.1) --++ (0,-.2);
\draw (4,-.4) node {\scriptsize $\chi$};
\draw (5,-.4) node {\scriptsize $\le$};
\draw (6,.1) --++ (0,-.2);
\draw (6,-.4) node {\scriptsize $\underline{S_1}$};
\draw (6.85,-.4) node {\scriptsize $\le$};
\draw (8,.1) --++ (0,-.2);
\draw (8,-.4) node {\scriptsize $\overline{S_1} = b$};

\draw [thick,decorate,decoration={brace,amplitude=10pt,mirror},xshift=0.2pt,yshift=-0.2pt]
(-1.9,-.7) -- (-.2,-.7) node[black,midway,yshift=-0.6cm] {\scriptsize $J_2$};
\draw [thick,decorate,decoration={brace,amplitude=10pt,mirror},xshift=0.4pt,yshift=-0.4pt]
(2.1,-.7) -- (3.9,-.7) node[black,midway,yshift=-0.6cm] {\scriptsize $J_4$};
\draw [thick,decorate,decoration={brace,amplitude=10pt,mirror},xshift=0.4pt,yshift=-0.4pt]
(4.1,-.7) -- (5.9,-.7) node[black,midway,yshift=-0.6cm] {\scriptsize $J_3$};
\draw [thick,decorate,decoration={brace,amplitude=10pt,mirror},xshift=0.2pt,yshift=-0.2pt]
(8.1,-.7) -- (9.9,-.7) node[black,midway,yshift=-0.6cm] {\scriptsize $J_1$};

\end{tikzpicture}
\caption{The sets of equation (\ref{eqf'domain2}), with $b > \chi > a$,
$K^{(i)} = [\inf S_i, \sup S_i] \setminus S_i$ for $i \in \{1,3\}$,
$\underline{S_1} := \inf S_1$, $\overline{S_1} := \sup S_1$,
etc.}
\label{figf'domain}
\end{figure}

We next prove lemma \ref{lemN2} which examine the roots of the
`non-asymptotic' functions, $f_{(t,s),n}', f_n', \tilde{f}_n'$.
Recall, equations (\ref{eqftsn'}, \ref{eqfn'}, \ref{eqtildefn'})
give the following for all $n>N$:
\begin{align}
\label{eqftsn'2}
f_{(t,s),n}'(w)
&= C_n(w) - \frac{1-\eta_n}{w-\chi_n}
= \frac1n \sum_{x \in P_n} \frac1{w-x} - \frac{1-\eta_n}{w-\chi_n}, \\
\label{eqfn2'}
f_n'(w)
&= C_n(w) - \frac{1-\frac{s_n-1}n}{w-v_n}
= \frac1n \sum_{x \in P_n} \frac1{w-x} - \frac{1-\frac{s_n-1}n}{w-v_n}, \\
\label{eqtildefn2'}
\tilde{f}_n'(w)
&= C_n(w) - \frac{1-\frac{r_n+1}n}{w-v_n}
= \frac1n \sum_{x \in P_n} \frac1{w-x} - \frac{1-\frac{r_n+1}n}{w-u_n}.
\end{align}
where $P_n$ and $\mu_n$ and $C_n$ are defined in equation (\ref{eqPn}), and
$(\chi_n,\eta_n), (u_n,r_n), (v_n,s_n)$ are defined in definition
\ref{defLowRigNonAsy} and equation (\ref{equnrnvnsn2}). The above functions
have domains $\C \setminus (P_n \cup \{\chi_n\}$ and $\C \setminus (P_n \cup \{v_n\})$
and $\C \setminus (P_n \cup \{u_n\})$ respectively. Also recall, definition \ref{defxiN}
gives the following for all $n > N$:
\begin{align}
\nonumber
&t - 4\xi > s + 4\xi >
s - 4\xi > b + 4\xi >
b - 4\xi > \chi + 4\xi >
\chi - 4\xi > a + 4\xi, \\
\label{eqn1}
&b+4\xi > x_1^{(n)} > b-4\xi
\text{ and } a+4\xi > x_n^{(n)} > a-4\xi, \\
\nonumber
&\chi + 4\xi > \{\chi_n, v_n , u_n\} > \chi - 4\xi, \\
\nonumber
&1 - 2\xi > 1 - \eta + 2\xi
> \{1-  \eta_n, 1 - \tfrac{s_n-1}n, 1 - \tfrac{r_n+1}n\} > 1 - \eta - 2\xi > 2\xi,
\end{align}
Note the above implies that $\xi < \frac18 (t-s), \frac1{16} (t-b), \frac1{24} (t-\chi),
\frac18(s-b), \frac14(1-\eta)$, etc.

\begin{proof}[Proof of lemma \ref{lemN2}:]
Fix $\xi = \xi(t,s) > 0$ and $N = N(t,s) \ge 1$ as in definition \ref{defxiN}.
First note, equations (\ref{eqfts'}, \ref{eqn1}) imply that $B(t,2\xi)$ and $B(s,2\xi)$
are disjoint open subsets of $(\C \setminus \R) \cup (b+4\xi,+\infty)$, and $f_{(t,s)}'$
is well-defined and analytic in $B(t,2\xi) \cup B(s,2\xi)$. Parts (1,2,3) then follow
trivially from parts (1,2) of lemma \ref{lemf'}.

Next note equations (\ref{eqftsn'2}, \ref{eqn1}) imply that $f_{(t,s),n}'$ is
well-defined and analytic in $B(t,2\xi) \cup B(s,2\xi)$ for all $n>N$.
Part (4) then follows trivially from equation (\ref{eqftsn'2}) and
definition \ref{defLowRigNonAsy}. Consider (5). Note, equations
(\ref{eqfts'}, \ref{eqftsn'2}) give the following for all $n>N$:
\begin{align*}
|f_{(t,s),n}''(t) - f_{(t,s)}''(t)|
&\le |C_n'(t) - C'(t)|
+ \left| \frac{1-\eta_n}{(t-\chi_n)^2} - \frac{1-\eta}{(t-\chi)^2} \right|, \\
|f_{(t,s),n}''(s) - f_{(t,s)}''(s)|
&\le |C_n'(s) - C'(s)|
+ \left| \frac{1-\eta_n}{(s-\chi_n)^2} - \frac{1-\eta}{(s-\chi)^2} \right|.
\end{align*}
Recall that $\chi_n = \chi_n(t,s)$ and $\chi = \chi_\OO(t,s)$, and similarly
for $\eta_n$ and $\eta$. Equation (\ref{eqL4}) and part (2) then give
$f_{(t,s),n}''(t) \to f_{(t,s)}''(t) > 0$ and $f_{(t,s),n}''(s) \to f_{(t,s)}''(s) < 0$
as $n\to\infty$. This proves (5). 

Consider (6). We will show the following:
\begin{equation}
\tag{i}
\inf_{w \in \partial B(t,\xi)} |f_{(t,s)}'(w)| > 0
\hspace{0.5cm} \mbox{and} \hspace{0.5cm}
\lim_{n \to \infty} \sup_{w \in \text{cl}(B(t,\xi))} |f_{(t,s)}'(w) - f_{(t,s),n}'(w)| = 0.
\end{equation}
This implies that we can choose $N$ such that the following is
satisfied for all $n>N$:
\begin{equation*}
\inf_{w \in \partial B(t,\xi)} |f_{(t,s)}'(w)|
> \sup_{w \in \text{cl}(B(t,\xi))} |f_{(t,s)}'(w) - f_{(t,s),n}'(w)|. 
\end{equation*}
Then parts (1,2,3) and Rouch\'{e}'s theorem imply that $f_{(t,s),n}'$
has exactly $1$ root in $B(t,\xi)$ for all $n>N$ and
counting multiplicities. Similarly we can choose $N$ such that
$f_{(t,s),n}'$ has exactly $1$ root in $B(s,\xi)$ for all $n>N$ and
counting multiplicities. Moreover, part (4) implies
that these roots are $t$ and $s$ respectively. This proves (6).

Consider (7). First note, for all $n>N$, definition \ref{defxiN} and equations
(\ref{eqftsn'2}, \ref{eqfn2'}, \ref{eqtildefn2'}, \ref{eqn1})
give $B(t,2n^{-\frac12}) \subset B(t,\xi)$
and $B(s,2n^{-\frac12}) \subset B(s,\xi)$, and $f_n', \tilde{f}_n'$ are
well-defined and analytic in $B(t,2\xi) \cup B(s,2\xi)$. Fix $n>N$. Equations
(\ref{eqftsn'2}, \ref{eqfn2'}) then give the following:
\begin{align*}
|f_n'(t) - f_{(t,s),n}'(t)|
&= \left| \frac{1-\eta_n}{t-\chi_n}
- \frac{1-\frac{s_n-1}n}{t-v_n} \right|, \\
|f_n''(t) - f_{(t,s),n}''(t)|
&= \left| \frac{1-\eta_n}{(t-\chi_n)^2}
- \frac{1-\frac{s_n-1}n}{(t-v_n)^2} \right|.
\end{align*}
Recall that $f_{(t,s),n}'(t) = 0$ (see part (4)),
$\frac{C_n(t)}{1-\eta_n} = \frac1{t-\chi_n}$ (see definition \ref{defLowRigNonAsy}),
and equation (\ref{equnrnvnsn2}). Combined these give,
\begin{align*}
|f_n'(t)|
= |f_n'(t) - f_{(t,s),n}'(t)|
&= \left| \frac{1-\eta_n}{t-\chi_n} - \frac{1-\eta_n}{t-\chi_n}
\frac{1 - \frac{m_n}{t-\chi_n} n^{-\frac12} - \frac{y_{2,n}}{1-\eta_n} n^{-1}}
{1 - \frac{m_n}{t-\chi_n} n^{-\frac12} - \frac{y_{1,n}}{t-\chi_n} n^{-1}} \right|, \\
|f_n''(t) - f_{(t,s),n}''(t)|
&= \left| \frac{1-\eta_n}{(t-\chi_n)^2} - \frac{1-\eta_n}{(t-\chi_n)^2}
\frac{1 - \frac{m_n}{t-\chi_n} n^{-\frac12} - \frac{y_{2,n}}{1-\eta_n} n^{-1}}
{(1 - \frac{m_n}{t-\chi_n} n^{-\frac12} - \frac{y_{1,n}}{t-\chi_n} n^{-1})^2} \right|.
\end{align*}
Therefore,
\begin{align*}
|f_n'(t)|
&:= \frac{|1-\eta_n|}{|t-\chi_n|} \;
\frac{|- \frac{y_{1,n}}{t-\chi_n} n^{-1} + \frac{y_{2,n}}{1-\eta_n} n^{-1}|}
{|1 - \frac{m_n}{t-\chi_n} n^{-\frac12} - \frac{y_{1,n}}{t-\chi_n} n^{-1}|}, \\
|f_n''(t) - f_{(t,s),n}''(t)|
&:= \frac{|1-\eta_n|}{|t-\chi_n|^2} \;
\frac{|(1 - \frac{m_n}{t-\chi_n} n^{-\frac12}
- \frac{y_{1,n}}{t-\chi_n} n^{-1})^2
- ( 1 - \frac{m_n}{t-\chi_n} n^{-\frac12}
- \frac{y_{2,n}}{1-\eta_n} n^{-1} )|}
{|1 - \frac{m_n}{t-\chi_n} n^{-\frac12} - \frac{y_{1,n}}{t-\chi_n} n^{-1}|^2},
\end{align*}
and so $|f_n'(t)| = B_{1,n} n^{-1}$ and
$|f_n''(t) - f_{(t,s),n}''(t)| = B_{2,n} n^{-\frac12}$ where,
\begin{align*}
B_{1,n}
&:= \frac{|1-\eta_n|}{|t-\chi_n|} \;
\frac{|- \frac{y_{1,n}}{t-\chi_n} + \frac{y_{2,n}}{1-\eta_n}|}
{|1 - \frac{m_n}{t-\chi_n} n^{-\frac12} - \frac{y_{1,n}}{t-\chi_n} n^{-1}|}, \\
B_{2,n}
&:= \frac{|1-\eta_n|}{|t-\chi_n|^2} \;
\frac{|- \frac{m_n}{t-\chi_n}
- 2 \frac{y_{1,n}}{t-\chi_n} n^{-\frac12}
+ \frac{y_{2,n}}{1-\eta_n} n^{-\frac12}
+ (\frac{m_n}{t-\chi_n} n^{-\frac12} + \frac{y_{1,n}}{t-\chi_n} n^{-1})^2 n^\frac12|}
{|1 - \frac{m_n}{t-\chi_n} n^{-\frac12} - \frac{y_{1,n}}{t-\chi_n} n^{-1}|^2}.
\end{align*}
Similarly we can show that $|\tilde{f}_n'(s)| = \tilde{B}_{1,n} n^{-1}$ and
$|\tilde{f}_n''(s) - f_{(t,s),n}''(s)| = \tilde{B}_{2,n} n^{-\frac12}$ where,
\begin{align*}
\tilde{B}_{1,n}
&:= \frac{|1-\eta_n|}{|s-\chi_n|} \;
\frac{|- \frac{\tilde{y}_{1,n}}{s-\chi_n} + \frac{\tilde{y}_{2,n}}{1-\eta_n}|}
{|1 - \frac{\tilde{m}_n}{s-\chi_n} n^{-\frac12} - \frac{\tilde{y}_{1,n}}{s-\chi_n} n^{-1}|}, \\
\tilde{B}_{2,n}
&:= \frac{|1-\eta_n|}{|s-\chi_n|^2} \;
\frac{|- \frac{\tilde{m}_n}{s-\chi_n}
- 2 \frac{\tilde{y}_{1,n}}{s-\chi_n} n^{-\frac12}
+ \frac{\tilde{y}_{2,n}}{1-\eta_n} n^{-\frac12}
+ (\frac{\tilde{m}_n}{s-\chi_n} n^{-\frac12} + \frac{\tilde{y}_{1,n}}{s-\chi_n} n^{-1})^2 n^\frac12|}
{|1 - \frac{\tilde{m}_n}{s-\chi_n} n^{-\frac12} - \frac{\tilde{y}_{1,n}}{s-\chi_n} n^{-1}|^2}.
\end{align*}
Finally recall (see equation (\ref{equnrnvnsn2})) that
$m_n, \tilde{m}_n, y_{1,n}, y_{2,n}, \tilde{y}_{1,n}, \tilde{y}_{2,n} = O(1)$
for all $n$ sufficiently large, and (see equation (\ref{eqn1})) that
$\frac32(1-\eta) > 1-\eta_n > \frac12 (1-\eta) > 0$
and $t-\chi_n > \frac56 (t-\chi) > 0$ and $s-\chi_n > \frac34 (s-\chi) > 0$.
Combined with the above expressions this gives
$B_{1,n}, B_{2,n}, \tilde{B}_{1,n}, \tilde{B}_{2,n} = O(1)$
for all $n$ sufficiently large. This proves (7).
Part (8) follows trivially from parts (5,7).

Consider (9). Recall that $f_{(t,s)}', f_n'$ are well-defined and analytic
in $B(s,2\xi)$ for all $n>N$. We will show the following:
\begin{equation}
\tag{ii}
\inf_{w \in \partial B(s,\xi)} |f_{(t,s)}'(w)| > 0
\hspace{0.5cm} \mbox{and} \hspace{0.5cm}
\lim_{n \to \infty} \sup_{w \in \text{cl}(B(s,\xi))} |f_{(t,s)}'(w) - f_n'(w)| = 0.
\end{equation}
This implies that we can choose $N$ such that the following is
satisfied for all $n>N$:
\begin{equation*}
\inf_{w \in \partial B(s,\xi)} |f_{(t,s)}'(w)|
> \sup_{w \in \text{cl}(B(s,\xi))} |f_{(t,s)}'(w) - f_n'(w)|. 
\end{equation*}
Then parts (1,2,3) and Rouch\'{e}'s theorem imply that $f_n'$ has exactly
$1$ root in $B(s,\xi)$, for all $n>N$ and counting multiplicities. Denote
this by $s_n$. Moreover, equation (\ref{eqfn2'}) implies that roots of $f_n'$
occur in complex conjugate pairs, and so $s_n$ must be real-valued. More
exactly, $s_n \in (s-\xi,s+\xi)$  for all $n>N$.

Next, for all $n>N$, recall that $f_{(t,s),n}', f_n'$ are well-defined and
analytic in $B(t,2\xi)$, and note that $B(t, 2n^{-\frac12}) \subset B(t, \xi)$
(see definition \ref{defxiN}). Also recall (see part (5)) that
$f_{(t,s),n}'(t) = 0$ for all $n>N$. Then, for all
$w \in \text{cl}(B(t, n^{-\frac12}))$ and $n>N$, Taylors theorem gives,
\begin{equation*}
f_{(t,s),n}'(w) = f_{(t,s),n}''(t) (w-t)
+ \int_t^w dz f_{(t,s),n}'''(z) (w-z),
\end{equation*}
where the integral is along the straight line from $t$ to $w$.
It follows that,
\begin{equation*}
\inf_{w \in \partial B(t, n^{-\frac12})} |f_{(t,s),n}'(w)|
\ge |f_{(t,s),n}''(t)| (n^{-\frac12})
- \sup_{z \in \text{cl}(B(t, n^{-\frac12}))} |f_{(t,s),n}'''(z)|
(n^{-\frac12})^2,
\end{equation*}
for all $n>N$. Recall, part (5) gives $|f_{(t,s),n}''(t)| > \frac12 |f_{(t,s)}''(t)|$
for all $n>N$. Moreover, we will show that we can choose $N$ such that the following
is satisfied for all $n>N$:
\begin{equation}
\tag{iii}
\sup_{z \in \text{cl}(B(t, n^{-\frac12}))} |f_{(t,s),n}'''(z)| n^{-\frac12}
< \tfrac14 |f_{(t,s)}''(t)|,
\end{equation}
Combined the above give, for all $n>N$,
\begin{equation*}
\inf_{w \in \partial B(t, n^{-\frac12})} |f_{(t,s),n}'(w)|
> \tfrac14 |f_{(t,s)}''(t)| (n^{-\frac12}).
\end{equation*}
Finally, we will show that we can choose $N$ such that for all $n>N$:
\begin{equation}
\tag{iv}
\tfrac14 |f_{(t,s)}''(t)| (n^{-\frac12})
>\sup_{w \in \text{cl}(B(t, n^{-\frac12}))} | f_{(t,s),n}'(w) - f_n'(w)|,
\end{equation}
Therefore, for all $n>N$,
\begin{equation*}
\inf_{w \in \partial B(t, n^{-\frac12})} |f_{(t,s),n}'(w)| >
\sup_{w \in \text{cl}(B(t, n^{-\frac12}))} | f_{(t,s),n}'(w) - f_n'(w) |.
\end{equation*}
Then parts (4,5,6) and Rouch\'{e}'s theorem imply that $f_n'$ has exactly $1$
root in $B(t,n^{-\frac12})$, for all $n>N$ and counting multiplicities. Denote
this by $t_n$. Moreover, equation (\ref{eqfn2'}) implies that roots of $f_n'$
occur in complex conjugate pairs, and $n^{-\frac12} < \frac12 \xi$ for all $n>N$
(see definition \ref{defxiN}), and so
$t_n \in (t-n^{-\frac12}, t-n^{-\frac12}) \subset (t - \frac12 \xi,t + \frac12 \xi)$.
This proves (9).

Consider (10). Recall that $f_{(t,s)}', \tilde{f}_n'$ are
well-defined and analytic in $B(t,2\xi)$ for all $n>N$. We will show the following:
\begin{equation}
\tag{v}
\inf_{w \in \partial B(t,\xi)} |f_{(t,s)}'(w)| > 0
\hspace{0.5cm} \mbox{and} \hspace{0.5cm}
\lim_{n \to \infty} \sup_{w \in \text{cl}(B(t,\xi))} |f_{(t,s)}'(w) - \tilde{f}_n'(w)| = 0.
\end{equation}
This implies that we can choose $N$ such that the following is
satisfied for all $n>N$:
\begin{equation*}
\inf_{w \in \partial B(t,\xi)} |f_{(t,s)}'(w)|
> \sup_{w \in \text{cl}(B(t,\xi))} |f_{(t,s)}'(w) - \tilde{f}_n'(w)|. 
\end{equation*}
Then parts (1,2,3) and Rouch\'{e}'s theorem imply that $\tilde{f}_n'$
has exactly $1$ root in $B(t,\xi)$, for all $n>N$ and counting multiplicities.
Denote it by $\tilde{t}_n$.
Moreover, equation (\ref{eqtildefn2'}) implies that roots of $\tilde{f}_n'$
occur in complex conjugate pairs, and so $\tilde{t}_n \in (t-\xi,t+\xi)$.

Next note, similar to part (9), $f_{(t,s),n}'(s) = 0$
for all $n>N$, and so Taylors theorem gives the following:
\begin{equation*}
\inf_{w \in \partial B(s, n^{-\frac12})} |f_{(t,s),n}'(w)|
\ge |f_{(t,s),n}''(s)| (n^{-\frac12})
- \sup_{z \in \text{cl}(B(s, n^{-\frac12}))} |f_{(t,s),n}'''(z)|
(n^{-\frac12})^2.
\end{equation*}
Recall that part (5) gives $|f_{(t,s),n}''(s)| > \frac12 |f_{(t,s)}''(s)|$
for all $n>N$. We will show that we can choose $N$ such that the following
are satisfied for all $n>N$:
\begin{align}
\tag{vi}
\sup_{z \in \text{cl}(B(s, n^{-\frac12}))} |f_{(t,s),n}'''(z)| n^{-\frac12}
&< \tfrac14 |f_{(t,s)}''(s)|, \\
\tag{vii}
\tfrac14 |f_{(t,s)}''(s)| (n^{-\frac12})
&>\sup_{w \in \text{cl}(B(s, n^{-\frac12}))} | f_{(t,s),n}'(w) - \tilde{f}_n'(w)|.
\end{align}
Then parts (4,5,6) and Rouch\'{e}'s theorem imply that $\tilde{f}_n'$ has exactly
$1$ root in $B(s,n^{-\frac12})$, for all $n>N$ and counting multiplicities. Denote
this by $\tilde{s}_n$. Moreover, equation (\ref{eqtildefn2'}) implies that roots
of $\tilde{f}_n'$ occur in complex conjugate pairs, and $n^{-\frac12} < \frac12 \xi$,
and so $\tilde{s}_n \in(s-n^{-\frac12}, s-n^{-\frac12}) \subset (s - \frac12 \xi,s + \frac12 \xi)$.
This proves (10).

Consider (11). First recall, part (4) gives $f_{(t,s),n}'(t) = f_{(t,s),n}'(s) = 0$
for all $n>N$. We can then use equation (\ref{eqftsn'2}), and proceed similarly
to the proof of part (4) of lemma \ref{lemf'} to get,
\begin{equation*}
f_{(t,s),n}''(t) = \frac1n \sum_{x \in P_n} \frac{\chi_n-x}{(t-x)^2 (t-\chi_n)}
\hspace{.5cm} \text{and} \hspace{.5cm}
f_{(t,s),n}''(s) = \frac1n \sum_{x \in P_n} \frac{\chi_n-x}{(s-x)^2 (s-\chi_n)},
\end{equation*}
for all $n>N$. Moreover, parts (9,10) give
$f_n'(t_n) = \tilde{f}_n'(\tilde{s_n}) = 0$ for all $n>N$. We can then use
equations (\ref{eqfn2'}, \ref{eqtildefn2'}) and proceed similarly to get,
\begin{equation*}
f_n''(t_n) = \frac1n \sum_{x \in P_n} \frac{v_n-x}{(t_n-x)^2 (t_n-v_n)}
\hspace{.5cm} \text{and} \hspace{.5cm}
\tilde{f}_n''(\tilde{s}_n) = \frac1n \sum_{x \in P_n} \frac{u_n-x}{(\tilde{s}_n-x)^2 (\tilde{s}_n-u_n)},
\end{equation*}
for all $n>N$. Next note, equation (\ref{equnrnvnsn2}) and parts (9,10) give
the following for all $n>N$: $|v_n - \chi_n| \le |m_n| n^{-\frac12} + |y_{1,n}| n^{-1}$,
$|u_n - \chi_n| \le |\tilde{m}_n| n^{-\frac12} + |\tilde{y}_{1,n}| n^{-1}$,
$|t_n-t| < n^{-\frac12} < \frac12 \xi$, and $|\tilde{s}_n-s| < n^{-\frac12} < \frac12 \xi$.
Finally recall that $x_1^{(n)} = \max P_n$ and $x_n^{(n)} = \min P_n$ (see equation
(\ref{eqPn})), and note equation (\ref{eqn1}) gives the following for all $n>N$ and
$x \in P_n$: $\max \{|\chi_n-x|, |u_n-x|, |v_n-x|\} < \min \{ 2(b-\chi), 2(\chi-a) \}$,
$|t-x| > \frac34 (t-b) > 0$, $|t-\chi_n| > \frac56 (t-\chi) > 0$,
$|s-x| > \frac12 (s-b) > 0$, $|s-\chi_n| > \frac34 (s-\chi) > 0$,
$|t_n-x| > \frac{23}{32} (t-b) > 0$, $|t_n-v_n| > \frac{39}{48} (t-\chi) > 0$,
$|\tilde{s}_n-x| > \frac7{16} (s-b) > 0$, and $|\tilde{s}-u_n| > \frac{23}{32} (s-\chi) > 0$. 
Combined, the above imply that we can choose $N$ sufficiently large
such that the following are also satisfied for all $n>N$:
\begin{equation*}
|f_{(t,s),n}''(t) - f_n''(t_n)| < \tfrac14 |f_{(t,s)}''(t)|
\hspace{.5cm} \text{and} \hspace{.5cm}
|f_{(t,s),n}''(s) - \tilde{f}_n''(\tilde{s}_n)| < \tfrac14 |f_{(t,s)}''(s)|.
\end{equation*}
Finally recall (see part (5)) that $f_{(t,s),n}''(t) > \frac12 f_{(t,s)}''(t) > 0$
and  $f_{(t,s),n}''(s) < \frac12 f_{(t,s)}''(s) < 0$ for all $n>N$.
This proves (11).

Consider (i). Note, the first part of (i) follows from the
extreme value theorem, since $f_{(t,s)}'$ is analytic in $B(t,2\xi)$
(see parts (1-3)). We prove the second
part of (i) via contradiction: Assume that the second part does not hold. Then
there exists a $\d > 0$ for which, for all $n\ge1$, there exists
some $p_n \ge n$ and $z_n \in \text{cl}(B(t,\xi))$ with
$\d < | f_{(t,s),n}'(z_n) - f_{(t,s),p_n}'(z_n) |$. Choosing $\{z_n\}_{n\ge1}$
to be convergent, and denoting the limit by $z$, the triangle inequality gives
\begin{equation}
\label{eqlemRootsNonAsy1}
\d < | f_{(t,s)}'(z_n) - f_{(t,s)}'(z) |
+ | f_{(t,s)}'(z) - f_{(t,s),p_n}'(z) |
+ | f_{(t,s),p_n}'(z) - f_{(t,s),p_n}'(z_n) |.
\end{equation}
Note, $| f_{(t,s)}'(z_n) - f_{(t,s)}'(z) | \to 0$ since $z_n \to z$,
$\{z,z_1,z_2\ldots\} \subset \text{cl}(B(t,\xi))$, and $f_{(t,s)}'$ is
analytic in $B(t,2\xi)$. Also, since $z \in \text{cl}(B(t,\xi))$, equations
(\ref{eqL4}, \ref{eqfts'}, \ref{eqftsn'2}) imply that
$| f_{(t,s)}'(z) - f_{(t,s),p_n}'(z) | \to 0$. Finally, equation
(\ref{eqftsn'2}) implies that,
\begin{equation*}
| f_{(t,s),p_n}'(z) - f_{(t,s),p_n}'(z_n) |
\le \sup_{x \in P_n} \left| \frac1{z-x} - \frac1{z_n-x} \right|
+ \left| \frac{1-\eta_n}{z-\chi_n} - \frac{1-\eta_n}{z_n-\chi_n} \right|.
\end{equation*}
Then, since $z_n \to z$ and $\{z,z_1,z_2\ldots\} \subset \text{cl}(B(t,\xi))$,
equation (\ref{eqn1}) implies that $| f_{p_n}'(z) - f_{p_n}'(z_n) | \to 0$.
The above observations contradict equation (\ref{eqlemRootsNonAsy1}),
and so our assumption is false. This proves the second part of (i).
Parts (ii,v) have similar proofs.

Consider (iii). First note, for all $n>N$, equation (\ref{eqftsn'2}) gives,
\begin{equation*}
f_{(t,s),n}'''(z)
= \frac1n \sum_{x \in P_n} \frac2{(z-x)^3} - \frac{2(1-\eta_n)}{(z-\chi_n)^3},
\end{equation*}
for all $z \in \text{cl}(B(t, n^{-\frac12}))$. Next recall that
$n^{-\frac12} < \frac12 \xi$ for all $n>N$(see definition \ref{defxiN}),
$x_1^{(n)} = \max P_n$ (see equation (\ref{eqPn})), and note equation
(\ref{eqn1}) gives the following for all $n>N$:
$1 > 1-\eta_n > 0$, $|z - \chi_n| > \frac{39}{48} (t-\chi) > 0$ for all
$z \in \text{cl}(B(t, n^{-\frac12}))$, and $|z-x| > \frac{23}{32} (t-b) > 0$
for all $z \in \text{cl}(B(t, n^{-\frac12}))$ and $x \in P_n$.
Thus, for all $n>N$,
\begin{equation*}
\sup_{z \in \text{cl}(B(t, n^{-\frac12}))} |f_{(t,s),n}'''(z)|
< \frac1n \sum_{x \in P_n} \frac2{(\frac{23}{32} (t-b))^3}
+ \frac2{(\frac{39}{48} (t-\chi))^3}
< \frac{2^3}{(t-b)^3} + \frac{2^2}{(t-\chi)^3}.
\end{equation*}
Part (iii) easily follows.

Consider (iv). Proceed similarly to the proof of part (7) above to get,
\begin{align*}
|f_{(t,s),n}'(w) - f_n'(w)|
&= \left| \frac{1-\eta_n}{w-\chi_n} \right| \;
\left| 1 - \frac{1 - \frac{m_n}{t-\chi_n} n^{-\frac12} - \frac{y_{2,n}}{1-\eta_n} n^{-1}}
{1 - \frac{m_n}{w-\chi_n} n^{-\frac12} - \frac{y_{1,n}}{w-\chi_n} n^{-1}} \right| \\
&= \left| \frac{1-\eta_n}{w-\chi_n} \right| \;
\left| \frac{(\frac{m_n}{t-\chi_n} - \frac{m_n}{w-\chi_n}) n^{-\frac12}
+ (\frac{y_{2,n}}{1-\eta_n} - \frac{y_{1,n}}{w-\chi_n}) n^{-1}}
{1 - \frac{m_n}{w-\chi_n} n^{-\frac12} - \frac{y_{1,n}}{w-\chi_n} n^{-1}} \right| 
\end{align*}
for all $n>N$ and $w \in \text{cl}(B(t, n^{-\frac12}))$. In particular
note that $|\frac{m_n}{t-\chi_n} - \frac{m_n}{w-\chi_n}| n^{-\frac12}
\le \frac{|m_n|}{|t-\chi_n| \; |w-\chi_n|} n^{-1}$ for all
$n>N$ and $w \in \text{cl}(B(t, n^{-\frac12}))$. Finally recall
that $m_n, y_{1,n}, y_{2,n} = O(1)$ for all $n$ sufficiently large
(see equation (\ref{equnrnvnsn2})), $n^{-\frac12} < \frac12 \xi$ for all $n>N$
(see definition \ref{defxiN}), and note equation (\ref{eqn1}) gives the
following for all $n>N$: $\frac32(1-\eta) > 1-\eta_n > \frac12 (1-\eta) > 0$,
$t-\chi_n > \frac56 (t-\chi) > 0$, and
$|w - \chi_n| > \frac{39}{48} (t-\chi) > 0$ for all
$w \in \text{cl}(B(t, n^{-\frac12}))$. This proves (iv).

Consider (vi). First note, for all $n>N$, equation (\ref{eqftsn'2}) gives,
\begin{equation*}
f_{(t,s),n}'''(z)
= \frac1n \sum_{x \in P_n} \frac2{(z-x)^3} - \frac{2(1-\eta_n)}{(z-\chi_n)^3},
\end{equation*}
for all $z \in \text{cl}(B(s, n^{-\frac12}))$. Next recall that
$n^{-\frac12} < \frac12 \xi$ for all $n>N$(see definition \ref{defxiN}),
$x_1^{(n)} = \max P_n$ (see equation (\ref{eqPn})), and note equation
(\ref{eqn1}) gives the following for all $n>N$:
$1 > 1-\eta_n > 0$, $|z - \chi_n| > \frac{23}{32} (s-\chi) > 0$ for all
$z \in \text{cl}(B(s, n^{-\frac12}))$, and $|z-x| > \frac7{16} (s-b) > 0$
for all $z \in \text{cl}(B(s, n^{-\frac12}))$ and $x \in P_n$.
Thus, for all $n>N$,
\begin{equation*}
\sup_{z \in \text{cl}(B(t, n^{-\frac12}))} |f_{(t,s),n}'''(z)|
< \frac1n \sum_{x \in P_n} \frac2{(\frac{7}{16} (s-b))^3}
+ \frac2{(\frac{23}{32} (s-\chi))^3}
< \frac{2^4}{(s-b)^3} + \frac{2^3}{(s-\chi)^3}.
\end{equation*}
Part (vi) easily follows.

Consider (vii). Proceed similarly to the proof of part (7) above to get,
\begin{align*}
|f_{(t,s),n}'(w) - \tilde{f}_n'(w)|
&= \left| \frac{1-\eta_n}{w-\chi_n} \right| \;
\left| 1 - \frac{1 - \frac{\tilde{m}_n}{s-\chi_n} n^{-\frac12} - \frac{\tilde{y}_{2,n}}{1-\eta_n} n^{-1}}
{1 - \frac{\tilde{m}_n}{w-\chi_n} n^{-\frac12} - \frac{\tilde{y}_{1,n}}{w-\chi_n} n^{-1}} \right| \\
&= \left| \frac{1-\eta_n}{w-\chi_n} \right| \;
\left| \frac{(\frac{\tilde{m}_n}{s-\chi_n} - \frac{\tilde{m}_n}{w-\chi_n}) n^{-\frac12}
+ (\frac{\tilde{y}_{2,n}}{1-\eta_n} - \frac{\tilde{y}_{1,n}}{w-\chi_n}) n^{-1}}
{1 - \frac{\tilde{m}_n}{w-\chi_n} n^{-\frac12} - \frac{\tilde{y}_{1,n}}{w-\chi_n} n^{-1}} \right| 
\end{align*}
for all $n>N$ and $w \in \text{cl}(B(s, n^{-\frac12}))$. In particular
note that $|\frac{\tilde{m}_n}{s-\chi_n} - \frac{\tilde{m}_n}{w-\chi_n}| n^{-\frac12}
\le \frac{|\tilde{m}_n|}{|s-\chi_n| \; |w-\chi_n|} n^{-1}$ for all
$n>N$ and $w \in \text{cl}(B(s, n^{-\frac12}))$. Finally recall
that $\tilde{m}_n, \tilde{y}_{1,n}, \tilde{y}_{2,n} = O(1)$ for all $n$ sufficiently large
(see equation (\ref{equnrnvnsn2})), $n^{-\frac12} < \frac12 \xi$ for all $n>N$
(see definition \ref{defxiN}), and note equation (\ref{eqn1}) gives the
following for all $n>N$: $\frac32(1-\eta) > 1-\eta_n > \frac12 (1-\eta) > 0$,
$s-\chi_n > \frac34 (s-\chi) > 0$, and
$|w - \chi_n| > \frac{23}{32} (s-\chi) > 0$ for all
$w \in \text{cl}(B(s, n^{-\frac12}))$. This proves (vii).
\end{proof}

Parts (8-11) of lemma \ref{lemN2} examine the behaviour of the roots of $f_n'$ and $\tilde{f}_n'$
in neighbourhoods of $t$ and $s$. Next we consider the remaining roots of $f_n'$ and $\tilde{f}_n'$
in their respective domains, outside of these neighbourhoods. Consider $f_n'$. First recall that
$\{x \in P_n : x > v_n \} \neq \emptyset$ and $\{x \in P_n : x < v_n \} \neq \emptyset$
for all $n>N$ (see definition \ref{defxiN}). Next note that
$x_1^{(n)} = \max \{x \in P_n : x > v_n\}$ and
$x_n^{(n)} = \min \{x \in P_n : x < v_n\}$ (see equation (\ref{eqPn})), and define,
\begin{equation*}
X_n(v_n) := \min \{x \in P_n : x > v_n\}
\hspace{.5cm} \text{and} \hspace{.5cm}
x_n(v_n) := \max \{x \in P_n : x < v_n\}.
\end{equation*}
Then, for all $n>N$, partition the domain of $f_n'$ as follows:
\begin{equation}
\label{eqfn'domain}
\C \setminus (P_n \cup \{v_n\}) = (\C \setminus \R) \cup J_n \cup K_n,
\end{equation}
where $J_n := \cup_{i=1}^4 J_{i,n}$, $K_n := \cup_{i=1,3} K_n^{(i)}$, and
\begin{itemize}
\item
$J_{1,n} := (x_1^{(n)}, + \infty)$.
\item
$J_{2,n} := (-\infty, x_n^{(n)})$.
\item
$J_{3,n} := (v_n, X_n(v_n))$.
\item
$J_{4,n} := (x_n(v_n), v_n)$.
\item
$K_n^{(1)} := [X_n(v_n), x_1^{(n)}] \setminus P_n$.
\item
$K_n^{(2)} := [x_n^{(n)}, x_n(v_n)] \setminus P_n$.
\end{itemize}
Partition each $K_n^{(i)}$ as $\{K_{1,n}^{(i)}, K_{2,n}^{(i)}, \ldots\}$,
a finite set of pairwise disjoint open intervals, unique up to order,
which satisfy $\{\inf I, \sup I\} \subset P_n$ for any
$I \in \{K_{1,n}^{(i)}, K_{2,n}^{(i)}, \ldots\}$. These sets
are depicted in figure \ref{figSnC_xi1}. Note that
$\sum_{i=1}^2 |\{K_{1,n}^{(i)},K_{2,n}^{(i)},\ldots\}| = |P_n| - 2 = n-2$
when $v_n \not\in P_n$, and
$\sum_{i=1}^2 |\{K_{1,n}^{(i)},K_{2,n}^{(i)},\ldots\}| = |P_n| - 3 = n-3$
when $v_n \in P_n$. Note, an analogous partition exists for
$\C \setminus (P_n \cup \{u_n\})$, the domain of $\tilde{f}_n'$,
and we denote the analogous quantities by $\tilde{J}_{1,n}, \tilde{J}_{2,n}$, etc.
Note, in particular, $J_{1,n} = \tilde{J}_{1,n} = (x_1^{(n)}, + \infty)$.

\begin{figure}[t]
\centering
\begin{tikzpicture}

\draw [dotted] (-2,0) --++(15.5,0);
\draw (-2.25,0) node {$\R$};
\draw (-2.25,.5) node {$\mathbb{H}$};

\draw (0,0) node {$\bullet$};
\draw (.5,0) node {$\times$};
\draw (1,0) node {$\bullet$};
\draw (1.5,0) node {$\times$};
\draw (2,0) node {$\bullet$};
\draw (4,0) node {$\bullet$};\
\draw (6,0) node {$\bullet$};
\draw (6.5,0) node {$\times$};
\draw (7,0) node {$\bullet$};
\draw (7.5,0) node {$\times$};
\draw (8,0) node {$\bullet$};
\draw (9,.1) --++(0,-.2);

\draw (0,-.4) node {\scriptsize $x_n^{(n)}$};
\draw (1,-.4) node {\scriptsize $<$};
\draw (2,-.4) node {\scriptsize $x_n(v_n)$};
\draw (3,-.4) node {\scriptsize $<$};
\draw (4,-.4) node {\scriptsize $v_n$};
\draw (5,-.4) node {\scriptsize $<$};
\draw (6,-.4) node {\scriptsize $X_n(v_n)$};
\draw (7,-.4) node {\scriptsize $<$};
\draw (8,-.4) node {\scriptsize $x_1^{(n)}$};
\draw (8.4,-.4) node {\scriptsize $<$};
\draw (9,-.4) node {\scriptsize $s-2\xi$};

\draw [dotted] (10.5,0) circle (.75cm);
\draw (9.6,.75) node {\scriptsize $B(s,\xi)$};
\draw (10.5,0) node {\scriptsize $\times$};
\draw (10.5,.25) node {\scriptsize $s_n$};

\draw [dotted] (12.5,0) circle (.75cm);
\draw (11.6,.75) node {\scriptsize $B(t,\xi)$};
\draw (12.5,0) node {\scriptsize $\times$};
\draw (12.5,.25) node {\scriptsize $t_n$};

\draw [thick,decorate,decoration={brace,amplitude=10pt,mirror},xshift=0.2pt,yshift=-0.2pt]
(-1.9,-.7) -- (-.2,-.7) node[black,midway,yshift=-0.6cm] {\scriptsize $J_{2,n}$};
\draw [thick,decorate,decoration={brace,amplitude=10pt,mirror},xshift=0.4pt,yshift=-0.4pt]
(2.1,-.7) -- (3.9,-.7) node[black,midway,yshift=-0.6cm] {\scriptsize $J_{4,n}$};
\draw [thick,decorate,decoration={brace,amplitude=10pt,mirror},xshift=0.4pt,yshift=-0.4pt]
(4.1,-.7) -- (5.9,-.7) node[black,midway,yshift=-0.6cm] {\scriptsize $J_{3,n}$};
\draw [thick,decorate,decoration={brace,amplitude=10pt,mirror},xshift=0.2pt,yshift=-0.2pt]
(8.1,-.7) -- (13.4,-.7) node[black,midway,yshift=-0.6cm] {\scriptsize $J_{1,n}$};

\end{tikzpicture}
\caption{The roots of $f_n'$ are represented by $\times$, and are
each of multiplicity $1$. Elements of $P_n \cup \{v_n\}$ are
represented by $\bullet$. Above, $K_n^{(1)} = [X_n(v_n), x_1^{(n)}] \setminus P_n$,
$K_n^{(3)} = [x_n^{(n)}, x_n(v_n)] \setminus P_n$.}
\label{figSnC_xi1}
\end{figure}
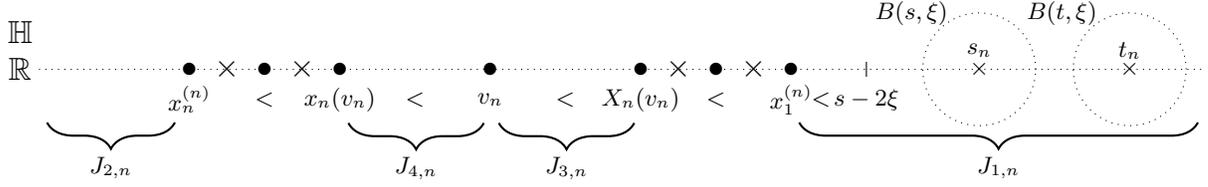

\begin{lem}
\label{lemfntildefn}
Fix $\xi,N$ as above and $n>N$, and define $t_n, s_n, \tilde{t}_n, \tilde{s}_n$
as in parts (9,10) of lemma \ref{lemN2}. Recall that
$(t-\xi, t+\xi) \cup (s-\xi, s+\xi) \subset J_{1,n}
= (x_1^{(n)},+\infty)$ and $t-\xi > s+\xi > s-\xi > x_1^{(n)}$
(see equation (\ref{eqn1})). Also recall that $t_n \in (t-\xi,t+\xi)$ and
$s_n \in (s-\xi, s+\xi)$ are roots of $f_n'$ of multiplicity $1$
(see parts (9) of lemma \ref{lemN2}). Then the following are satisfied for all $n>N$:
\begin{enumerate}
\item
$f_n'$ has a root of multiplicity $1$ at
$t_n \in (t-\xi,t+\xi) \subset J_{1,n} = (x_1^{(n)}, +\infty)$,
a root of multiplicity $1$ at
$s_n \in (s-\xi,s+\xi) \subset  J_{1,n} = (x_1^{(n)}, +\infty)$,
and $0$ roots in $J_{1,n} \setminus \{t_n,s_n\}$. Moreover,
$f_n |_{J_{1,n}}$ is real-valued, is strictly increasing in $(x_1^{(n)},s_n)$,
has a local maximum at $s_n$ ($f_n'(s_n) = 0$ and $f_n''(s_n) < 0$),
is strictly decreasing in $(s_n,t_n)$, has a local minimum at $t_n$
($f_n'(t_n) = 0$ and $f_n''(t) > 0$), and is strictly
increasing in $(t_n,+\infty)$.
\item
$f_n'$ has $0$ roots in $\C \setminus \R$, and in each of
$\{J_{2,n},J_{3,n},J_{4,n}\}$.
\item
$f_n'$ has exactly $1$ root, counting multiplicities, in each of
$\cup_{i=1}^2 \{K_1^{(i)},K_2^{(i)},\ldots\}$.
\end{enumerate}
Analogous results hold for $\tilde{f}_n$ with the analogous roots of $\tilde{f}_n'$,
$\tilde{t}_n \in (t-\xi, t+\xi)$ and $\tilde{s}_n \in (s-\xi, s+\xi)$.
\end{lem}

\begin{proof}
Fix $n>N$. We will show the following:
\begin{enumerate}
\item[(i)]
$f_n'$ has $\sum_{i=1}^2 |\{K_{1,n}^{(i)},K_{2,n}^{(i)},\ldots\}| + 2$ roots in
$\C \setminus (P_n \cup \{v_n\}) = (\C \setminus \R) \cup J_n \cup K_n$.
\item[(ii)]
$f_n'$ an odd number of roots in each of $\cup_{i=1}^2 \{K_{1,n}^{(i)},K_{2,n}^{(i)},\ldots\}$.
\item[(iii)]
$f_n |_{J_{1,n}}$ is real-valued, and $\lim_{w \to +\infty} w f_n'(w) > 2\xi > 0$. 
\end{enumerate}
Then, since $t_n \in (t-\xi,t+\xi) \subset J_{1,n}$ and
$s_n \in (s-\xi, s+\xi) \subset J_{1,n}$ are roots of $f_n'$ of
multiplicity $1$, parts (i,ii) and a simple counting argument imply that
the following: $f_n'$ has a root of multiplicity $1$ at
$t_n \in (t-\xi, t+\xi) \subset J_{1,n}$, a root of multiplicity $1$ at
$s_n \in (s-\xi, s+\xi) \subset J_{1,n}$, $0$ roots in each of
$\{\C \setminus \R, J_{1,n} \setminus \{t_n,s_n\}, J_{2,n},J_{3,n},J_{4,n}\}$,
and $1$ root in each of $\cup_{i=1}^2 \{K_{1,n}^{(i)},K_{2,n}^{(i)},\ldots\}$.
Moreover, since $f_n'$ has a root of multiplicity $1$ at both
$t_n,s_n \in J_{1,n}$ with $t_n > s_n$, and $0$ roots in
$J_{1,n} \setminus \{t_n,s_n\} = (x_1^{(n)}, +\infty) \setminus \{t_n,s_n\}$,
part (iii) implies that $f_n''(t_n) > 0$ and $f_n''(s_n) < 0$.
The above prove parts (1,2,3).

Consider (i). First note, for all $w \in \C \setminus (P_n \cup \{v_n\})$,
equation (\ref{eqfn2'}) gives,
\begin{equation*}
f_n'(w)
= \frac1n \sum_{x \in P_n \setminus \{v_n\}} \frac1{w-x}
- \frac{1-\frac{s_n-1}n - \frac1n 1_{(v_n \in P_n)}}{w-v_n},
\end{equation*}
Therefore, $f_n'(w) = \frac1n \frac1{w - v_n}
( \prod_{y \in P_n \setminus \{v_n\}} \frac1{w-y} ) Q_n(w)$, where
$Q_n$ is the polynomial,
\begin{align*}
Q_n(w)
&= (w - v_n) \sum_{x \in P_n \setminus \{v_n\}}
\bigg( \prod_{y \in P_n \setminus \{v_n,x\}} (w-y) \bigg) \\
&- ( n-(s_n-1) - 1_{(v_n \in P_n)} )
\bigg( \prod_{y \in P_n \setminus \{v_n\}} (w - y) \bigg).
\end{align*}
Note that $Q_n$ has no roots in $P_n \cup \{v_n\}$, and so the roots
of $Q_n$ and $f_n'$ coincide. Also note that $Q_n$ is a polynomial of
degree $|P_n| = n$ when $v_n \not\in P_n$, and of degree $|P_n|-1 = n-1$
when $v_n \in P_n$. Therefore, counting multiplicities, $f_n'$ has $n$ roots in
$\C \setminus (P_n \cup \{v_n\})$ when $v_n \not\in P_n$, and $n-1$
roots in $\C \setminus (P_n \cup \{v_n\})$ when $v_n \in P_n$. Finally
recall (see equation (\ref{eqfn'domain}) and the subsequent remarks)
$\sum_{i=1}^2 |\{K_{1,n}^{(i)},K_{2,n}^{(i)},\ldots\}| = n-2$ when
$v_n \not\in P_n$, and
$\sum_{i=1}^2 |\{K_{1,n}^{(i)},K_{2,n}^{(i)},\ldots\}| = n-3$ when
$v_n \in P_n$. This proves (i).

Consider (ii). Fix $i \in \{1,2\}$, and any interval
$I_n \in \{K_{1,n}^{(i)},K_{2,n}^{(i)},\ldots\}$. Recall that
$\inf I_n$ and $\sup I_n$ are either both consecutive elements of
$\{x \in  P_n : x > v_n\}$, or both consecutive
elements of $\{x \in  P_n : x < v_n\}$ (see
equation (\ref{eqfn'domain})). In both cases, equation
(\ref{eqfn2'}) gives,
\begin{equation*}
\lim_{w \in \R, w \uparrow \sup I_n} f_n'(w) = -\infty
\hspace{0.5cm} \text{and} \hspace{0.5cm}
\lim_{w \in \R, w \downarrow \inf I_n} f_n'(w) = +\infty.
\end{equation*}
This proves (ii).

Consider (iii). Fix $n>N$. First note, equation (\ref{eqfn2'})
implies that $f_n |_{J_{1,n}}$ is real-valued, and
$\lim_{w \to +\infty} w f_n'(w) = \frac{s_n-1}n$. Equation
(\ref{eqn1}) then gives $\lim_{w \to +\infty} w f_n'(w) > 2\xi > 0$.
This proves (iii).
\end{proof}

Finally, we examine Taylor expansions of $f_n$ in neighbourhoods
of $t$, and $\tilde{f}_n$ in neighbourhoods of $s$:
\begin{lem}
\label{lemTay}
Fix $\xi,N$ as above, and define $t_n, s_n, \tilde{t}_n, \tilde{s}_n$
as in parts (9,10) of lemma \ref{lemN2}. Fix $\theta \in (\tfrac13,\tfrac12)$
as in definition \ref{defxiN}. For all $n>N$, define:
\begin{align*}
b_n := |t_n + i n^{-\theta} - t| n^\theta
\hspace{.5cm} &\text{and} \hspace{.5cm}
\tilde{b}_n := |\tilde{s}_n + i n^{-\theta} - s| n^\theta, \\
\alpha_n := \text{Arg}(t_n + i n^{-\theta} - t)
\hspace{.5cm} &\text{and} \hspace{.5cm}
\tilde{\alpha}_n := \text{Arg}(\tilde{s}_n + i n^{-\theta} - s), \\
D_n := ( \tfrac12 |f_n''(t)| )^\frac12 \ge 0
\hspace{.5cm} &\text{and} \hspace{.5cm}
\tilde{D}_n := ( \tfrac12 |\tilde{f}_n''(s)| )^{\frac12} \ge 0.
\end{align*}
Then the following is satisfied for all $n>N$:
\begin{enumerate}
\item
$1 \le b_n < 2$ and $1 \le \tilde{b}_n < 2$,
$|b_n - 1| < n^{-\frac12+\theta}$
and 
$|\tilde{b}_n - 1| \le n^{-\frac12+\theta}$.
\item
$|\alpha_n - \tfrac\pi2| < n^{-\frac12+\theta}$
and 
$|\tilde{\alpha}_n - \tfrac\pi2| < n^{-\frac12+\theta}$.
\item
$D_n^2 > \frac18 |f_{(t,s)}''(t)| > 0$
and 
$\tilde{D}_n^2 > \frac18 |f_{(t,s)}''(s)| > 0$.
\end{enumerate}
Next note, for all $n>N$ equations (\ref{eqfn}, \ref{eqtildefn}, \ref{eqPn}, \ref{eqn1})
imply that $f_n,\tilde{f}_n$ are well-defined and analytic in the disjoint
sets $B(t,2\xi)$ and $B(s,2\xi)$. Recall, for all
$n>N$, that $n^{-\frac12} < \frac12 \xi$ (see definition \ref{defxiN}),
and $n^{-\theta} b_n < 2 \xi$ and $n^{-\theta} \tilde{b}_n < 2 \xi$ (see
definition \ref{defxiN} and part (1)). Finally,
$E_{1,n}, \tilde{E}_{1,n}, E_{2,n}, \tilde{E}_{2,n} = O(1)$ for all $n$ sufficiently large,
where $E_{1,n}, \tilde{E}_{1,n}, E_{2,n}, \tilde{E}_{2,n} > 0$ are defined
in the proof, and the following is satisfied for all $n>N$:
\begin{enumerate}
\setcounter{enumi}{3}
\item
$\sup_{w \in B(t, n^{-\frac12})} |f_n(w)-f_n(t)| \le E_{1,n} n^{-1}$.
\item
$\sup_{z \in B(s, n^{-\frac12})} |\tilde{f}_n(z)-\tilde{f}_n(s)| \le \tilde{E}_{1,n} n^{-1}$.
\item
$\sup_{w \in \text{cl}(B(0, n^{\frac12-\theta} b_n D_n ))}
|n f_n(t + n^{-\frac12} D_n^{-1} w) - n f_n(t) - w^2|
\le E_{2,n} n^{1-3\theta}$.
\item
$\sup_{w \in \text{cl}(B(0, n^{\frac12-\theta} \tilde{b}_n \tilde{D}_n ))}
|n \tilde{f}_n(s + n^{-\frac12} \tilde{D}_n^{-1} w) - n \tilde{f}_n(s) + w^2|
\le \tilde{E}_{2,n} n^{1-3\theta}$.
\end{enumerate}
\end{lem}

\begin{proof}
Fix $n>N$. Consider (1). First note, since
$b_n = \sqrt{1 + (t_n-t)^2 n^{2\theta}}$, it trivially follows that $b_n \ge 1$.
Next recall that $b_n = \sqrt{1 + (t_n-t)^2 n^{2\theta}}$,
$|t_n-t| n^\theta < n^{-\frac12 + \theta} < 1$ (see
definition \ref{defxiN} and part (9) of lemma \ref{lemN2}), and note that
$\sqrt{1 + x^2} < \sqrt2$ for all $x \in [0,1)$. This gives $b_n < \sqrt2$.
Finally recall that $b_n = \sqrt{1 + (t_n-t)^2 n^{2\theta}}$,
$|t_n-t| n^\theta < n^{-\frac12 + \theta} < 1$, and note that
$|\sqrt{1 + x^2} - 1| \le |x|$ for all $x \in [0,1)$ with equality only
when $x=0$. This gives $|b_n - 1| < n^{-\frac12+\theta}$.
We can similarly show that $1 \le \tilde{b}_n < \sqrt2$, and
$|\tilde{b}_n - 1| < n^{-\frac12+\theta}$. This proves (1).

Consider (2). First note, the definition of $\alpha_n$ gives,
\begin{equation*}
\alpha_n = \left\{ \begin{array}{rcl}
\arctan ( \frac1{(t_n-t) n^\theta} ) & ; & \text{when } t_n-t > 0, \\
\frac\pi2 & ; & \text{when } t_n-t = 0, \\
\arctan ( \frac1{(t_n-t) n^\theta} )  + \pi & ; & \text{when } t_n-t < 0.
\end{array} \right.
\end{equation*}
Recall that $|t_n-t| n^\theta < n^{-\frac12 + \theta} < 1$,
and note that $|\arctan(\frac1x) - \frac\pi2| \le |x|$ for all $x \in (0,1)$
with $\lim_{x \downarrow 0} |\arctan(\frac1x) - \frac\pi2| = 0$,
and $|\arctan(\frac1x) + \frac\pi2| \le |x|$ for all $x \in (-1,0)$
with $\lim_{x \uparrow 0} |\arctan(\frac1x) + \frac\pi2| = 0$. This gives
$|\alpha_n - \tfrac\pi2| < n^{-\frac12+\theta}$. Similarly we can show
that $|\tilde{\alpha}_n - \tfrac\pi2| < n^{-\frac12+\theta}$. This proves (2).
Consider (3). Recall that $D_n^2 = \tfrac12 |f_n''(t)|$ and
$\tilde{D}_n^2 = \tfrac12 |\tilde{f}_n''(s)|$.
Also recall (see part (8) of lemma \ref{lemN2}) 
that $f_n''(t) > \frac14 f_{(t,s)}''(t) > 0$ and 
$\tilde{f}_n''(s) < \frac14 f_{(t,s)}''(s) < 0$.
This proves (3).

Consider (4). First recall that $f_n$ is well-defined and analytic in
$B(t, n^{-\frac12})$. Then, for all $w \in B(t, n^{-\frac12})$, Taylors theorem gives,
\begin{equation*}
f_n(w) = f_n(t) + f_n'(t) (w-t) + \int_t^w dz \; f_n''(z) (w-z),
\end{equation*}
where the integral is along the straight line from $t$ to $w$.
Therefore,
\begin{equation*}
|f_n(w) - f_n(t)| \le |f_n'(t)| (n^{-\frac12})
+ \sup_{z \in B(t, n^{-\frac12})} |f_n''(z)| (n^{-\frac12})^2,
\end{equation*}
for all $w \in B(t, n^{-\frac12})$.
Next recall that $|f_n'(t)| = B_{1,n} n^{-1}$ where $B_{1,n} = O(1)$ for
all $n$ sufficiently large (see proof of part (7) of lemma \ref{lemN2}).
Therefore,
\begin{equation*}
|f_n(w) - f_n(t)| \le B_{1,n} (n^{-\frac32})
+ \sup_{z \in B(t, n^{-\frac12})} |f_n''(z)| (n^{-\frac12})^2,
\end{equation*}
for all $w \in B(t, n^{-\frac12})$. Next note, equation (\ref{eqfn2'}) gives,
\begin{equation*}
f_n''(z)
= - \frac1n \sum_{x \in P_n} \frac1{(z-x)^2} + \frac{1-\frac{s_n-1}n}{(z-v_n)^2},
\end{equation*}
for all $z \in B(t, n^{-\frac12})$. Next recall that
$n^{-\frac12} < \frac12 \xi$ (see definition \ref{defxiN}),
$x_1^{(n)} = \max P_n$ (see equation (\ref{eqPn})), and note equation
(\ref{eqn1}) gives the following:
$1 > 1-\frac{s_n-1}n > 0$, $|z-v_n| > \frac{39}{48} (t-\chi) > 0$ for all
$z \in B(t, n^{-\frac12})$, and $|z-x| > \frac{23}{32} (t-b) > 0$
for all $z \in B(t, n^{-\frac12})$ and $x \in P_n$.
Thus,
\begin{equation*}
\sup_{z \in B(t, n^{-\frac12})} |f_n''(z)|
< \frac1n \sum_{x \in P_n} \frac1{(\frac{23}{32} (t-b))^2}
+ \frac1{(\frac{39}{48} (t-\chi))^2}
< \frac2{(t-b)^2} + \frac2{(t-\chi)^2}.
\end{equation*}
Combined, the above give $|f_n(w) - f_n(t)| \le E_{1,n} n^{-1}$ for all
$w \in B(t, n^{-\frac12})$, where
\begin{equation*}
E_{1,n} := B_{1,n} n^{-\frac12} + \frac2{(t-b)^2} + \frac2{(t-\chi)^2}.
\end{equation*}
Recall that $B_{1,n} = O(1)$ for all $n$ sufficiently large. Thus
$E_{1,n} = O(1)$ for all $n$ sufficiently large. This proves (4).

Consider (5). Proceed similarly to the proof of part (4) to get,
\begin{equation*}
|\tilde{f}_n(w) - \tilde{f}_n(s)| \le \tilde{B}_{1,n} (n^{-\frac32})
+ \sup_{z \in B(s, n^{-\frac12})} |\tilde{f}_n''(z)| (n^{-\frac12})^2,
\end{equation*}
for all $w \in B(s, n^{-\frac12})$. Next note, equation (\ref{eqtildefn2'}) gives,
\begin{equation*}
\tilde{f}_n''(z)
= - \frac1n \sum_{x \in P_n} \frac1{(z-x)^2} + \frac{1-\frac{r_n+1}n}{(z-u_n)^2},
\end{equation*}
for all $z \in B(s, n^{-\frac12})$. Recall that $n^{-\frac12} < \frac12 \xi$,
$x_1^{(n)} = \max P_n$, and note equation (\ref{eqn1}) gives the following:
$1 > 1-\frac{r_n+1}n > 0$, $|z-u_n| > \frac{23}{32} (s-\chi) > 0$ for all
$z \in B(s, n^{-\frac12})$, and $|z-x| > \frac{7}{16} (s-b) > 0$
for all $z \in B(s, n^{-\frac12})$ and $x \in P_n$.
Thus,
\begin{equation*}
\sup_{z \in B(s, n^{-\frac12})} |f_n''(z)|
< \frac1n \sum_{x \in P_n} \frac1{(\frac{7}{16} (t-b))^2}
+ \frac1{(\frac{23}{32} (s-\chi))^2}
< \frac{2^3}{(t-b)^2} + \frac2{(s-\chi)^2}.
\end{equation*}
Combined, the above give $|\tilde{f}_n(w) - \tilde{f}_n(t)| \le \tilde{E}_{1,n} n^{-1}$
for all $w \in B(s, n^{-\frac12})$, where
\begin{equation*}
\tilde{E}_{1,n} := \tilde{B}_{1,n} n^{-\frac12} + \frac{2^3}{(s-b)^2} + \frac2{(s-\chi)^2}.
\end{equation*}
This proves (5).

Consider (6). First note Taylors theorem gives,
\begin{equation*}
f_n(t + n^{-\frac12} D_n^{-1} w) = f_n(t)
+ f_n'(t) (n^{-\frac12} D_n^{-1} w)
+ \tfrac12 f_n''(t) (n^{-\frac12} D_n^{-1} w)^2
+ \tfrac12 \int dz \; f_n'''(z) (t + n^{-\frac12} D_n^{-1} w-z)^2,
\end{equation*}
for all $w \in \text{cl}(B(0, n^{\frac12-\theta} b_n D_n))$,
where the integral is along the straight line from $t$ to
$t + n^{-\frac12} D_n^{-1} w$. Recall that $f_n''(t) > 0$
(see part (8) of lemma \ref{lemN2}), and $D_n^2 = \frac12 f_n''(t)$
(see statement of this lemma). Therefore,
\begin{equation*}
f_n(t + n^{-\frac12} D_n^{-1} w) = f_n(t)
+ f_n'(t) (n^{-\frac12} D_n^{-1} w)
+ n^{-1} w^2 + \tfrac12 \int dz \; f_n'''(z) (t + n^{-\frac12} D_n^{-1} w-z)^2.
\end{equation*}
for all $w \in \text{cl}(B(0, n^{\frac12-\theta} b_n D_n))$.
Next note, since $n^{-\frac12} D_n^{-1} w \in B(0, n^{-\theta} b_n)$
for all $w \in B(0, n^{\frac12-\theta} b_n D_n)$, and since
the integral is along the straight line from $t$ to
$t + n^{-\frac12} D_n^{-1} w$,
\begin{equation*}
|f_n(t + n^{-\frac12} D_n^{-1} w) - f_n(t) - n^{-1} w^2|
\le |f_n'(t)| (n^{-\theta} b_n)
+ \tfrac12 \sup_{z \in \text{cl}(B(t, n^{-\theta} b_n))} |f_n'''(z)| (n^{-\theta} b_n)^3,
\end{equation*}
for all $w \in \text{cl}(B(0, n^{\frac12-\theta} b_n D_n))$. Next recall
that $|f_n'(t)| = B_{1,n} n^{-1}$ (see proof of part (7) of lemma \ref{lemN2}),
$b_n < 2$ (see part (1)), and $n^{-\theta} b_n < 2 \xi$ (see definition \ref{defxiN}
and part (1)). Then,
\begin{equation*}
|f_n(t + n^{-\frac12} D_n^{-1} w) - f_n(t) - n^{-1} w^2|
< (B_{1,n} n^{-1}) (n^{-\theta} 2)
+ \tfrac12 \sup_{z \in \text{cl}(B(t, 2\xi))} |f_n'''(z)| (n^{-\theta} 2)^3,
\end{equation*}
for all $w \in \text{cl}(B(0, n^{\frac12-\theta} b_n D_n))$. Finally note,
equation (\ref{eqfn2'}) gives,
\begin{equation*}
f_n'''(z)
= \frac1n \sum_{x \in P_n} \frac2{(z-x)^3} - \frac{2(1-\frac{s_n-1}n)}{(z-v_n)^3},
\end{equation*}
for all $z \in \text{cl}(B(t, 2\xi))$. Recall $x_1^{(n)} = \max P_n$,
and note equation (\ref{eqn1}) gives the following:
$1 > 1-\frac{s_n-1}n > 0$, $|z - v_n| > \frac34 (t-\chi) > 0$ for all
$z \in \text{cl}(B(t, 2\xi))$, and $|z-x| > \frac58 (t-b) > 0$
for all $z \in \text{cl}(B(t, 2\xi))$ and $x \in P_n$.
Therefore,
\begin{equation*}
\sup_{z \in \text{cl}(B(t, 2\xi))} |f_n'''(z)|
< \frac1n \sum_{x \in P_n} \frac2{(\frac58 (t-b))^3}
+ \frac2{(\frac34 (t-\chi))^3}
< \frac{2^4}{(t-b)^3} + \frac{2^3}{(t-\chi)^3}.
\end{equation*}
Combined, the above give
$|f_n(t + n^{-\frac12} D_n^{-1} w) - f_n(t) - n^{-1} w^2| < E_{2,n} n^{-3\theta}$
for all $w \in \text{cl}(B(0, n^{\frac12-\theta} b_n D_n))$, where
\begin{equation*}
E_{2,n} := 2 B_{1,n} n^{2\theta-1} +
\frac{2^6}{(t-b)^3} + \frac{2^5}{(t-\chi)^3}.
\end{equation*}
Recall that $B_{1,n} = O(1)$ for all $n$ sufficiently large and
$2\theta-1 < 0$. Thus $E_{2,n} = O(1)$ for all $n$ sufficiently large.
This proves (6).

Consider (7). First note Taylors theorem gives,
\begin{equation*}
\tilde{f}_n(s + n^{-\frac12} \tilde{D}_n^{-1} w) = \tilde{f}_n(s)
+ \tilde{f}_n'(s) (n^{-\frac12} \tilde{D}_n^{-1} w)
+ \tfrac12 \tilde{f}_n''(s) (n^{-\frac12} \tilde{D}_n^{-1} w)^2
+ \tfrac12 \int dz \; \tilde{f}_n'''(z) (s + n^{-\frac12} \tilde{D}_n^{-1} w-z)^2,
\end{equation*}
for all $w \in \text{cl}(B(0, n^{\frac12-\theta} \tilde{b}_n \tilde{D}_n))$,
where the integral is along the straight line from $s$ to
$s + n^{-\frac12} \tilde{D}_n^{-1} w$. Recall that $\tilde{f}_n''(s) < 0$
(see part (8) of lemma \ref{lemN2}), and $\tilde{D}_n^2 = - \frac12 \tilde{f}_n''(s)$
(see statement of this lemma). Therefore,
\begin{equation*}
\tilde{f}_n(s + n^{-\frac12} \tilde{D}_n^{-1} w) = \tilde{f}_n(s)
+ \tilde{f}_n'(s) (n^{-\frac12} \tilde{D}_n^{-1} w)
- n^{-1} w^2
+ \tfrac12 \int dz \; \tilde{f}_n'''(z) (s + n^{-\frac12} \tilde{D}_n^{-1} w-z)^2,
\end{equation*}
for all $w \in \text{cl}(B(0, n^{\frac12-\theta} \tilde{b}_n \tilde{D}_n))$
Then, proceed similarly to part (6) to get,
\begin{equation*}
|\tilde{f}_n(s + n^{-\frac12} \tilde{D}_n^{-1} w) - \tilde{f}_n(s) + n^{-1} w^2|
< (\tilde{B}_{1,n} n^{-1}) (n^{-\theta} 2)
+ \tfrac12 \sup_{z \in \text{cl}(B(s, 2\xi))} |\tilde{f}_n'''(z)| (n^{-\theta} 2)^3,
\end{equation*}
for all $w \in \text{cl}(B(0, n^{\frac12-\theta} \tilde{b}_n \tilde{D}_n))$.
Next note, equation (\ref{eqtildefn2'}) gives,
\begin{equation*}
\tilde{f}_n'''(z)
= \frac1n \sum_{x \in P_n} \frac2{(z-x)^3} - \frac{2(1-\frac{r_n+1}n)}{(z-u_n)^3},
\end{equation*}
for all $z \in \text{cl}(B(s, 2\xi))$. Recall $x_1^{(n)} = \max P_n$,
and note equation (\ref{eqn1}) gives the following:
$1 > 1-\frac{r_n+1}n > 0$, $|z - u_n| > \frac58 (s-\chi) > 0$ for all
$z \in \text{cl}(B(s, 2\xi))$, and $|z-x| > \frac14 (s-b) > 0$
for all $z \in \text{cl}(B(s, 2\xi))$ and $x \in P_n$.
Therefore,
\begin{equation*}
\sup_{z \in \text{cl}(B(s, 2\xi))} |f_n'''(z)|
< \frac1n \sum_{x \in P_n} \frac2{(\frac14 (s-b))^3}
+ \frac2{(\frac58 (s-\chi))^3}
< \frac{2^7}{(s-b)^3} + \frac{2^4}{(s-\chi)^3}.
\end{equation*}
Combined, the above give
$|\tilde{f}_n(s + n^{-\frac12} \tilde{D}_n^{-1} w) - \tilde{f}_n(s) + n^{-1} w^2|
< \tilde{E}_{2,n} n^{-3\theta}$ for all
$w \in \text{cl}(B(0, n^{\frac12-\theta} \tilde{b}_n \tilde{D}_n))$, where
\begin{equation*}
\tilde{E}_{2,n} := 2 \tilde{B}_{1,n} n^{2\theta-1} +
\frac{2^9}{(s-b)^3} + \frac{2^6}{(s-\chi)^3}.
\end{equation*}
This proves (7).
\end{proof}

\subsection{The contours of descent/ascent}
\label{secCont}

In this section we define the contours to be used in the steepest
descent analysis. First define:
\begin{definition}
\label{defRnIn}
As in the previous section, fix $\xi$ and $N$ and $\theta \in (\frac13,\frac12)$,
and define $t_n, \tilde{s}_n, u_n,r_n,v_n,s_n$ for all $n>N$.
Recall that $t_n - v_n > 0$ and $\tilde{s}_n - u_n > 0$ (see parts (9,10)
of lemma \ref{lemN2} and equation (\ref{eqn1})). Define for all $n>N$:
\begin{align*}
q_n
&:= |t_n + i n^{-\theta} - v_n|, \\
\tilde{q}_n
&:= |\tilde{s}_n + i n^{-\theta} - u_n|.
\end{align*}
Also define $R_n : (0,1) \to \R$ and $I_n : (0,1) \to \R$ as follows
for all $n>N$:
\begin{align*}
R_n(y)
&:= (\tilde{s}_n - u_n)
\left( 1 - \frac12 \frac{(\tilde{q}_n)^2}{(\tilde{s}_n-u_n)^2} \log(y) \right) y, \\
I_n(y)
&:= \sqrt{(\tilde{q}_n)^2 y - R_n(y)^2},
\end{align*}
for all $y \in (0,1)$.
\end{definition}
The next lemma shows that $I_n$ is well-defined
and other useful properties: 
\begin{lem}
\label{lemRnIn}
For all $n>N$:
\begin{enumerate}
\item
$R_n$ strictly increases in $(0,1)$ with $\lim_{y \downarrow 0} R_n(y) = 0$
and $\lim_{y \uparrow 1} R_n(y) = \tilde{s}_n - u_n$.
\item
$(\tilde{q}_n)^2 y - R_n(y)^2 > 0$ for all $y \in (0,1)$, and so $I_n(y)$ is
well-defined and $I_n(y) > 0$ for all $y \in (0,1)$. Moreover,
$\lim_{y \downarrow 0} I_n(y) = 0$ and $\lim_{y \uparrow 1} I_n(y) = n^{-\theta}$.
\end{enumerate}
\end{lem}

\begin{proof}
Fix $n>N$. Consider (1). First note, definition \ref{defRnIn} gives
$\tilde{s}_n - u_n > 0$ and,
\begin{equation*}
R_n'(y) = (\tilde{s}_n - u_n)
\left( 1 - \frac12 \frac{(\tilde{q}_n)^2}{(\tilde{s}_n-u_n)^2}
- \frac12 \frac{(\tilde{q}_n)^2}{(\tilde{s}_n-u_n)^2} \log(y) \right),
\end{equation*}
for all $y \in (0,1)$. Next recall 
$\tilde{q}_n = |\tilde{s}_n + i n^{-\theta} - u_n|$ (see definition
\ref{defRnIn}), and so
\begin{equation*}
1 - \frac12 \frac{\tilde{q}_n^2} {(\tilde{s}_n-u_n)^2}
= \frac12 - \frac12 \frac{n^{-2\theta}}{(\tilde{s}_n-u_n)^2}.
\end{equation*}
Thus $1 - \frac12 \frac{\tilde{q}_n^2} {(\tilde{s}_n-u_n)^2} > 0$
since $n^{-\theta} < \xi < \frac1{16} (s-\chi)$ (see definition
\ref{defxiN} and equation (\ref{eqn1})) and
$\tilde{s}_n - u_n > \frac{23}{32} (s-\chi) > 0$ (see equation
(\ref{eqn1}) and part (10) of lemma \ref{lemN2}). Moreover,
note that $\log(y) < 0$ for all $y \in (0,1)$. Combined, the above give
$R_n'(y) > 0$ for all $y \in (0,1)$. Moreover, definition
\ref{defRnIn} easily gives $\lim_{y \downarrow 0} R_n(y) = 0$
and $\lim_{y \uparrow 1} R_n(y) = \tilde{s}_n - u_n$.
This proves (1).

Consider (2). First note, part (1) gives $R_n(y) > 0$ for all $y \in (0,1)$.
Thus, to prove (2), it is thus sufficient to show that $\tilde{q}_n \sqrt{y} > R_n(y)$ for
all $y \in (0,1)$, i.e. (see definition \ref{defRnIn}) that,
\begin{equation*}
\left( 1 + \frac{n^{-2\theta}}{(\tilde{s}_n-u_n)^2} \right)^\frac12
> \left( 1 - \frac12 \left( 1 + \frac{n^{-2\theta}}{(\tilde{s}_n-u_n)^2} \right) \log(y) \right) \sqrt{y}.
\end{equation*}
We will show that the following are satisfied for all $y \in (0,1)$:
\begin{align*}
&\text{(i)} \hspace{.5cm}
\left( 1 + \frac{n^{-2\theta}}{(\tilde{s}_n-u_n)^2} \right)^\frac12
> 1 - \frac{n^{-2\theta}}{(\tilde{s}_n-u_n)^2} y \log(y). \\
&\text{(ii)} \hspace{.5cm}
1 - \frac{n^{-2\theta}}{(\tilde{s}_n-u_n)^2} y \log(y)
> \left( 1 - \left( 1 + \frac{n^{-2\theta}}{(\tilde{s}_n-u_n)^2} \right) \log(y) \right) y.
\end{align*}
Replacing $y$ in (i,ii) by $\sqrt{y}$ gives the required inequality. This proves (2).

Consider (i). Note that the RHS of this inequality is positive
for all $y \in (0,1)$, since $0 > y \log(y)$. Thus, squaring both
sides, it is sufficient to show that,
\begin{equation*}
1 + \frac{n^{-2\theta}}{(\tilde{s}_n-u_n)^2}
> 1 - 2 \frac{n^{-2\theta}}{(\tilde{s}_n-u_n)^2} y \log(y)
+ \frac{n^{-4\theta}}{(\tilde{s}_n-u_n)^4} (y \log(y))^2,
\end{equation*}
for all $y \in (0,1)$. Rewriting, it is sufficient to show that,
\begin{equation*}
1 + 2 y \log(y) > \frac{n^{-2\theta}}{(\tilde{s}_n-u_n)^2} (y \log(y))^2,
\end{equation*}
for all $y \in (0,1)$. Next note that $0 > y \log(y) > -e^{-1}$ for
all $y \in (0,1)$. Thus it is sufficient to show that,
\begin{equation*}
1 - 2 e^{-1} > \frac{n^{-2\theta}}{(\tilde{s}_n-u_n)^2} e^{-2}.
\end{equation*}
Finally note that $n^{-\theta} < \xi < \frac1{16} (s-\chi)$ (see definition
\ref{defxiN} and equation (\ref{eqn1})) and
$\tilde{s}_n - u_n > \frac{23}{32} (s-\chi) > 0$ (see equation
(\ref{eqn1}) and part (10) of lemma \ref{lemN2}). Thus it is
sufficient to show that $1 - 2 e^{-1} > (\frac2{23})^2 e^{-2}$
This is trivially true. This proves (i).
Consider (ii). Rewriting, it is sufficient to show that
$1 > y (1-\log(y))$ for all $y \in (0,1)$. This is trivially
true. This proves (ii). 
\end{proof}

We now use the quantities in definition \ref{defRnIn} to define the contours
to be used in the steepest descent analysis. Extend the definition
of $R_n, I_n : (0,1) \to \R$ to the end-points $\{0,1\}$ using the
well-defined limits shown lemma \ref{lemRnIn}, and define:
\begin{definition}
\label{defDesAsc}
For all $n>N$, let $\g_n^+$ to be the contour which:
\begin{itemize}
\item
starts at $t \in (s, +\infty)$,
\item
then traverses the straight line from $t$ to $t_n + i n^{-\theta}$,
\item
then traverses the counter-clockwise arc of $\partial B(v_n, q_n)$ from
$t_n + i n^{-\theta}$ to $v_n - q_n$,
\item
then ends at $v_n - q_n$.
\end{itemize}
Note, $\g_n^+$ is trivially a continuous contour which begins and ends in $\R$, and is
otherwise contained in $\mathbb{H}$. Let $\g_n^-$ be the reflection of $\g_n^+$ in
$\R$, let $\g_n$ be the continuous closed contour given by $\g_n = \g_n^+ + \g_n^-$ with
counter-clockwise orientation.

For all $n>N$, let $\G_n^+$ to be the contour which:
\begin{itemize}
\item
starts at $s \in (b, t)$,
\item
then traverses the straight line from $s$ to $\tilde{s}_n + i n^{-\theta}$,
\item
then traverses the contour $y \mapsto u_n + R_n(1-y) + i I_n(1-y)$ for $y \in [0,1]$,
\item
then ends at $u_n$.
\end{itemize}
Note, lemma \ref{lemRnIn} implies that $\G_n^+$ is a continuous contour which
begins and ends in $\R$, and is otherwise contained in $\mathbb{H}$.
Define $\G_n^-$ and $\G_n$ analogously to above.
\end{definition}

The next result prove properties of the contours which are useful for the steepest descent analysis,
and figure \ref{figContoursAscDes} depicts $\g_n^+$ and $\G_n^+$:
\begin{lem}
\label{lemDesAsc}
The following are satisfied for all $n>N$:
\begin{enumerate}
\item
$\g_n$ contains $v_n$ and $\G_n$.
\item
$\G_n$ contains $\{x \in P_n : x > u_n\}$ and does not contain any of
$\{x \in P_n : x < u_n\}$.
\item
$\text{Re}(f_n(w)) \le \text{Re}(f_n(t_n + i n^{-\theta}))$ for all
$w$ on that section of $\g_n^+$ given by the counter-clockwise arc
of $\partial B(v_n, q_n)$ from $t_n + i n^{-\theta}$ to $v_n - q_n$.
\item
$\text{Re}(\tilde{f}_n(z)) \ge \text{Re}(\tilde{f}_n(\tilde{s}_n + i n^{-\theta}))
= \text{Re}(\tilde{f}_n(u_n + R_n(1) + i I_n(1)))$ for all $z$ on that section of
$\G_n^+$ given by the contour $y \mapsto u_n + R_n(1-y) + i I_n(1-y)$ for $y \in [0,1]$.
\item
$|w-z| \ge \frac12 (t-s)$ for all $w \in \g_n$ and $z \in \G_n$.
\item
$|\g_n| \le 8(t-\chi)$, where $|\cdot|$ represents the length.
\item
$|\G_n| \le 8(s-\chi)$.
\end{enumerate}
\end{lem}

\begin{proof}
Fix $n>N$. Consider (1). First note, definition \ref{defDesAsc} trivially
implies that $\g_n$ contains $v_n$. Next recall (see definition \ref{defDesAsc}),
$\g_n^+$ starts at $t$ and ends at $v_n-q_n = v_n - |t_n + i n^{-\theta} - v_n|$,
and $\G_n^+$ starts at $s$ and ends at $u_n$. Moreover, both contours are otherwise
contained in $\mathbb{H}$, and $t > s > u_n > v_n - |t_n + i n^{-\theta} - v_n|$
(this follows from equation (\ref{eqn1}), and since $n^{-\theta} < \xi$ by definition
\ref{defxiN}, and since $|t_n-t| < \frac12 \xi$ by part (9) of lemma \ref{lemN2}).
Thus $\g_n^+$ contains the start and end points of $\G_n^+$. Moreover,
we will show in the proof of part (5), below, that $|w-z| \ge \frac12 (t-s)$ for all
$w \in \g_n^+$ and $z \in \G_n^+$. Combined, the above imply that
$\g_n^+$ contains $\G_n^+$. This proves (1).
Consider (2). Note that definition \ref{defDesAsc} and lemma \ref{lemRnIn} imply that
$\G_n^+$ starts at $s$, ends at $u_n$, and is otherwise contained in $\mathbb{H}$.
Also, equations (\ref{eqPn}, \ref{eqn1}) give $s > x_1^{(n)} = \max P_n$. Part
(2) easily follows.

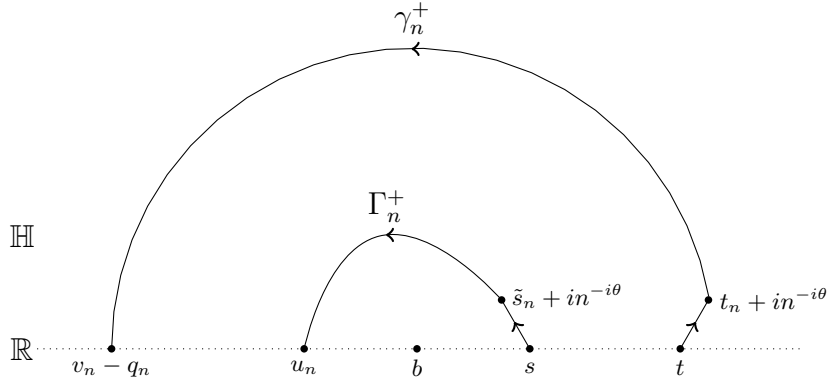
\begin{figure}
\centering

\begin{tikzpicture};

\fill [black] (8,0) circle (.05cm);
\draw (8,-.25) node {\scriptsize $t$};
\draw (8,0) --++(.375,.65);
\draw[arrows=->,line width=1pt](8.2,.34641)--(8.205,.35507);
\fill [black] (8.375,.65) circle (.05cm);
\draw (9.25,.65) node {\scriptsize $t_n + i n^{-i\theta}$};
\draw [domain=10:180] plot ({4.44+(4)*(cos(\x))}, {(4)*(sin(\x))});
\fill [black] (0.44,0) circle (.05cm);
\draw (0.44,-.25) node {\scriptsize $v_n-q_n$};
\draw[arrows=->,line width=1pt](4.44,3.99)--(4.43,3.99);
\draw (4.44,4.4) node {$\gamma_n^+$};

\fill [black] (6,0) circle (.05cm);
\draw (6,-.25) node {\scriptsize $s$};
\draw (6,0) --++(-.375,.65);
\draw[arrows=->,line width=1pt](5.8,.34641)--(5.795,.35507);
\fill [black] (5.625,.65) circle (.05cm);
\draw (6.5,.7) node {\scriptsize $\tilde{s}_n + i n^{-i\theta}$};
\draw plot [smooth, tension=1] coordinates
{ (5.625,.65) (4,1.5) (3,0)};
\fill [black] (3,0) circle (.05cm);
\draw (3,-.25) node {\scriptsize $u_n$};
\draw[arrows=->,line width=1pt](4.1,1.515)--(4.09,1.515);
\draw (4.1,1.9) node {$\Gamma_n^+$};

\draw (-.75,1.5) node {$\mathbb{H}$};
\draw [dotted] (-.55,0) --++(10.15,0);
\draw (-.75,0) node {$\mathbb{R}$};
\fill [black] (4.5,0) circle (.05cm);
\draw (4.5,-.25) node {\scriptsize $b$};

\end{tikzpicture}
\caption{The contours $\g_n^+$ and $\G_n^+$ defined in definition \ref{defDesAsc}, and whose properties
are examined in lemma \ref{lemDesAsc}. We remind the reader that $t_n \to t$ and $s_n \to s$ and
$u_n \to \chi$ and $v_n - q_n \to \chi - (t - \chi)$ as $n \to \infty$,
and that $t > s > b > \chi > \chi - (t -\chi)$.}
\label{figContoursAscDes}
\end{figure}

Consider (3). Define $g_n(y) := \text{Re}(f_n(v_n + q_n e^{iy}))$ for all $y \in \R$.
Note, to prove (3), it is sufficient to show that $g_n'(y) < 0$ for all $y \in (0,\pi)$.
To prove this, first note, for all $y \in \R$, equations (\ref{eqfn}, \ref{eqPn}) give,
\begin{align*}
g_n(y)
&= \frac1{2n} \sum_{x \in P_n} \log |v_n + q_n e^{iy} - x|^2
- (1 - \tfrac{s_n-1}n) \log|v_n + q_n e^{iy} - v_n| \\
&= \frac1{2n} \sum_{x \in P_n} \log ((v_n-x)^2 + 2(v_n-x) q_n \cos(y) + q_n^2)
- (1 - \tfrac{s_n-1}n) \log (q_n),
\end{align*}
where log is the natural logarithm.
Differentiate to get $g_n'(y) = h_n(y) \sin(y)$ and
$g_n''(y) = h_n'(y) \sin(y) + h_n(y) \cos(y)$ for all $y \in \R$, where
\begin{equation*}
h_n(y)
= - \frac1{n} \sum_{x \in P_n} \frac{(v_n-x) q_n}{(v_n-x)^2 + 2(v_n-x) q_n \cos(y) + q_n^2}.
\end{equation*}
We will show:
\begin{enumerate}
\item[(i)]
$g_n'(0) = 0$ and $g_n''(0) = h_n(0)$ and $h_n(0) < 0$.
\item[(ii)]
Assume that there exists a $Y \in (0,\pi)$ for which $g_n'(Y) = 0$.
Then $g_n''(Y) = h_n'(Y) \sin(Y)$ and $h_n'(Y) \sin(Y) < 0$.
\end{enumerate}
Part (i) implies that $0$ is a local maximum of $g_n : \R \to \R$.
Moreover, part (ii) implies that any extrema of $g_n'$ in $(0,\pi)$ is
also a local maximum. It follows that that $g_n$ has no extrema in $(0,\pi)$,
and that $g_n'(y) < 0$ for all $y \in (0,\pi)$. This proves (3).

Consider (4). Define $G_n(y) := \text{Re}(\tilde{f}_n(u_n + R_n(y) + i I_n(y)))$
for all $y \in [0,1]$. To prove (4), it is sufficient to show that
$G_n'(y) < 0$ for all $y \in (0,1)$. Remark that for all
$y \in (0,1)$, equations (\ref{eqtildefn}, \ref{eqPn}) give,
\begin{equation*}
G_n(y)
= \frac1{2n} \sum_{x \in P_n} \log |u_n + R_n(y) + i I_n(y) - x|^2
- \tfrac12 (1 - \tfrac{r_n+1}n) \log|R_n(y) + i I_n(y)|^2,
\end{equation*}
where log is the natural logarithm. Then, since $R_n(y)^2 + I_n(y)^2 = (\tilde{q}_n)^2 y$
for all $y \in (0,1)$ (see definition \ref{defRnIn}),
\begin{equation*}
G_n(y)
= \frac1{2n} \sum_{x \in P_n} \log ((u_n -x)^2 + 2(u_n-x) R_n(y) + (\tilde{q}_n)^2 y)
- \tfrac12 (1 - \tfrac{r_n+1}n) \log ((\tilde{q}_n)^2 y).
\end{equation*}
Therefore, for all $y \in (0,1)$,
\begin{equation*}
G_n'(y)
= \frac1{2n} \sum_{x \in P_n} \frac{2(u_n-x) R_n'(y) + (\tilde{q}_n)^2}
{(u_n -x)^2 + 2(u_n-x) R_n(y) + (\tilde{q}_n)^2 y}
- \tfrac12 (1 - \tfrac{r_n+1}n) \frac1y.
\end{equation*}
Next recall that $\tilde{f}_n'(\tilde{s}_n) = 0$ (see part (10) of lemma
\ref{lemN2}). Equation (\ref{eqtildefn2'}) thus gives,
\begin{equation*}
1-\tfrac{r_n+1}n = \frac1n \sum_{x \in P_n} \frac{\tilde{s}_n-u_n}{\tilde{s}_n-x}.
\end{equation*}
Therefore, for all $y \in (0,1)$,
\begin{equation*}
G_n'(y)
= \frac1{2n} \sum_{x \in P_n} \frac{2(u_n-x) R_n'(y) + (\tilde{q}_n)^2}
{(u_n -x)^2 + 2(u_n-x) R_n(y) + (\tilde{q}_n)^2 y}
- \frac1{2n} \sum_{x \in P_n} \frac{\tilde{s}_n-u_n}{(\tilde{s}_n-x) y}.
\end{equation*}
Rewriting gives, for all $y \in (0,1)$,
\begin{equation*}
G_n'(y)
= \frac1{2n} \sum_{x \in P_n} \frac{(u_n-x) H_n(y) + (u_n-x)^2 M_n(y)}
{((u_n -x)^2 + 2(u_n-x) R_n(y) + (\tilde{q}_n)^2 y) (\tilde{s}_n-x) y},
\end{equation*}
where,
\begin{align*}
H_n(y)
&:= 2 (\tilde{s}_n-u_n) y R_n'(y) + (\tilde{q}_n)^2 y - 2 (\tilde{s}_n-u_n) R_n(y), \\
M_n(y)
&:= 2 y R_n'(y) - (\tilde{s}_n-u_n).
\end{align*}
We will show:
\begin{enumerate}
\item[(iii)]
$H_n(y) = 0$ for all $y \in (0,1)$.
\item[(iv)]
$M_n(y) < 0$ for all $y \in (0,1)$.
\end{enumerate}
Finally recall that $(u_n -x)^2 + 2(u_n-x) R_n(y) + (\tilde{q}_n)^2 y
= |u_n + R_n(y) + i I_n(y) - x|^2 > 0$ for all $y \in (0,1)$,
and $\tilde{s}_n-x > \frac7{16} (s-b) > 0$ for all $x \in P_n$
(see equation (\ref{eqn1}) and part (10) of lemma \ref{lemN2}).
Combined, the above give $G_n'(y) < 0$ for all $y \in (0,1)$.
This proves (4).

Consider (5). First recall that $|t_n+in^{-\theta}-t| < 2n^{-\theta}$
(see part (1) of lemma \ref{lemTay}), and that part of $\g_n^+$ outside
$B(t, |t_n+in^{-\theta}-t|)$ is a subset of $\partial B(v_n, q_n)$ (see
definition \ref{defDesAsc}). Also, $|\tilde{s}_n+in^{-\theta}-s| < 2n^{-\theta}$
(see part (1) of lemma \ref{lemTay}), and that part of $\G_n^+$ outside
$B(s, |\tilde{s}_n+in^{-\theta}-s|)$ is a subset of the contour
$y \mapsto u_n + R_n(y) + i I_n(y)$ for $y \in [0,1]$ (see definition
\ref{defDesAsc}). We will show:
\begin{align*}
\tag{v}
&\inf_{w \in B(t, 2n^{-\theta})} \; \inf_{z \in B(s, 2n^{-\theta})} |w-z|
> \tfrac12 (t-s). \\
\tag{vi}
&\inf_{w \in B(t, 2n^{-\theta})} \; \inf_{y \in [0,1]} |w -(u_n + R_n(y) + i I_n(y))|
> \tfrac12 (t-s). \\
\tag{vii}
&\inf_{w \in \partial B(v_n, q_n)} \; \inf_{z \in B(s, 2n^{-\theta})} |w-z|
> \tfrac12 (t-s). \\
\tag{viii}
&\inf_{w \in \partial B(v_n, q_n)} \; \inf_{y \in [0,1]} |w -(u_n + R_n(y) + i I_n(y))|
> \tfrac12 (t-s).
\end{align*}
Combined, the above give $|w-z| \ge \frac12 (t-s)$ for all $w \in \g_n^+$ and $z \in \G_n^+$.
Finally recall that $\g_n^+$ is a continuous contour which begins and ends in $\R$ and is
otherwise in $\mathbb{H}$, $\g_n^-$ is the reflection of $\g_n^+$ in
$\R$, and $\g_n = \g_n^+ + \g_n^-$. Similarly for $\G_n^+, \G_n^-, \G_n$.
This proves (5).

Consider (6). Definitions \ref{defRnIn} and \ref{defDesAsc} trivially give
$|\g_n| \le 2|t_n + in^{-\theta} - t| + 2\pi |t_n + i n^{-\theta} - v_n|$.
Therefore,
\begin{equation*}
|\g_n| \le 2|t_n-t| + 2\pi |t_n - v_n| + 2(1+\pi) n^{-\theta}. 
\end{equation*}
Next, definition \ref{defxiN}, part (9) of lemma
\ref{lemN2}, and equation (\ref{eqn1}) give the following:
$|t_n - t| < \frac12 \xi < \frac1{48} (t-\chi)$,
$|t_n - v_n| < \frac{57}{48} (t-\chi)$,
and $n^{-\theta} < \xi < \frac1{24} (t-\chi)$.
Combined, the above prove (6).

Consider (7). Note, definition \ref{defDesAsc} gives,
\begin{equation*}
|\Gamma_n| = 2 |\tilde{s}_n + in^{-\theta} - s|
+ 2 \int_0^1 dy \sqrt{ (R_n'(y))^2 + (I_n'(y))^2 }.
\end{equation*}
Denote, for simplicity, the constant $c_n = (\frac{n^{-\theta}}{\tilde{s}_n-u_n})^2$.
Note $c_n < (\frac2{23})^2$ since $n^{-\theta} < \xi < \frac1{16} (s-\chi)$
(see definition \ref{defxiN} and equation (\ref{eqn1})), and
$\tilde{s}_n-u_n > \frac{23}{32} (s-\chi)$ (see part (10) of lemma \ref{lemN2}
and equation (\ref{eqn1})). Moreover, definition \ref{defRnIn} gives,
\begin{align*}
|\Gamma_n| = 2 |\tilde{s}_n + in^{-\theta} - s|
+ (1+c_n) (\tilde{s}_n-u_n) \int_0^1  \frac{dy}{\sqrt{y}} \\
\sqrt{ \frac{1-y(1-\log(y)) + c_n y(1+\log(y)) }
{1 - y(1 - \frac12 \log(y))^2 + c_n(1 + y \log(y) - \frac12 y \log(y)^2) + c_n^2 (-\frac14 y \log(y)^2)} }.
\end{align*}
Recall that $|\tilde{s}_n+in^{-\theta}-s| < 2n^{-\theta}$
(see part (1) of lemma \ref{lemTay}). Next, we state (without
proof) the following inequalities, which hold for all $y \in (0,1)$:
\begin{itemize}
\item
$1-y(1-\log(y)) \le 2(y-1)^2$.
\item
$y(1+\log(y)) \le 1$.
\item
$1 - y(1 - \frac12 \log(y))^2 \ge \frac14 (y-1)^2$.
\item
$1 + y \log(y) - \frac12 y \log(y)^2 \ge \frac14$.
\item
$-\frac14 y \log(y)^2 \ge -\frac14$.
\end{itemize}
Combined the above give:
\begin{equation*}
|\Gamma_n|
\le 4 n^{-\theta}
+ (1+c_n) (\tilde{s}_n-u_n) \int_0^1 \frac{dy}{\sqrt{y}}
\sqrt{ \frac{2(y-1)^2 + c_n (1) }{\frac14 (y-1)^2 + c_n(\frac14) + c_n^2 (-\frac14)} }.
\end{equation*}
Thus, since $c_n < (\frac2{23})^2 < \frac12$,
\begin{align*}
|\Gamma_n|
&< 4 n^{-\theta}
+ (1+c_n) (\tilde{s}_n-u_n) \int_0^1 \frac{dy}{\sqrt{y}}
\sqrt{ \frac{2(y-1)^2 + c_n }{\frac14 (y-1)^2 + \frac18 c_n} } \\
&= 4 n^{-\theta}
+ (1+c_n) (\tilde{s}_n-u_n) \int_0^1 \frac{dy}{\sqrt{y}} \sqrt{8} \\
&= 4 n^{-\theta}
+ 4 \sqrt2 (1+c_n) (\tilde{s}_n-u_n).
\end{align*}
Finally recall $c_n < (\frac2{23})^2$, and note definition \ref{defxiN},
part (10) of lemma \ref{lemN2}, and equation (\ref{eqn1}) give the following:
$n^{-\theta} < \xi < \frac1{16} (s-\chi)$, and
$|\tilde{s}_n-u_n| < \frac{41}{32} (s-\chi)$.
Combined, the above prove (7).

Consider (i). Recall that $g_n'(y) = h_n(y) \sin(y)$ and
$g_n''(y) = h_n'(y) \sin(y) + h_n(y) \cos(y)$ for all $y \in \R$.
It trivially follows that $g_n'(0) = 0$ and $g_n''(0) = h_n(0)$.
It remains to show that $h_n(0) < 0$. To see this first note,
the expression for $h_n : \R \to \R$ gives,
\begin{equation*}
h_n(0)
= - \frac1n \sum_{x \in P_n} \frac{(v_n-x) q_n}{(v_n-x + q_n)^2}.
\end{equation*}
Next recall (see proof of part (11) of lemma \ref{lemN2}) that,
\begin{equation*}
(t_n-v_n) q_n f_n''(t_n) = \frac1n \sum_{x \in P_n} \frac{(v_n-x) q_n}{(t_n-x)^2}.
\end{equation*}
Therefore,
\begin{align*}
h_n(0)
&\le - (t_n-v_n) q_n f_n''(t_n) + |h_n(0) + (t_n-v_n) q_n f_n''(t_n)| \\
&\le - (t_n-v_n) q_n f_n''(t_n) + \max_{x \in P_n}
\frac{|v_n-x| q_n |t_n - v_n - q_n|  (|t_n - x| + |v_n - x + q_n|)}{|t_n-x|^2 |v_n-x + q_n|^2}.
\end{align*}
Next note, part (10) of lemma \ref{lemN2} and equation (\ref{eqn1}) give $t_n-v_n > 0$,
and part (11) of lemma \ref{lemN2} gives $f_n''(t_n) > \frac14 f_{(t,s)}''(t) > 0$, and so
\begin{equation*}
h_n(0)
< - \tfrac14 (t_n-v_n) q_n |f_{(t,s)}''(t)| + \max_{x \in P_n}
\frac{|v_n-x| q_n |t_n - v_n - q_n|  (|t_n - x| + |v_n - x + q_n|)}{|t_n-x|^2 |v_n-x + q_n|^2}.
\end{equation*}
Finally recall $x_1^{(n)} = \max P_n$ and $x_n^{(n)} = \min P_n$ (see equation (\ref{eqPn})),
$|t_n-t| < \frac12 \xi$ (see part (9) of lemma \ref{lemN2}),
$|q_n - (t_n-v_n)| < n^{-\theta} < \xi$ (see definitions \ref{defxiN} and \ref{defRnIn}),
and note equation (\ref{eqn1}) gives the following for all $x \in P_n$:
$t_n - v_n > \frac{39}{48} (t - \chi)$,
$q_n > t_n - v_n - \xi > \frac{37}{48} (t - \chi)$,
$|v_n-x| < \min \{ 2(b-\chi), 2(\chi-a) \} \le b-a$,
$q_n < t_n - v_n + \xi < \frac{59}{48} (t - \chi)$,
$|t_n - v_n - q_n| < n^{-\theta}$,
$|t_n-x| < \frac{41}{32} (t-b)$,
$|v_n - x + q_n| = |(t_n-x) + q_n - (t_n-v_n)| < |t_n - x| + \xi < \frac{43}{32} (t-b)$,
$|t_n-x| > \frac{23}{32} (t-b)$,
$|v_n-x + q_n| = |(t_n-x) + q_n - (t_n-v_n)| > |t_n - x| - \xi > \frac{21}{32} (t-b)$. 
Combined, the above give,
\begin{align*}
h_n(0)
&< -\tfrac14 (\tfrac{39}{48} (t - \chi)) (\tfrac{37}{48} (t - \chi)) |f_{(t,s)}''(t)|
+ \frac{(b-a) (\tfrac{59}{48} (t - \chi)) n^{-\theta} \; (\frac{41}{32} (t-b) + \frac{43}{32} (t-b))}
{(\frac{23}{32} (t-b))^2 (\frac{21}{32} (t-b))^2} \\
&< - \tfrac18 (t - \chi)^2 |f_{(t,s)}''(t)| + 8 n^{-\theta} \frac{(b-a) (t-\chi)}{(t-b)^3}.
\end{align*}
Definition \ref{defxiN} finally gives $h_n(0) < 0$. This proves (i).

Consider (ii). Recall that $g_n'(y) = h_n(y) \sin(y)$ and
$g_n''(y) = h_n'(y) \sin(y) + h_n(y) \cos(y)$ for all $y \in \R$.
Thus, since $Y \in (0,\pi)$ and $g_n'(Y) = 0$, $\sin(Y) \neq 0$,
$h_n(Y) = 0$, and $g_n''(Y) = h_n'(Y) \sin(Y)$. Finally note,
the expression for $h_n : \R \to \R$ gives,
\begin{equation*}
h_n'(Y) \sin(Y) = - \frac1n \sum_{x \in P_n}
\frac{2 (v_n-x)^2 q_n^2 \sin(Y)^2}{((v_n-x)^2 + 2(v_n-x) q_n \cos(Y) + q_n^2)^2}.
\end{equation*}
This proves (ii).

Part (iii) follows trivially from the expressions for $R_n$ and $R_n'$ in definition
\ref{defRnIn} and lemma \ref{lemRnIn}. Consider (iv). First note, for all $y \in (0,1)$,
the expressions for $R_n$ and $R_n'$ give,
\begin{equation*}
M_n(y) = (\tilde{s}_n-u_n) \bigg( 2y - \frac{(\tilde{q}_n)^2}{(\tilde{s}_n-u_n)^2} y - 1
- \frac{(\tilde{q}_n)^2}{(\tilde{s}_n-u_n)^2} y \log(y) \bigg).
\end{equation*}
Therefore, since $(\tilde{q}_n)^2 = (\tilde{s}_n-u_n)^2 + n^{-2\theta}$ (see definition
\ref{defRnIn}),
\begin{equation*}
M_n(y) = (\tilde{s}_n-u_n) \bigg( y-1 - y \log(y)
- \frac{n^{-2\theta}}{(\tilde{s}_n-u_n)^2} y (1+\log(y)) \bigg),
\end{equation*}
for all $y \in (0,1)$. Note that $\tilde{s}_n-u_n > 0$ (see part (10) of
lemma \ref{lemN2} and equation (\ref{eqn1})). Also recall
(see proof of part (7) above) that $\frac{n^{-2\theta}}{(\tilde{s}_n-u_n)^2} < (\frac{2}{23})^2$.
Finally, we state that $y - 1 - y \log(y) \le 2e^{-1} - 1$ and
$y (1+\log(y)) \ge - e^{-2}$ for all $y \in (0,e^{-1}]$, and
$y - 1 - y \log(y) < 0$ and $y (1+\log(y)) > 0$ for
all $y \in (e^{-1},1)$. Combined the above prove
that $M_n(y) < 0$ for all $y \in (0,1)$. This proves (iv).

Consider (v). Note, definition \ref{defxiN} and equation (\ref{eqn1}) gives $t>s$
and $n^{-\theta} < \xi < \tfrac18 (t-s)$. This proves (v). Consider (vi). Note,
\begin{equation*}
\inf_{w \in B(t, 2n^{-\theta})} \; \inf_{y \in [0,1]} |w -(u_n + R_n(y) + i I_n(y))|
\ge \inf_{w \in B(t, 2n^{-\theta})} |w-u_n| - \sup_{y \in [0,1]} |R_n(y) + i I_n(y)|.
\end{equation*}
Next note, for all $y \in [0,1]$, definition \ref{defRnIn} gives
$|R_n(y) + i I_n(y)| = \tilde{q}_n \sqrt{y} =
|\tilde{s}_n + i n^{-\theta} - u_n| \sqrt{y}$. Therefore,
\begin{equation*}
\inf_{w \in B(t, 2n^{-\theta})} \; \inf_{y \in [0,1]} |w -(u_n + R_n(y) + i I_n(y))|
\ge \inf_{w \in B(t, 2n^{-\theta})} |w-u_n| - |\tilde{s}_n + i n^{-\theta} - u_n|.
\end{equation*}
Finally recall that $n^{-\theta} < \xi$, $|\tilde{s}_n-s| < \frac12 \xi$
(see part (10) of lemma \ref{lemN2}), and note equation (\ref{eqn1}) gives the following:
$|w-u_n| \ge t - u_n - 2\xi$ for all $w \in B(t, 2n^{-\theta})$,
$|\tilde{s}_n + i n^{-\theta} - u_n| < \tilde{s}_n - u_n + \xi < s - u_n + \tfrac32 \xi$,
and $\xi < \tfrac18 (t-s)$. Combined, the above prove (vi).

Consider (vii). Note that
\begin{equation*}
\inf_{w \in \partial B(v_n, q_n)} \; \inf_{z \in B(s, 2n^{-\theta})} |w-z|
\ge \inf_{w \in \partial B(v_n, q_n)} |w-v_n| - \sup_{z \in B(s, 2n^{-\theta})} |v_n-z|.
\end{equation*}
In addition, $\inf_{w \in \partial B(v_n, q_n)} |w-v_n| = q_n = |t_n + i n^{-\theta} - v_n|$
(see definition \ref{defRnIn}). Therefore,
\begin{equation*}
\inf_{w \in \partial B(v_n, q_n)} \; \inf_{z \in B(s, 2n^{-\theta})} |w-z|
\ge |t_n + i n^{-\theta} - v_n| - \sup_{z \in B(s, 2n^{-\theta})} |v_n-z|.
\end{equation*}
Finally recall that $n^{-\theta} < \xi$, $|t_n-t| < \frac12 \xi$
(see part (9) of lemma \ref{lemN2}), and note equation (\ref{eqn1}) gives the following:
$|t_n + i n^{-\theta} - v_n| > t_n - v_n - \xi > t - v_n - \tfrac32 \xi$,
$|z-v_n| \le s - v_n + 2\xi$ for all $z \in B(s, 2n^{-\theta})$,
and $\xi < \tfrac18 (t-s)$. Combined, the above prove (vii).

Consider (viii). Note,
\begin{align*}
&\inf_{w \in \partial B(v_n, q_n)} \; \inf_{y \in [0,1]} |w -(u_n + R_n(y) + i I_n(y))| \\
&\ge \inf_{w \in \partial B(v_n, q_n)} |w-v_n| - |v_n-u_n| - \sup_{y \in [0,1]} |R_n(y) + i I_n(y)|.
\end{align*}
Proceed as in the proofs of parts (vi,vii) above to get,
\begin{align*}
&\inf_{w \in \partial B(v_n, q_n)} \; \inf_{y \in [0,1]} |w -(u_n + R_n(y) + i I_n(y))| \\
&\ge |t_n + i n^{-\theta} - v_n| - |v_n-u_n| - |\tilde{s}_n + i n^{-\theta} - u_n|.
\end{align*}
Next recall that $|t_n + i n^{-\theta} - v_n| > t - v_n - \frac32 \xi$
(see proof of part (vii)), and $|\tilde{s}_n + i n^{-\theta} - u_n| < s - u_n + \frac32 \xi$
(see proof of part (vi)). Therefore,
\begin{equation*}
\inf_{w \in \partial B(v_n, q_n)} \; \inf_{y \in [0,1]} |w -(u_n + R_n(y) + i I_n(y))|
> t - s - 3\xi - 2|v_n-u_n|.
\end{equation*}
Finally recall that $|v_n - u_n| < \frac12 \xi$ (see definition \ref{defxiN})
and $\xi < \tfrac18 (t-s)$ (see equation (\ref{eqn1})). Combined, the above
prove (viii).
\end{proof}

\subsection{Proof of theorem \ref{thmdecay} via steepest descent analysis}
\label{secproof}

Fix $\xi,N$ and $\theta \in (\frac13,\frac12)$ as in the previous two sections.
Fix $n>N$. Define $t_n, \tilde{s}_n, u_n,r_n,v_n,s_n, \g_n, \G_n$ as in
the previous two sections. Recall (see parts (1,2) of lemma \ref{lemDesAsc})
that $\G_n$ contains $\{x \in P_n : x > u_n\}$ and does not contain any of
$\{x \in P_n : x < u_n\}$, and $\g_n$ contains $v_n$ and $\G_n$. Equations
(\ref{eqKnrnunsnvn1}, \ref{eqPn}) thus gives,
\begin{equation*}
K_n((u_n,r_n),(v_n,s_n))
= \frac{(n-s_n)!}{(n-r_n-1)!} \; J_n - \phi_{r_n,s_n}(u_n,v_n),
\end{equation*}
where
\begin{equation*}
J_n := \frac1{(2\pi i)^2} \int_{\g_n} dw \int_{\G_n} dz \;
\frac1{w-z} \frac{(z - u_n)^{n-r_n-1}}{(w - v_n)^{n-s_n+1}}
\prod_{x \in P_n} \left( \frac{w - x}{z - x} \right).
\end{equation*}

Define $b_n, \tilde{b}_n, \alpha_n, \tilde{\alpha}_n$ as in lemma
\ref{lemTay}, and so $n^{-\theta} b_n = |t_n + i n^{-\theta} - t|$
and $\alpha_n = \text{Arg}(t_n + i n^{-\theta} - t)$, and
$n^{-\theta} \tilde{b}_n = |\tilde{s}_n + i n^{-\theta} - s|$ 
and $\tilde{\alpha}_n = \text{Arg}(\tilde{s}_n + i n^{-\theta} - s)$.
Recall (see parts (1,2) of lemma \ref{lemTay}) that
$1 \le b_n < 2$ and $1 \le \tilde{b}_n < 2$, and
$\max\{|b_n - 1|, |\tilde{b}_n - 1|, |\alpha_n - \tfrac\pi2|, |\tilde{\alpha}_n - \tfrac\pi2|\}
< n^{-\frac12+\theta}$. Also, definition \ref{defDesAsc} implies that
we can partition $\g_n$ and $\G_n$ as follows:
\begin{equation}
\label{eqConLocRem}
\g_n = \g_n^{(l)} + \g_n^{(r)}
\hspace{.5cm} \text{and} \hspace{.5cm}
\G_n = \G_n^{(l)} + \G_n^{(r)},
\end{equation}
where:
\begin{itemize}
\item
$\g_n^{(l)}$ is that {\em local section} of $\g_n$ given by the lines
from $t_n - i n^{-\theta} = t + n^{-\theta} b_n e^{-i \alpha_n}$ to $t$,
and from $t$ to $t_n + i n^{-\theta} = t + n^{-\theta} b_n e^{i \alpha_n}$.
\item
$\G_n^{(l)}$ is that {\em local section} of $\G_n$ given by the lines
from $\tilde{s}_n - i n^{-\theta} = s + n^{-\theta} \tilde{b}_n e^{-i \tilde{\alpha}_n}$ to $s$,
and from $s$ to $\tilde{s}_n + i n^{-\theta} = s + n^{-\theta} \tilde{b}_n e^{i \tilde{\alpha}_n}$.
\item
$\g_n^{(r)}$ and $\G_n^{(r)}$ are
(respectively) the {\em remaining sections} of $\g_n$ and $\G_n$.
\end{itemize}
Then,
\begin{equation}
\label{eqJn11Jn12}
J_n = J_n^{(l,l)} + J_n^{(l,r)} + J_n^{(r,l)} + J_n^{(r,r)},
\end{equation}
where,
\begin{equation*}
J_n^{(l,l)} := \frac1{(2\pi i)^2} \int_{\g_n^{(l)}} dw \int_{\G_n^{(l)}} dz \;
\frac1{w-z} \frac{(z - u_n)^{n-r_n-1}}{(w - v_n)^{n-s_n+1}}
\prod_{x \in P_n} \left( \frac{w - x}{z - x} \right).
\end{equation*}
The other three terms on the RHS of equation (\ref{eqJn11Jn12}) are defined
analogously. As we shall see in the following lemmas, the asymptotic
behaviour of $J_n^{(l,l)}$ dominates the other terms.

Consider first $J_n^{(l,l)}$. Define $D_n, \tilde{D}_n$ as in lemma
\ref{lemTay}, and recall that $D_n > (\frac18 |f_{(t,s)}''(t)|)^\frac12 > 0$ and
$\tilde{D}_n > (\frac18 |f_{(t,s)}''(s)|)^\frac12 > 0$ (see part (3) of
lemma \ref{lemTay}). Then:
\begin{lem}
\label{lemJn11}
The following is satisfied:
\begin{equation*}
\left| J_n^{(l,l)}
- \frac{\exp( n f_n(t) - n \tilde{f}_n(s))}{4 \pi (t-s) D_n \tilde{D}_n} \; n^{-1} \right|
< \frac{\exp( n f_n(t) - n \tilde{f}_n(s))}{4 \pi (t-s) D_n \tilde{D}_n} \; n^{-3\theta} \; F_n,
\end{equation*}
where $F_n > 0$ is defined in the proof and satisfies $F_n = O(1)$ for all $n$
sufficiently large.
\end{lem}

\begin{proof}
First, equations (\ref{eqfn}, \ref{eqtildefn}, \ref{eqPn}, \ref{eqJn11Jn12}) give
\begin{equation*}
J_n^{(l,l)} = \frac1{(2\pi i)^2} \int_{\g_n^{(l)}} dw \int_{\G_n^{(l)}} dz \;
\frac{\exp(n f_n(w) - n \tilde{f}_n(z))}{w-z},
\end{equation*}
where (see equation (\ref{eqConLocRem})):
\begin{itemize}
\item
$\g_n^{(l)}$ is the lines from
$t + n^{-\theta} b_n e^{-i \alpha_n}$ to $t$,
and from $t$ to $t + n^{-\theta} b_n e^{i \alpha_n}$.
\item
$\G_n^{(l)}$ is the lines from
$s + n^{-\theta} \tilde{b}_n e^{-i \tilde{\alpha}_n}$ to $s$,
and from $s$ to $s + n^{-\theta} \tilde{b}_n e^{i \tilde{\alpha}_n}$.
\end{itemize}
A change of variables then gives,
\begin{equation*}
J_n^{(l,l)} = 
\frac{n^{-1}}{(2\pi i)^2 D_n \tilde{D}_n} \int_{h_n} dw \int_{H_n} dz \;
\frac{\exp(n f_n(t + n^{-\frac12} D_n^{-1} w)
- n \tilde{f}_n(s + n^{-\frac12} \tilde{D}_n^{-1} z))}
{t-s + n^{-\frac12} D_n^{-1} w - n^{-\frac12} \tilde{D}_n^{-1} z},
\end{equation*}
where:
\begin{itemize}
\item
$h_n$ is the lines from $n^{\frac12-\theta} b_n D_n e^{-i \alpha_n}$ to $0$,
and from $0$ to $n^{\frac12-\theta} b_n D_n e^{i \alpha_n}$.
\item
$H_n$ is the lines from $n^{\frac12-\theta} \tilde{b}_n \tilde{D}_n e^{-i \tilde{\alpha}_n}$
to $0$, and from $0$ to $n^{\frac12-\theta} \tilde{b}_n \tilde{D}_n e^{i \tilde{\alpha}_n}$.
\end{itemize}
$h_n$ and $H_n$ are shown in figure \ref{figConReScaleCase1}.
Note, letting $\text{cl}$ denote closure in $\mathbb{C}$, $h_n \subset \text{cl}(B(0, n^{\frac12-\theta} b_n D_n))$ and
$H_n \subset \text{cl}(B(0, n^{\frac12-\theta} \tilde{b}_n \tilde{D}_n))$.
Parts (6,7) of lemma \ref{lemTay} then give,
\begin{equation*}
J_n^{(l,l)} = 
\frac{n^{-1}}{(2\pi i)^2 D_n \tilde{D}_n} \int_{h_n} dw \int_{H_n} dz \;
\frac{\exp(n f_n(t) - n \tilde{f}_n(s) + w^2 + z^2 + n^{1-3\theta} g_n(w,z))}
{t-s + n^{-\frac12} D_n^{-1} w - n^{-\frac12} \tilde{D}_n^{-1} z},
\end{equation*}
where
$n^{1-3\theta} g_n(w,z) := n f_n(t + n^{-\frac12} D_n^{-1} w) - n f_n(t) - w^2
- n \tilde{f}_n(s + n^{-\frac12} \tilde{D}_n^{-1} z)) + n \tilde{f}_n(s) - z^2$
satisfies,
\begin{equation}
\label{eqBdd1}
\sup_{(w,z) \in \text{cl}(B(0, n^{\frac12-\theta} b_n D_n))
\times \text{cl}(B(0, n^{\frac12-\theta} \tilde{b}_n \tilde{D}_n))}
|g_n(w,z)| \le E_{2,n} + \tilde{E}_{2,n},
\end{equation}
and $E_{2,n} + \tilde{E}_{2,n} = O(1)$ for all $n$ sufficiently large.
Next recall that $|\alpha_n - \frac\pi2| \le n^{-\frac12+\theta}$ and
$|\alpha_n - \frac\pi2| \le n^{-\frac12+\theta}$ (see part (2) of lemma \ref{lemTay})),
where $\theta \in (\frac13,\frac12)$, and define:
\begin{itemize}
\item
$k_n$ is the line from $n^{\frac12-\theta} b_n D_n e^{-i \frac\pi2}$ to
$n^{\frac12-\theta} b_n D_n e^{i \frac\pi2}$. $c_n$ is the smallest arcs
of $\partial B(0, n^{\frac12-\theta} b_n D_n)$
from $n^{\frac12-\theta} b_n D_n e^{-i \alpha_n}$ to
$n^{\frac12-\theta} b_n D_n e^{-i \frac\pi2}$, and from
$n^{\frac12-\theta} b_n D_n e^{i \frac\pi2}$ to
$n^{\frac12-\theta} b_n D_n e^{i \alpha_n}$.
\item
$K_n$ is the line from $n^{\frac12-\theta} \tilde{b}_n \tilde{D}_n e^{-i \frac\pi2}$ to
$n^{\frac12-\theta} \tilde{b}_n \tilde{D}_n e^{i \frac\pi2}$. $C_n$ is the smallest arcs
of $\partial B(0, n^{\frac12-\theta} \tilde{b}_n \tilde{D}_n)$ from
$n^{\frac12-\theta} \tilde{b}_n \tilde{D}_n e^{-i \tilde{\alpha}_n}$ to
$n^{\frac12-\theta} \tilde{b}_n \tilde{D}_n e^{-i \frac\pi2}$, and from
$n^{\frac12-\theta} \tilde{b}_n \tilde{D}_n e^{i \frac\pi2}$ to
$n^{\frac12-\theta} \tilde{b}_n \tilde{D}_n e^{i \tilde{\alpha}_n}$.
\end{itemize}
These contours are also shown in figure \ref{figConReScaleCase1}.
Then, noting that $h_n$ and $c_n + k_n$ have the same initial and final
points, and similarly for $H_n$ and $C_n + K_n$,
\begin{equation}
\label{eqJ1nllIn}
J_n^{(l,l)}
=  n^{-1} \; \frac{\exp(n f_n(t) - n \tilde{f}_n(s))}{(2\pi)^2 D_n \tilde{D}_n}
(I_n^{(k,K)} + I_n^{(k,C)} + I_n^{(c,K)} + I_n^{(c,C)}),
\end{equation}
where,
\begin{equation*}
I_n^{(k,K)}
:= - \int_{k_n} dw \int_{K_n} dz \; \frac{\exp(w^2 + z^2 + n^{1-3\theta} g_n(w,z))}
{t-s + n^{-\frac12} D_n^{-1} w - n^{-\frac12} \tilde{D}_n^{-1} z},
\end{equation*}
and the other three terms on the RHS are defined analogously. Next write,
\begin{equation}
\label{eqI1n}
I_n^{(k,K)} = I_{1,n} + I_{2,n} + I_{3,n},
\end{equation}
where,
\begin{align*}
I_{1,n}
&:= - \int_{k_n} dw \int_{K_n} dz \; \frac{\exp(w^2 + z^2)}{t-s}, \\
I_{2,n}
&:= - \int_{k_n} dw \int_{K_n} dz \;
\left( \frac{\exp(w^2 + z^2 + n^{1-3\theta} g_n(w,z))}{t-s} -
\frac{\exp(w^2 + z^2)}{t-s} \right), \\
I_{3,n}
&:= - \int_{k_n} dw \int_{K_n} dz \;
\left( \frac{\exp(w^2 + z^2 + n^{1-3\theta} g_n(w,z))}
{t-s + n^{-\frac12} D_n^{-1} w - n^{-\frac12} \tilde{D}_n^{-1} z}
- \frac{\exp(w^2 + z^2 + n^{1-3\theta} g_n(w,z))}{t-s} \right).
\end{align*}
We will show:
\begin{enumerate}
\item[(i)]
$|I_{1,n} - \frac\pi{t-s}|
< \exp(- n^{1-2\theta} (D_n^2 \wedge \tilde{D}_n^2)) \frac\pi{t-s}$.
\item[(ii)]
$|I_{2,n}|
\le n^{1-3\theta} (E_{2,n} + \tilde{E}_{2,n}) \frac{2\pi}{t-s}$.
\item[(iii)]
$|I_{3,n}|
\le n^{-\theta} \frac{2^5 \pi}{(t-s)^2}$.
\item[(iv)]
$|I_n^{(k,C)}|
< \exp(-\tfrac14 n^{1-2\theta} \tilde{D}_n^2)
\frac{2^6 \tilde{D}_n}{t-s}$.
\item[(v)]
$|I_n^{(c,K)}|
< \exp(-\tfrac14 n^{1-2\theta} D_n^2) \frac{2^6 D_n}{t-s}$.
\item[(vi)]
$|I_n^{(c,C)}|
< \exp(-\tfrac14 n^{1-2\theta} (D_n^2 + \tilde{D}_n^2)) \frac{2^7 D_n \tilde{D}_n}{t-s}$.
\end{enumerate}
Define,
\begin{equation*}
F_n : = n^{-1+3\theta} \; \tfrac{t-s}\pi \; ( |I_{1,n} - \tfrac\pi{t-s}| + |I_{2,n}| + |I_{3,n}| + |I_n^{(k,C)}| + |I_n^{(c,K)}| + |I_n^{(c,C)}| ).
\end{equation*}
Recall that $\theta \in (\frac13,\frac12)$, $D_n > (\frac18 |f_{(t,s)}''(t)|)^\frac12 > 0$
and $\tilde{D}_n > (\frac18 |f_{(t,s)}''(s)|)^\frac12 > 0$ (see part (3) of
lemma \ref{lemTay}), and $E_{2,n} + \tilde{E}_{2,n} = O(1)$.
It follows that $F_n = O(1)$ for all $n$ sufficiently large. The required result then
follows from equations (\ref{eqJ1nllIn}, \ref{eqI1n}) and parts (i-vi).

\begin{figure}[t]
\centering
\begin{tikzpicture};

\draw (-3.5,1.5) node {$\mathbb{H}$};
\draw [dotted] (-3,0) --++(6,0);
\draw (-3.5,0) node {$\R$};

\draw [dotted] (0,0) circle (2cm);
\draw [fill] (0,0) circle (.05cm);
\draw (-.1,-.2) node {\scriptsize $0$};
\draw (1.75,2.5) node {\scriptsize $n^{\frac12-\theta} b_n D_n e^{i \alpha_n}$};
\draw[arrows=->,line width=0.5pt](1.75,2.3)--(.62,1.95);
\draw (1.75,-2.5) node {\scriptsize $n^{\frac12-\theta} b_n D_n e^{-i \alpha_n}$};
\draw[arrows=->,line width=0.5pt](1.75,-2.3)--(.62,-1.95);
\draw (-1.75,2.5) node {\scriptsize $n^{\frac12-\theta} b_n D_n e^{i \frac\pi2}$};
\draw[arrows=->,line width=0.5pt](-1.5,2.3)--(-.1,2.05);
\draw (-1.75,-2.5) node {\scriptsize $n^{\frac12-\theta} b_n D_n e^{-i \frac\pi2}$};
\draw[arrows=->,line width=0.5pt](-1.5,-2.3)--(-.1,-2.05);

\draw (0,0) --++ (.6180,1.902);
\draw[arrows=->,line width=1pt](.309,.951)--(.312,.960);
\draw (.7,.951) node {$h_n$};
\draw (0,0) --++ (.6180,-1.902);
\draw[arrows=->,line width=1pt](.312,-.960)--(.309,-.951);
\draw (.7,-.951) node {$h_n$};

\draw (0,-2) --++ (0,4);
\draw[arrows=->,line width=1pt](0,1)--(0,1.01);
\draw (-.3,1) node {$k_n$};
\draw[arrows=->,line width=1pt](0,-1.01)--(0,-1);
\draw (-.3,-1) node {$k_n$};

\draw [domain=72:90] plot ({2*cos(\x)}, {2*sin(\x)});
\draw[arrows=->,line width=1pt](.293,1.978)--(.313,1.975);
\draw (.313,2.2) node {$c_n$};
\draw [domain=-90:-72] plot ({2*cos(\x)}, {2*sin(\x)});
\draw[arrows=->,line width=1pt](.313,-1.975)--(.293,-1.978);
\draw (.313,-2.2) node {$c_n$};

\draw [dotted] (4,0) --++(6,0);

\draw [dotted] (7,0) circle (2cm);
\draw [fill] (7,0) circle (.05cm);
\draw (6.9,-.2) node {\scriptsize $0$};
\draw (8.75,2.5) node
{\scriptsize $n^{\frac12-\theta} \tilde{b}_n \tilde{D}_n e^{i \tilde{\alpha}_n}$};
\draw[arrows=->,line width=0.5pt](8.75,2.3)--(7.62,1.95);
\draw (8.75,-2.5) node
{\scriptsize $n^{\frac12-\theta} \tilde{b}_n \tilde{D}_n e^{-i \tilde{\alpha}_n}$};
\draw[arrows=->,line width=0.5pt](8.75,-2.3)--(7.62,-1.95);
\draw (5.25,2.5) node {\scriptsize $n^{\frac12-\theta} \tilde{b}_n \tilde{D}_n e^{i \frac\pi2}$};
\draw[arrows=->,line width=0.5pt](5.5,2.3)--(6.9,2.05);
\draw (5.25,-2.5) node {\scriptsize $n^{\frac12-\theta} \tilde{b}_n \tilde{D}_n e^{-i \frac\pi2}$};
\draw[arrows=->,line width=0.5pt](5.5,-2.3)--(6.9,-2.05);

\draw (7,0) --++ (.6180,1.902);
\draw[arrows=->,line width=1pt](7.309,.951)--(7.312,.960);
\draw (7.7,.951) node {$H_n$};
\draw (7,0) --++ (.6180,-1.902);
\draw[arrows=->,line width=1pt](7.312,-.960)--(7.309,-.951);
\draw (7.7,-.951) node {$H_n$};

\draw (7,-2) --++ (0,4);
\draw[arrows=->,line width=1pt](7,1)--(7,1.01);
\draw (6.7,1) node {$K_n$};
\draw[arrows=->,line width=1pt](7,-1.01)--(7,-1);
\draw (6.7,-1) node {$K_n$};

\draw [domain=72:90] plot ({7+2*cos(\x)}, {2*sin(\x)});
\draw[arrows=->,line width=1pt](7.293,1.978)--(7.313,1.975);
\draw (7.313,2.2) node {$C_n$};
\draw [domain=-90:-72] plot ({7+2*cos(\x)}, {2*sin(\x)});
\draw[arrows=->,line width=1pt](7.313,-1.975)--(7.293,-1.978);
\draw (7.313,-2.2) node {$C_n$};

\end{tikzpicture}
\caption{Left: The circle $B(0,n^{\frac12-\theta} b_n D_n)$, and the contours
$h_n$, $k_n$, $c_n$. Right: The circle $B(0,n^{\frac12-\theta} \tilde{b}_n \tilde{D}_n)$, and
the contours $H_n$, $K_n$, $C_n$. Recall, $\theta \in (\frac13, \frac12)$,
$\alpha_n = \frac\pi2 + O(n^{-\frac12+\theta})$, and
$\tilde{\alpha}_n = \frac\pi2 + O(n^{-\frac12+\theta})$.}
\label{figConReScaleCase1}
\end{figure}
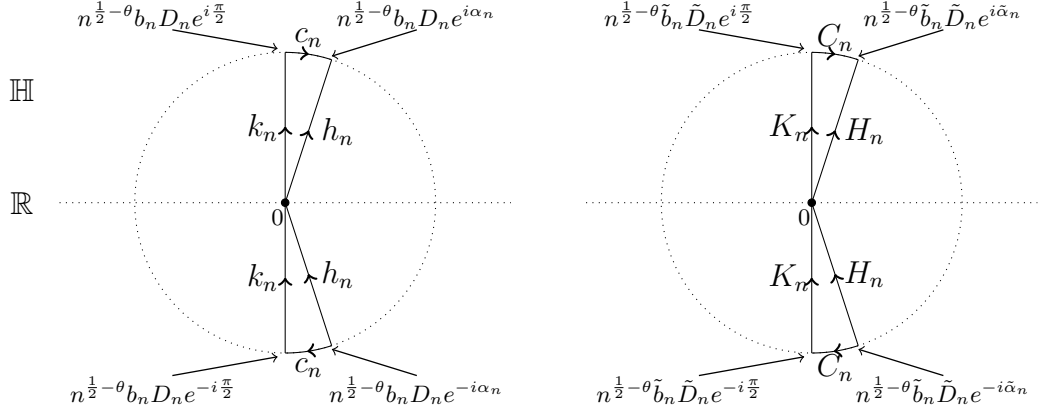

Consider (i). First recall that $\theta \in (\frac13, \frac12)$, $k_n(x) = i x$ for all
$x \in [-n^{\frac12-\theta} b_n D_n, n^{\frac12-\theta} b_n D_n]$
and $K_n(y) = i y$ for all
$y \in [-n^{\frac12-\theta} \tilde{b}_n \tilde{D}_n, n^{\frac12-\theta} \tilde{b}_n \tilde{D}_n]$.
Equation (\ref{eqI1n}) thus gives,
\begin{equation*}
I_{1,n} = \frac1{t-s}
\int_{-n^{\frac12-\theta} b_n D_n}^{n^{\frac12-\theta} b_n D_n} dx
\int_{-n^{\frac12-\theta} \tilde{b}_n \tilde{D}_n}^{n^{\frac12-\theta} \tilde{b}_n \tilde{D}_n} dy \;
\exp(-x^2 - y^2).
\end{equation*}
Therefore,
\begin{equation*}
I_{1,n}
< \frac1{t-s} \int_{-\infty}^\infty dx \int_{-\infty}^\infty dy  \; \exp(- x^2 - y^2)
= \frac{\pi}{t-s}.
\end{equation*}
Next recall that $b_n, \tilde{b}_n \ge 1$ (see part (1) of lemma \ref{lemTay}). Therefore,
\begin{align*}
I_{1,n}
&\ge \frac1{t-s}
\int_0^{n^{\frac12-\theta} (D_n \wedge \tilde{D}_n)} dr
\int_{-\pi}^{\pi} d\phi \; r \exp(- r^2) \\
&= \frac\pi{t-s} \left( 1 - \exp(- n^{1-2\theta} (D_n^2 \wedge \tilde{D}_n^2)) \right).
\end{align*}
Combined, the above prove (i).

Consider (ii). First recall that $\theta \in (\frac13, \frac12)$, $k_n(x) = i x$ for all
$x \in [-n^{\frac12-\theta} b_n D_n, n^{\frac12-\theta} b_n D_n]$
and $K_n(y) = i y$ for all
$y \in [-n^{\frac12-\theta} \tilde{b}_n \tilde{D}_n, n^{\frac12-\theta} \tilde{b}_n \tilde{D}_n]$.
Equation (\ref{eqI1n}) thus gives,
\begin{equation*}
I_{2,n}
= \int_{-n^{\frac12-\theta} b_n D_n}^{n^{\frac12-\theta} b_n D_n} dx
\int_{-n^{\frac12-\theta} \tilde{b}_n \tilde{D}_n}^{n^{\frac12-\theta} \tilde{b}_n \tilde{D}_n} dy
\; \frac{\exp(- x^2 - y^2 )}{t-s} \left( \exp(n^{1-3\theta} g_n(ix,iy)) - 1 \right).
\end{equation*}
Next recall that $n^{1-3\theta} (E_{2,n} + \tilde{E}_{2,n}) < 1$ (see definition \ref{defxiN}),
and note that $|\exp(x) - 1| \le 2|x|$ when $|x| < 1$. Equation (\ref{eqBdd1}) then gives,
\begin{equation*}
\sup_{|x| \le  n^{\frac12-\theta} b_n D_n, \;
|y| \le n^{\frac12-\theta} \tilde{b}_n \tilde{D}_n}
\left| \exp(n^{1-3\theta} g_n(ix,iy)) - 1 \right|	
< 2 n^{1-3\theta} (E_{2,n} + \tilde{E}_{2,n}). 
\end{equation*}
Therefore,
\begin{align*}
|I_{2,n}|
&< \int_{-n^{\frac12-\theta} b_n D_n}^{n^{\frac12-\theta} b_n D_n} dx
\int_{-n^{\frac12-\theta} \tilde{b}_n \tilde{D}_n}^{n^{\frac12-\theta} \tilde{b}_n \tilde{D}_n} dy
\; \frac{\exp(- x^2 - y^2 )}{t-s} (2 n^{1-3\theta} (E_{2,n} + \tilde{E}_{2,n})) \\
&< \int_{-\infty}^{\infty} dx \int_{-\infty}^{\infty} dy \;
\frac{\exp(- x^2 - y^2 )}{t-s} (2 n^{1-3\theta} (E_{2,n} + \tilde{E}_{2,n})).
\end{align*}
This proves (ii).

Consider (iii). First recall that $\theta \in (\frac13, \frac12)$, $k_n(x) = i x$ for all
$x \in [-n^{\frac12-\theta} b_n D_n, n^{\frac12-\theta} b_n D_n]$
and $K_n(y) = i y$ for all
$y \in [-n^{\frac12-\theta} \tilde{b}_n \tilde{D}_n, n^{\frac12-\theta} \tilde{b}_n \tilde{D}_n]$.
Equation (\ref{eqI1n}) thus gives,
\begin{align*}
I_{3,n}
= \int_{-n^{\frac12-\theta} b_n D_n}^{n^{\frac12-\theta} b_n D_n} dx
\int_{-n^{\frac12-\theta} \tilde{b}_n \tilde{D}_n}^{n^{\frac12-\theta} \tilde{b}_n \tilde{D}_n} dy
&\; \frac{\exp(- x^2 - y^2 + n^{1-3\theta} g_n(ix,iy))}{t-s} \times \\
&\left( - \frac{n^{-\frac12} D_n^{-1} ix
- n^{-\frac12} \tilde{D}_n^{-1} iy}
{t-s + n^{-\frac12} D_n^{-1} ix - n^{-\frac12} \tilde{D}_n^{-1} iy} \right).
\end{align*}
Next recall that $n^{1-3\theta} (E_{2,n} + \tilde{E}_{2,n}) < 1$ (see definition \ref{defxiN}),
and note that $|\exp(x)| < 4$ when $|x| < 1$. Equation (\ref{eqBdd1}) thus gives,
\begin{equation*}
\sup_{|x| \le  n^{\frac12-\theta} b_n D_n, \;
|y| \le n^{\frac12-\theta} \tilde{b}_n \tilde{D}_n}
\exp(n^{1-3\theta} |g_n(ix, iy)|)
< 4. 
\end{equation*}
Next recall that $b_n, \tilde{b}_n < 2$ (see part (1) of lemma \ref{lemTay}),
and $n^{-\theta} < \xi$ (see definition \ref{defxiN}), and so
$|n^{-\frac12} D_n^{-1} ix| < 2n^{-\theta} < 2\xi$ for all
$|x| \le  n^{\frac12-\theta} b_n D_n$,
and $|n^{-\frac12} \tilde{D}_n^{-1} iy| < 2n^{-\theta} < 2\xi$ for all
$|y| \le n^{\frac12-\theta} \tilde{b}_n \tilde{D}_n$.
Also, $\xi < \frac18 (t-s)$ (see equation (\ref{eqn1})), and so
\begin{equation*}
\sup_{|x| \le  n^{\frac12-\theta} b_n D_n, \;
|y| \le n^{\frac12-\theta} \tilde{b}_n \tilde{D}_n}
\left| - \frac{n^{-\frac12} D_n^{-1} ix - n^{-\frac12} \tilde{D}_n^{-1} iy}
{t-s + n^{-\frac12} D_n^{-1} ix - n^{-\frac12} \tilde{D}_n^{-1} iy} \right|
< \frac{2 n^{-\theta} + 2 n^{-\theta}}{t-s - 2\xi - 2\xi}
< \frac{4 n^{-\theta}}{\frac12 (t-s)}. 
\end{equation*}
Combined the above give,
\begin{align*}
|I_{3,n}|
&< \int_{-n^{\frac12-\theta} b_n D_n}^{n^{\frac12-\theta} b_n D_n} dx
\int_{-n^{\frac12-\theta} \tilde{b}_n \tilde{D}_n}^{n^{\frac12-\theta} \tilde{b}_n \tilde{D}_n} dy
\; \frac{\exp(- x^2 - y^2 )}{t-s} \frac{2^5 n^{-\theta}}{t-s} \\
&< \int_{-\infty}^{\infty} dx \int_{-\infty}^{\infty} dy \;
\frac{\exp(- x^2 - y^2 )}{t-s} \frac{2^5 n^{-\theta}}{t-s}.
\end{align*}
This proves (iii).

Consider (iv). First recall equation (\ref{eqJ1nllIn}):
\begin{equation*}
I_n^{(k,C)}
:= - \int_{k_n} dw \int_{C_n} dz \; \frac{\exp(w^2 + z^2 + n^{1-3\theta} g_n(w,z))}
{t-s + n^{-\frac12} D_n^{-1} w - n^{-\frac12} \tilde{D}_n^{-1} z},
\end{equation*}
where $k_n$ and $C_n$ are given in figure \ref{figConReScaleCase1}. Next recall
that $\theta < \frac12$ and $|\tilde{\alpha}_n - \frac\pi2| < n^{-\frac12+\theta}$
(see part (2) of lemma \ref{lemTay}).
It thus follows that $|\text{Arg}(w)| = \frac\pi2$ for all $w$ on $k_n$,
and $||\text{Arg}(z)| - \frac\pi2| < n^{-\frac12+\theta}$
for all $z$ on $C_n$. Therefore $\text{Re} (w^2) = - |w|^2$
for all $w$ on $k_n$. Moreover, since $n^{-\frac12+\theta} < \frac12$ (see definition
\ref{defxiN}), $\text{Re} (z^2) = |z|^2 \cos(2\text{Arg}(z)) < - \frac14 |z|^2$ for all
$z$ on $C_n$. Therefore,
\begin{equation*}
\left| \exp(w^2 + z^2) \right|
= \exp(\text{Re}(w^2 + z^2))
< \exp(- |w|^2 - \tfrac14 |z|^2),
\end{equation*}
for all $w$ on $k_n$ and $z$ on $C_n$. Next, proceed similarly to part (iii)
to get:
\begin{equation*}
\left| \frac{\exp(n^{1-3\theta} g_n(w, z))}
{t-s + n^{-\frac12} D_n^{-1} w - n^{-\frac12} \tilde{D}_n^{-1} z} \right|
< \frac4{\frac12(t-s)}
= \frac8{t-s},
\end{equation*}
for all $w$ on $k_n$ and $z$ on $C_n$. Recall that $k_n(x) = i x$ for all
$x \in [-n^{\frac12-\theta} b_n D_n, n^{\frac12-\theta} b_n D_n]$, and that
$C_n$ is composed of $2$ arcs of $\partial B(0, n^{\frac12-\theta} \tilde{b}_n \tilde{D}_n)$
with total length $2 n^{\frac12-\theta} \tilde{b}_n \tilde{D}_n |\tilde{\alpha}_n - \frac\pi2|
< 2 n^{\frac12-\theta} \tilde{b}_n \tilde{D}_n n^{-\frac12+\theta} = 2 \tilde{b}_n \tilde{D}_n$.
Combined the above give,
\begin{align*}
|I_n^{(k,C)}|
&< \int_{-\infty}^\infty dx \; (2 \tilde{b}_n \tilde{D}_n) \;
\exp(- x^2 - \tfrac14 n^{1-2\theta} \tilde{b}_n^2 \tilde{D}_n^2) \frac8{t-s} \\
&= \sqrt{\pi} \; (2 \tilde{b}_n \tilde{D}_n) \;
\exp(-\tfrac14 n^{1-2\theta} \tilde{b}_n^2 \tilde{D}_n^2) \frac8{t-s}.
\end{align*}
Finally note that $\sqrt{\pi} < 2$, and $1 \le \tilde{b}_n < 2$
(see part (1) of lemma \ref{lemTay}). This proves (iv).

Consider (v). Proceed similar to case (iv) to get the following
for all $w$ on $c_n$ and $z$ on $K_n$:
$||\text{Arg}(w)| - \frac\pi2| < |\alpha_n - \frac\pi2| < n^{-\frac12+\theta} < \frac12$,
$|\text{Arg}(z)| = \frac\pi2$,
\begin{align*}
\left| \exp(w^2 + z^2) \right|
&< \exp(- \tfrac14 |w|^2 - |z|^2), \\
\left| \frac{\exp(n^{1-3\theta} g_n(w, z))}
{t-s + n^{-\frac12} D_n^{-1} w - n^{-\frac12} \tilde{D}_n^{-1} z} \right|
&< \frac8{t-s}.
\end{align*}
Next recall that $c_n$ is composed of $2$ arcs of $\partial B(0, n^{\frac12-\theta} b_n D_n)$
with total length $2 n^{\frac12-\theta} b_n D_n |\alpha_n - \frac\pi2| < 2 b_n D_n$,
and $K_n(y) = i y$ for all
$y \in [-n^{\frac12-\theta} \tilde{b}_n \tilde{D}_n, n^{\frac12-\theta} \tilde{b}_n \tilde{D}_n]$.
Combine the above with equation (\ref{eqJ1nllIn}) to get,
\begin{align*}
|I_n^{(c,K)}|
&< (2 b_n D_n) \; \int_{-\infty}^\infty dy \;  
\exp(-\tfrac14 n^{1-2\theta} b_n^2 D_n^2 - y^2 ) \frac8{t-s} \\
&= (2 b_n D_n) \; \sqrt{\pi} 
\exp(-\tfrac14 n^{1-2\theta} b_n^2 D_n^2) \frac8{t-s}.
\end{align*}
Finally note that $\sqrt{\pi} < 2$, and $1 \le b_n < 2$ (see part (1) of lemma \ref{lemTay}).
This proves (v).

Consider (vi). Proceed similar to previous cases (iv,v) to get the following
for all $w$ on $c_n$ and $z$ on $C_n$:
$||\text{Arg}(w)| - \frac\pi2| < |\alpha_n - \frac\pi2| < n^{-\frac12+\theta} < \frac12$,
$||\text{Arg}(z)| - \frac\pi2| < |\tilde{\alpha}_n - \frac\pi2| < n^{-\frac12+\theta} < \frac12$,
\begin{align*}
\left| \exp(w^2 + z^2) \right|
&< \exp(- \tfrac14 |w|^2 - \tfrac14 |z|^2), \\
\left| \frac{\exp(n^{1-3\theta} g_n(w, z))}
{t-s + n^{-\frac12} D_n^{-1} w - n^{-\frac12} \tilde{D}_n^{-1} z} \right|
&< \frac8{t-s}.
\end{align*}
Next recall that $c_n$ is composed of $2$ arcs of $\partial B(0, n^{\frac12-\theta} b_n D_n)$
with total length $2 n^{\frac12-\theta} b_n D_n |\alpha_n - \frac\pi2| < 2 b_n D_n$,
and $C_n$ is composed of $2$ arcs of $\partial B(0, n^{\frac12-\theta} \tilde{b}_n \tilde{D}_n)$
with total length $2 n^{\frac12-\theta} \tilde{b}_n \tilde{D}_n |\tilde{\alpha}_n - \frac\pi2|
< 2 \tilde{b}_n \tilde{D}_n$. Combine the above with equation (\ref{eqJ1nllIn}) to get,
\begin{equation*}
|I_n^{(c,C)}|
< (2 b_n D_n) \; (2 \tilde{b}_n \tilde{D}_n) \;
\exp(-\tfrac14 n^{1-2\theta} b_n^2 D_n^2 -
\tfrac14 n^{1-2\theta} \tilde{b}_n^2 \tilde{D}_n^2 ) \frac8{t-s}.
\end{equation*}
Finally recall that $1 \le b_n, \tilde{b}_n < 2$. This proves (vi).
\end{proof}

Next we examine the asymptotic behaviour of the remaining terms
of equation (\ref{eqJn11Jn12}):
\begin{lem}
\label{lemJn12}
The following is satisfied:
\begin{equation*}
|J_n^{(l,r)} + J_n^{(r,l)} + J_n^{(r,r)}|
< \frac{\exp(n f_n(t) - n \tilde{f}_n(s))}{t-s} \;
\exp( - \tfrac14 n^{1-2\theta} (D_n^2 \wedge \tilde{D}_n^2))
\; n^{-\theta} \; G_n,
\end{equation*}
where $G_n > 0$ is defined in the proof and satisfies $G_n = O(1)$ for all $n$
sufficiently large.
\end{lem}

\begin{proof}
We will show:
\begin{align}
\tag{i}
|J_n^{(l,r)}|
&< \frac{2^3 \exp(n f_n(t) - n \tilde{f}_n(s))}{t-s}
\; \exp( - \tfrac14 n^{1-2\theta} \tilde{D}_n^2)
\; n^{-\theta} \; (s-\chi), \\
\tag{ii}
|J_n^{(r,l)}|
&< \frac{2^3 \exp(n f_n(t) - n \tilde{f}_n(s))}{t-s}
\; \exp(- \tfrac14 n^{1-2\theta} D_n^2)
\; n^{-\theta} \; (t-\chi), \\
\tag{iii}
|J_n^{(r,r)}|
&< \frac{2^4 \exp(n f_n(t) - n \tilde{f}_n(s))}{t-s} \;
\exp(- \tfrac14 n^{1-2\theta} (D_n^2 + \tilde{D}_n^2)) \;
(t-\chi) (s-\chi).
\end{align}
Recall that $\theta \in (\frac13, \frac12)$, and
$D_n^2 > \frac18 |f_{(t,s)}''(t)| > 0$
and  $\tilde{D}_n^2 > \frac18 |f_{(t,s)}''(s)| > 0$
(see part (3) of lemma \ref{lemTay}). The required result
then easily follows from parts (i,ii,iii) with,
\begin{equation*}
G_n = 2^3 (s-\chi) + 2^3 (t-\chi)
+ 2^4 \exp(- \tfrac14 n^{1-2\theta} (D_n^2 \wedge \tilde{D}_n^2)) \; n^\theta\;
(t-\chi) (s-\chi).
\end{equation*}

Consider (i). Note, equations (\ref{eqfn}, \ref{eqtildefn}, \ref{eqPn}, \ref{eqJn11Jn12}) give,
\begin{equation*}
|J_n^{(l,r)}| \le \frac1{(2\pi)^2} |\g_n^{(l)}| |\G_n^{(r)}|
\sup_{(w,z) \in \g_n^{(l)} \times \G_n^{(r)}}
\left| \frac{\exp(n f_n(w) - n \tilde{f}_n(z))}{w-z} \right|.
\end{equation*}
Recall (see equation (\ref{eqConLocRem})) that $\g_n^{(l)}$ is
the lines from $t + n^{-\theta} b_n e^{-i \alpha_n}$ to $t$,
and from $t$ to $t + n^{-\theta} b_n e^{i \alpha_n}$.
Therefore $|\g_n^{(l)}| = 2 n^{-\theta} b_n < 4 n^{-\theta}$ (see part (1) of
lemma \ref{lemTay}). Next recall (see definition \ref{defDesAsc} and equation
(\ref{eqConLocRem})) that $\G_n^{(r)}$ traverses the contour
$x \mapsto u_n + R_n(1-x) + i I_n(1-x)$ for $x \in [0,1]$, and its
reflection in $\R$. Combine the above with parts (4,5,7) of lemma
\ref{lemDesAsc} to get,
\begin{equation*}
|J_n^{(l,r)}| < \frac1{(2\pi)^2} (4 n^{-\theta}) (8(s-\chi))
\sup_{w \in \g_n^{(l)}}
\left| \frac{\exp(n f_n(w) - n \tilde{f}_n(\tilde{s}_n + in^{-\theta}))}
{\frac12 (t-s)} \right|.
\end{equation*}
Thus, since $\frac{(4)(8)}{(2\pi)^2} < 1$, and
$\tilde{s}_n + i n^{-\theta} = s + n^{-\theta} \tilde{b}_n e^{i\tilde{\alpha}_n}$
(see lemma \ref{lemTay}),
\begin{equation*}
|J_n^{(l,r)}| < \frac{2 (s-\chi)}{t-s} \; n^{-\theta} \;
\sup_{w \in h_n} \left| \exp(n f_n(t + n^{-\frac12} D_n^{-1} w)
- n \tilde{f}_n(s + n^{-\frac12} \tilde{D}_n^{-1} z_n)) \right|,
\end{equation*}
where $h_n \subset B(0,n^{\frac12-\theta} b_n D_n)$ is defined in figure
\ref{figConReScaleCase1}, and
$z_n \in \partial B(0,n^{\frac12-\theta} \tilde{b}_n \tilde{D}_n)$ is defined
by $z_n := n^{\frac12-\theta} \tilde{b}_n \tilde{D}_n e^{i \tilde{\alpha}_n}$.
Note that $h_n \subset \text{cl}(B(0, n^{\frac12-\theta} b_n D_n))$ and
$z_n \in \partial B(0, n^{\frac12-\theta} \tilde{b}_n \tilde{D}_n)$. Parts
(6,7) of lemma \ref{lemTay} then give:
\begin{equation*}
|J_n^{(l,r)}| < \frac{2 (s-\chi)}{t-s} \; n^{-\theta} \;
\sup_{w \in h_n} \left| \exp(n f_n(t) - n \tilde{f}_n(s) +
w^2 + z_n^2 + n^{1-3\theta} g_n(w,z_n)) \right|,
\end{equation*}
where
$n^{1-3\theta} g_n(w,z) := n f_n(t + n^{-\frac12} D_n^{-1} w) - n f_n(t) - w^2
- n \tilde{f}_n(s + n^{-\frac12} \tilde{D}_n^{-1} z)) + n \tilde{f}_n(s) - z^2$.
We then proceed similarly to part (iii) of the previous lemma to get,
\begin{equation*}
|J_n^{(l,r)}| < \frac{8 (s-\chi)}{t-s} \; n^{-\theta} \;
\sup_{w \in h_n} \left| \exp(n f_n(t) - n \tilde{f}_n(s) +
w^2 + z_n^2) \right|.
\end{equation*}
Recall that $|\text{Arg}(w)| = \alpha_n$ for all $w$ on $h_n$,
$|\text{Arg}(z_n)| = \tilde{\alpha}_n$,
$|\alpha_n - \tfrac\pi2| < n^{-\frac12+\theta}$ and
$|\tilde{\alpha}_n - \tfrac\pi2| \le n^{-\frac12+\theta}$ (see part
(2) of lemma \ref{lemTay}), and
$|z_n| = n^{\frac12-\theta} \tilde{b}_n \tilde{D}_n$. We then proceed
similarly to parts (iv, v, vi) of the previous lemma to get
$\text{Re}(w^2) <  - \frac14 |w|^2 \le 0$ for all $w$ on $h_n$, and
$\text{Re} ((z_n)^2) < - \frac14 |z_n|^2
= - \frac14 n^{1-2\theta} \tilde{b}_n^2 \tilde{D}_n^2
\le - \frac14 n^{1-2\theta} \tilde{D}_n^2$.
Combined, the above prove (i).

Consider (ii). First, proceed similarly to part (i) to get
$|\G_n^{(l)}| < 4 n^{-\theta}$. Next recall (see definition \ref{defDesAsc}
and equation (\ref{eqConLocRem})) that $\g_n^{(r)}$ the counter-clockwise arc
of $\partial B(v_n, q_n)$ from $t_n + i n^{-\theta}$ to $v_n-q_n$,
and its reflection in $\R$. Equations (\ref{eqfn}, \ref{eqtildefn}, \ref{eqPn},
\ref{eqJn11Jn12}), and parts (3,5,6) of lemma \ref{lemDesAsc}
thus give,
\begin{equation*}
|J_n^{(r,l)}| < \frac1{(2\pi)^2} (8(t-\chi)) (4 n^{-\theta})
\sup_{z \in \G_n^{(l)}}
\left| \frac{\exp(n f_n(t_n + i n^{-\theta}) - n \tilde{f}_n(z))}
{\frac12 (t-s)} \right|.
\end{equation*}
Thus, since $t_n + i n^{-\theta} = t + n^{-\theta} b_n e^{i\alpha_n}$,
\begin{equation*}
|J_n^{(r,l)}| < \frac{2 (t-\chi)}{t-s} \; n^{-\theta} \;
\sup_{z \in H_n} \left| \exp(n f_n(t + n^{-\frac12} D_n^{-1} w_n)
- n \tilde{f}_n(s + n^{-\frac12} \tilde{D}_n^{-1} z)) \right|,
\end{equation*}
where $w_n \in \partial B(0,n^{\frac12-\theta} b_n D_n)$ is defined
by $w_n := n^{\frac12-\theta} b_n D_n e^{i \alpha_n}$, and
$H_n \subset B(0,n^{\frac12-\theta} \tilde{b}_n \tilde{D}_n)$ is defined
in figure \ref{figConReScaleCase1}. Proceed similarly to part (i)
to then get,
\begin{equation*}
|J_n^{(r,l)}| < \frac{8 (t-\chi)}{t-s} \; n^{-\theta} \;
\sup_{z \in H_n} \left| \exp(n f_n(t) - n \tilde{f}_n(s) +
w_n^2 + z^2) \right|,
\end{equation*}
$\text{Re} ((w_n)^2) < - \frac14 |w_n|^2 = - \frac14 n^{1-2\theta} b_n^2 D_n^2
\le - \frac14 n^{1-2\theta} D_n^2$, and $\text{Re}(z^2) < - \frac14 |z|^2 \le 0$
for all $z$ on $H_n$. Combined, the above prove (ii).

Consider (iii). First, proceeding similarly to parts (i,ii),
\begin{equation*}
|J_n^{(r,r)}| < \frac1{(2\pi)^2} (8(t-\chi)) (8(s-\chi))
\left| \frac{\exp(n f_n(t_n + i n^{-\theta}) - n \tilde{f}_n(\tilde{s}_n + i n^{-\theta}))}
{\frac12 (t-s)} \right|.
\end{equation*}
Then, define $w_n \in \partial B(0,n^{\frac12-\theta} b_n D_n)$
and $z_n \in \partial B(0,n^{\frac12-\theta} \tilde{b}_n \tilde{D}_n)$ as above,
and proceed similarly to parts (i,ii) to get,
\begin{align*}
|J_n^{(r,r)}|
&< \frac{4 (t-\chi) (s-\chi)}{t-s} \;
\left| \exp(n f_n(t + n^{-\frac12} D_n^{-1} w_n)
- n \tilde{f}_n(s + n^{-\frac12} \tilde{D}_n^{-1} z_n)) \right| \\
&< \frac{16 (t-\chi) (s-\chi)}{t-s} \;
\left| \exp(n f_n(t) - n \tilde{f}_n(s) + w_n^2 + z_n^2) \right|,
\end{align*}
$\text{Re} ((w_n)^2) < - \frac14 n^{1-2\theta} D_n^2$, and
$\text{Re} ((z_n)^2) < - \frac14 n^{1-2\theta} \tilde{D}_n^2$.
Combined, the above prove (iii).
\end{proof}

Finally we prove theorem \ref{thmdecay}:
\begin{proof}[Proof of theorem \ref{thmdecay}:]
First recall (see equation (\ref{eqJn11Jn12})) that
$J_n = J_n^{(l,l)} + J_n^{(l,r)} + J_n^{(r,l)} + J_n^{(r,r)}$.
Lemmas \ref{lemJn11} and \ref{lemJn12} thus gives,
\begin{align*}
\left| n J_n
- \frac{\exp( n f_n(t) - n \tilde{f}_n(s))}{4 \pi (t-s) D_n \tilde{D}_n} \right|
&< \frac{\exp( n f_n(t) - n \tilde{f}_n(s))}{4 \pi (t-s) D_n \tilde{D}_n} \; n^{1-3\theta} \; F_n \\
&+ \frac{\exp(n f_n(t) - n \tilde{f}_n(s))}{t-s} \;
\exp( - \tfrac14 n^{1-2\theta} (D_n^2 \wedge \tilde{D}_n^2)) \; n^{1-\theta} \; G_n,
\end{align*}
where $F_n$ and $G_n$ are defined in the proof of lemmas \ref{lemJn11} and \ref{lemJn12}
(respectively). Moreover, equations (\ref{eqKnrusvFixTopLine}, \ref{eqKnrnunsnvn1})
trivially give $\phi_{r_n, s_n}(u_n,v_n) = 0$ and
$K_n((u_n,r_n),(v_n,s_n)) = (1-\tfrac{s_n}n) \; nJ_n$ when $r_n = s_n$ for all $n>N$,
as required.
\end{proof}

\section{The behaviour of the roots of \texorpdfstring{$f_{(\chi,\eta)}'$}{Lg}}
\label{sectrofn'}

In this section we examine the behaviour of the roots of the function
$f_{(\chi,\eta)}'$ given in equation (\ref{eqf'}). Only the following assumptions are
required in this section:
\begin{itemize}
\item
$\mu$ is a probability measure on $\R$ with compact support, $\supp(\mu) \subset [a,b]$ with
$\{a,b\} \subset \supp(\mu)$, and $(\chi,\eta) \in [a,b] \times [0,1]$ is fixed.
\item
Assume that $b > a$ to avoid that degenerate
case where $\mu$ is a single atom of mass 1. This implies that
$\mu[\{\chi\}] \in [0,1)$.
\end{itemize}

Recall (see equation (\ref{eqf'})),
\begin{equation*}
f_{(\chi,\eta)}'(w)
= \int_{(\chi,b]} \frac{\mu[dx]}{w-x}
- \frac{1-\eta - \mu[\{\chi\}]}{w-\chi}
+ \int_{[a,\chi)} \frac{\mu[dx]}{w-x},
\end{equation*}
for all $w \in \C \setminus \R$. The above expression has a unique
analytic extension to the set $\C \setminus (S_1 \cup S_2 \cup S_3)$,
where $S_i := S_i(\chi,\eta)$ for all $i \in \{1,2,3\}$ are defined by: 
\begin{equation*}
S_1 := \supp(\mu |_{(\chi,b]}),
\hspace{0.5cm}
S_2 := \left\{ \begin{array}{rcl}
\{\chi\} & ; & \text{when } \mu[\{\chi\}] \neq 1-\eta, \\
\emptyset & ; & \text{when } \mu[\{\chi\}] = 1-\eta,
\end{array} \right.
\hspace{0.5cm}
S_3 := \supp(\mu |_{[a,\chi)}).
\end{equation*}
Note $S_1 = \emptyset$ when $b = \chi$, and $S_3 = \emptyset$ when $\chi = a$.
Thus, since $b > a$, $\supp(\mu) \subset [a,b]$ with
$\{a,b\} \subset \supp(\mu)$, $(\chi,\eta) \in [a,b] \times [0,1]$,
and $\mu[\{\chi\}] \in [0,1)$, the following 12 cases exhaust all
possibilities:
\begin{itemize}
\item[(a)]
$b > \chi > a$, $1 > \eta > 0$, $1-\eta > \mu[\{\chi\}]$,
and $S_1 \neq \emptyset$, $S_2 = \{\chi\}$, $S_3 \neq \emptyset$.
\item[(b)]
$b > \chi > a$, $1 > \eta = 0$, $1-\eta > \mu[\{\chi\}]$,
and $S_1 \neq \emptyset$, $S_2 = \{\chi\}$, $S_3 \neq \emptyset$.
\item[(c)]
$b > \chi > a$, $1 \ge \eta > 0$, $1-\eta < \mu[\{\chi\}]$,
and $S_1 \neq \emptyset$, $S_2 = \{\chi\}$, $S_3 \neq \emptyset$.
\item[(d)]
$b > \chi > a$, $1 \ge \eta > 0$, $1-\eta = \mu[\{\chi\}]$,
and $S_1 \neq \emptyset$, $S_2 = \emptyset$, \;\;\; $S_3 \neq \emptyset$.
\item[(e)]
$b > \chi = a$, $1 > \eta > 0$, $1-\eta > \mu[\{\chi\}]$,
and $S_1 \neq \emptyset$, $S_2 = \{\chi\}$, $S_3 = \emptyset$.
\item[(f)]
$b > \chi = a$, $1 > \eta = 0$, $1-\eta > \mu[\{\chi\}]$,
and $S_1 \neq \emptyset$, $S_2 = \{\chi\}$, $S_3 = \emptyset$.
\item[(g)]
$b > \chi = a$, $1 \ge \eta > 0$, $1-\eta < \mu[\{\chi\}]$,
and $S_1 \neq \emptyset$, $S_2 = \{\chi\}$, $S_3 = \emptyset$.
\item[(h)]
$b > \chi = a$, $1 \ge \eta > 0$, $1-\eta = \mu[\{\chi\}]$,
and $S_1 \neq \emptyset$, $S_2 = \emptyset$, \;\;\; $S_3 = \emptyset$.
\item[(i)]
$b = \chi > a$, $1 > \eta > 0$, $1-\eta > \mu[\{\chi\}]$,
and $S_1 = \emptyset$, $S_2 = \{\chi\}$, $S_3 \neq \emptyset$.
\item[(j)]
$b = \chi > a$, $1 > \eta = 0$, $1-\eta > \mu[\{\chi\}]$,
and $S_1 = \emptyset$, $S_2 = \{\chi\}$, $S_3 \neq \emptyset$.
\item[(k)]
$b = \chi > a$, $1 \ge \eta > 0$, $1-\eta < \mu[\{\chi\}]$,
and $S_1 = \emptyset$, $S_2 = \{\chi\}$, $S_3 \neq \emptyset$.
\item[(l)]
$b = \chi > a$, $1 \ge \eta > 0$, $1-\eta = \mu[\{\chi\}]$,
and $S_1 = \emptyset$, $S_2 = \emptyset$, \;\;\; $S_3 \neq \emptyset$.
\end{itemize}
Moreover note:
\begin{itemize}
\item
$b = \sup S_1 \ge \inf S_1 \ge \chi \ge \sup S_3 \ge \inf S_3 = a$
for possibilities (a-d).
\item
$b = \sup S_1 \ge \inf S_1 \ge \chi = a$ for possibilities (e-h).
\item
$b = \chi \ge \sup S_3 \ge \inf S_3 = a$ for possibilities (i-l).
\end{itemize}
The sets, $S_1,S_2,S_3$, for the above possibilities are depicted in figure
\ref{figf'Supports}. Note, since $\mu[S_1] + \mu[\{\chi\}] + \mu[S_3] = 1$,
we trivially have,
\begin{equation}
\label{eqf'2}
f_{(\chi,\eta)}'(w) 
= \int_{S_1} \frac{\mu[dx]}{w-x}
- \frac{\mu[S_1] + \mu[S_3] - \eta}{w-\chi}
+ \int_{S_3} \frac{\mu[dx]}{w-x},
\end{equation}
for all $w \in \C \setminus (S_1 \cup S_2 \cup S_3)$. 

Next write the domain of $f_{(\chi,\eta)}'$ as the disjoint union:
\begin{equation*}
\C \setminus (S_1 \cup S_2 \cup S_3) = (\C \setminus \R) \cup J \cup K,
\end{equation*}
where $J := \cup_{i=1}^4 J_i$, $K := \R \setminus (S \cup J)$, and
\begin{itemize}
\item
$J_1 := (\sup S_1,+\infty)$.
\item
$J_2 := (-\infty, \inf S_3)$.
\item
$J_3 := (\chi,\inf S_1)$ when $S_1 \neq \emptyset$ and $S_2 = \{\chi\}$
and $\inf S_1 > \chi$. Otherwise, $J_3 := \emptyset$.
\item
$J_4 := (\sup S_3, \chi)$ when $S_3 \neq \emptyset$ and $S_2 = \{\chi\}$
and $\chi > \sup S_3$. Otherwise, $J_4 := \emptyset$.
\end{itemize}
Note that $K \subset \R$ is open, and so it can be partitioned as
$K = \cup_{k=1}^\infty K_k$, where $\{K_1, K_2,\ldots\}$ is a set of
pairwise disjoint open intervals. This partition is unique up to order, and is
either empty, finite, or countable. The above sets for the different
possibilities are also depicted in figure \ref{figf'Supports}.

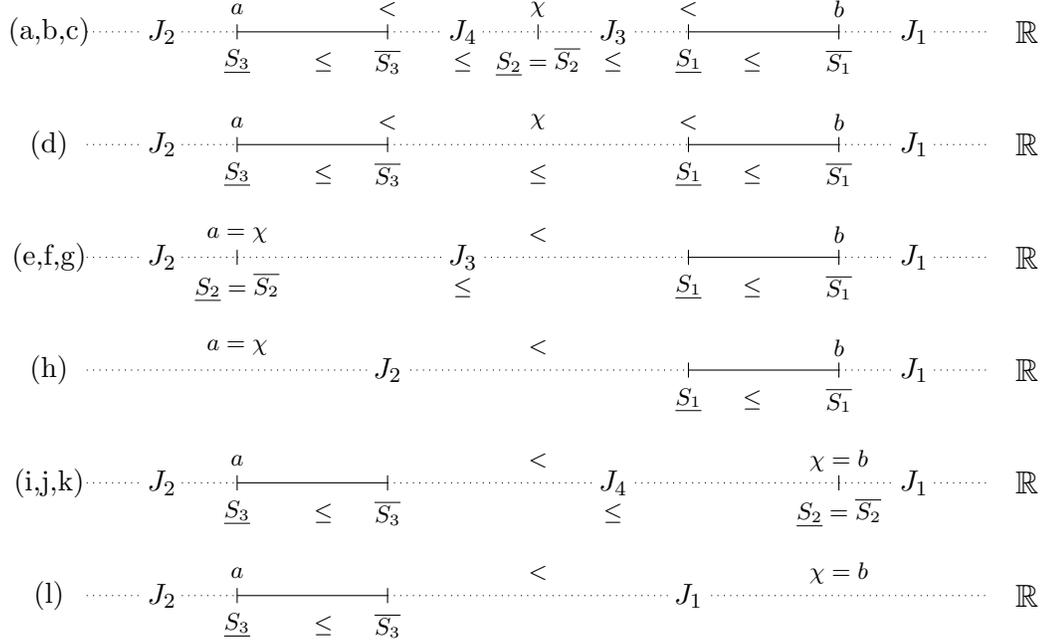
\begin{figure}[t]
\centering
\begin{tikzpicture}

\draw (-2.5,0) node {\small (a,b,c)};
\draw [dotted] (-2,0) --++ (2,0);
\draw [fill,white] (-1,0) circle (.25cm);
\draw (-1,0) node {\small $J_2$};
\draw (0,0) --++ (2,0);
\draw [dotted] (2,0) --++ (4,0);
\draw [fill,white] (3,0) circle (.25cm);
\draw (3,0) node {\small $J_4$};
\draw [fill,white] (5,0) circle (.25cm);
\draw (5,0) node {\small $J_3$};
\draw (6,0) --++ (2,0);
\draw [dotted] (8,0) --++ (2,0);
\draw [fill,white] (9,0) circle (.25cm);
\draw (9,0) node {\small $J_1$};
\draw (10.5,0) node {$\R$};

\draw (0,.1) --++ (0,-.2);
\draw (0,-.4) node {\scriptsize $\underline{S_3}$};
\draw (1.15,-.4) node {\scriptsize $\le$};
\draw (2,.1) --++ (0,-.2);
\draw (2,-.4) node {\scriptsize $\overline{S_3}$};
\draw (3,-.4) node {\scriptsize $\le$};
\draw (4,.1) --++ (0,-.2);
\draw (4,-.4) node {\scriptsize $\underline{S_2} = \overline{S_2}$};
\draw (5,-.4) node {\scriptsize $\le$};
\draw (6,.1) --++ (0,-.2);
\draw (6,-.4) node {\scriptsize $\underline{S_1}$};
\draw (6.85,-.4) node {\scriptsize $\le$};
\draw (8,.1) --++ (0,-.2);
\draw (8,-.4) node {\scriptsize $\overline{S_1}$};

\draw (0,.3) node {\scriptsize $a$};
\draw (2,.3) node {\scriptsize $<$};
\draw (4,.3) node {\scriptsize $\chi$};
\draw (6,.3) node {\scriptsize $<$};
\draw (8,.3) node {\scriptsize $b$};

\draw (-2.5,-1.5) node {\small (d)};
\draw [dotted] (-2,-1.5) --++ (2,0);
\draw [fill,white] (-1,-1.5) circle (.25cm);
\draw (-1,-1.5) node {\small $J_2$};
\draw (0,-1.5) --++ (2,0);
\draw [dotted] (2,-1.5) --++ (4,0);
\draw (6,-1.5) --++ (2,0);
\draw [dotted] (8,-1.5) --++ (2,0);
\draw [fill,white] (9,-1.5) circle (.25cm);
\draw (9,-1.5) node {\small $J_1$};
\draw (10.5,-1.5) node {$\R$};

\draw (0,-1.4) --++ (0,-.2);
\draw (0,-1.9) node {\scriptsize $\underline{S_3}$};
\draw (1.15,-1.9) node {\scriptsize $\le$};
\draw (2,-1.4) --++ (0,-.2);
\draw (2,-1.9) node {\scriptsize $\overline{S_3}$};
\draw (4,-1.9) node {\scriptsize $\le$};
\draw (6,-1.4) --++ (0,-.2);
\draw (6,-1.9) node {\scriptsize $\underline{S_1}$};
\draw (6.85,-1.9) node {\scriptsize $\le$};
\draw (8,-1.4) --++ (0,-.2);
\draw (8,-1.9) node {\scriptsize $\overline{S_1}$};

\draw (0,-1.2) node {\scriptsize $a$};
\draw (2,-1.2) node {\scriptsize $<$};
\draw (4,-1.2) node {\scriptsize $\chi$};
\draw (6,-1.2) node {\scriptsize $<$};
\draw (8,-1.2) node {\scriptsize $b$};

\draw (-2.5,-3) node {\small (e,f,g)};
\draw [dotted] (-2,-3) --++ (2,0);
\draw [fill,white] (-1,-3) circle (.25cm);
\draw (-1,-3) node {\small $J_2$};
\draw [dotted] (0,-3) --++ (6,0);
\draw [fill,white] (3,-3) circle (.25cm);
\draw (3,-3) node {\small $J_3$};
\draw (6,-3) --++ (2,0);
\draw [dotted] (8,-3) --++ (2,0);
\draw [fill,white] (9,-3) circle (.25cm);
\draw (9,-3) node {\small $J_1$};
\draw (10.5,-3) node {$\R$};

\draw (0,-2.9) --++ (0,-.2);
\draw (0,-3.4) node {\scriptsize $\underline{S_2} = \overline{S_2}$};
\draw (3,-3.4) node {\scriptsize $\le$};
\draw (6,-2.9) --++ (0,-.2);
\draw (6,-3.4) node {\scriptsize $\underline{S_1}$};
\draw (6.85,-3.4) node {\scriptsize $\le$};
\draw (8,-2.9) --++ (0,-.2);
\draw (8,-3.4) node {\scriptsize $\overline{S_1}$};

\draw (0,-2.7) node {\scriptsize $a = \chi$};
\draw (4,-2.7) node {\scriptsize $<$};
\draw (8,-2.7) node {\scriptsize $b$};

\draw (-2.5,-4.5) node {\small (h)};
\draw [dotted] (-2,-4.5) --++ (8,0);
\draw [fill,white] (2,-4.5) circle (.25cm);
\draw (2,-4.5) node {\small $J_2$};
\draw (6,-4.5) --++ (2,0);
\draw [dotted] (8,-4.5) --++ (2,0);
\draw [fill,white] (9,-4.5) circle (.25cm);
\draw (9,-4.5) node {\small $J_1$};
\draw (10.5,-4.5) node {$\R$};

\draw (6,-4.4) --++ (0,-.2);
\draw (6,-4.9) node {\scriptsize $\underline{S_1}$};
\draw (6.85,-4.9) node {\scriptsize $\le$};
\draw (8,-4.4) --++ (0,-.2);
\draw (8,-4.9) node {\scriptsize $\overline{S_1}$};

\draw (0,-4.2) node {\scriptsize $a = \chi$};
\draw (4,-4.2) node {\scriptsize $<$};
\draw (8,-4.2) node {\scriptsize $b$};

\draw (-2.5,-6) node {\small (i,j,k)};
\draw [dotted] (-2,-6) --++ (2,0);
\draw [fill,white] (-1,-6) circle (.25cm);
\draw (-1,-6) node {\small $J_2$};
\draw (0,-6) --++ (2,0);
\draw [dotted] (2,-6) --++ (6,0);
\draw [fill,white] (5,-6) circle (.25cm);
\draw (5,-6) node {\small $J_4$};
\draw [dotted] (8,-6) --++ (2,0);
\draw [fill,white] (9,-6) circle (.25cm);
\draw (9,-6) node {\small $J_1$};
\draw (10.5,-6) node {$\R$};

\draw (0,-5.9) --++ (0,-.2);
\draw (0,-6.4) node {\scriptsize $\underline{S_3}$};
\draw (1.15,-6.4) node {\scriptsize $\le$};
\draw (2,-5.9) --++ (0,-.2);
\draw (2,-6.4) node {\scriptsize $\overline{S_3}$};
\draw (5,-6.4) node {\scriptsize $\le$};
\draw (8,-5.9) --++ (0,-.2);
\draw (8,-6.4) node {\scriptsize $\underline{S_2} = \overline{S_2}$};

\draw (0,-5.7) node {\scriptsize $a$};
\draw (4,-5.7) node {\scriptsize $<$};
\draw (8,-5.7) node {\scriptsize $\chi = b$};

\draw (-2.5,-7.5) node {\small (l)};
\draw [dotted] (-2,-7.5) --++ (2,0);
\draw [fill,white] (-1,-7.5) circle (.25cm);
\draw (-1,-7.5) node {\small $J_2$};
\draw (0,-7.5) --++ (2,0);
\draw [dotted] (2,-7.5) --++ (8,0);
\draw [fill,white] (6,-7.5) circle (.25cm);
\draw (6,-7.5) node {\small $J_1$};
\draw (10.5,-7.5) node {$\R$};

\draw (0,-7.4) --++ (0,-.2);
\draw (0,-7.9) node {\scriptsize $\underline{S_3}$};
\draw (1.15,-7.9) node {\scriptsize $\le$};
\draw (2,-7.4) --++ (0,-.2);
\draw (2,-7.9) node {\scriptsize $\overline{S_3}$};

\draw (0,-7.2) node {\scriptsize $a$};
\draw (4,-7.2) node {\scriptsize $<$};
\draw (8,-7.2) node {\scriptsize $\chi = b$};

\end{tikzpicture}
\caption{The sets $S_1, S_2, S_3, J_1, J_2, J_3, J_4$
for possibilities (a-l). When one of these sets is not depicted,
it is understood to be empty. Also, $J_3$ is empty when $\inf S_1 = \chi$,
and $J_4$ is empty when $\chi = \sup S_3$. Above, $\overline{S_i} := \sup S_i$
and $\underline{S_i} := \inf S_i$. Recall that
$K = \R \setminus (S \cup J) = \cup_{k=1}^\infty K_k$, where
$\{K_1, K_2,\ldots\}$ are disjoint open intervals. Finally, note that 
$[\inf S_i, \sup S_i] \setminus S_i$ is either empty or (finite or countable)
union of intervals from $\{K_1,K_2,\ldots\}$.}
\label{figf'Supports}
\end{figure}

\begin{rem}
\label{remMultiplicity}
For the remainder of this section, whenever we say a number of roots, it should
be implicitly understood that we mean that number of roots counting multiplicities.
\end{rem}

The behaviour of the roots of $f_{(\chi,\eta)}'$ for the above possibilities is the following:
\begin{thm}
\label{thmf'}
For (a), $f_{(\chi,\eta)}'$ has at most $2$ roots in each of
$\{\C \setminus \R, J_1, J_2, J_3, J_4\}$, and at most $3$ roots in each of
$\{K_1, K_2, \ldots\}$. Moreover, when $f_{(\chi,\eta)}'$ has either $1$ or $2$
roots in some fixed $I \in \{\C \setminus \R, J_1, J_2, J_3, J_4\}$, then $f_{(\chi,\eta)}'$
has $0$ roots in each of $\{\C \setminus \R, J_1, J_2, J_3, J_4\} \setminus \{I\}$,
and at most $1$ root in each of $\{K_1, K_2, \ldots\}$.
Finally, when $f_{(\chi,\eta)}'$ has either $2$ or $3$ roots in
some fixed $L \in \{K_1, K_2, \ldots\}$, then $f_{(\chi,\eta)}'$
has $0$ roots in each of $\{\C \setminus \R, J_1, J_2, J_3, J_4\}$,
and at most $1$ root in each of $\{K_1, K_2, \ldots\} \setminus \{L\}$.

For (b), $f_{(\chi,\eta)}'$ has at most $1$ root in each of
$\{J_1, J_2\} \cup \{K_1,K_2,\ldots\}$, and $0$ roots in each of
$\{\C \setminus \R, J_3, J_4\}$. Moreover, when $f_{(\chi,\eta)}'$ has $1$
root in some fixed $I \in \{J_1, J_2\}$, then $f_{(\chi,\eta)}'$
has $0$ roots in $\{J_1, J_2\} \setminus \{I\}$.

For (c), $f_{(\chi,\eta)}'$ has $0$ roots in each of
$\{\C \setminus \R, J_1, J_2\}$, and at most $1$ root in
each of $\{J_3, J_4\} \cup \{K_1, K_2, \ldots \}$.
For (d), $f_{(\chi,\eta)}'$ has $0$ roots in each of
$\{\C \setminus \R, J_1, J_2\}$, and at most $1$ root in
each of $\{K_1, K_2, \ldots \}$.

For (e), $f_{(\chi,\eta)}'$ has $0$ roots in each of
$\{\C \setminus \R, J_1, J_3\}$, and at most $1$ root in
each of $\{J_2\} \cup \{K_1, K_2, \ldots \}$. For (f),
$f_{(\chi,\eta)}'$ has $0$ roots in each of
$\{\C \setminus \R, J_1, J_2, J_3\}$, and at most $1$ root in
each of $\{K_1, K_2, \ldots \}$. For (g), $f_{(\chi,\eta)}'$
has $0$ roots in each of $\{\C \setminus \R, J_1, J_2\}$,
and at most $1$ root in each of $\{J_3\} \cup \{K_1, K_2, \ldots \}$.
For (h), $f_{(\chi,\eta)}'$ has $0$ roots in each of
$\{\C \setminus \R, J_1, J_2\}$, and at most $1$ root in
each of $\{K_1, K_2, \ldots \}$.

For (i), $f_{(\chi,\eta)}'$ has $0$ roots in each of
$\{\C \setminus \R, J_2, J_4\}$, and at most $1$ root in
each of $\{J_1\} \cup \{K_1, K_2, \ldots \}$. For (j),
$f_{(\chi,\eta)}'$ has $0$ roots in each of
$\{\C \setminus \R, J_1, J_2, J_4\}$, and at most $1$ root in
each of $\{K_1, K_2, \ldots \}$. For (k), $f_{(\chi,\eta)}'$
has $0$ roots in each of $\{\C \setminus \R, J_1, J_2\}$,
and at most $1$ root in each of $\{J_4\} \cup \{K_1, K_2, \ldots \}$.
For (l), $f_{(\chi,\eta)}'$ has $0$ roots in each of
$\{\C \setminus \R, J_1, J_2\}$, and at most $1$ root in
each of $\{K_1, K_2, \ldots \}$.
\end{thm}

\begin{proof}
We will prove the result only for possibilities (a,b) when the supports
are as given on the top of figure \ref{figthmroots}. The remaining
results follow from similar considerations.

Consider (a,b) where the supports are given as on the top of figure
\ref{figthmroots}. First note, equation (\ref{eqf'2}) trivially
implies the following:
\begin{itemize}
\item[(i)]
Non-real roots of $f_{(\chi,\eta)}'$ occur in complex conjugate pairs.
\end{itemize}
Next, inspired by equation (\ref{eqf'2}), define the following for all $n \ge 1$:
\begin{equation}
\label{eqthmroots}
g_n (w) 
:= \frac1n \sum_{x \in X_n} \frac1{w-x}
- \frac{\frac{m+l}n-\eta}{w-\chi}
+ \frac1n \sum_{y \in Y_n} \frac1{w-y},
\end{equation}
for all $w \in \C \setminus (X_n \cup \{\chi\} \cup Y_n)$, where:
\begin{itemize}
\item
$m := m(n)$ is a positive integer ($\ge 4$) with $\frac{m}n \to \mu[S_1] > 0$ as $n \to \infty$.
\item
$l := l(n)$ is a positive integer  ($\ge 2$) with $\frac{l}n \to \mu[S_3] > 0$ as $n \to \infty$.
\item
$X_n$ is a set of $m$ distinct real-numbers with 
$\{a_2, a_1\} \subset X_n \subset [\underline{X_n}, a_2] \cup [a_1, \overline{X_n}]$
for all $n$, $\underline{X_n} \to \underline{S_1}$ and $\overline{X_n} \to \overline{S_1}$
as $n \to \infty$, and $\frac1n \sum_{x \in X_n} \delta_x \to \mu |_{(\chi,b]}$
weakly as $n \to \infty$.
\item
$Y_n$ is a set of $l$ distinct real-numbers with
$\{\underline{S_3}\} \subset Y_n \subset [\underline{S_3},\overline{S_3})$
for all $n$, $\overline{Y_n} \uparrow \overline{S_3} = \chi$ as $n \to \infty$,
and $\frac1n \sum_{y \in Y_n} \delta_y \to \mu |_{[a,\chi)}$ weakly
as $n \to \infty$.
\end{itemize}
These are depicted on the bottom of figure \ref{figthmroots}.
Equations (\ref{eqf'2}, \ref{eqthmroots}), the above convergence as $n \to \infty$,
and Rouch\'e's theorem imply the following:
\begin{itemize}
\item[(ii)]
Suppose that $z \in \C \setminus (S_1 \cup S_2 \cup S_3) = (\C \setminus \R) \cup (J_1 \cup J_2 \cup J_3 \cup K_1)$
is a root of $f_{(\chi,\eta)}'$ of multiplicity $k\ge1$. Fix $\e > 0$
for which $B(z,\e) \subset \C \setminus (S_1 \cup S_2 \cup S_3)$, and $z$ is the unique root of
$f_{(\chi,\eta)}'$ in $B(z,\e)$. Then, for all $n$ sufficiently large,
$g_n$ has $k$ roots in $B(z,\e)$.
\end{itemize}

\begin{figure}[t]
\centering
\begin{tikzpicture}

\draw [dotted] (-3,0) --++ (1.5,0);
\draw [fill,white] (-2.25,0) circle (.25cm);
\draw (-2.25,0) node {\small $J_2$};
\draw (-1.5,0) --++ (3,0);
\draw [dotted] (1.5,0) --++ (1.5,0);
\draw [fill,white] (2.25,0) circle (.25cm);
\draw (2.25,0) node {\small $J_3$};
\draw (3,0) --++ (1.5,0);
\draw [dotted] (4.5,0) --++ (1.5,0);
\draw [fill,white] (5.25,0) circle (.25cm);
\draw (5.25,0) node {\small $K_1$};
\draw (6,0) --++ (1.5,0);
\draw (7.5,0) [dotted] --++ (1.5,0);
\draw [fill,white] (8.25,0) circle (.25cm);
\draw (8.25,0) node {\small $J_1$};
\draw (9.5,0) node {$\R$};

\draw (-1.5,.1) --++ (0,-.2);
\draw (-1.5,-.4) node {\scriptsize $\underline{S_3}$};
\draw (0,-.4) node {\scriptsize $<$};
\draw (1.5,.1) --++ (0,-.2);
\draw (1.5,-.4) node {\scriptsize $\overline{S_3}$};
\draw (2.25,-.4) node {\scriptsize $<$};
\draw (3,.1) --++ (0,-.2);
\draw (3,-.4) node {\scriptsize $\underline{S_1}$};
\draw (3.75,-.4) node {\scriptsize $<$};
\draw (4.5,.1) --++ (0,-.2);
\draw (4.5,-.4) node {\scriptsize $a_2$};
\draw (5.25,-.4) node {\scriptsize $<$};
\draw (6,.1) --++ (0,-.2);
\draw (6,-.4) node {\scriptsize $a_1$};
\draw (6.75,-.4) node {\scriptsize $<$};
\draw (7.5,.1) --++ (0,-.2);
\draw (7.5,-.4) node {\scriptsize $\overline{S_1}$};

\draw (-1.5,.3) node {\scriptsize $a$};
\draw (0,.3) node {\scriptsize $<$};
\draw (1.5,.3) node {\scriptsize $\chi$};
\draw (4.5,.3) node {\scriptsize $<$};
\draw (7.5,.3) node {\scriptsize $b$};


\draw [dotted] (-3,-1.5) --++ (12,0);
\draw [fill,white] (-2.25,-1.5) circle (.25cm);
\draw (-2.25,-1.5) node {\scriptsize $J_2$};
\draw (-1.5,-1.5) node {\scriptsize $\times$};
\draw (-1.075,-1.5) node {\scriptsize $\times$};
\draw (-.65,-1.5) node {\scriptsize $\times$};
\draw (-.225,-1.5) node {\scriptsize $\times$};
\draw (.2,-1.5) node {\scriptsize $\times$};
\draw (.625,-1.5) node {\scriptsize $\times$};
\draw (1.05,-1.5) node {\scriptsize $\times$};
\draw [fill,white] (2.05,-1.5) circle (.25cm);
\draw (2.05,-1.5) node {\scriptsize $J_{3,n}$};
\draw (2.8,-1.5) node {\scriptsize $\times$};
\draw (3.225,-1.5) node {\scriptsize $\times$};
\draw (3.65,-1.5) node {\scriptsize $\times$};
\draw (4.075,-1.5) node {\scriptsize $\times$};
\draw (4.5,-1.5) node {\scriptsize $\times$};
\draw [fill,white] (5.25,-1.5) circle (.25cm);
\draw (5.25,-1.5) node {\scriptsize $K_1$};
\draw (6,-1.5) node {\scriptsize $\times$};
\draw (6.425,-1.5) node {\scriptsize $\times$};
\draw (6.85,-1.5) node {\scriptsize $\times$};
\draw (7.275,-1.5) node {\scriptsize $\times$};
\draw (7.7,-1.5) node {\scriptsize $\times$};
\draw [fill,white] (8.35,-1.5) circle (.25cm);
\draw (8.35,-1.5) node {\scriptsize $J_{1,n}$};
\draw (9.5,-1.5) node {$\R$};

\draw (-1.5,-1.2) node {\scriptsize $\underline{Y_n}$};
\draw (-1.5,-.75) node {\scriptsize $\rotatebox{90}{=}$};
\draw (-.5,-1.2) node {\scriptsize $<$};
\draw (1,-1.2) node {\scriptsize $\overline{Y_n}$};
\draw (1.25,-1.2) node {\scriptsize $<$};
\draw (1.5,.-1.4) --++ (0,-.2);
\draw (1.5,-1.2) node {\scriptsize $\chi$};
\draw (2.15,-1.2) node {\scriptsize $<$};
\draw (2.8,-1.2) node {\scriptsize $\underline{X_n}$};
\draw (3.65,-1.2) node {\scriptsize $<$};
\draw (4.5,-1.2) node {\scriptsize $a_2$};
\draw (5.25,-1.2) node {\scriptsize $<$};
\draw (6,-1.2) node {\scriptsize $a_1$};
\draw (6.85,-1.2) node {\scriptsize $<$};
\draw (7.7,-1.2) node {\scriptsize $\overline{X_n}$};

\end{tikzpicture}
\caption{Top: An example support for possibilities (a,b). In words,
$S_1$ is the union of two intervals,
$S_1 = [\underline{S_1}, a_2] \cup [a_1,\overline{S_1}]$ with
$\overline{S_1} > a_1 > a_2 > \underline{S_1}$.
Also, $S_3$ is a single interval, $S_3 = [\underline{S_3}, \overline{S_3}]$
with $\overline{S_3} > \underline{S_3}$. Moreover,
$\underline{S_1} > \chi = \overline{S_3}$,
and so $J_3 = (\chi, \underline{S_1})$ and $J_4 = \emptyset$.
Finally, $K = K_1 = (a_2,a_1)$.
Bottom: Examples of the sets $X_n$ and $Y_n$ defined in equation
(\ref{eqthmroots}). Elements of $X_n$ and $Y_n$ are denoted by
$\times$.}
\label{figthmroots}
\end{figure}
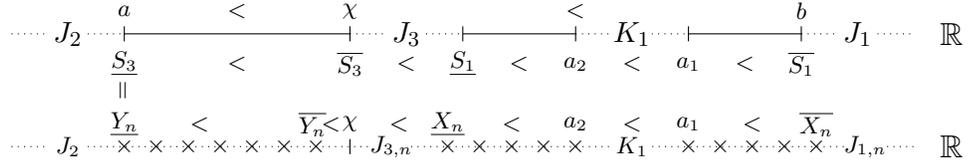

Next, we will show, for all $n\ge1$:
\begin{itemize}
\item[(iii)]
$g_n$ has $m+l$ roots in $\C \setminus (X_n \cup \{\chi\} \cup Y_n\})$
for possibility (a), and at least $m+l-1$ roots in
$\C \setminus (X_n \cup \{\chi\} \cup Y_n\})$ for possibility (b).
\item[(iv)]
$g_n$ has an odd number of roots in any open interval bounded by any two
consecutive elements of $X_n$ for possibilities (a,b). Similarly for
any two consecutive elements of $Y_n$.
\item[(v)]
$g_n$ has an even number of roots in each of
$\{\C \setminus \R, J_{1,n}, J_2, J_{3,n}\}$ for possibility (a), where
$J_{1,n} := (\overline{X_n}, +\infty)$ and $J_{3,n} := (\chi, \underline{X_n})$
(see figure \ref{figthmroots}). $g_n$ has an even number of roots in each of
$\{\C \setminus \R, J_{3,n}\}$ for possibility (b).
\end{itemize}
We will then use parts (iii,iv,v) to show:
\begin{itemize}
\item[(vi)]
For possibility (a) and all $n\ge1$, $g_n$ has either $0$ or $2$ roots in each of
$\{\C \setminus \R, J_{1,n}, J_2, J_{3,n}\}$, and either $1$ or $3$ roots
in $K_1$. Moreover, when $g_n$ has $2$ roots in some fixed
$I \in \{\C \setminus \R, J_{1,n}, J_2, J_{3,n}\}$, then $g_n$ has $0$
roots in each of $\{\C \setminus \R, J_{1,n}, J_2, J_{3,n}\} \setminus \{I\}$,
and $1$ root in $K_1$. Finally, when $g_n$ has $3$ roots in $K_1$, then
$g_n$ has $0$ roots in each of $\{\C \setminus \R, J_{1,n}, J_2, J_{3,n}\}$. 
\item[(vii)]
For possibility (b) and all $n\ge1$, $g_n$ has $0$ roots in each of
$\{\C \setminus \R, J_{3,n}\}$, $1$ root in $K_1$, and at most
$1$ root in each of $\{J_{1,n}, J_2\}$. Moreover, when $g_n$ has
$1$ root in some fixed $I \in \{J_{1,n}, J_2\}$, then $g_n$ has $0$ roots in
$\{J_{1,n}, J_2\} \setminus \{I\}$.
\end{itemize}
Finally, we will use parts (i,ii,vi,vii) to show:
\begin{itemize}
\item[(viii)]
For (a), $f_{(\chi,\eta)}'$ has at most $2$ roots in each of
$\{\C \setminus \R, J_1, J_2, J_3\}$, and at most $3$ roots in $K_1$.
Moreover, when $f_{(\chi,\eta)}'$ has either $1$ or $2$
roots in some fixed $I \in \{\C \setminus \R, J_1, J_2, J_3\}$, then
$f_{(\chi,\eta)}'$ has $0$ roots in each of
$\{\C \setminus \R, J_1, J_2, J_3\} \setminus \{I\}$, and at most $1$
root in $K_1$. Finally, when $f_{(\chi,\eta)}'$ has either $2$ or $3$
roots in $K_1$, then $f_{(\chi,\eta)}'$ has $0$ roots in each of
$\{\C \setminus \R, J_1, J_2, J_3\}$. 
\item[(ix)]
For (b), $f_{(\chi,\eta)}'$ has $0$ roots in each of $\{\C \setminus \R, J_3\}$,
at most $1$ root in each of $\{J_1, J_2, K_1\}$. Moreover, when $f_{(\chi,\eta)}'$ has $1$
root in some fixed $I \in \{J_1, J_2\}$, then $f_{(\chi,\eta)}'$
has $0$ roots in each of $\{J_1, J_2\} \setminus \{I\}$.
\end{itemize}
Parts (viii,ix) prove the required results for possibilities (a,b) when
the supports are as given on the top of figure \ref{figthmroots}.

Consider (iii). Recall that the sets $\{X_n, \{\chi\}, Y_n\}$ are mutually disjoint,
$X_n$ consists of $m\ge4$ distinct elements, and $Y_n$ consists of $l\ge2$ elements.
Define the following polynomial:
\begin{equation*}
p_n (w) 
:= \frac1n \sum_{x \in X_n \cup Y_n}
\bigg( \prod_{y \in X_n \cup Y_n, y \neq x} (w-y) \bigg) (w-\chi)
- (\tfrac{m+l}n-\eta) \bigg( \prod_{y \in X_n \cup Y_n} (w-y) \bigg),
\end{equation*}
for all $w \in \C$.
Recall that $\eta>0$ for possibility (a), and $\eta=0$ for possibility
(b). Therefore $p_n$ has degree $m+l$ for (a), and degree at least $m+l-1$
for (b). Next note that $p_n$ has $0$ roots in $X_n \cup \{\chi\} \cup Y_n$,
as can be seen by substitution. Also, equation (\ref{eqthmroots}) implies
that the roots of $p_n$ and $g_n$ in $\C \setminus (X_n \cup \{\chi\} \cup Y_n)$
coincide, up to multiplicities. This proves (iii).

Consider (iv). Let $x$ and $y$ denote any two consecutive elements of $X_n$,
or any two consecutive elements of $Y_n$, with $y>x$. Note, equation
(\ref{eqthmroots}) implies that $g_n |_{(x,y)}$ is real-valued and continuous,
and:
\begin{equation*}
\lim_{w \in \R, w \downarrow x} g_n(w) = +\infty
\hspace{.5cm} \text{and} \hspace{.5cm}
\lim_{w \in \R, w \uparrow y} g_n(w) = -\infty.
\end{equation*}
Therefore $g_n$ has an odd number of roots in $(x,y)$. This proves (iv).

Consider (v). First note, equation (\ref{eqthmroots}) implies that
non-real roots of $g_n$ occur in complex conjugate pairs. Therefore $g_n$ has
an even number of roots in $\C \setminus \R$. Next note, equation (\ref{eqthmroots})
implies that $g_n |_{(\chi, \underline{X_n})}$ is real-valued and continuous, and:
\begin{equation*}
\lim_{w \in \R, w \downarrow \chi} g_n(w) = - \infty
\hspace{.5cm} \text{and} \hspace{.5cm}
\lim_{w \in \R, w \uparrow \underline{X_n}} g_n(w) = -\infty.
\end{equation*}
Therefore $g_n$ has an even number of roots in $J_{3,n} = (\chi, \underline{X_n})$.
Finally note, equation (\ref{eqthmroots}) implies that
$g_n |_{(\overline{X_n},+\infty)}$ is real-valued and continuous, and:
\begin{equation*}
\lim_{w \in \R, w \downarrow \overline{X_n}} g_n(w) = +\infty
\hspace{.5cm} \text{and} \hspace{.5cm}
\lim_{w \in \R, w \uparrow +\infty} w g_n(w) = \eta.
\end{equation*}
Therefore, since $\eta>0$ for possibility (a), $g_n$ has an even number
of roots in $J_{1,n} = (\overline{X_n},+\infty)$. Similarly, for (a),
$g_n$ has an even number of roots in $J_2 = (-\infty,\underline{X_n})$.

Consider (vi). Note, part (iv) and figure \ref{figthmroots} imply that
$g_n$ has at least $m-1$ roots in $[\underline{X_n}, \overline{X_n}]$.
More specifically, recalling that $\{a_2,a_1\} \subset X_n$,
$g_n$ has at least $m-2$ roots in  $[\underline{X_n}, a_2] \cup [a_1, \overline{X_n}]$,
and at least $1$ root in $(a_2,a_1) = K_1$. Similarly, $g_n$ has
at least $l-1$ roots in $[\underline{Y_n}, \overline{Y_n}] = [\underline{S_3}, \overline{Y_n}]$.
Part (iii) and figure \ref{figthmroots} thus imply that $g_n$ has at most
$2$ roots in $(\C \setminus \R) \cup (J_{1,n} \cup J_2 \cup J_{3,n})$, and at
most $3$ roots in $(\C \setminus \R) \cup (J_{1,n} \cup J_2 \cup J_{3,n} \cup K_1)$.
Part (vi) then follows from parts (iv,v). Part (vii) can be shown similarly.

Consider (viii). First suppose that $z \in \C \setminus \R$ is a
root of $f_{(\chi,\eta)}'$ of multiplicity $k\ge1$. Fix $\e>0$ such
that $B(z,\e) \subset \C \setminus \R$, and $z$ is the unique root in $B(z,\e)$.
Note, part (i) implies that $\overline{z}$ is also a root of multiplicity $k$,
and $\overline{z}$ is the unique root in $B(\overline{z},\e)$. 
Then, for all $n$ sufficiently large, part (ii) implies that $g_n$ has
$k$ roots in both $B(z,\e)$ and $B(\overline{z},\e)$. Thus, since
$B(z,\e)$ and $B(\overline{z},\e)$ are disjoint subsets of $\C \setminus \R$,
$g_n$ has at least $2k\ge2$ roots in $\C \setminus \R$.
Finally recall, part (vi) implies that $g_n$ has either $0$ or
$2$ roots in $\C \setminus \R$. Therefore $k=1$, and so $z$ and $\overline{z}$
are roots of $f_{(\chi,\eta)}'$ of multiplicity $1$.

Next suppose that $z,w \in \C \setminus \R$ are roots of $f_{(\chi,\eta)}'$ of
multiplicity $k=1$ and $l \in \{0,1\}$ respectively ($l=0$ means $w$ is not a root),
and $w \not\in \{z,\overline{z}\}$. Fix $\e>0$ such that
$\{B(z,\e), B(\overline{z},\e), B(w,\e), B(\overline{w},\e)\}$ are
disjoint subsets of $\C \setminus \R$, $z$ is the unique root in $B(z,\e)$,
and $w$ is the unique root in $B(w,\e)$. Then we can proceed similarly to
above to show, for all $n$ sufficiently large, that $g_n$ has $k$ roots in
each of $\{B(z,\e), B(\overline{z},\e)\}$, and $l$ roots in each of
$\{B(w,\e), B(\overline{w},\e)\}$. Thus $g_n$ has at least $2k + 2l$ roots
in $\C \setminus \R$. Thus, since $k=1$, part (vi) implies $l=0$.
Combined, the above show, when $z \in \C \setminus \R$ is a root of
$f_{(\chi,\eta)}'$, that $z$ is a root of $f_{(\chi,\eta)}'$ of multiplicity $1$,
$\overline{z}$ is a root of $f_{(\chi,\eta)}'$ of multiplicity $1$, and
$f_{(\chi,\eta)}'$ has $0$ roots in $(\C \setminus \R) \setminus \{z,\overline{z}\}$.
Thus $f_{(\chi,\eta)}'$ has either $0$ or $2$ roots in $\C \setminus \R$.

Next suppose and $z \in J_1$ is a root of multiplicity $k\ge1$. Fix $\e>0$
such that $B(z,\e) \subset (\C \setminus \R) \cup J_1$, and $z$ is the
unique root in $B(z,\e)$. Then, for all $n$ sufficiently large, part (ii)
implies that $g_n$ has $k$ roots in $B(z,\e) \subset (\C \setminus \R) \cup J_1$.
Recall, $J_1 = (\overline{S_1},+\infty)$ and $J_{1,n} = (\overline{X_n},+\infty)$
and $\overline{X_n} \to \overline{S_1}$ as $n \to \infty$. Therefore,
for all $n$ sufficiently large, $B(z,\e) \subset (\C \setminus \R) \cup J_{1,n}$,
and so $g_n$ has at least $k\ge1$ roots in
$(\C \setminus \R) \cup J_{1,n}$. Finally recall, part (vi) implies that
$g_n$ has either $0$ or $2$ roots in $(\C \setminus \R) \cup J_{1,n}$.
Therefore $k=1$ or $k=2$, and so $z$ is a root of $f_{(\chi,\eta)}'$ of
multiplicity at most $2$.

Next suppose that $z,w \in J_1$ are roots of $f_{(\chi,\eta)}'$ of multiplicity
$k \in \{1,2\}$ and $l \in \{0,1,2\}$ respectively, and $w \neq z$. Fix $\e>0$ 
such that $B(z,\e)$ and $B(w,\e)$ are disjoint subsets of
$(\C \setminus \R) \cup J_1$, $z$ is the unique root in $B(z,\e)$,
and $w$ is the unique root in $B(w,\e)$. Then we can proceed similarly to
above to show, for all $n$ sufficiently large, that $g_n$ has $k$ roots in
$B(z,\e) \subset (\C \setminus \R) \cup J_{1,n}$, and $l$ roots in $B(w,\e)
\subset (\C \setminus \R) \cup J_{1,n}$. Thus $g_n$ has at least $k+l$ roots
in $(\C \setminus \R) \cup J_{1,n}$. Therefore, part (vi) implies 
that $l=0$ when $k=2$, and $l \in \{0,1\}$ when $k=1$. This implies, when
$z \in J_1$ is a root of $f_{(\chi,\eta)}'$ of multiplicity $2$, that
$f_{(\chi,\eta)}'$ has $0$ roots in $J_1 \setminus \{z\}$. Moreover,
when $z \in J_1$ is a root of $f_{(\chi,\eta)}'$ of multiplicity $1$,
$f_{(\chi,\eta)}'$ has a root of multiplicity at most $1$ in $J_1 \setminus \{z\}$.
Similarly it can be shown, when $z,w \in J_1$ are distinct roots of
$f_{(\chi,\eta)}'$ of multiplicity $1$, that
$f_{(\chi,\eta)}'$ has $0$ roots in $J_1 \setminus \{z,w\}$.
Therefore $f_{(\chi,\eta)}'$ has at most $2$ roots in $J_1$.

Next suppose that $z \in J_1$ is a root of multiplicity $k \in \{1,2\}$,
and $w \in J_2$ is a root of multiplicity $l\ge0$. Fix $\e>0$ such that
$B(z,\e) \subset (\C \setminus \R) \cup J_1$,
$B(w,\e) \subset (\C \setminus \R) \cup J_2$, $z$ is the unique root in $B(z,\e)$,
and $w$ is the unique root in $B(w,\e)$. Then we can proceed similarly to
above to show, for all $n$ sufficiently large, that $g_n$ has $k \in \{1,2\}$ roots in
$B(z,\e) \subset (\C \setminus \R) \cup J_{1,n}$, and $l\ge0$ roots in
$B(w,\e) \subset (\C \setminus \R) \cup J_2$. Note, since non-real roots of
$g_n$ occur in complex conjugate pairs, one of the following
must be satisfied for the roots in $B(z,\e)$:
\begin{itemize}
\item
$k \in \{1,2\}$, $g_n$ has $0$ roots in $B(z,\e) \setminus (z-\e,z+\e) \subset \C \setminus \R$,
and either $1$ or $2$ roots in $(z-\e,z+\e) \subset J_{1,n}$.
\item
$k=2$, $g_n$ has $2$ roots in $B(z,\e) \setminus (z-\e,z+\e) \subset \C \setminus \R$,
and $0$ roots in $(z-\e,z+\e) \subset J_{1,n}$.
\end{itemize}
In the first case, part (vi) implies that $g_n$ has $2$ roots in $J_{1,n}$,
and $0$ roots in $\C \setminus \R$ and $J_2$. Therefore, since
$B(w,\e) \subset (\C \setminus \R) \cup J_2$, $g_n$ has $0$ roots in $B(w,\e)$.
In the second case, part (vi) implies that $g_n$ has $2$ roots in
$B(z,\e) \setminus (z-\e,z+\e) \subset \C \setminus \R$,
$0$ roots in $(\C \setminus \R) \setminus B(z,\e)$, and $0$ roots in $J_2$. Therefore,
since $B(w,\e) \subset (\C \setminus \R) \cup J_2$, and since $B(z,\e)$ and $B(w,\e)$
are disjoint, $g_n$ has $0$ roots in $B(w,\e)$. In both cases, this gives $l=0$.
Therefore, when $z \in J_1$ is a root of $f_{(\chi,\eta)}'$ of multiplicity
$1$ or $2$, $f_{(\chi,\eta)}'$ has $0$ roots in $J_2$.

We finally state that the rest of part (viii), and part (ix), can be shown
using similar arguments.
\end{proof}

Recall the definitions of $\LL$ and $\EE = \EE^+ \cup \EE^- \cup \EE_0 \cup \EE_1$
given in definitions \ref{defLiq} and \ref{defEdge}. We end this section by using
theorem \ref{thmf'} to refine these definitions:
\begin{cor}
\label{corf'}
We have:
\begin{enumerate}
\item
Possibility (a) of theorem \ref{thmf'} is satisfied whenever
$(\chi,\eta) \in \LL \cup \EE^+ \cup \EE^- \cup \OO$, and so $b > \chi > a$,
$1 > \eta > 0$, $1-\eta > \mu[\{\chi\}]$, and $S_1 \neq \emptyset$,
$S_2 = \{\chi\}$, $S_3 \neq \emptyset$. When $(\chi,\eta) \in \LL$,
$f_{(\chi,\eta)}'$ has a unique root in $\mathbb{H}$, and this root is of multiplicity
$1$. When $(\chi,\eta) \in \EE^+$, $f_{(\chi,\eta)}'$ has a unique repeated
root in $(\chi,+\infty) \setminus \supp(\mu)$, and this is of
multiplicity either $2$ or $3$. When $(\chi,\eta) \in \EE^-$, $f_{(\chi,\eta)}'$
has a unique repeated root in $(-\infty,\chi) \setminus \supp(\mu)$,
and this is of multiplicity either $2$ or $3$. When $(\chi,\eta) \in \OO$,
$f_{(\chi,\eta)}'$ has a root of multiplicity $1$ in $(b,+\infty)$, and
has at most $2$ roots in $(b,+\infty)$.
\item
$\chi \in \R \setminus \supp(\mu)$ and $\eta = 1$ when $(\chi,\eta) \in \EE_0$.
Moreover, possibility (d) of theorem \ref{thmf'} is satisfied, and so
$b > \chi > a$, $\eta = 1$, and $S_1 \neq \emptyset$, $S_2 = \emptyset$,
$S_3 \neq \emptyset$. Finally, $f_{(\chi,\eta)}'$ has a root of multiplicity $1$ at $\chi$.
\item
$\chi \in \supp(\mu)$, $1 > \mu[\{\chi\}] > 0$, and $\eta = 1 - \mu[\{\chi\}]$
when $(\chi,\eta) \in \EE_1$. Moreover, one of possibilities (d,h,l) is
satisfied. For (d), $b > \chi > a$, $S_1 \neq \emptyset$, $S_2 = \emptyset$,
$S_3 \neq \emptyset$, and $f_{(\chi,\eta)}'$ has either $0$ or $1$ root at $\chi$.
For (h), $\chi = a$, $S_1 \neq \emptyset$, $S_2 = S_3 = \emptyset$,
and $f_{(\chi,\eta)}'$ has $0$ roots at $\chi$.  For (l), $\chi = b$, $S_1 = S_2 = \emptyset$,
$S_3 \neq \emptyset$, and $f_{(\chi,\eta)}'$ has $0$ roots at $\chi$.
\item
$\{\LL, \EE^+, \EE^-, \EE_0, \EE_1, \OO\}$ is pairwise disjoint.
\end{enumerate}
\end{cor}

\begin{proof}
Consider (1) when $(\chi,\eta) \in \LL$. First note, definition \ref{defLiq}
implies that $f_{(\chi,\eta)}'$ has roots in $\C \setminus \R$. Next note that
this can only happen when possibility (a) of theorem \ref{thmf'} is satisfied.
Finally note, possibility (a) of theorem \ref{thmf'} implies that $f_{(\chi,\eta)}'$ has at
most $2$ roots in $\C \setminus \R$. Thus, since non-real
roots of $f_{(\chi,\eta)}'$ occur in complex conjugate pairs, $f_{(\chi,\eta)}'$ has exactly
$1$ roots in $\mathbb{H}$, and this is of multiplicity $1$. This
proves (1) when $(\chi,\eta) \in \LL$.

Consider (1) when $(\chi,\eta) \in \EE^+$. First note, definition \ref{defEdge}
implies that $f_{(\chi,\eta)}'$ has a repeated root in $(\chi,+\infty) \setminus \supp(\mu)$.
Next note that this can only happen when possibility (a) of theorem \ref{thmf'}
is satisfied. Finally note, possibility (a) of theorem \ref{thmf'} implies that
this root has multiplicity either $2$ or $3$. This proves (1) when
$(\chi,\eta) \in \EE^+$. We can similarly prove (1) when $(\chi,\eta) \in \EE^-$.

Consider (1) when $(\chi,\eta) \in \OO$. First note that definition \ref{defLowRig}
implies that $\chi<b$, $\eta >0$, and $f_{(\chi,\eta)}'$ has a root of multiplicity $1$ in
$J_1 = (b,+\infty)$. Next note that this can only happen when possibility (a)
of theorem \ref{thmf'} is satisfied. Finally note that possibility (a) of theorem
\ref{thmf'} implies that $f_{(\chi,\eta)}'$ has at most $2$ roots in $J_1$.
This proves (1) when $(\chi,\eta) \in \OO$.

Consider (2). Recall that $(\chi,\eta) \in \EE_0$. First note that
equation (\ref{eqR2}) and definition \ref{defEdge} imply that
$\chi \in \R \setminus \supp(\mu)$ (and so $\mu[\{\chi\}] = 0$),
$C(\chi) = 0$, and $\eta = 1$. Next note that since $\eta = 1$ and
$\mu[\{\chi\}] = 0$, equations (\ref{eqf'0}, \ref{eqCauTrans})
give $f_{(\chi,\eta)}'(w) = C(w)$ for all $w \in \C \setminus \supp(\mu)$.
Therefore $f_{(\chi,\eta)}'(\chi) = C(\chi) = 0$. Also, since
$1 - \eta = \mu[\{\chi\}] (=0)$,
one of possibilities (d,h,l) of theorem \ref{thmf'} is satisfied.
Moreover, since $C(\chi) = 0$, equation (\ref{eqCauTrans}) trivially
implies that $\chi \neq a$ and $\chi \neq b$. Therefore possibility
(d) must be satisfied. Finally note, possibility (d) of theorem \ref{thmf'}
implies that $\chi$ is a root of $f_{(\chi,\eta)}'$ of multiplicity $1$.

Consider (3). Recall that $(\chi,\eta) \in \EE_1$. First note that
equation (\ref{eqR2}) and definition \ref{defEdge} imply that
$\chi \in \supp(\mu)$, $\mu[\{\chi\}] > 0$, and $\eta = 1 - \mu[\{\chi\}]$.
Thus, since $1 - \eta = \mu[\{\chi\}]$,
one of possibilities (d,h,l) of theorem \ref{thmf'} is satisfied.
For possibility (d), note that theorem \ref{thmf'} implies that
$b > \chi > a$, $S_1 \neq \emptyset$, $S_2 = \emptyset$,
$S_3 \neq \emptyset$, and $f_{(\chi,\eta)}'$ has either $0$ or $1$ root at $\chi$.
Similarly, theorem \ref{thmf'} gives the required results for those possibilities (h,l).
This proves (3).

Consider (4). Suppose first that $(\chi,\eta) \in \LL$.
Part (1) of this result thus implies that possibility (a)
of theorem \ref{thmf'} is satisfied, and that $f_{(\chi,\eta)}'$ has
a root in $\C \setminus \R$. Possibility (a) of theorem \ref{thmf'}
further implies that $f_{(\chi,\eta)}'$ has no real-valued repeated roots, and so
$(\chi,\eta) \not\in \EE^+ \cup \EE^-$ (see definition \ref{defEdge}).
Moreover, possibility (a) implies that $f_{(\chi,\eta)}'$ has no roots
in $J_1 = (b,+\infty)$, and so $(\chi,\eta) \not\in \OO$ (see definition \ref{defLowRig}).
Finally, none of possibilities (d,h,l) are satisfied, and so
parts (2) and (3) of this lemma imply that
$(\chi,\eta) \not\in \EE_0 \cup \EE_1$.

Next suppose that $(\chi,\eta) \in \EE^+$. Part (1) of this
result thus implies that possibility (a) of theorem \ref{thmf'}
is satisfied, and that $f_{(\chi,\eta)}'$ has a unique repeated root in
$(\chi,+\infty) \setminus \supp(\mu)$. Possibility (a) of theorem \ref{thmf'}
further implies that $f_{(\chi,\eta)}'$ has no repeated roots in $(-\infty,\chi) \setminus \supp(\mu)$,
and so $(\chi,\eta) \not\in \EE^-$ (see definition \ref{defEdge}).
Moreover, possibility (a) implies that $f_{(\chi,\eta)}'$ has no roots of
multiplicity $1$ in $J_1 = (b,+\infty)$, and so $(\chi,\eta) \not\in \OO$
(see definition \ref{defLowRig}). Finally, none of possibilities
(d,h,l) are satisfied, and so parts (2) and (3) of this lemma imply that
$(\chi,\eta) \not\in \EE_0 \cup \EE_1$.

Next suppose that $(\chi,\eta) \in \EE^-$. Then, similar arguments to those
used above show that $(\chi,\eta) \not\in \OO \cup \EE_0 \cup \EE_1$.
Next suppose that $(\chi,\eta) \in \OO$. Then, similar arguments to those
used above show that $(\chi,\eta) \not\in \EE_0 \cup \EE_1$.
Finally suppose that $(\chi,\eta) \in \EE_0$. Then $\eta = 1$,
and definition \ref{defEdge} trivially imples that
$(\chi,\eta) \not\in \EE_1$. This proves (4).
\end{proof}

\section{An application to a  problem from Quantum Information Theory}\label{sec:application}

Let us consider the following problem. We fix a parameter $t\in (0,1)$ and an integer $k\ge 1$, and
take a sequence 
$V_n$ of random subspaces of $\C^k\otimes \C^n$ of dimension $d=d_n\sim tkn$.
Here, random means taken uniformly according to the uniform measure on the Grassmann manifold.
For a given $x\in \C^k\otimes \C^n$, we recall that its \emph{singular value decomposition} is
$$x=\sum_i \sqrt{\lambda_i(x)} e_i(x)\otimes f_i(x),$$ where $\lambda_1(x)\ge\lambda_2(x)\ge\ldots \ge 0$, 
and both $(e_i)$ and $(f_i)$ are families of orthonormal vectors.
$\lambda_i(x)$ are always uniquely defined. As for $(e_i(x))$ and $(f_i(x))$ they are generically defined up to a phase
(here, generically means that 
this statement holds true if all $\lambda_i(x)$ are distinct, and this is actually a necessary and sufficient condition). 

It follows from Pythagoras' theorem that $\sum \lambda_i(x)= ||x||_2^2$.
We are interested in the subset $K_{k,t,n}$ of $\R^k$ of all possible singular values $x$ for
$x\in V_n$ of norm $1$. 
The set $K_{k,t,n}$ is actually random, and it is a subset of the
probability simplex $\Delta_k= \{\lambda_1,\ldots , \lambda_k, \lambda_i\ge 0, \sum \lambda_i=1\}$ .
As per our definition of singular values, this set should consist of non-increasing eigenvalues, but
for convenience we make an abuse of language we consider instead the symmetrized version of this set, i.e. any 
permutation of coordinates leaves the set $K_{k,t,n}$ invariant.

It was proved in \cite{BeCoNe12} (Theorem 1.2) that this set actually converges in the Hausdorff distance
to a set $K_{k,t}$ defined as $K_{k,t}=\{(a_1,\ldots , b_k)\in \Delta_k, \forall 
(a_1,\ldots , 1_k)\in \Delta_k, \sum a_ib_i\le ||a||_t\}$, where
$||a||_t=||(a_1,\ldots , a_k)||_t$ is the free compressed $t$-norm, as introduced
in the first section, cf Equation \eqref{soft-def-free-compressed} -- see also definition \ref{defEdge}.
Recall that the \emph{Hausdorff distance} between two compact subsets $K,S$ of a complete metric space 
is the infimum over all $\varepsilon >0$ such that $K\subset B(S,\varepsilon )$ and 
$S\subset B(K,\varepsilon )$, where $B(S,\varepsilon )$ is the ball of `center' $S$ and radius $\varepsilon$, i.e.
the collection of all elements that are $\varepsilon$-close to $S$. 
We also proved that it is true for the boundary of sets viewed as subsets of the affine space of real $k$-tuples that add up to $1$
in the sense that the Hausdorff distance between $\partial K_{k,t,n}$ and $\partial K_{k,t}$ converges to zero almost surely. 
Thanks to the main result, we are able to upgrade the results mentioned earlier in this section as follows:

\begin{thm}
There exist constants $C$ and a polynomial function $h(\varepsilon )$ such that for any $\varepsilon\in (0,1)$,
$P(d(K_{k,t,n},K_{k,t})\ge \varepsilon )\le Ce^{-nh(\varepsilon ) }$.
\end{thm}

We do not give a complete proof of this result, as it is essentially contained in \cite{BeCoNe12},
however let us try to give a sense of the important ideas. 
It follows from linear algebra considerations that an element of $V_n$ will satisfy 
$\sum_{i=1}^k \lambda_i(x)a_i |<e_i(x),h_i>|^2\ge \alpha$ if and only if, 
calling $p$ the orthogonal projection onto $V_n$,
$p(\sum a_i e_ie_i^*) p$,  has operator norm at least $\alpha$.
Our main theorem \ref{thmdecay} allows us to estimate this quantity very precisely.

In order to prove the result, we need to be able to obtain such an estimate for all $k$-tuples $(a_i)$, $(h_i)$ simultaneously, 
where $(a_i)\in \R_+^k$, and $(h_i)$ is a family of orthonormal vectors. 
Thanks to this estimate, we are able to estimate 
$$P(|<(a_1,\ldots ,a_k),K_{k,t,n}>-<(a_1,\ldots ,a_k),K_{k,t}>|\ge \varepsilon ),$$ 
and find it to be less than $Ce^{-nh(\varepsilon ) }$.

In this problem, $k$ is fixed, so we can take a finite $\eta$-net of $(a_i)\in \R_+^k,(h_i)$ for an appropriate metric, on the product
of real eigenvalues and eigenvectors up to a phase -- which, for this purpose, can be thought of as the convex set of trace one
semidefinite selfadjoint matrices. By passing, let us note that this set is also known as the set of density matrices in 
QIT. 
Thanks to this net argument, and by a continuity argument, we can then take the sup over all probability vectors 
$(a_1,\ldots , a_k)$ and estimate again
$$P\left( \sup_{(a_1,\ldots ,a_k)}|
< ( a_1,\ldots ,a_k),K_{k,t,n}>-<(a_1,\ldots ,a_k),K_{k,t}>|\ge \varepsilon \right),$$ 
and bound it alike by $Ce^{-nh(\varepsilon ) }$, with constant worsened to take into account $\eta$ and a union bound reasoning. 
This gives the desired result. Note that although we show the existence of actual constants and of an exponential speed of convergence,
making the constants $c,h$ is probably a difficult task, first because it requires to make every constant of 
subsection \ref{secAtoms} explicit, and secondly because it asks to understand in detail the procedure of optimizing the sup
over all probability vectors. Partial work in this direction was completed  \cite{BeCoNe16},
though the problem under consideration was simpler and yet required considerably involved developments in free probability theory.

In particular, given a continuous function, this result allows us to give estimates for 
$$P(|\min \{f(x),x\in K_{k,t,n}\} - \min \{f(x),x\in K_{k,t}\}| >\varepsilon)$$
and we obtain similar upper bounds, of type $C\exp (-nh(\varepsilon ))$.
Thanks to the results of \cite{BeCoNe16}, it was known that the minimum output entropy for generic quantum channels
can be generically violated if and only if the parameter $k\ge 183$, however, no estimate was available for the 
required dimension $n$ of the input space, nor was any technique available to attack this problem. 
This paper contributes to solving this problem in the sense that combining the above result in the case 
where $f$ is the entropy function $H$, together with the calculations of $\min \{H(x),x\in K_{k,t}\}$ of 
\cite{BeCoNe16} yield a path towards answering this problem.

\end{document}